\documentclass[letterpaper,11pt,reqno]{amsart}

\usepackage[margin=1in]{geometry}

\usepackage{hyperref}
\hypersetup{
     colorlinks   = true,
     citecolor    = black,
     linkcolor    = blue,
     urlcolor     = blue
}

\usepackage{graphicx,amsmath,amssymb,amsthm,paralist,color,tikz-cd}
\usepackage{mathrsfs}
\usepackage{mathtools}
\usepackage{cite}
\usepackage{setspace}
\usepackage[color=blue!30]{todonotes}

\usepackage{amscd}
\usepackage{bbm}

\newtheorem{thm}{Theorem}[section]
\newtheorem{lem}[thm]{Lemma}

\newtheorem{prop}[thm]{Proposition}

\newtheorem{cor}[thm]{Corollary}

\theoremstyle{definition}
\newtheorem{definition}[thm]{Definition}
\newtheorem*{definition-nono}{Definition}
\newtheorem{claim}[thm]{Claim}

\newtheorem{remark}[thm]{Remark}

\newtheorem*{acknowledgement}{Acknowledgements}
\newtheorem*{examples}{Examples}
\newtheorem*{conjecture-nonum}{Conjecture}

\newtheoremstyle{case}{}{}{}{}{}{:}{ }{}
\theoremstyle{case}

\newcommand{\N}{\mathbb{N}}
\newcommand{\Z}{\mathbb{Z}}

\newcommand{\R}{\mathbb{R}}
\newcommand{\C}{\mathbb{C}}
\renewcommand{\H}{\mathbb{H}}

\newcommand{\mc}{\mathcal}

\newcommand{\mf}{\mathfrak}

\newcommand{\mrm}{\mathrm}

\renewcommand{\a}{\alpha}
\renewcommand{\b}{\beta}
\newcommand{\g}{\gamma}
\newcommand{\G}{\Gamma}
\renewcommand{\d}{\delta}
\newcommand{\e}{\varepsilon}
\renewcommand{\l}{\lambda}
\renewcommand{\L}{\Lambda}
\newcommand{\w}{\omega}
\newcommand{\s}{\sigma}
\newcommand{\vp}{\varphi}
\renewcommand{\t}{\tau}
\renewcommand{\th}{\theta}

\renewcommand{\k}{\kappa}

\newcommand{\set}[1]{\left\{#1\right\}}
\renewcommand{\r}{\rightarrow}

\def\multiset#1#2{\ensuremath{\left(\kern-.3em\left(\genfrac{}{}{0pt}{}{#1}{#2}\right)\kern-.3em\right)}}
\newcommand{\norm}[1]{\left\lVert#1\right\rVert}

\newcommand{\Acal}{\mc{A}}
\newcommand{\Bcal}{\mc{B}}
\newcommand{\Ccal}{\mc{C}}
\newcommand{\Dcal}{\mc{D}}
\newcommand{\Ecal}{\mc{E}}

\newcommand{\Gcal}{\mc{G}}

\newcommand{\Kcal}{\mc{K}}
\newcommand{\Lcal}{\mc{L}}
\newcommand{\Mcal}{\mc{M}}
\newcommand{\Ncal}{\mc{N}}
\newcommand{\Ocal}{\mc{O}}
\newcommand{\Pcal}{\mc{P}}
\newcommand{\Qcal}{\mc{Q}}
\newcommand{\Rcal}{\mc{R}}
\newcommand{\Scal}{\mc{S}}

\newcommand{\Vcal}{\mc{V}}
\newcommand{\Wcal}{\mc{W}}

\newcommand{\diam}[1]{\mrm{diam}\left(#1\right)}
\newcommand{\supp}{\mrm{supp}}
\newcommand{\Lie}{\mrm{Lie}}

\numberwithin{equation}{section}

\newcommand{\SL}{\mrm{SL}_2(\R)}

\newcommand{\bms}{\mrm{m}^{\mrm{BMS}}}

\newcommand{\ps}{\mu^{\mrm{PS}}}
\newcommand{\wu}[1]{N^+_{#1}}
\newcommand{\uT}{\mrm{T}^1}

\newcommand{\Nab}{N_{\mrm{ab}}}
\newcommand{\Edisc}{\Ecal_{\mrm{dis}}}

\newcommand{\mxu}{\mu_x^u}

\newcommand{\Ad}{\mrm{Ad}}
\newcommand{\inj}{\mrm{inj}}
\newcommand{\dist}{\mrm{dist}}

\newcommand{\M}{\mathbb{M}}
\newcommand{\K}{\mathbb{K}}
\newcommand{\id}{\mrm{id}}

\newcommand{\F}{\mathscr{F}} 
\renewcommand{\P}{\mathbb{P}}
\renewcommand{\Re}{\mrm{Re}}
\renewcommand{\Im}{\mrm{Im}}

\newcommand{\Bstar}{\Bcal_\star}

\newcommand{\yrho}{y_{\rho}}
\newcommand{\murho}{\mu^u_{y_\rho}}

\newcommand{\murhoi}{\mu^u_{y^i_\rho}}
\newcommand{\yrhoi}{y_\rho^i}
\newcommand{\npls}{\mf{n}^+}
\newcommand{\nonepls}{\npls_\a}
\newcommand{\ntwopls}{\npls_{2\a}}
\newcommand{\nminus}{\mf{n}^-}
\newcommand{\noneminus}{\nminus_\a}
\newcommand{\ntwominus}{\nminus_{2\a}}
\newcommand{\muxu}{\mu_x^u}
\newcommand{\zeps}{Z^{(\epsilon)}}

\usepackage{amsfonts}
\usepackage[sort=use]{glossaries}
\usepackage{glossary-mcols}
\setglossarystyle{mcolindex}

\newglossary[cn-glg]{constants}{cn-gls}{cn-glo}{Constants}

\glsdisablehyper

 \newcommand*{\newconstant}[3][]{%
  \newglossaryentry{#2}{
    name={$#3$},text={#3},description={}}%
}

\newglossaryentry{pj}
{
  name={\ensuremath{p_j}},
  description={}
} 
\newconstant{fvarpi}{f_\varpi}
\newconstant{alpha}{\alpha}
\newconstant{Kj}{K_j}
\newconstant{iotaj}{\iota_j}
\newconstant{pjw}{p_{j,w}}
 \newconstant{F}{F}
\newconstant{Pcalj0}{\Pcal_j^0}
\newconstant{rho0}{\rho_0}
\newconstant{gjw}{g_j^w}
\newconstant{xj}{x_j}
\newconstant{N1j}{N_1^+(j)}
\newconstant{Trho}{T_\rho}
\newconstant{xrhoell}{x_{\rho,\ell}}
\newconstant{phitilde}{\widetilde{\phi}_{\rho,\ell}}
\newconstant{Frho}{\mathscr{F}_{\rho}}
\newconstant{Fstar}{\F_\star}
\newconstant{D}{D}
\newconstant{T0}{T_0}
\newconstant{w}{w}
\newconstant{Irhoj}{I_{\rho,j}}
\newconstant{Wrho}{W_\rho}
\newconstant{Well}{W_\ell}
\newconstant{tauell}{\t_\ell}
\newconstant{Jphiell}{J\Phi_\ell}
\newconstant{phirhoelleta}{\phi_{\rho,\ell,\eta}}
\newconstant{phirhoell}{\phi_{\rho,\ell}}
\newconstant{yrho}{y_\rho}
\newconstant{eta}{\eta}
\newconstant{Jrho}{J_\rho}
\newconstant{kappa}{\k}
\newconstant{Crhoj}{C_{\rho,j,i}}
\newconstant{Srhoj}{S_{\rho,j,i}}
\newconstant{m}{m}
\newconstant{j}{j}
\newconstant{beta}{\b}
\newconstant{Delta}{\Delta}
\newconstant{Delta+}{\Delta_+}
\newconstant{sigma}{\s}
\newconstant{estar}{e^\star_{1,0}}
\newconstant{normstar}{\norm{\cdot}^\star_1}
\newconstant{Psirho}{\Psi_\rho}
\newconstant{tildephirhoelleta}{\tilde{\phi}_{\rho,\ell,\eta}}
\newconstant{chieta}{\chi_\eta}
\newconstant{gamma}{\g}
\newconstant{ggamma}{g^\g}
\newconstant{Fgamma}{F_\g}
\newconstant{J0}{J_0}
\newconstant{Jb}{J_b}
\newconstant{epsilon1}{\e_1}
\newconstant{C}{C}
\newconstant{sigmaa}{\s_a}
\newconstant{B}{B}
\newconstant{Mstar}{\M^\star}
\newconstant{chi}{\chi}
\newconstant{cstar}{c_\star}
\newconstant{Fm}{F_m}
\newconstant{varkappa}{\varkappa}
\newconstant{r}{r}
\newconstant{sigma0}{\s_0}
\newconstant{varrho}{\varrho}
\newconstant{kappa0}{\k_0}
\newconstant{varsigma}{\varsigma}
\newconstant{M}{M}
\newconstant{Rrho}{\Rcal_\rho}
\newconstant{Ri}{\Rcal_i}
\newconstant{Erho}{\Ecal_\rho}
\newconstant{Ei}{\Ecal_i}
\newconstant{sell}{s_{\rho,\ell}}
\newconstant{lambda}{\l}
\newconstant{R}{R}
\newconstant{Rp}{\Rcal_\wp}
\newconstant{H}{H}
\newconstant{T1}{T_1}
\newconstant{A}{A}

\title[Exponential Mixing \& Additive Combinatorics]{Exponential Mixing Via Additive Combinatorics}
\author{Osama Khalil}
\address{Department of Mathematics, Statistics, and Computer Science, University of Illinois Chicago}
\email{okhalil@uic.edu}

\date{}

\AtBeginDocument{%
   \def\MR#1{}
}

\begin{document}

\begin{abstract}
    We prove that the geodesic flow on a geometrically finite locally symmetric space of negative curvature is exponentially mixing with respect to the Bowen-Margulis-Sullivan measure.
    The approach is based on constructing a suitable anisotropic Banach space on which the infinitesimal generator of the flow admits an essential spectral gap.
    A key step in the proof involves estimating certain oscillatory integrals against the Patterson-Sullivan measure.
    For this purpose, we prove a general result of independent interest asserting that the Fourier transform of measures on $\R^d$ that do not concentrate near proper affine hyperplanes enjoy polynomial decay outside of a sparse set of frequencies.
    As an intermediate step, we show that the $L^q$-dimension ($1<q\leq \infty$) of iterated self-convolutions of such measures tend towards that of the ambient space.
    Our analysis also yields that the Laplace transform of the correlation function of smooth observables extends meromorphically to the entire complex plane in the convex cocompact case and to a strip of explicit size beyond the imaginary axis in the case the manifold admits cusps.

\end{abstract}

\maketitle

\section{Introduction}

\subsection{Exponential mixing and Pollicott-Ruelle resonances}
Let $\mc{X}$ be the unit tangent bundle of a quotient of a real, complex, quaternionic, or a Cayley hyperbolic space by a discrete, geometrically finite, non-elementary group of isometries $\G$.
Denote by $g_t$ the geodesic flow on $\mc{X}$ and by $\bms$ the Bowen-Margulis-Sullivan probability measure of maximal entropy for $g_t$.
Let $\d_\G$ be the critical exponent of $\G$.
We refer the reader to Section~\ref{sec:prelims} for definitions. 
The following is the main result of this article in its simplest form.
\begin{thm}\label{thm:intro mixing}
The geodesic flow on $\mc{X}$ is exponentially mixing with respect to $\bms$.
More precisely, there exists $\s_0=\s_0(\mc{X})>0$ such that for all $f\in C_c^3(\mc{X})$, $g\in C_c^2(\mc{X})$ and $t\geq 0$,
\begin{equation*}
    \int_\mc{X} f\circ g_t \cdot g\;d\bms
    = \int_\mc{X} f\;d\bms \int_\mc{X}  g\;d\bms + \norm{f}_{C^3}  O_g\left( e^{-\s_0 t}\right).
\end{equation*}
The implicit constant depends on $g$ through its $C^2$-norm and the injectivity radius of its support.
\end{thm}

The results also hold for functions with unbounded support and controlled growth in the cusp; cf.~Section~\ref{sec:Butterley}.
Theorem~\ref{thm:intro mixing} follows immediately from the following more precise result showing that the correlation function admits a finite resonance expansion.

\begin{thm}\label{thm:intro strip}
There exists $\s>0$ such that the following holds.
There exist finitely many complex numbers $\l_1, \dots, \l_N$ with $-\s<\Re(\l_i)<0$, finite-rank projectors $\Pi_i$, and nilpotent matrices $\Ncal_i$ acting on the range of $\Pi_i$ for each $i$, such that
for all $f\in C_c^3(\mc{X})$ with $\int_\mc{X}f\;d\bms=0$, $g\in C_c^2(\mc{X})$ and $t\geq 0$, we have
\begin{align*}
     \int_\mc{X} f\circ g_t \cdot g\;d\bms =
     \sum_{i=1}^N e^{t\l_i} \int_{\mc{X}} g\cdot e^{t \Ncal_i} \Pi_i(f)\;d\bms
     +\norm{f}_{C^3} O_{g}\left(e^{-\s t}\right).
\end{align*}
The implicit constant depends on $g$ through its $C^2$-norm and the injectivity radius of its support.
\end{thm}

\begin{remark}
    The constant $\s$ in Theorem~\ref{thm:intro strip} depends only on non-concentration parameters of Patterson-Sullivan (PS) measures near proper generalized sub-spheres of the boundary at infinity; cf.~Corollary~\ref{cor:aff non-conc of projections} for details.
    In particular, Theorem~\ref{thm:intro strip} implies that
    $\s$ does not change if we replace $\G$ with a finite index subgroup.
    The interested reader is referred to~\cite{MageeNaud-IHES,MageeNaud-Alon} for recent developments on a closely related problem yielding uniform resonance-free regions for the Laplacian operator on random covers of convex cocompact hyperbolic surfaces.

\end{remark} 

The ``eigenvalues" $\l_i$ above are known as \emph{Pollicott-Ruelle resonances}.
Theorem~\ref{thm:intro mixing} follows from the above result by taking $\s_0$ to be the absolute value of the largest real part of the $\l_i$'s. 
The reader is referred to Section~\ref{sec:Butterley} for a more precise discussion of the Banach spaces on which the operators $\Pi_i$ live.

Given two bounded functions $f$ and $g$ on $\mc{X}$, the associated correlation function is defined by
\begin{align*}
    \rho_{f,g}(t) := \int_\mc{X} f\circ g_t \cdot g\;d\bms, \qquad t\in\R.
\end{align*}
Its (one-sided) Laplace transform is defined for any $z\in \C$ with positive real part $\Re(z)$ as follows:
\begin{align*}
    \hat{\rho}_{f,g}(z) := \int_0^\infty e^{-zt} \rho_{f,g}(t)\;dt.
\end{align*}
Theorem~\ref{thm:intro strip} implies that, for suitably smooth $f$ and $g$, $\hat{\rho}_{f,g}$ admits a meromorphic continuation to the half plane $\Re(z)>-\s$ with the only possible poles occurring at $\set{\l_i}$.

Our analysis also yields the following result.
Let $\d_\G$ denote the critical exponent of $\G$ and define
\begin{align}\label{eq:sigma(Gamma)}
    \s(\G) := \begin{cases} \infty, & \text{if } \G \text{ is convex cocompact}, \\
    \min\set{\d_\G, 2\d_\G-k_{\max}, k_{\min}}, & \text{otherwise},
    \end{cases}
\end{align}
where $k_{\max}$ and $k_{\min}$ denote the maximal and minimal ranks of parabolic fixed points of $\G$ respectively; cf.~Section~\ref{sec:global measure formula} for the definition of the rank of a cusp.

\begin{thm}\label{thm:intro meromorphic}
Let $r\in \N$.
For all $f,g\in C_c^{r+2}(\mc{X})$, $\hat{\rho}_{f,g}$ is analytic in the half plane $\Re(z)>0$ and admits a meromorphic continuation to the half plane:
\begin{equation*}
    \Re(z)> -\min\set{r,\s(\G)/2},
\end{equation*}
with $0$ being the only pole on the imaginary axis.
In particular, when $\G$ is convex cocompact and $f,g\in C_c^\infty(\mc{X})$, $\hat{\rho}_{f,g}$ admits a meromorphic extension to the entire complex plane.

\end{thm}

Theorem~\ref{thm:intro meromorphic} is deduced from an analogous result on the meromorphic continuation of the family of resolvent operators $z\mapsto R(z)$,
\begin{align}\label{eq:resolvent intro}
    R(z):=\int_0^\infty e^{-zt} \Lcal_t \;dt:C_c(\mc{X}) \to C(\mc{X}) ,
\end{align}
defined initially for $z$ with large enough $\Re(z)$,
where $\Lcal_t$ is the transfer operator given by $f\mapsto f\circ g_t$; cf.~Theorem~\ref{thm:resolvent spectrum2} for a precise statement. 
Analogous results regarding resolvents were obtained for Anosov flows in~\cite{GiuliettiLiveraniPollicott} and Axiom A flows in~\cite{DyatlovGuillarmou-long,DyatlovGuillarmou-short} leading to a resolution of a conjecture of Smale on the meromorphic continuation of the Ruelle zeta function; cf.~\cite{Smale}.
We refer the reader to~\cite{GiuliettiLiveraniPollicott} for a discussion of the history of the latter problem.

\subsection{$L^q$-flattening of measures on $\R^d$ under convolution}

The key new ingredient in our proof of Theorem~\ref{thm:intro mixing} is the statement that the conditional measures of the BMS measure along the strong unstable foliation enjoy polynomial Fourier decay outside of a very sparse set of frequencies; cf.~Corollary~\ref{cor:flattening intro}.

The key step in the proof is an $L^q$-flattening result for convolutions of measures on $\R^d$ of independent interest. Roughly speaking, it states that the $L^q$-dimension (Def.~\ref{def:L2 dim}) of a measure $\mu$ improves under iterated self-convolutions unless $\mu$ is concentrated near proper affine hyperplanes in $\R^d$ at many scales.
The proof of this result provided in Section~\ref{sec:flattening} can be read independently of the rest of the article.

We formulate here a special case of our results under the following non-concentration condition and refer the reader to Definition~\ref{def:aff non-conc} for a much weaker condition under which these results hold.

We need some notation before stating the result. Let $\Dcal_k$ denote the dyadic partition of $\R^d$ by translates of the cube $2^{-k}[0,1)^d$ by $2^{-k}\Z^d$.
We recall the notion of $L^q$-dimension of measures.

\begin{definition}\label{def:L2 dim}
    For $q>1$, the \textit{$L^q$-dimension} of a Borel probability measure $\mu$ on $\R^d$, denoted $\dim_q \mu$, is defined to be
    \begin{align*}
        \dim_q\mu : 
        =\liminf_{k\to\infty} \frac{-\log_2 \sum_{P\in \Dcal_k} \mu(P)^q}{(q-1) k}.
    \end{align*}
    The \textit{Frostman exponent} of $\mu$, denoted $\dim_\infty \mu$, is defined to be
    \begin{align*}
    \dim_\infty \mu :=
    \liminf_{k\to \infty} \frac{\log_2 \max_{P\in\Dcal_k} \mu(P)}{-k}.
\end{align*}
\end{definition}

We say that Borel measure $\mu$ on $\R^d$ is \textit{uniformly affinely non-concentrated} if for every $\e>0$, there exists $\d(\e)>0$ so that $\d(\e)\to 0$ as $\e\to 0$ and for all $x\in \mrm{supp}(\mu)$, $0<r\leq 1$, and every affine hyperplane $W<\R^d$, we have
        \begin{align}\label{eq:uniform affine non-conc}
            \mu( W^{(\e r)} \cap B(x,r)) \leq \d(\e) \mu(B(x,r)),
        \end{align}
        where $W^{(r)}$ and $B(x,r)$ denote the $r$-neighborhood of $W$ and the $r$-ball around $x$ respectively.

The following is our main result on flattening under convolution with non-concentrated measures.

\begin{thm}\label{thm:ellq improvement}
    Let $1<q<\infty$ and $\eta>0$ be given.
    Then, there exists $\e=\e(q,\eta)>0$ such that if $\mu$ is any compactly supported Borel probability measure on $\R^d$ which is uniformly affinely non-concentrated, then
    \begin{align*}
        \dim_q (\mu\ast \nu) > \dim_q\nu +\e,
    \end{align*}
    for every compactly supported probability measure $\nu$ on $\R^d$ with $\dim_q \nu \leq d-\eta$.

   In particular, $\dim_\infty\mu^{\ast n}$ converges to $d$ at a rate depending only on the non-concentration parameters of $\mu$,
   and, hence, the same holds for $\dim_q \mu^{\ast n}$ for all $q>1$.
\end{thm}

\begin{remark}
    We refer the reader to Section~\ref{sec:flattening} where a quantitative form of Theorem~\ref{thm:ellq improvement} is obtained under a much weaker \emph{non-uniform} non-concentration condition; cf.~Def.~\ref{def:aff non-conc}.
This quantitative form is necessary for our applications and the weaker hypothesis is essential in the presence of cusps.
   
\end{remark}

The $L^2$-dimension case of Theorem~\ref{thm:ellq improvement} has the following immediate corollary asserting that the Fourier transform of affinely non-concentrated measures enjoys polynomial decay outside of a very sparse set of frequencies. 

\begin{cor}\label{cor:flattening intro}
Let $\mu$ be as in Theorem~\ref{thm:ellq improvement} and denote by $\hat{\mu}$ its Fourier transform.
Then, for every $\e>0$, there is $\t>0$, depending only on the non-concentration parameters of $\mu$, such that for all $T\geq 1$,
\begin{align*}
    \left|\set{\norm{\xi}\leq T: |\hat{\mu}(\xi)|>T^{-\t}}\right| \leq C_{\e,\mu} T^\e,
\end{align*}
where $|\cdot|$ denotes the Lebesgue measure on $\R^d$, and $C_{\e,\mu}\geq 1$ is a constant depending on $\e$, the diameter of the support of $\mu$, and its non-concentration parameters.
\end{cor}

\begin{remark}\label{rem:examples+expanded ball}
\begin{enumerate}
    \item A large class of dynamically defined measures, which includes self-conformal measures, is known to be affinely non-concentrated; cf.~\cite[Proposition 4.7 and Corollary 4.9]{RossiShmerkin} for measures on the real line and the results surveyed in~\cite[Section 1.3]{Dasetal} for measures in higher dimensions under suitable irreducibility hypotheses\footnote{The results referenced in~\cite{Dasetal} require the open set condition, while~\cite{RossiShmerkin} does not.}.
    In particular, Theorem~\ref{thm:ellq improvement} applies to these measures generalizing prior known special cases for certain self-similar measures on $\R$ by different methods; cf.~\cite{FengLau,MosqueraShmerkin}.

    \item In~\cite{BaYu}, it was observed that the proofs of Theorem~\ref{thm:ellq improvement}, and its quantitative form Theorem~\ref{thm:quant ellq improvement}, go through under the following weaker form of~\eqref{eq:uniform affine non-conc} allowing the ball on the right side to have a larger radius:
    \begin{align}\label{eq:expanded ball affine non-conc}
            \mu( W^{(\e r)} \cap B(x,r)) \leq \d(\e) \mu(B(x,cr)),
    \end{align}
    where $c\geq 1$ is a fixed constant. 
    This property holds for instance for certain self-similar measures which do not satisfy~\eqref{eq:uniform affine non-conc}, e.g.~in the absence of separation conditions.

    \item Our proof in fact shows that Theorem~\ref{thm:ellq improvement} holds for \textit{projections} of non-concentrated measures; cf.~Theorem~\ref{thm:quant ellq improvement} and Corollary~\ref{cor:flattening}.
    Beyond the intrinsic interest in the study of projections of fractal measures, this stronger form is essential in our proof of exponential mixing outside the case of real hyperbolic spaces;
    cf.~Section~\ref{sec:outline} for further discussion. 
\end{enumerate}
        
\end{remark}

Corollary~\ref{cor:flattening intro} generalizes the work of Kaufman~\cite{Kaufman} and Tsujii~\cite{Tsujii-selfsimilar} for self-similar measures on $\R$ by different methods.
Theorem~\ref{thm:ellq improvement} was obtained for measures on the real line by Rossi and Shmerkin in~\cite{RossiShmerkin} under the uniform non-concentration hypothesis above. 
Their work builds crucially on a $1$-dimensional inverse theorem due to Shmerkin in~\cite{Shmerkin-Furstenberg} which was the key ingredient in his groundbreaking solution of Furstenberg's intersection conjecture.
Proposition~\ref{prop:discretized flattening} can be regarded as a higher dimensional substitute for Shmerkin's inverse theorem.
A similar higher dimensional inverse theorem for $L^q$-dimension was announced by Shmerkin in his ICM survey~\cite[Section 3.8.3]{Shmerkin-ICM}.

In Section~\ref{sec:friendly}, we show that Corollary~\ref{cor:flattening intro} applies to PS measures when $\mc{X}$ is real hyperbolic (and to certain projections of these measures in the other cases, see discussion in Section~\ref{sec:outline} below).

For convex cocompact hyperbolic surfaces, Bourgain and Dyatlov showed that PS measures in fact have polynomially decaying Fourier transform~\cite{BourgainDyatlov}.
Their methods are different to ours and are based on Bourgain's sum-product estimates.
Their result was extended to convex cocompact Schottky real hyperbolic $3$-manifolds in~\cite{LiNaudPan} by similar methods.
These results imply Corollary~\ref{cor:flattening intro} in these special cases, however Corollary~\ref{cor:flattening intro} also applies to measures whose Fourier transform does not tend to $0$ at infinity (e.g.~the coin tossing measure on the middle $1/3$ Cantor set).
In forthcoming work, we apply our methods to generalize these results to hyperbolic manifolds of any dimension which are not necessarily of Schottky type.

\subsection{Polynomial decay near proper subvarieties}

Theorem~\ref{thm:ellq improvement} has the following important consequence regarding polynomial decay of the PS mass of neighborhoods of certain proper subvarieties of the boundary at infinity, which are saturated along the vertical foliation. 
This result is of independent interest.
Denote by $N^+$ the expanding horospherical group associated to $g_t$ for $t>0$, the orbits of which give rise to the strong unstable foliation.
    Let $N^+_r$ be the $r$-ball around identity in $N^+$ (cf.~Section~\ref{sec:Carnot} for definition of the metric on $N^+$).
    Let $\Nab^+ $ denote the abelianization $N^+/[N^+,N^+]$. 
    Finally, let $\Omega\subseteq \mc{X}$ be the non-wandering set for the geodesic flow; i.e. the closure of the set of its periodic orbits.

\begin{thm}\label{thm:friendly intro}
    Let $x\in \Omega$. Then, there exist $C,\k>0$ such that for all $\e>0$ the following holds.
    Let $\Lcal \subset N^+$ be the preimage of any proper affine subspace of the abelianization $\Nab^+$ and let $\Lcal^{(\e)}$ be its $\e$-neighborhood.
    Then, 
    \begin{align*}
        \muxu \left(\Lcal^{(\e)} \cap N_1^+ \right) \leq C\e^\k \muxu(N_1^+).
    \end{align*}
    The constants $C$ and $\k$ can be chosen to be uniform as $x$ varies in any fixed compact set.
\end{thm}
We refer the reader to Theorem~\ref{thm:flat implies friendly} for a more general version of this result.
Theorem~\ref{thm:friendly intro} was obtained in~\cite[Lemma 3.8]{Dasetal} in the case of real hyperbolic spaces by completely different methods. 
It is worth noting that our proof of exponential mixing only uses Theorem~\ref{thm:friendly intro} in the case when the space is \textit{not} real hyperbolic; cf.~Remark~\ref{rem:role of friendliness} for further discussion.
\subsection{Exponential recurrence from the cusp}

    An important ingredient in our arguments is the following exponential decay result on the measure of the set of orbits with long cusp excursions, which is of independent interest.
    Let the notation be in as in Theorem~\ref{thm:friendly intro}.
  \begin{thm}\label{thm:exp recurrence intro}
  Let $\s(\G)$ be as in~\eqref{eq:sigma(Gamma)} and let $0<\b<\s(\G)/2$ be given. 
For every $\e >0$, there exists a compact set $K\subseteq \Omega$ and $T_0>0$ such that the following holds for all $T>T_0, 0<\th<1$ and $x\in  \Omega$.
Let $\chi_K$ be the indicator function of $K$.
Then,
    \begin{align*}
        \mu_x^u
        \left(n\in N_1^+: \int_0^T \chi_K(g_t nx)\;dt \leq (1- \th) T\right)
        \ll_{\b,x,\e} e^{-(\b\th -\e )T} \mu_x^u(N_1^+).
    \end{align*}
    The implicit constant is uniform as $x$ varies in any fixed compact set.
\end{thm}

The reader is referred to Theorem~\ref{thm:exp recurrence} for a stronger and more precise statement.
Theorem~\ref{thm:exp recurrence intro} implies that the Hausdorff dimension of the set of points in $N_1^+x$ whose forward orbit asymptotically spends all of its time in the cusp is at most $\s(\G)/2$.
This bound is not sharp and can likely be improved using a refinement of our methods.
We hope to return to this problem in future work.

\subsection{Prior results}

In the case $\G$ is convex cocompact, Theorem~\ref{thm:intro mixing} is a special case of the results of~\cite{Stoyanov} which extend the arguments of Dolgopyat~\cite{Dolgopyat} to Axiom A flows under certain assumptions on the regularity of the foliations and the holonomy maps.
The special case of convex cocompact hyperbolic surfaces was treated in earlier work of Naud~\cite{Naud-Cantor}. The extension to frame flows on convex cocompact manifolds was treated in~\cite{SarkarWinter,ChowSarkar}.

In the case of real hyperbolic manifolds with $\d_\G$ strictly greater than half the dimension of the boundary at infinity, Theorem~\ref{thm:intro mixing} was obtained in~\cite{EdwardsOh}, with much more precise and explicit estimates on the size of the essential spectral gap. The methods of~\cite{EdwardsOh} are unitary representation theoretic, building on the work of Lax and Phillips in~\cite{LaxPhillips}, for which the restriction on the critical exponent is necessary.
Earlier instances of the results of~\cite{EdwardsOh} under more stringent assumptions on the size of $\d_\G$ were obtained by Mohammadi and Oh in~\cite{MohammadiOh}, albeit the latter results are stronger in that they in fact hold for the frame flow rather than the geodesic flow.

The case of real hyperbolic geometrically finite manifolds with cusps and arbitrary critical exponent was only recently resolved independently in~\cite{LiPan} where a symbolic coding of the geodesic flow was constructed. This approach builds on extensions of Dolgopyat's method to suspension flows over shifts with infinitely many symbols; cf.~\cite{AraujoMelbourne,AGY}.
The extension of their result to frame flows was carried out in~\cite{LiPanSarkar}.

Finally, we refer the reader to~\cite{DyatlovGuillarmou-long} and the references therein for a discussion of the history of the microlocal approach to the problem of spectral gaps via anisotropic Sobolev spaces.

\subsection{Outline of the argument}\label{sec:outline} 

The article has several parts that can be read independently of one another. For the convenience of the reader, we give a brief outline of those parts.

The first part consists of Sections~\ref{sec:prelims}-\ref{sec:linear expand}.
After recalling some basic facts in Section~\ref{sec:prelims}, we prove a key doubling result, Proposition~\ref{prop:doubling}, in Section~\ref{sec:doubling} for the conditional measures of $\bms$ along the strong unstable foliation.

In Section~\ref{section: height function rank 1}, we construct a Margulis function which shows, roughly speaking, that generic orbits with respect to $\bms$ are biased to return to the thick part of the manifold. In Section~\ref{sec:linear expand}, we prove a statement on average expansion of vectors in linear representations which is essential for our construction of the Margulis function. The main difficulty in the latter result in comparison with the classical setting lies in controlling the \emph{shape} of sublevel sets of certain polynomials in order to estimate their measure with respect to conditional measures of $\bms$ along the unstable foliation.

The second part consists of Sections~\ref{sec: spectral gap} and~\ref{sec:ess radius}.
In Section~\ref{sec: spectral gap}, we define anisotropic Banach spaces arising as completions of spaces of smooth functions with respect to a dynamically relevant norm and study the norm of the transfer operator as well as the resolvent in their actions on these spaces in Section~\ref{sec:ess radius}.
The proof of Theorem~\ref{thm:intro meromorphic} is completed in Section~\ref{sec:ess radius}. The approach of these two sections follows closely the ideas of~\cite{GouezelLiverani,GouezelLiverani2,AvilaGouezel}, originating in~\cite{BlankKellerLiverani}.
Theorem~\ref{thm:exp recurrence intro} is deduced from this analysis in Section~\ref{sec:exp recur}.

 The third part concerns a Dolgopyat-type estimate which is a key technical estimate in the proof of Theorems~\ref{thm:intro mixing} and~\ref{thm:intro strip}. 
 Its proof occupies Section~\ref{sec:Dolgopyat} with auxiliary technical results in Sections~\ref{sec:mollifiers},~\ref{sec:temporal function}, and~\ref{sec:friendly}. 
 Readers familiar with the theory of anisotropic spaces may skip directly to Section~\ref{sec:Dolgopyat}, taking the results on recurrence from the cusps from previous sections as a black box.

The Dolgopyat-type estimate, obtained in Theorem~\ref{thm:Dolgopyat}, provides a contraction on the norm of resolvents with large imaginary parts.
Theorems~\ref{thm:intro mixing} and~\ref{thm:intro strip} are deduced from this result in Section~\ref{sec:Butterley}.
A sketch of its proof is given in Section~\ref{sec:sketch}.
The principle behind Theorem~\ref{thm:Dolgopyat}, due to Dolgopyat, is to exploit the non-joint integrability of the stable and unstable foliations via certain oscillatory integral estimates; cf.~\cite{Dolgopyat,Liverani,GiuliettiLiveraniPollicott,GiuliettiLiveraniPolicott-Erratum,BaladiDemersLiverani}.

A major difficulty in implementing this philosophy lies in
estimating these oscillatory integrals against Patterson-Sullivan measures, which are \emph{fractal} in nature in general.
In particular, we cannot argue using the standard integration by parts method in previous works on exponential mixing of SRB measures using the method of anisotropic spaces, see for instance~\cite{Liverani,GiuliettiLiveraniPollicott,GiuliettiLiveraniPolicott-Erratum,BaladiDemersLiverani}, where the unstable conditionals are of Lebesgue class.

We deal with this difficulty using Corollary~\ref{cor:flattening} by taking advantage of the fact that the estimate in question is an \emph{average} over oscillatory integrals.
This idea is among the main contributions of this article. 
We hope this method can be fruitful in establishing rates of mixing of hyperbolic flows in greater generality.

In the case of variable curvature (i.e. when $\mc{X}$ is not real hyperbolic), the action of the derivative of the geodesic flow on the strong unstable distribution is non-conformal which causes significant additional difficulties in the analysis, particularly in the presence of cusps.
We deal with this difficulty by working with the \emph{projection} of the unstable conditionals to the directions of slowest expansion and show that these projections also satisfy the conclusion of Corollary~\ref{cor:flattening}.
See Remarks~\ref{rem:role of friendliness} and~\ref{rem:projections} for further discussion.

In Section~\ref{sec:temporal function}, we obtain a linearization of the so-called temporal distance function. 
In Section~\ref{sec:friendly}, we verify the non-concentration hypotheses of Corollary~\ref{cor:flattening intro} (more precisely, we verify the weaker hypothesis of Corollary~\ref{cor:flattening}) for the projection of the unstable conditionals of $\bms$ onto the directions with weakest expansion.
This allows us to apply Corollary~\ref{cor:flattening} towards estimating the oscillatory integrals arising in Section~\ref{sec:Dolgopyat}.

Finally, Section~\ref{sec:flattening} is dedicated to the proof of Theorem~\ref{thm:ellq improvement} and Corollary~\ref{cor:flattening intro}.
Among the key ingredients in the proof are the asymmetric Balog-Szemer\'edi-Gowers Lemma due to Tao and Vu (Theorem~\ref{thm:asymmetric Balog-Szemeredi-Gowers}) as well as Hochman's inverse theorem for the entropy of convolutions (Theorem~\ref{thm:Hochman}).
This section can be read independently from the rest of the article.

\begin{acknowledgement}
 The author thanks the Hausdorff Research Institute for Mathematics at the Universit\"at Bonn for its hospitality during the trimester program ``Dynamics: Topology and Numbers'' in Spring 2020 where part of this research was conducted.
This research is supported in part by the NSF under grant number DMS-2247713. 
The author thanks Hee Oh, Peter Sarnak, and Pablo Shmerkin for helpful discussions regarding this project. 
The author also thanks F\'elix Lequen and the referees for detailed reading of the article and for numerous corrections and comments that significantly improved the exposition.
\end{acknowledgement}

\section{Preliminaries}

    \label{sec:prelims}
    
    We recall here some background and definitions on geometrically finite manifolds.
    
	\subsection{Geometrically finite manifolds}

	The standard reference for the material in this section is~\cite{Bowditch1993}.
	Suppose $G$ is the group of orientation preserving isometries of a real, complex, quaternionic or Cayley hyperbolic space, denoted $\H^d_\K$, of dimension $d\geq 2$, where $\K\in\set{\R,\C,\H,\mathbb{O}}$. In the case $\K=\mathbb{O}$, then $d=2$.
	
    Fix a basepoint $o\in \H^d_\K$. Then, $G$ acts transitively on $\H^d_\K$ and the stabilizer $K$ of $o$ is a maximal compact subgroup of $G$.
    We shall identify $\H^d_\K$ with $K\backslash G$.
    Denote by $A=\set{g_t:t\in\R}$ a $1$-parameter subgroup of $G$ inducing the geodesic flow on the unit tangent bundle of $\H^d_\K$.
    Let $M<K$ denote the centralizer of $A$ inside $K$ so that the unit tangent bundle $\uT\H^d_\K$ may be identified with ${M}\backslash{G}$.
    In Hopf coordinates, we can identify $\uT\H^d_\K$ with $\R\times (\partial\H^d_\K\times \partial\H^d_\K \setminus\Delta)$, where $\partial \H^d_\K$ denotes the boundary at infinity and $\Delta$ denotes the diagonal.
    
    Let $\G<G$ be an infinite discrete subgroup of $G$.
    The limit set of $\G$, denoted $\L_\G$, is the set of limit points of the orbit $\G\cdot o$ on $\partial \H^d_\K$.
    Note that the discreteness of $\G$ implies that all such limit points belong to the boundary.
    Moreover, this definition is independent of the choice of $o$ in view of the negative curvature of $\H^d_\K$.
    We often use $\L$ to denote $\L_\G$ when $\G$ is understood from context.
    We say $\G$ is \textit{non-elementary} if $\L_\G$ is infinite.
    
    The \textit{hull} of $\L_\G$, denoted $\mrm{Hull}(\L_\G)$, is the smallest convex subset of $\H^d_\K$ containing all the geodesics joining points in $\L_\G$.
    The convex core of the manifold $\H^d_\K/\G$ is the smallest convex subset containing the image of $\mrm{Hull}(\L_\G)$.
    We say $\H^d_\K/\G$ is \textit{geometrically finite} (resp.~\textit{convex cocompact}) if the closed $1$-neighborhood of the convex core has finite volume (resp. is compact), cf.~\cite{Bowditch1993}.
    The non-wandering set for the geodesic flow is the closure of the set of vectors in the unit tangent bundle whose orbit accumulates on itself.
    In Hopf coordinates, this set, denoted $\Omega$, coincides with the projection of $\R\times (\L_\G\times\L_\G-\Delta)$ mod $\G$.

	A useful equivalent definition of geometric finiteness is that the limit set of $\G$ consists entirely of radial and bounded parabolic limit points; cf.~\cite{Bowditch1993}.
    This characterization of geometric finiteness will be of importance to us and so we recall here the definitions of these objects.

    A point $\xi \in \L$ is said to be a \textit{radial point} if any geodesic ray terminating at $\xi$ returns infinitely often to a bounded subset of $\H^d_\K/\G$.
    The set of radial limit points is denoted by $\L_r$.
    
    Denote by $N^+$ the expanding horospherical subgroup of $G$ associated to $g_t$, $t\geq 0$.
    A point $p \in \L$ is said to be a \textit{parabolic point} if the stabilizer of $p$ in $\G$, denoted by $\G_p$, is conjugate in $G$ to an unbounded subgroup of $MN^+$.
    A parabolic limit point $p$ is said to be \textit{bounded} if $ \left(\L-\set{p}\right)/\G_p $ is compact.
    An equivalent charachterization is that $p\in \L$ is parabolic if and only if any geodesic ray terminating at $p$ eventually leaves every compact subset of $\H^d_\K/\G$.
    The set of parabolic limit points will be denoted by $\L_p$.

   Given $g\in G$, we denote by $g^+$ the coset of $P^-g$ in the quotient $P^-\backslash G$, where $P^-=N^-AM$ is the stable parabolic group associated to $\set{g_t:t\geq 0}$.
   Similarly, $g^-$ denotes the coset $P^+g$ in $P^+\backslash G$.
   Since $M$ is contained in $P^\pm$, such a definition makes sense for vectors in the unit tangent bundle $M\backslash G$.
   Geometrically, for $v\in M\backslash G$, $v^+$ (resp.~$v^-$) is the forward (resp.~backward) endpoint of the geodesic determined by $v$ on the boundary of $\H^d_\K$.
   Given $x\in G/\G$, we say $x^{\pm}$ belongs to $\L$ if the same holds for any representative of $x$ in $G$; this notion being well-defined since $\L$ is $\G$ invariant.

	\subsection*{Notation} Throughout the remainder of the article, we fix a discrete non-elementary geometrically finite group $\G$ of isometries of some (irreducible) rank one symmetric space $\H_\K^d$ and denote by $X$ the quotient $G/\G$, where $G$ is the isometry group of $\H^d_\K$.

\subsection{Standard horoballs}  \label{sec:cuspnbhd}
    Since parabolic points are fixed points of elements of $\G$, $\L$ contains only countably many such points.
    Moreover, $\G$ contains at most finitely many conjugacy classes of parabolic subgroups.
    This translates to the fact that $\L_p$ consists of finitely many $\G$ orbits.

     Let $\set{p_1, \dots, p_s}\subset \partial \H^d_\K$ be a maximal set of nonequivalent parabolic fixed points under the action of $\G$. 
     As a consequence of geometric finiteness of $\G$, one can find a finite disjoint collection of \textit{open} horoballs $H_1, \dots, H_s \subset \H^d_\K$ with the following properties (cf.~\cite{Bowditch1993}):
    
    \begin{enumerate}
    	\item $H_i$ is centered on $p_i$, for $i = 1,\dots, s$.
	    \item $\overline{H_i}\G  \cap \overline{H_j}\G  = \emptyset$ for all $i\neq j$.
        \item For all $i \in \set{1,\dots,s}$ and $\g_1,\g_2 \in \G$
        \begin{equation*}
            \overline{H_i}\g_1 \cap  \overline{H_i}\g_2 \neq \emptyset
            \Longrightarrow \overline{H_i}\g_1 =  \overline{H_i}\g_2, 
            \g_1^{-1}\g_2\in \G_{p_i}.
        \end{equation*}
        \item $\mrm{Hull}(\L_\G)\setminus (\bigcup_{i=1}^s H_i\G)$ is compact mod $\G$.
    \end{enumerate}
    
    \begin{remark}\label{rem:o outside cusp}
        We shall assume throughout the remainder of the article that our fixed basepoint $o$ lies outside these standard horoballs, i.e.
        \begin{equation*}
            o \notin \bigcup_{i=1}^s \overline{H_i} \G.
        \end{equation*}
    \end{remark}


\subsection{Conformal Densities and the BMS Measure}

    The \textit{critical exponent}, denoted $\d_\G$, is defined to be the infimum over all real number $s\geq 0$ such that the Poincar\'e series
    \begin{align}\label{eq:Poincare}
        P_\G(s,o) := \sum_{\g\in\G} e^{-s \dist(o,\g\cdot o)}
    \end{align}
    converges.
    We shall simply write $\d$ for $\d_\G$ when $\G$ is understood from context.
    The Busemann function is defined as follows: given $x,y\in \H^d_\K$ and $\xi\in \partial \H^d_\K$, let $\g:[0,\infty)\to\H^d_\K$ denote a geodesic ray terminating at $\xi$ and define
    \begin{equation*}
        \b_\xi(x,y) = \lim_{t\to\infty}
        \dist (x,\g(t)) - \dist(y,\g(t)).
    \end{equation*}
    A $\G$-invariant conformal density of dimension $s$ is a collection of Radon measures $\set{\nu_x:x\in \H^d_\K}$ on the boundary satisfying
    \begin{equation*}
        \g_\ast \nu_x = \nu_{\g x}, \qquad \text{and} \qquad
        \frac{d\nu_{y}}{d\nu_x}(\xi) = e^{s\b_{\xi}(x,y)}, \qquad
        \forall x,y\in\H^d_\K, \xi\in \partial \H^d_\K, \g\in\G.
    \end{equation*}
    
    Given a pair of conformal densities $\set{\mu_x}$ and $\set{\nu_x}$ of dimensions $s_1$ and $s_2$ respectively, we can form a $\G$ invariant measure on $\uT\H^d_\K$, denoted by $m^{\mu,\nu}$ as follows: for $x=(\xi_1,\xi_2,t)\in \uT\H^d_\K$
    \begin{equation}\label{eq:BMS}
        dm^{\mu,\nu}(\xi_1,\xi_2,t) = e^{s_1\b_{\xi_1}(o,x)+s_2\b_{\xi_2}(o,x)}\;d\mu_o(\xi_1)\;d\nu_o(\xi_2)\;dt.
    \end{equation}
    Moreover, the measure $m^{\mu,\nu}$ is invariant by the geodesic flow.
     
    When $\G$ is geometrically finite and $\K=\R$, Patterson~\cite{Patterson} and Sullivan~\cite{Sullivan} showed the existence of a unique (up to scaling) $\G$-invariant conformal density of dimension $\d_\G$, denoted $\set{\ps_x:x\in \H^d_\R}$.
    Geometric finiteness also implies that the measure $m^{\ps,\ps}$ descends to a finite measure of full support on $\Omega$ and is the unique measure of maximal entropy for the geodesic flow. This measure is called the Bowen-Margulis-Sullivan (BMS for short) measure and is denoted $\bms$.
    
    Since the fibers of the projection from $G/\G$ to $\uT\H^d_\K/\G$ are compact and parametrized by the group $M$, we can lift such a measure to $G/\G$, also denoted $\bms$, by taking locally the product with the Haar probability measure on $M$.
    Since $M$ commutes with the geodesic flow, this lift is invariant under the group $A$. 
    We refer the reader to~\cite{Roblin} and~\cite{PaulinPollicottSchapira} and references therein for details of the construction in much greater generality than that of $\H^d_\K$.

\subsection{Stable and unstable foliations and leafwise measures}

    The fibers of the projection $G\to \uT\H^d_\K$ are given by the compact group $M$, which is the centralizer of $A$ inside the maximal compact group $K$.
    In particular, we may lift $\bms$ to a measure on $G/\G$, also denoted $\bms$, and given locally by the product of $\bms$ with the Haar probability measure on $M$.
    The leafwise measures of $\bms$ on $N^+$ orbits are given as follows:
    \begin{equation}\label{eq:unstable conditionals}
        d\mu_x^u(n) = e^{\d_\G \b_{(nx)^+}(o,nx)}d\ps_o((nx)^+).
    \end{equation}
    They satisfy the following equivariance property under the geodesic flow:
    \begin{equation}\label{eq:g_t equivariance}
        \mu_{g_tx}^u = e^{\d t} \mrm{Ad}(g_t)_\ast \mu_{x}^u.
    \end{equation}
    Moreover, it follows readily from the definitions that for all $n\in N^+$,
    \begin{align}\label{eq:N equivariance}
       (n)_\ast \mu_{nx}^u =  \mu_x^u,
    \end{align}
    where $(n)_\ast \mu_{nz}^u$ is the pushforward of $\mu_{nz}^u$ under the map $u\mapsto un$ from $N^+$ to itself.
    Finally, since $M$ normalizes $N^+$ and leaves $\bms$ invariant, this implies that these conditionals are $\Ad(M)$-invariant: 
    \begin{align}\label{eq:M equivariance}
        \mu^u_{mx}  = \Ad(m)_\ast\mu_x^u, \qquad m\in M.
    \end{align}
    
\subsection{Cygan metrics}
\label{sec:Carnot}
We recall the definition of the Cygan metric on $N^+$, denoted $d_{N^+}$. These metrics are right invariant under translation by $N^+$, and satisfy the following convenient scaling property under conjugation by $g_t$. For all $r>0$, if $N_r^+$ denotes the ball of radius $r$ around identity in that metric and $t\in \R$, then
\begin{equation}\label{eq:scaling of Carnot metric}
    \mrm{Ad}(g_t)(N_r^+) = N_{e^tr}^+.
\end{equation}

To define the metric, we need some notation which we use throughout the article.
For $x\in\K$, denote by $\bar{x}$ its $\K$-conjugate and by $|x|:=\sqrt{\bar{x}x}$ its modulus.
This modulus extends to a norm on $\K^n$ by setting
\begin{align*}
    \norm{u}^2 := \sum_i |u_i|^2,
    \qquad u=(u_1,\dots,u_n)\in\K^n.
\end{align*}

We let $\mrm{Im}\K$ denote those $x\in\K$ such that $\bar{x}=-x$.
For example, $\mrm{Im}\K$ is the pure imaginary numbers and the subspace spanned by the quaternions $i,j$ and $k$ in the cases $\K=\C$ and $\K=\H$ respectively.
For $u\in\K$, we write $\mrm{Re}(u)=(u+\bar{u})/2$ and $\mrm{Im}(u)=(u-\bar{u})/2$.

The Lie algebra $\mf{n}^+$ of $N^+$ splits under $\Ad(g_t)$ into eigenspaces as $\mf{n}^+_\a\oplus \mf{n}^+_{2\a}$, where $\mf{n}^+_{2\a}=0$ if and only if $\K=\R$.
Moreover, we have the identification $\mf{n}^+_\a \cong \K^{d-1}$ and $\mf{n}^+_{2\a}\cong \mrm{Im}(\K)$ as real vector spaces; cf.~\cite[Section 19]{Mostow}.
We denote by $\norm{\cdot}'$ the following quasi-norm on $\mf{n}^+$: 
\begin{align}\label{eq:Carnot norm}
    \norm{(u,s)}' :=  \left(\norm{u}^4 + |s|^2 \right)^{1/4}, \qquad (u,s)\in \mf{n}^+_\a\oplus \mf{n}^+_{2\a} .
\end{align}
With this notation, the distance of $n:=\exp(u,s)$ to identity is given by:
\begin{equation}\label{eq:Carnot}
    d_{N^+}(n, \id) := \norm{(u,s)}'.
\end{equation}
Given $n_1,n_2\in N^+$, we set $d_{N^+}(n_1,n_2) = d_{N^+}(n_1n_2^{-1},\id )$.

\subsection{Local stable holonomy}\label{sec:holonomy}

    We recall the definition of (stable) holonomy maps. We give a simplified discussion of this topic which is sufficient in our homogeneous setting.
    Let $x=u^-y$ for some $y\in \Omega$ and $u^-\in N^-_2$.
Since the product map $N^-\times A \times M \times N^+\to G$ is a diffeomorphism near identity, we can choose the norm on the Lie algebra so that the following holds. We can find maps $p^-:N_1^+\to P^-=N^-AM$ and $u^+:N_2^+\to N^+$ so that
\begin{align}\label{eq:switching order of N- and N+}
    nu^- = p^-(n)u^+(n), \qquad \forall n\in N_2^+.
\end{align}
Then, it follows by~\eqref{eq:unstable conditionals} that for all $n\in N_2^+$, we have
\begin{align*}
    d\mu_y^u(u^+(n)) = e^{-\d \b_{(nx)^+}(u^+(n)y,nx)}d\mu_x^u(n).
\end{align*}
Moreover, by further scaling the metrics if necessary, we can ensure that these maps are diffeomorphisms onto their images.
In particular, writing $\Phi(nx)=u^+(n)y$, we obtain the following change of variables formula: for all $f\in C(N_2^+)$,
\begin{align}\label{eq:stable equivariance}
    \int f(n) \;d\mu_x^u(n) = 
    \int f((u^+)^{-1}(n)) e^{\d \b_{\Phi^{-1}(ny)^+}(ny,\Phi^{-1}(ny))}\;d\mu_y^u(n).
\end{align}

\begin{remark}\label{rem:commutation of stable and unstable}
To avoid cluttering the notation with auxiliary constants, we shall assume that the $N^-$ component of $p^-(n)$ belongs to $N_2^-$ for all $n\in N_2^+$ whenever $u^-$ belongs to $N_1^-$.
\end{remark}


\subsection{Notational convention} Throughout the article, given two quantities $A$ and $B$, we use the Vinogradov notation $A\ll B$ to mean that there exists a constant $C\geq 1$, possibly depending on $\G$ and the dimension of $G$, such that $|A|\leq C B$. In particular, this dependence on $\G$ is suppressed in all of our implicit constants, except when we wish to emphasize it.
    The dependence on $\G$ may include for instance the diameter of the complement of our choice of cusp neighborhoods inside $\Omega$ and the volume of the unit neighborhood of $\Omega$.
    We write $A\ll_{x,y} B$ to indicate that the implicit constant depends on parameters $x$ and $y$.
    We also write $A=O_x(B)$ to mean $A\ll_x B$.


\section{Doubling Properties of Leafwise Measures}
\label{sec:doubling}

The goal of this section is to prove the following useful consequence of the global measure formula on the doubling properties of the leafwise measures. The result is an immediate consequence of Sullivan's shadow lemma in the case $\G$ is convex cocompact. In particular, the content of the following result is the uniformity, even in the case $\Omega$ is not compact.
The argument is based on the topological transitivity of the geodesic flow when restricted to $\Omega$.

Define the following exponents:
\begin{align}\label{eq:Delta}
    \gls{Delta} &:= \min\set{\d, 2\d-k_{\max}, k_{\min}},
    \nonumber\\
    \gls{Delta+} &:= \max\set{\d, 2\d-k_{\min}, k_{\max}}.
\end{align}
where $k_{\max}$ and $k_{\min}$ denote the maximal and minimal ranks of parabolic fixed points of $\G$ respectively.
If $\G$ has no parabolic points, we set $k_{\max}=k_{\min} = \d$, so that $\Delta=\Delta_+=\d$.

\begin{prop}[Global Doubling and Decay]
\label{prop:doubling}

For every $0<\s\leq 5$, $x\in N_2^-\Omega$ and $0<r\leq 1$, we have
\begin{equation*}
    \mu_x^u(N_{\s r}^+)\ll
    \begin{cases}
       \s^{\Delta} \cdot \mu_x^u(N_r^+) &
       \forall 0<\s\leq 1, 0<r\leq 1, \\
    \s^{\Delta_+} \cdot \mu_x^u(N_r^+)
    & \forall \s> 1, 0<r\leq 5/\s.
    \end{cases}
\end{equation*}
\end{prop}

\begin{remark}
The above proposition has very different flavor when applied with $\s<1$, compared with $\s>1$.
In the former case, we obtain a global rate of decay of the measure of balls on the boundary, centered in the limit set. In the latter case, we obtain the so-called Federer property for our leafwise measures.

\end{remark}

\begin{remark}
The restriction that $r\leq 5/\s$ in the case $\s>1$ allows for a uniform implied constant. The proof shows that in fact, when $\s>1$, the statement holds for any $0<r\leq 1$, but with an implied constant depending on $\s$.
\end{remark}

\subsection{Global Measure Formula}
\label{sec:global measure formula}

Our basic tool in proving Proposition~\ref{prop:doubling} is the extension of Sullivan's shadow lemma known as the global measure formula, which we recall in this section.

Given a parabolic fixed point $p\in \L$, with stabilizer $\G_p\subset \G$, we define \textit{the rank of} $p$ to be twice the critical exponent of the Poincar\'e series $P_{\G_p}(s,o)$ associated with $\G_p$; cf.~\eqref{eq:Poincare}.

Given $\xi\in \partial\H^d_\K$, we let $[o\xi)$ denote the geodesic ray. For $t\in \R_+$, denote by $\xi(t)$ the point at distance $t$ from $o$ on $[o\xi)$.
For $x\in\H^d_\K$, define the $\Ocal(x)$ to be the \textit{shadow} of unit ball $B(x,1)$ in $\H^d_\K$ on the boundary as viewed from $o$.
More precisely,
\begin{equation*}
    \Ocal(x):= \set{\xi\in\partial \H^d_\K: [o\xi)\cap B(x,1)\neq \emptyset}.
\end{equation*}
Shadows form a convenient, dynamically defined, collection of neighborhoods of points on the boundary.

The following generalization of Sullivan's shadow lemma gives precise estimates on the measures of shadows with respect to Patterson-Sullivan measures.

    \begin{thm}[Theorem 3.2,~\cite{Schapira-QND}]
    \label{thm:global measure formula}
    
    There exists $C=C(\G,o)\geq 1$ such that for every $\xi\in \L$ and all $t>0$,
    \begin{equation*}
        C^{-1} 
        \leq \frac{\ps_o( \Ocal(\xi(t)))}{e^{-\d t } e^{d(t)(k(\xi(t))-\d)}} \leq
        C ,
    \end{equation*}
    where 
    \begin{equation*}
        d(t) = \dist(\xi(t),\G\cdot o),
    \end{equation*}
    and $k(\xi(t))$ denotes the rank of a parabolic fixed point $p$ if $\xi(t)$ is contained in a standard horoball centered at $p$ and otherwise $k(\xi(t))=\d$.
 
   \end{thm}
A version of Theorem~\ref{thm:global measure formula} was obtained earlier for real hyperbolic spaces in~\cite{StratmannVelani} and for complex and quaternionic hyperbolic spaces in~\cite{Newberger}.

\subsection{Proof of Proposition~\ref{prop:doubling}}
Assume that $\s \leq 1$, the proof in the case $\s> 1$ is similar.

Fix a non-negative $C^\infty$ bump function $\psi$ supported inside $N_{1}^+$ and having value identically $1$ on $N_{1/2}^+$.
Given $\e>0$, let
$   \psi_{\e}(n) =\psi(\Ad(g_{-\log \e})(n))$. 
Note that the condition that $\psi_\e(\id)=\psi(\id)=1$ implies that for $x\in X$ with $x^+\in \L$, 
\begin{equation}
    \mu_x^u(\psi_\e)>0, \qquad \forall \e>0.
\end{equation}
Note further that for any $r>0$, we have that
$\chi_{N_r^+}\leq  \psi_r\leq \chi_{N_{2r}^+}$.

First, we establish a uniform bound over $x\in\Omega$.
Consider the following function $f_\s:\Omega \to (0,\infty)$:
\begin{equation*}
    f_\s(x) = \sup_{0<r\leq 1} \frac{\mu_x^u(\psi_{\s r})}{\mu_x^u(\psi_r)}.
\end{equation*}
We claim that it suffices to prove that
\begin{equation}\label{eq:bound on f_sigma}
    f_\s(x)\ll \s^{\Delta},
\end{equation}
uniformly over all $x\in \Omega$ and $0<\s\leq 1$.
Indeed, fix some $0< r\leq 1$ and $0<\s\leq 1$. 
By enlarging our implicit constant if necessary, we may assume that $\s\leq 1/4$.
From the above properties of $\psi$, we see that
\begin{align*}
    \mu_x^u(N_{\s r}^+) \leq \mu_x^u(\psi_{(4\s) (r/2)})
    \ll \s^{\Delta} \mu_x^u(\psi_{ r/2})
    \leq \s^{\Delta} \mu_x^u(N_r^+).
\end{align*}

Hence, it remains to prove~\eqref{eq:bound on f_sigma}.
By~\cite[Lemme 1.16]{Roblin}, for each given $r>0$, the map $x\mapsto \mu_x^u(\psi_{\s r})/\mu_x^u(\psi_r)$ is a continuous function on $\Omega$.
Indeed, the weak-$\ast$ continuity of the map $x\mapsto \mu_x^u$ is the reason we work with bump functions instead of indicator functions directly.
Moreover, continuity of these functions implies that $f_\s$ is lower semi-continuous.

The crucial observation regarding $f_\s$ is as follows. In view of~\eqref{eq:g_t equivariance}, we have for $t\geq 0$,
\begin{equation*}
    f_\s(g_tx) =\sup_{0<r\leq e^{-t}} \frac{\mu_x^u(\psi_{\s r})}{\mu_x^u(\psi_r)} \leq f_\s(x).
\end{equation*}
Hence, for all $B\in \R$, the sub-level sets $\Omega_{<B}:=\set{f_\s <B}$ are invariant by $g_t$ for all $t\geq 0$.
On the other hand, the restriction of the (forward) geodesic flow to $\Omega$ is topologically transitive.
In particular, any invariant subset of $\Omega$ with non-empty interior must be dense in $\Omega$.
Hence, in view of the lower semi-continuity of $f_\s$, to prove~\eqref{eq:bound on f_sigma}, it suffices to show that $f_\s$ satisfies~\eqref{eq:bound on f_sigma} for all $x$ in some open subset of $\Omega$.

Recall we fixed a basepoint $o\in \H_\K^d$ belonging to the hull of the limit set.
Let $x_o\in G$ denote a lift of $o$ whose projection to $G/\G$ belongs to $\Omega$.
Let $E$ denote the unit neighborhood of $x_o$.
We show that $E\cap \Omega \subset \set{f_\s \ll \s^{\Delta}}$. 
Without loss of generality, we may further assume that $\s< 1/2$, by enlarging the implicit constant if necessary.

First, note that the definition of the conditional measures $\mu_x^u$ immediately gives
   \begin{equation*}
       \mu_x^u|_{N_{4}^+} \asymp \ps_o|_{\left(N_{4}^+\cdot x\right)^+},\qquad  \forall x\in E.
   \end{equation*}
   It follows that
   \begin{equation*}
   \ps_o((N_{r}^+\cdot x)^+) \ll
       \mu_x^u(\psi_r)\ll \ps_o((N_{2r}^+\cdot x)^+),
   \end{equation*}
  for all $0\leq r\leq 2$ and $x\in E$.
  Hence, it will suffice to show that for all $0<\s<1$,
  \begin{equation*}
      \frac{\ps_o((N_{\s r}^+\cdot x)^+)}{\ps_o((N_{ r}^+\cdot x)^+)} \ll \s^{\Delta}.
  \end{equation*}

   To this end, there is a constant $C_1\geq 1$ such that the following holds; cf.~\cite[Theorem 2.2]{Corlette}\footnote{The quoted result in~\cite{Corlette} is stated in terms of the so-called Carnot-Caratheodory metric $d_{cc}$ on $N^+$, which enjoys the same scaling property in~\eqref{eq:scaling of Carnot metric}. In particular, this metric is Lipschitz equivalent to the Cygan metric in~\eqref{eq:Carnot} by compactness of the unit sphere in the latter and continuity of the map $n\mapsto d_{cc}(n,\id)$.}.
   For all $x\in E$, if $\xi = x^+$, then, the shadow $S_r=\set{(nx)^+: n\in N_r^+}$ satisfies 
   \begin{equation}\label{eq:approximate shadow}
      \Ocal(\xi(|\log r|+C_1))\subseteq  S_r \subseteq \Ocal(\xi(|\log r|-C_1)),
      \qquad \forall 0<r\leq 2.
   \end{equation}
   Here, and throughout the rest of the proof, if $s\leq 0$, we use the convention 
   \[\Ocal(\xi(s)) = \Ocal(\xi(0))=\partial \H_\K^d.\]
   Fix some arbitrary $x\in E$ and let $\xi=x^+$.
   To simplify notation, set for any $t, r >0$,
   \begin{align*}
        t_\s &:= \max\set{|\log \s r|-C_1,0},
        &t_r &:=  |\log r|+C_1,
        \\
       d(t) &:= \dist(\xi(t),\G\cdot o),  
       &       k(t) &:= k(\xi(t)),
   \end{align*}
   where $k(\xi(t))$ is as in the notation of Theorem~\ref{thm:global measure formula}. 
   
   By further enlarging the implicit constant, we may assume for the rest of the argument that
   \begin{equation*}
       -\log \s > 2C_1.
   \end{equation*}
    This insures that $t_\s \geq t_r$ and avoids some trivialities.

    Let $0<r\leq 1$ be arbitrary.
    We define constants $\s_0:=\s\leq \s_1 \leq \s_2 \leq \s_3:=1$ as follows.
    If $\xi(t_\s)$ is in the complement of the cusp neighborhoods, we set $\s_1=\s$.
    Otherwise, we define $\s_1$ by the property that $\xi(|\log \s_1 r|)$ is the first point along the geodesic segment joining $\xi(t_\s)$ and $\xi(t_r)$ (traveling from the former point to the latter) meets the boundary of the horoball containing $\xi(t_\s)$.
    Similarly, if $\xi(t_r)$ is outside the cusp neighborhoods,
    we set $\s_2=1$.
    Otherwise, we define $\s_2$ by the property that $\xi(|\log \s_2 r|)$ is the first point along the same segment, now traveling from $\xi(t_r)$ towards $\xi(t_\s)$, which intersects the boundary of the horoball containing $\xi(t_r)$.
    Define
    \begin{equation*}
        t_{\s_0}:= t_\s, \qquad  t_{\s_3}:= t_r, \qquad
        t_{\s_i}:= |\log \s_i r| \quad \text{for } i=1,2.
    \end{equation*}
    
    In this notation, we first observe that
    $  k(t_{\s_1})=k(t_{\s_2})=\d$. 
    In particular, Theorem~\ref{thm:global measure formula} yields
    \begin{equation*}
     \frac{\ps_o(S_{\s_1 r})}{\ps_o(S_{\s_2 r})}
        \ll \left(\frac{\s_1}{\s_2}\right)^\d.
    \end{equation*}

    Note further that since geodesics in $\H_\K^d$ are unique distance minimizers, we have that the distance between $\xi(t_{\s_i})$ and $\xi(t_{\s_{i+1}})$ is equal to  $|t_{\s_i}-t_{\s_{i+1}}|$, for $i=0,2$.
    Moreover, by our choice of basepoint $o$ and standard horoballs (cf.~Remark~\ref{rem:o outside cusp}), we have that
    \begin{align*}
        \G\cdot o \cap \bigcup_{j=1}^s H_j = \emptyset.
    \end{align*}
    Let $H(\s_0)$ denote the element of the collection of standard horoballs $\Gamma \cdot H_j$, $j=1,\dots, s$, which contains the point $\xi(t_{\s_0})$ if the latter point is inside a cusp neighborhood, and otherwise set $H(\s_0)$ to be the unit ball around $o$.
    Then, there is a constant $C_2\geq 1$, depending only on on the constant $C_1$ as well as 
    the distance between the orbit $\Gamma \cdot o$ and the standard horoballs $H_j$, such that
    \begin{align*}
        d(t_{\s_0}) 
        &\leq \dist(\xi(t_{\s_0}), \partial H(\s_0) ) 
        + \dist ( \partial H(\s_0) , \Gamma \cdot o)
        \nonumber\\
        &\leq \dist( \xi(t_{\s_0}), \xi(t_{\s_{1}})) 
        + \dist ( \partial H(\s_0) , \Gamma \cdot o) 
        \leq -\log (\s_0/\s_{1}) + C_2,
    \end{align*}
    where $\partial H(\s_0)$ denotes the boundary of $H(\s_0)$.
    Similarly, we also get that
    \begin{equation*}
        d(t_{\s_3}) \leq 
         \dist( \xi(t_{\s_2}), \xi(t_{\s_{3}}) ) +C_2 
         \leq -\log (\s_2/\s_{3})+C_2.
    \end{equation*}

    Hence,
    it follows using Theorem~\ref{thm:global measure formula} and the above discussion that
    \begin{align*}
        \frac{\ps_o(S_{\s_0 r})}{\ps_o(S_{\s_1r})}
        \ll 
        \left(\frac{\s_0}{\s_1}\right)^\d 
        e^{d(t_{\s_0})(k(t_{\s_0})-\d)} \ll
        \begin{cases}
        \left(\frac{\s_0}{\s_1}\right)^{2\d - k(t_{\s_0})} & \text{ if } k(t_{\s_0}) \geq \d, \\
        \left(\frac{\s_0}{\s_1}\right)^{\d} & \text{ otherwise}.
        \end{cases}
    \end{align*}
    Similarly, we obtain
    \begin{align*}
        \frac{\ps_o(S_{\s_2 r})}{\ps_o(S_{\s_3 r})}
         \ll \left(\frac{\s_2}{\s_3}\right)^\d 
        e^{-d(t_{\s_3})(k(t_{\s_3})-\d)} \ll
        \begin{cases}
            \left(\frac{\s_2}{\s_3}\right)^{k(t_{\s_3})}
            & \text{ if } k(t_{\s_3}) \leq \d, \\
            \left(\frac{\s_2}{\s_3}\right)^{\d} & \text{ otherwise}.
        \end{cases}
    \end{align*}
    In all cases, we get for $i=0,1,2$ that
    \begin{align*}
         \frac{\ps_o(S_{\s_i r})}{\ps_o(S_{\s_{i+1} r})} \ll 
         \left(\frac{\s_i}{\s_{i+1}}\right)^\Delta,
    \end{align*}
    where $\Delta$ is as in the statement of the proposition.
    Therefore, using the following trivial identity
    \begin{equation*}
        \frac{\ps_o(S_{\s r})}{\ps_o(S_r)}
        =\frac{\ps_o(S_{\s_0 r})}{\ps_o(S_{\s_1r})} \frac{\ps_o(S_{\s_1 r})}{\ps_o(S_{\s_2 r})}
        \frac{\ps_o(S_{\s_2 r})}{\ps_o(S_r)},
    \end{equation*}
    we see that
    $  f(x)\ll \s^{\Delta}$.
    As $x\in E$ was arbitrary, we find that $E\subset \set{f_\s \ll \s^{\Delta}}$, thus concluding the proof in the case $\s\leq 1$.
    Note that in the case $\s>1$, the constants $\s_i$ satisfy $\s_i/\s_{i+1}\geq 1$, so that combining the $3$ estimates requires taking the maximum over the exponents, yielding the bound with $\Delta_+$ in place of $\Delta$ in this case.

Now, let $r\in (0,1]$ and suppose $x=u^-y$ for some $y\in \Omega$ and $u^-\in N^-_2$.
By~\cite[Theorem 2.2]{Corlette}, the analog of~\eqref{eq:approximate shadow} holds, but with shadows from the viewpoint of $x$ and $y$, in place of the fixed basepoint $o$. Recalling the map $n\mapsto u^+(n)$ in~\eqref{eq:switching order of N- and N+}, one checks that this implies that this map is Lipschitz on $N_1^+$ with respect to the Cygan metric, with Lipschitz constant $\asymp C_1$. Moreover, the Jacobian of the change of variables associated to this map with respect to the measures $\mu_x^u$ and $\mu_y^u$ is bounded on $N_1^+$, independently of $y$ and $u^-$; cf.~\eqref{eq:stable equivariance} for a formula for this Jacobian.
Hence, the estimates for $x\in N_2^-\Omega$ follow from their counterparts for points in $\Omega$.


\section{Margulis Functions In Infinite Volume}
	\label{section: height function rank 1}

    We construct Margulis functions on $\Omega$ which allow us to obtain quantitative recurrence estimates to compact sets.
    Our construction is similar to the one in~\cite{BQ-RandomWalkRecurrence} in the case of lattices in rank $1$ groups.
    We use geometric finiteness of $\G$ to establish the analogous properties more generally.
    The idea of Margulis functions originated in~\cite{EskinMargulisMozes}.
    
    Throughout this section, we assume $\G$ is a non-elementary, geometrically finite group containing parabolic elements. The following is the main result of this section.
    A similar result in the special case of quotients of $\SL$ follows from combining Lemma 9.9 and Proposition 7.6 in~\cite{MohammadiOh-Isolation}.

    \begin{thm}\label{thm:Margulis function}
    Let $\Delta >0$ denote the constant in~\eqref{eq:Delta}.
    For every $0<\b<\Delta/2$,
    there exists a proper function $V_\b:N_1^-\Omega\to \R_+$ such that the following holds.
    There is a constant $c \geq 1$ such that for all $x\in N_1^-\Omega$ and $t \geq 0$, 
     \begin{equation*}
         \frac{1}{\mu_x^u(N_1^+)} \int_{N_1^+} V_\b(g_tnx)\;d\mu_x^u(n) \leq c e^{-\b  t}V_\b(x) + c.
     \end{equation*}
    \end{thm}

    Our key tool in establishing Theorem~\ref{thm:Margulis function} is Proposition~\ref{prop:linear expand}, which is a statement regarding average expansion of vectors in linear represearntations of $G$.
    The fractal nature of the conditional measures $\mu_x^u$ poses serious difficulties in establishing this latter result.
    
    \subsection{Construction of Margulis functions}

    Let $p_1,\dots, p_d\in \L$ be a maximal set of inequivalent parabolic fixed points and for each $i$, let $\G_i$ denote the stabilizer of $p_i$ in $\G$.
    Let $P_i<G$ denote the parabolic subgroup of $G$ fixing $p_i$.
    Denote by $U_i$ the unipotent radical of $P_i$ and by $A_i$ a maximal $\R$-split torus inside $P_i$.
    Then, each $U_i$ is a maximal connected unipotent subgroup of $G$ admitting a closed (but not necessarily compact) orbit from identity in $G/\G$.
    As all maximal unipotent subgroups of $G$ are conjugate, we fix elements $h_i\in G$ so that $h_iU_ih_i^{-1}=N^+$.
    Note further that $G$ admits an Iwasawa decomposition of the form $G=KA_iU_i$ for each $i$, where $K$ is our fixed maximal compact subgroup.

    Denote by $W$ the adjoint representation of $G$ on its Lie algebra.
    The specific choice of representation is not essential for the construction, but is convenient for making some parameters more explicit.
    We endow $W$ with a norm that is invariant by $K$.
    
    Let $0\neq v_0\in W$ denote a vector that is fixed by $N^+$.
     In particular, $v_0$ is a highest weight vector for the diagonal group $A$ (with respect to the ordering determined by declaring the roots in $N^+$ to be positive).
    Let $v_i=h_iv_0/\norm{h_i v_0}$.
    Note that each of the vectors $v_i$ is fixed by $U_i$ and is a weight vector for $A_i$.
    In particular, there is an additive character $\chi_i:A_i\to \R$ such that
    \begin{equation}\label{eq:chi_i}
        a \cdot v_i = e^{\chi_i(a)}v_i, \qquad\forall a\in A_i.
    \end{equation}
    We denote by $A_i^+$ the subsemigroup of $A_i$ which expands $U_i$ (i.e.~the positive Weyl chamber determined by $U_i$).
    We let $\a_i:A_i\to \R$ denote the simple root of $A_i$ in $\mrm{Lie}(U_i)$.
    Then,
    \begin{equation}\label{eq:multiple of simple root}
    \chi_i = \chi_\K\a_i, \qquad 
        \chi_\K =
        \begin{cases}
            1, & \text{if } \K=\R, \\
            2 & \text{if } \K=\C, \H,\mathbb{O}.
        \end{cases}
    \end{equation}

    Given $\b>0$, we define a function $V_\b: G/\G \r \R_+$ as follows:
    \begin{equation} \label{defn: height function}
    V_\b(g\G) := \max_{w \in \bigcup_{i=1}^d g\G\cdot v_i } \norm{w}^{-\b/\chi_\K}.
    \end{equation}
    The fact that $V_\b(g\G)$ is indeed a maximum will follow from Lemma~\ref{lem:isolation}.

\subsection{Linear expansion}

    The following result is our key tool in establishing the contraction estimate on $V_\b$ in Theorem~\ref{thm:Margulis function}.
    \begin{prop}\label{prop:linear expand}
    For every $0\leq \b<\Delta/2$, there exists $C=C(\b)\geq 1$ so that for all $t>0$, $x\in N_1^-\Omega$, and all non-zero vectors $v$ in the orbit $G\cdot v_0 \subset W$, we have 
    \begin{equation*}
       \frac{1}{\mu_x^u(N_1^+)} \int_{N_1^+} \norm{g_t n \cdot v}^{-\b/\chi_\K} \;d\mu^u_x(n) \leq C e^{-\b  t} \norm{v}^{-\b/\chi_\K}.
    \end{equation*} 
    \end{prop}
    
    We postpone the proof of Proposition~\ref{prop:linear expand} to Section~\ref{sec:linear expand}.
    Let $\pi_+:W\to W^+$ denote the projection onto the highest weight space of $g_t$.
    The difficulty in the proof of Proposition~\ref{prop:linear expand} beyond the case $G=\SL$ lies in controlling the \emph{shape} of the subset of $N^+$ on which $\norm{\pi_+(n\cdot v)}$ is small, so that we may apply the decay results from Proposition~\ref{prop:doubling}, that are valid only for balls of the form $N_\e^+$.
    We deal with this problem by using a convexity trick.
    A suitable analog of the above result holds for any non-trivial linear representation of $G$.

    The following proposition establishes several geometric properties of the functions $V_\b$ which are useful in proving, and applying, Theorem~\ref{thm:Margulis function}.
    This result is proved in Section~\ref{sec:geometric properties}.
    
    \begin{prop}
    \label{prop:height function properties}
    Suppose $V_\b$ is as in~\eqref{defn: height function}. Then,
    	\begin{enumerate}
    	\item \label{item:V is proper} For every $x$ in the unit neighborhood of $\Omega$,
   we have that
    	\begin{equation*}
    	   \inj(x)^{-1} \ll_\G V_\b^{\chi_\K/\b}(x),
    	\end{equation*}
    	where $\inj(x)$ denotes the injectivity radius at $x$.
    In particular, $V_\b$ is proper on $\Omega$.
    	\item \label{item:log Lipschitz} For all $g\in G$ and all $x\in X$,
        	\[ \norm{g}^{-\b} V_\b(x) \leq V_\b(gx) \leq \norm{g^{-1}}^{\b} V_\b(x). \]

        \item \label{item:isolation} There exists a constant $\e_0 >0$ such that for all $x = g\G \in X$, there exists at most one vector $v \in \bigcup_i g\G \cdot v_i$ satisfying $\norm{v} \leq \e_0$. 
    	\end{enumerate}
    \end{prop}

  \subsection{Proof of Theorem~\ref{thm:Margulis function}}
  In this section, we use Proposition~\ref{prop:height function properties} to translate the linear expansion estimates in Proposition~\ref{prop:linear expand} into a contraction estimate for the functions $V_\b$.

    	Let $t_0>0$ be be given and define
        \[ \omega_0 := \sup_{n\in N_1^+} \max\set{\norm{g_{t_0} n}^{1/\chi_\K}, \norm{(g_{t_0} n)^{-1}}^{1/\chi_\K} }, \]
        where $\norm{\cdot}$ denotes the operator norm of the action of $G$ on $W$.
        Then, for all $n \in N_1^+$ and all $x\in X$, we have
        \begin{equation} \label{eqn: Lipschitz property}
        	\omega_0^{-1} V_1(x)\leq V_1(g_{t_0} n x) 
            	\leq \omega_0 V_1(x),
        \end{equation}
        where $V_1=V_\b$ for $\b=1$.
        
        Let $\e_0$ be as in Proposition~\ref{prop:height function properties}\eqref{item:isolation}.
        Suppose $x\in X$ is such that $V_1(x) \leq \omega_0/\e_0$.
        Then, by~\eqref{eqn: Lipschitz property}, for any $\b>0$, we have that
        \begin{equation} \label{eqn: bounded case}
        \frac{1}{\mu_x^u(N_1^+)} \int_{N_1^+} V_\b(g_{t_0} n x) \;d\mu_x^u(n) \leq 
        	B_0:=(\omega_0^2 \e_0^{-1})^{\b}.
        \end{equation}
         
        Now, suppose $x\in N_1^-\Omega$ is such that $V_1(x) \geq \omega_0/\e_0$ and write $x=g\G$ for some $g\in G$.
        Then, by Proposition~\ref{prop:height function properties}\eqref{item:isolation}, there exists a unique vector $v_\star \in \bigcup_i g\G \cdot v_i$ satisfying $V_1(x) = \norm{v_\star}^{-1/\chi_\K}$.
        Moreover, by~\eqref{eqn: Lipschitz property}, we have that $V_1(g_{t_0} n x) \geq 1/\e_0$ for all $n\in N_1^+$.
        And, by definition of $\omega_0$, for all $n\in N_1^+$, $\norm{g_{t_0} n v_\star}^{1/\chi_\K} \leq \e_0$.
        Thus, applying Proposition~\ref{prop:height function properties}\eqref{item:isolation} once more, we see that $g_{t_0} n v_\star$ is the unique vector in $\bigcup_i g_{t_0}n g\G \cdot v_i $ satisfying 
        \begin{equation*}
        	V_\b(g_{t_0}n x) = \norm{g_{t_0}nv_\star}^{-1/\chi_\K},
        	\qquad\forall n\in N_1^+.
        \end{equation*}
        
        Moreover, since the vectors $v_i$ all belong to the $G$-orbit of $v_0$, it follows that $v_\star$ also belongs to $G\cdot v_0$.
        Thus, we may apply Proposition~\ref{prop:linear expand} as follows.
        Fix some $\b>0$ and let $C=C(\b) \geq 1$ be the constant in the conclusion of the proposition. Then,
        \begin{equation*}
        	\frac{1}{\mu_x^u(N_1^+)} \int_{N_1^+} V_\b(g_{t_0} n x) d\mu_x^u
            	= \frac{1}{\mu_x^u(N_1^+)} \int_{N_1^+} \norm{g_{t_0}nv_\star}^{-\b/\chi_\K} d\mu_x^u
                \leq C e^{-\b  t_0} \norm{v_\star}^{-\b/\chi_\K} 
                = C e^{-\b  t_0} V_\b( x).
        \end{equation*}
        Combining this estimate with~\eqref{eqn: bounded case}, we obtain for any fixed $t_0$,
        \begin{equation}\label{eq:estimate for fixed t_0}
            \frac{1}{\mu_x^u(N_1^+)}\int_{N_1^+}
            V_\b(g_{t_0}nx)\;d\mu_x^u(n)
            \leq C e^{-\b t_0}V_\b(x) +B_0,
        \end{equation}
        for all $x\in \Omega$.
    We claim that there is a constant $c_1=c_1(\b)>0$ such that, if $t_0$ is large enough, depending on $\b$, then
         \begin{equation}\label{eq:induction claim}
         \frac{1}{\mu_x^u(N_1^+)} \int_{N_1^+} V_\b(g_{kt_0}nx)\;d\mu_x^u(n) \leq c_1^k e^{-\b  kt_0}V_\b(x) + 3c_1 B_0, 
     \end{equation}
     for all $k\in\N$. 
     By Proposition~\ref{prop:height function properties}, this claim completes the proof since $V_\b(g_ty)\ll V_\b(g_{\lfloor t/t_0\rfloor t_0}y)$, for all $t\geq 0$ and $y\in X$, with an implied constant depending only on $t_0$ and $\b$.
     
     The proof of~\eqref{eq:induction claim} is by now a standard argument, with the key ingredient in carrying it out being the doubling estimate Proposition~\ref{prop:doubling}.
    We proceed by induction. Let $k\in \N$ be arbitrary and assume that~\eqref{eq:induction claim} holds for such $k$.
    Let $\set{n_i\in \Ad(g_{kt_0})(N_1^+):i\in I}$ denote a finite collection of points in the support of $\mu_{g_{kt_0}x}^u$ such that $N_1^+n_i$ covers the part of the support inside $\Ad(g_{kt_0})(N_1^+)$.
    We can find such a cover with uniformly bounded multiplicity, depending only on $N^+$.
    That is
    \begin{equation*}
        \sum_{i\in I} \chi_{N_1^+n_i}(n) \ll \chi_{\cup_i N_1^+n_i}(n), \qquad \forall n\in N^+.
    \end{equation*}
    Let $x_i=n_ig_{kt_0}x$.
    By~\eqref{eq:estimate for fixed t_0}, and a change of variable, cf.~\eqref{eq:g_t equivariance} and~\eqref{eq:N equivariance}, we obtain
    \begin{align*}
        e^{\d kt_0}
        \int_{N_1^+}V_\b(g_{(k+1)t_0}nx)\;d\mu_x^u
        \leq
        \sum_{i\in I} \int_{N_1^+}V_\b(g_{t_0}nx_i)\;d\mu_{x_i}^u
        \leq 
        \sum_{i\in I}\mu_{x_i}^u(N_1^+)\left( Ce^{-\b t_0}V_\b(x_i)+B_0\right).
    \end{align*}
    It follows using Proposition~\ref{prop:height function properties} that
    $\mu_{y}^u(N_1^+)V_\b(y)\ll \int_{N_1^+}V_\b(ny)\;d\mu_y^u(n)$ for all $y\in X$.
    Hence,
    \begin{align*}
        \int_{N_1^+}V_\b(g_{(k+1)t_0}nx)\;d\mu_x^u(n)
        \ll  e^{-\d kt_0}
        \sum_{i\in I}
        \int_{N_1^+}
        \left( Ce^{-\b t_0}V_\b(nx_i)+B_0\right)\;d\mu^u_{x_i}(n).
    \end{align*}
    Note that since $g_t$ expands $N^+$ by at least $e^t$, we have
    \begin{equation*}
        \Acal_k:= \Ad(g_{-kt_0})\left(\bigcup_i N_1^+n_i\right) \subseteq N_2^+.
    \end{equation*}
    Using bounded multiplicity property of the cover, for any non-negative function $\vp$, we have
    \begin{align*}
        \sum_{i\in I}
        \int_{N_1^+} \vp(nx_i)\;d\mu^u_{x_i}
        = \int_{N^+} \vp(n g_{kt_0}x) 
        \sum_{i\in I} \chi_{N_1^+n_i}(n)
        \;d\mu^u_{g_{kt_0}x}
        \ll \int_{\bigcup_i N_1^+n_i} \vp(n g_{kt_0}x) \;d\mu^u_{g_{kt_0}x}.
    \end{align*}
    Changing variables back so the integrals take place against $\mu_x^u$, we obtain
    \begin{align*}
    e^{-\d kt_0}
        \sum_{i\in I}
        \int_{N_1^+}
        \big( Ce^{-\b t_0}  V_\b(nx_i)+B_0
         \big) 
        \;d\mu^u_{x_i}
        &\ll \int_{\Acal_k} 
        \left( Ce^{-\b t_0}V_\b(g_{kt_0}nx)+B_0\right)
        \;d\mu^u_{x}
        \nonumber\\
        &\leq
        Ce^{-\b t_0} \int_{N_2^+}
        V_\b(g_{kt_0}nx)
        \;d\mu^u_{x}
        + B_0\mu_x^u(N_2^+).
    \end{align*}
    
    To apply the induction hypothesis, we again pick a cover of $N_2^+$ by balls of the form $N_1^+n$, for a collection of points $n\in N_2^+$ in the support of $\mu_x^u$.
    We can arrange for such a collection to have a uniformly bounded cardinality and multiplicity.
    By essentially repeating the above argument, and using our induction hypothesis for $k$, in addition to the doubling property in Prop.~\ref{prop:doubling}, we obtain
    \begin{align*}
        Ce^{-\b t_0} \int_{N_2^+}
        V_\b(g_{kt_0}nx)
        \;d\mu^u_{x}
        + B_0\mu_x^u(N_2^+)
        \ll (C c_1^k e^{-\b(k+1)t_0} V_\b(x) + 2B_0C e^{-\b t_0}+ B_0)
        \mu_x^u(N_1^+),
    \end{align*}
    where we also used Prop.~\ref{prop:height function properties} to ensure that $V_\b(nx)\ll V_\b(x)$, for all $n\in N_3^+$.
    Taking $c_1$ to be larger than the product of $C$ with all the uniform implied constants accumulated thus far in the argument, we obtain
    \begin{align*}
    \frac{1}{\mu_x^u(N_1^+)}
        \int_{N_1^+}V_\b(g_{(k+1)t_0}nx)\;d\mu_x^u(n)
        \leq c_1^{k+1} e^{-\b (k+1)t_0} V_\b(x)+ 2c_1e^{-\b t_0} B_0 + c_1 B_0.
    \end{align*}
   This completes the proof.

  \subsection{Geometric properties of Margulis functions and proof of Proposition~\ref{prop:height function properties}}
\label{sec:geometric properties}

    In this section, we give a geometric interpretation of the functions $V_\b$ which allows us to prove Proposition~\ref{prop:height function properties}.
    Item~\eqref{item:log Lipschitz} follows directly from the definitions, so we focus on the remaining properties.

    The data in the definition of $V_\b$ allows us to give a linear description of cusp neighborhoods as follows.
    Given $g\in G$ and $i$, write $g=kau$ for some $k\in K$, $a\in A_i$ and $u\in U_i$.
    Geometrically, the size of the $A$ component in the Iwasawa decomposition $G=KA_iU_i$ corresponds to the value of the Busemann cocycle $|\b_{p_i}(Kg,o)|$, where $Kg$ is the image of $g$ in $K\backslash G$; cf.~\cite[Remark 6.5]{BenoistQuint-book} and the references therein for the precise statement. This has the following consequence. We can find $0<\e_i<1$ such that
    \begin{equation}\label{eq:e_i}
        \norm{\Ad(a)|_{\mrm{Lie}(U_i)}} <\e_i
        \Longleftrightarrow Kg\in H_{p_i},
    \end{equation}
    where $H_{p_i}$ is the standard horoball based at $p_i$ in $\H^d_\K\cong K\backslash G$.
    
     The functions $V_\b(x)$ roughly measure how far into the cusp $x$ is. More precisely, we have the following lemma.
    
    \begin{lem}\label{lem:V is proper}
    The restriction of $V_\b$ to any bounded neighborhood of $\Omega$ is a proper map.
    \end{lem}
    \begin{proof}
    In view of Property~\eqref{item:log Lipschitz} of Proposition~\ref{prop:height function properties}, it suffices to prove that $V_\b$ is proper on $\Omega$.
    Now, suppose that for some sequence $g_n \in G$, we have $g_n \G$ tends to infinity in $\Omega$.
    Then, since $\G$ is geometrically finite, this implies that the injectivity radius at $g_n\G$ tends to $0$.
    Hence, after passing to a subsequence, we can find $\g_n \in \G$ such that $g_n\g_n$ belongs to a single horoball among the horoballs constituting our fixed standard cusp neighborhood; cf.~Section~\ref{sec:cuspnbhd}.
    By modifying $\g_n$ on the right by a fixed element in $\G$ if necessary, we can assume that $Kg_n\g_n$ converges to one of the parabolic points $p_i$ (say $p_1$) on the boundary of $\H^d_\K\cong K\backslash G$.
    
    Moreover, geometric finiteness implies that $(\L_\G\setminus\set{p_1})/\G_1$ is compact. Thus, by multiplying $g_n\g_n$ by an element of $\G_1$ on the right if necessary, we may assume that $(g_n\g_n)^-$ belongs to a fixed compact subset of the boundary, which is disjoint from $\set{p_1}$.
    
    Thus, for all large $n$, we can write $g_n\g_n=k_n a_n u_n$, for $k_n\in K$, $a_n\in A_i$ and $u_n\in U_i$, such that the eigenvalues of $\Ad(a_n)$ are bounded above; cf.~\eqref{eq:e_i}.
    Moreover, as $(g_n\g_n)^-$ belongs to a compact set that is disjoint from $\set{p_1}$ and $(g_n\g_n)^+\to p_1$, the set $\set{u_n}$ is bounded.
    To show that $V_\b(g_n\G)\to\infty$, since $U_i$ fixes $v_i$ and $K$ is a compact group, it remains to show that $a_n$ contracts $v_i$ to $0$.
    Since $g_n\g_n$ is unbounded in $G$ while $k_n$ and $u_n$ remain bounded, this shows that the sequence $a_n$ is unbounded. Upper boundedness of the eigenvalues of $\Ad(a_n)$ thus implies the claim.
    \end{proof} 
    
    \begin{remark}
    The above lemma is false without restricting to $\Omega$ in the case $\G$ has infinite covolume since the injectivity radius is not bounded above on $G/\G$.
    Note also that this lemma is false in the case $\G$ is not geometrically finite, since the complement of cusp neighborhoods inside $\Omega$ is compact if and only if $\G$ is geometrically finite.
    \end{remark}

    The next crucial property of the functions $V_\b$ is the following linear manifestation of the existence of cusp neighborhoods consisting of disjoint horoballs.
    This lemma implies Proposition~\ref{prop:height function properties}\eqref{item:isolation}. 
    \begin{lem}\label{lem:isolation}
    There exists a constant $\e_0 >0$ such that for all $x = g\G \in X$, there exists at most one vector $v \in \bigcup_i g\G \cdot v_i$ satisfying $\norm{v} \leq \e_0$.
    \end{lem}
    \begin{remark}
     The constant $\e_0$ roughly depends on the distance from a fixed basepoint to the cusp neighborhoods.
    \end{remark}
    \begin{proof}[Proof of Lemma~\ref{lem:isolation}]

    Let $g\in G$ and $i$ be given.
    Write $g= kau$, for some $k\in K$, $a\in A_i$ and $u\in U_i$.
    Since $U_i$ fixes $v_i$ and the norm on $W$ is $K$-invariant, we have $\norm{g\cdot v_i}=\norm{a\cdot v_i}=e^{\chi_i(a)}$; cf.~\eqref{eq:chi_i}.
    Moreover, since $W$ is the adjoint representation, we have
    \begin{equation*}
        \norm{\Ad(a)|_{\mrm{Lie}(U_i)}} \asymp e^{\chi_i(a)},
    \end{equation*}
    and the implied constant, denoted $C$, depends only on the norm on the Lie algebra.

    Let $0<\e_i<1$ be the constants in~\eqref{eq:e_i} and define $\e_0:=\min_i \e_i/C$.
    Let $x=g\G\in G/\G$.
    Suppose that there are elements $\g_1,\g_2\in \G$ and vectors $v_{i_1},v_{i_2}$ in our finite fixed collection of vectors $v_i$ such that $\norm{g\g_j\cdot v_{i_j}}<\e_0$ for $j=1,2$.
    Then, the above discussion, combined with the choice of $\e_i$ in~\eqref{eq:e_i}, imply that $Kg\g_j$ belongs to the standard horoball $H_j$ in $\H^d_\K$ based at $p_{i_j}$.
    However, this implies that the two standard horoballs $H_1\g_1^{-1}$ and $H_2\g_2^{-1}$ intersect non-trivially.
    By choice of these standard horoballs, this implies that the two horoballs $H_j\g_j^{-1}$ are the same and that the two parabolic points $p_{i_j}$ are equivalent under $\G$. In particular, the two vectors $v_{i_1},v_{i_2}$ are in fact the same vector, call it $v_{i_0}$.
    It also follows that $\g_1^{-1}\g_2$ sends $H$ to itself and fixes the parabolic point it is based at.
    Thus, $\g_1^{-1}\g_2$ fixes $v_{i_0}$ by definition.
    But, then, we get that
    \begin{equation*}
        g\g_2\cdot v_{i_0} = g\g_1 (\g_1^{-1}\g_2)\cdot v_{i_0} = g\g_1 \cdot v_{i_0}.
    \end{equation*}
    This proves uniqueness of the vector in $\bigcup_i g\G\cdot v_i$ of norm $\leq \e_0$, if it exists, and concludes the proof.
   
    \end{proof}

    The following lemma verifies Proposition~\ref{prop:height function properties}\eqref{item:V is proper} relating the injectivity radius to $V_\b$.
    \begin{lem}
    For all $x$ in the unit neighborhood of $\Omega$, we have
    \begin{equation*}
        \inj(x)^{-1} \ll_\G V_\b^{\chi_\K/\b}(x),
    \end{equation*}
    where $\chi_\K$ is given in~\eqref{eq:multiple of simple root}. 
    \end{lem}
    
    \begin{proof}
    Let $x\in \Omega$ and set $\tilde{x}_0=Kx$.
    Let $x_0\in K\backslash G\cong \H^d_\K$ denote a lift of $\tilde{x}_0$.
    Then, $x_0$ belongs to the hull of the limit set of $\G$; cf.~Section~\ref{sec:prelims}.
    
    Since $\inj(\cdot)^{-1}$ and $V_\b$ are uniformly bounded above and below on the complement of the cusp neighborhoods inside $\Omega$, it suffices to prove the lemma under the assumption that $x_0$ belongs to some standard horoball $H$ based at a parabolic fixed point $p$.
    We may also assume that the lift $x_0$ is chosen so that $p$ is one of our fixed finite set of inequivalent parabolic points $\set{p_i}$.

    Geometric finiteness of $\G$ implies that there is a compact subset $\Kcal_p$ of $\partial \H^d_\K\backslash\set{p}$, depending only on the stabilizer $\G_p$ in $\G$, with the following property. Every point in the hull of the limit set is equivalent, under $\G_p$, to a point on the set of geodesics joining $p$ to points in $\Kcal_p$.
    Thus, after adjusting $x_0$ by an element of $\G_p$ if necessary, we may assume that $x_0$ belongs to this set.
    In particular, we can find $g\in G$ so that $x_0=Kg$ and $g$ can be written as $kau$ in the Iwasawa decomposition associated to $p$, for some $k\in K, a\in A_p$, and $u\in U_p$\footnote{The groups $A_p$ and $U_p$ were defined at the beginning of the section.} with the property that $\Ad(a)$ is contracting on $U_p$ and $u$ is of uniformly bounded size.

    Note that it suffices to prove the statement assuming the injectivity radius of $x$ is sufficiently small, depending only on the metric on $G$,
    while the distance of $x_0$ to the boundary of the cusp horoball $H_p$ is at least $1$.
    Now, let $\g\in \G$ be a non-trivial element such that $x_0\g$ is at distance at most $2\inj(x)$ from $x_0$. 
    Then, this implies that both $x_0$ and $x_0\g$ belong to $H_p$.
    Let $v=\g -\id$.
    In view of the discreteness of $\G$, we have that $\norm{v}\gg 1$.
    Since the exponential map is close to an isometry near the origin, we see that
    \begin{align*}
        \dist(g\g g^{-1},\id)
        \asymp \norm{g\g g^{-1}-\id} 
        =  \norm{gv g^{-1}} 
        = \norm{\Ad(au)(v)}\geq e^{\chi_\K\a(a)}\norm{\Ad(u)(v)},
    \end{align*}
    where $\chi_\K$ is given in~\eqref{eq:multiple of simple root} and we used $K$-invariance of the norm.
    Here, $\a$ is the simple root of $A_p$ in the Lie algebra of $U_p$ and $e^{\chi_\K\a(a)}$ is the smallest eigenvalue of $\Ad(a)$ on the Lie algebra of the parabolic group stabilizing $p$.
    Note that since $x_0$ belongs to $H_p$, $\a(a)$ is strictly negative.
    
    Recalling that $u$ belongs to a uniformly bounded neighborhood of identity in $G$ and that $\norm{v}\gg 1$, it follows that $\dist(g\g g^{-1},\id)\gg e^{\chi_\K\a(a)}$.
    Since $\g$ was arbitrary, this shows that the injectivity radius at $x$ satisfies the same lower bound.
    
   Finally, let $v_p\in\set{v_i}$ denote the vector fixed by $U_p$. Using the above Iwasawa decomposition, we see that $V_\b^{1/\b}(x) \geq \norm{a v_p}^{-1/\chi_\K} = e^{-\chi_p(a)/\chi_\K}$, where $\chi_p$ is the character on $A_p$ determined by $v_p$, cf.~\eqref{eq:chi_i}. This concludes the proof in view of~\eqref{eq:multiple of simple root} and the fact that $\chi_p=\chi_\K \a$.
    \end{proof}

    Finally, we record the following useful quantitative form of Lemma~\ref{lem:V is proper} which follows by similar arguments to those discussed in this section.
    We leave the details to the reader.
    \begin{lem}\label{lem:ht vs dist}
        For all $x$ in a bounded neighborhood of $\Omega$, we have $e^{\dist(x,o)} \ll V_\b(x)^{O_\b(1)}$.
    \end{lem}


\section{Shadow Lemmas, Convexity, and Linear Expansion}
\label{sec:linear expand}

The goal of this section is to prove Proposition~\ref{prop:linear expand} estimating the average rate of expansion of vectors with respect to leafwise measures.
This completes the proof of Theorem~\ref{thm:Margulis function}.

    \subsection{Proof of Proposition~\ref{prop:linear expand}}

    We may assume without loss of generality that $\norm{v}=1$.
    Let $W^+$ denote the highest weight subspace of $W$ for $A_+=\set{g_t:t>0}$. Denote by $\pi_+$ the projection from $W$ onto $W^+$.
    In our choice of representation $W$, the eigenvalue of $A_+$ in $W^+$ is $e^{\chi_\K t}$, , where $\chi_\K$ is given in~\eqref{eq:multiple of simple root}.
    It follows that
    \begin{align*}
         \frac{1}{\mu_x^u(N_1^+)}\int_{N_1^+} \norm{g_t n \cdot v}^{-\b/\chi_\K} \;d\mu^u_x(n)
         \leq e^{-\b  t}  \frac{1}{\mu_x^u(N_1^+)}\int_{N_1^+} \norm{\pi_+( n \cdot v)}^{-\b/\chi_\K} \;d\mu^u_x(n).
    \end{align*}
    Hence, it suffices to show that, for a suitable choice of $\b$, the integral on the right side is uniformly bounded, independently of $v$ and $x$ (but possibly depending on $\b$).
    
    For simplicity, set $\b_\K=\b/\chi_\K$.
    A simple application of Fubini's Theorem yields
    \begin{align*}
        \int_{N_1^+} \norm{\pi_+( n \cdot v)}^{-\b_\K} \;d\mu^u_x(n)
        =\int_0^\infty \mu_x^u \bigg( n\in N_1^+: \norm{\pi_+( n \cdot v)}^{\b_\K} \leq t^{-1} \bigg) \;dt.
    \end{align*}
    For $v\in W$, we define a polynomial map on $N^+$ by $n\mapsto p_v(n):= \norm{\pi_+(n\cdot v)}^2$ and set
    \begin{equation*}
       S(v,\e):= \set{n\in N^+: p_v(n)\leq \e}.
    \end{equation*}
    To apply Proposition~\ref{prop:doubling}, we wish to efficiently estimate the radius of a ball in $N^+$ containing the sublevel sets $  S\big(v, t^{-2/\b_\K}\big)  \cap N_1^+$.
    We have the following claim.
    \begin{claim}\label{claim:diameter estimate}
    There exists a constant $C_0>0$, such that, for all $\e>0$, the diameter of $S(v,\e)\cap N_1^+$ is at most $ C_0 \e^{1/4 \chi_\K}$.
    \end{claim}
    
    We show how this claim concludes the proof.
    By estimating the integral over $[0,1]$ trivially, we get
    \begin{align}\label{eq:splitting into large and small radii}
        \int_0^\infty \mu_x^u \bigg( n\in N_1^+:& \norm{\pi_+( n \cdot v)}^{\b_\K} \leq t^{-1} \bigg) \;dt 
        \leq 
         \mu_x^u(N_1^+)
        + \int_{1}^{\infty}
        \mu_x^u \bigg(
         S\big(v, t^{-2/\b_\K}\big) 
        \cap N_1^+\bigg)\;dt. 
    \end{align}

    Claim~\ref{claim:diameter estimate} implies that if $\mu_x^u\left(S(v,\e)\cap N_1^+\right)>0$ for some $\e>0$, then $S(v,\e) \cap N_1^+$ is contained in a ball of radius $2C_0\e^{1/4\chi_\K}$, centered at a point in the support of the measure $\mu_x^u|_{N_1^+}$.
   Recalling that $\b_\K=\b /\chi_\K$, we thus obtain
   \begin{align}\label{eq:from sublevel sets to balls}
    \int_{1}^{\infty}
        \mu_x^u \bigg(
         S\big(v, t^{-2/\b_\K}\big)
        \cap N_1^+\bigg)\;dt
        \leq 
        \int_{1}^{\infty}
        \sup_{n\in \mrm{supp}(\mu_x^u)\cap N_{1}^+} \mu_x^u \bigg( B_{N^+}\big(n,2C_0 t^{-1/2\b}\big)\bigg)\;dt,
    \end{align}
   where for $n\in N^+$ and $r>0$, $B_{N^+}(n,r)$ denotes the ball of radius $r$ centered at $n$.
    
   To estimate the integral on the right side of~\eqref{eq:from sublevel sets to balls}, we use the doubling results in Proposition~\ref{prop:doubling}.
   Note that if $n\in \mrm{supp}(\mu_x^u)$, then $(nx)^+$ belongs to the limit set $\L_\G$.
   Since $x\in N_1^-\Omega$ by assumption, this implies that $nx$ belongs to $N_2^-\Omega$ for all $n\in N_1^+$ in the support of $\mu_x^u$; cf.~Remark~\ref{rem:commutation of stable and unstable}.
   Hence, changing variables using~\eqref{eq:N equivariance} and applying Proposition~\ref{prop:doubling}, we obtain for all $n\in \mrm{supp}(\mu_x^u)\cap N_{1}^+$,
   \begin{align*}
       \mu_x^u\bigg( B_{N^+}\big(n,2C_0 t^{-1/2\b}\big)\bigg)
       = \mu_{nx}^u\bigg( B_{N^+}\big(\id,2C_0 t^{-1/2\b}\big)\bigg)
       \ll t^{-\Delta/2\b} \mu^u_{nx}(N_1^+).
   \end{align*}
   Moreover, for $n\in N_1^+$, we have, again by Proposition~\ref{prop:doubling}, that
   \begin{align*}
       \mu^u_{nx}(N_1^+) \leq \mu_x^u(N_2^+)\ll \mu_x^u(N_1^+).
   \end{align*}
   Put together, this gives
   \begin{align*}
       \int_{1}^{\infty}
        \sup_{n\in \mrm{supp}(\mu_x^u)\cap N_{1}^+} \mu_x^u \bigg( B_{N^+}\big(n,2C_0 t^{-1/2\b}\big)\bigg)\;dt
        \ll \mu_x^u(N_1^+) 
        \int_{1}^{\infty} t^{-\Delta/2\b} \;dt.
   \end{align*}
    The integral on the right side above converges whenever $\b<\Delta/2$, which concludes the proof.

\subsection{Preliminary facts}

    Towards the proof of Claim~\ref{claim:diameter estimate}, we begin by recalling the Bruhat decomposition of $G$.
    Denote by $P^-$ the subgroup $MAN^-$ of $G$.
    \begin{prop}[Theorem 5.15,~\cite{BorelTits}]\label{prop:Bruhat}
    Let $w\in G$ denote a non-trivial Weyl ``element" satisfying $wg_t w^{-1}=g_{-t}$. Then,
    \begin{equation}\label{eq:Bruhat}
        G = P^-N^+ \bigsqcup P^- w.
    \end{equation}
    \end{prop}
    
     We shall need the following result, which is yet another reflection in linear representations of $G$ of the fact that $G$ has real rank $1$.

    \begin{prop}
    \label{prop:pengyu} 
    Let $V$ be a normed finite dimensional representation of $G$, and $v_0\in V$ be any highest weight vector for $g_t$ ($t>0$) with weight $e^{\l t}$ for some $\l\geq 0$.
    Let $v$ be any vector in the orbit $G\cdot v_0$ and define 
    \begin{equation*}
        G(v, V^{<\l}(g_t)) = \set{g\in G:  \lim_{t\r\infty} \frac{\log \norm{g_t gv}}{t} <\l}.
    \end{equation*}
   Then,
    there exists $g_v \in G$ such that
    
    \begin{equation*}
        G(v,V^{<\l}(g_t)) \subseteq P^- g_v.
    \end{equation*}
    
    \end{prop}
    \begin{proof}
    Let $h\in G$ be such that $v=hv_0$ and
    let $g\in G(v, V^{<\l}(g_t))$.
    By the Bruhat decomposition, either $gh=pn$ for some $p\in P^-$ and $n\in N^+$, or $gh=pw$ for some $p\in P^-$ and $w$ being the long Weyl ``element".
    Suppose we are in the first case, and note that $N^+$ fixes $v_0$ since it is a highest weight vector for $g_t$.
    Moreover, $\Ad(g_t)(p)$ converges to some element in $G$ as $t$ tends to $\infty$.
    Since $g_t gv= e^{\l t} \Ad(g_t)(p)v_0$, we see that $\log\norm{g_tgv}/t \to \l$ as $t$ tends to $\infty$, thus contradicting the assumption that $g$ belongs to $G(v, V^{<\l}(g_t))$.
    Hence, $gh$ must belong to $P^-w$. This implies the conclusion by taking $g_v:= wh^{-1}$.

    \end{proof}

    The following immediate corollary is the form we use this result in our arguments.
    \begin{cor}\label{cor:1 point}
    Let the notation be as in Proposition~\ref{prop:pengyu}. Then, $N^+\cap G(v,V^{<\l}(g_t))$ contains at most one point.
    \end{cor}
    \begin{proof}

    Recall the Bruhat decomposition of $G$ in Proposition~\ref{prop:Bruhat}.
    Let $g_v\in G$ be as in Proposition~\ref{prop:pengyu} and suppose that $n_0\in P^-g_v \cap N^+$.
    Let $p_0\in P^-$ be such that $n_0=p_0g_v$.

    First, assume $g_v = p_v n_v$ for some $p_v\in P^-$ and $n_v\in N^+$.
    Then, $n_0=p_0p_v n_v$ and, hence, $n_0n_v^{-1}\in P^-\cap N^+=\set{\id}$.
    In particular, $n_0 = n_v$, and the claim follows in this case.
    
    Now assume that $g_v = p_v w$ for some $p_v\in P^-$, so that $n_0 = p_0p_v w\in P^-w\cap N^+$.
    This is a contradiction, since the latter intersection is empty as follows from the Bruhat decomposition.
    
    \end{proof}


\subsection{Convexity and proof of Claim~\ref{claim:diameter estimate}}

 Let $B_1\subset \mrm{Lie}(N^+)$ denote a compact convex set whose image under the exponential map contains $N_1^+$ and denote by $B_2$ a compact convex set containing $B_1$ in its interior.

    Define $\mf{n}^+_1$ to be the unit sphere in the Lie algebra $\mf{n}^+$ of $N^+$ in the following sense:
    \begin{equation*}
        \mf{n}^+_1 := \set{u\in\mf{n}^+: d_{N^+}(\exp(u),\id)=1},
    \end{equation*}
    where $d_{N^+}$ is the Cygan metric on $N^+$; cf.~Sec.~\ref{sec:Carnot}.
    Given $u,b\in \mf{n}^+$, define a line $\ell_{u,b}:\R \to \mf{n}^+$ by
    \begin{equation*}
        \ell_{u,b}(t) := tu + b,
    \end{equation*}
    and denote by $\Lcal$ the space of all such lines $\ell_{u,b}$ such that $u\in \mf{n}^+_1$.
    We endow $\Lcal$ with the topology inherited from its natural identification with its $\mf{n}_1^+\times \mf{n}^+$.
    Then, the subset $\Lcal(B_1)$ of all such lines such that $b$ belongs to the compact set $B_1$ is compact in $\Lcal$.
       
Recall that a vector $v\in W$ is said to be unstable if the closure of the orbit $G\cdot v$ contains $0$.
Highest weight vectors are examples of unstable vectors.
Let $\mc{N}$ denote the null cone of $G$ in $W$, i.e., the closed cone consisting of all unstable vectors.
Let $\Ncal_1\subset \Ncal$ denote the compact set of unit norm unstable vectors.
Note that, for any $v\in \Ncal$, the restriction of $p_v$ to any $\ell\in \Lcal$ is a polynomial in $t$ of degree at most that of $p_v$. 
 We note further that the function
    \begin{align*}
        \rho(v,\ell) := \sup \set{ p_v(\ell(t)) : \ell(t)\in B_2 }
    \end{align*}
 is continuous and non-negative on the compact space $\Ncal_1 \times \Lcal(B_1)$. We claim that
    \begin{align*}
        \rho_\star := \inf \set{\rho(v,\ell): (v,\ell)\in \Ncal_1 \times \Lcal(B_1) } 
    \end{align*}
    is strictly positive.
Indeed, by continuity and compactness, it suffices to show that $\rho$ is non-vanishing. Suppose not and let $(v,\ell)$ be such that $\rho(v,\ell)=0$. 
Since $B_1$ is contained in the interior of $B_2$, the intersection 
\begin{align*}
   I(\ell):= \set{t\in \R: \ell(t)\in B_2}
\end{align*}
is an interval (by convexity of $B_2$) with non-empty interior. Since $p_v(\ell(\cdot))$ is a polynomial vanishing on a set of non-empty interior, this implies it vanishes identically.
On the other hand, Corollary~\ref{cor:1 point} shows that $p_v$ has at most $1$ zero in all of $\mf{n}^+$, a contradiction.

Positivity of $\rho_\star$ has the following consequence.
Our choice of the representation $W$ implies that the degree of the polynomial $p_v$ is at most $4\chi_\K$, where $\chi_\K$ is given in~\eqref{eq:multiple of simple root}. This can be shown by direct calculation in this case.\footnote{In general, such a degree can be calculated from the largest eigenvalue of $g_t$ in $W$; for instance by restricting the representation to suitable subalgebras of the Lie algebra of $G$ that are isomoprhic to $\mf{sl}_2(\R)$ and using the explicit description of $\mf{sl}_2(\R)$ representations.}
By the so-called $(C,\a)$-good property (cf.~\cite[Proposition 3.2]{Kleinbock-Clay}), we have for all $\e>0$
\begin{equation*}
    |\set{t\in I(\ell): p_v(\ell(t)) \leq \e}| \leq C_d \left(\e/\rho_\star \right)^{1/4\chi_\K} |I(\ell)|,
\end{equation*}
where $C_d>0$ is a constant depending only on the degree of $p_v$, and $|\cdot|$ denotes the Lebesgue measure on $\R$.

To use this estimate, we first note that the length of the intervals $I(\ell)$ is uniformly bounded over $\Lcal(B_1)$. Indeed, suppose for some $u=(u_\a,u_{2\a}), b\in\mf{n}^+$ and $\ell=\ell_{u,b}\in \Lcal(B_1)$, $I(\ell)$ has endpoints $t_1 <t_2$ so that the points $\ell(t_i)$ belong to the boundary of $B_2$.
Recall that the Lie algebra $\mf{n}^+$ of $N^+$ decomposes into $g_t$ eigenspaces as $\mf{n}^+_\a\oplus \mf{n}^+_{2\a}$, where $\mf{n}^+_{2\a}=0$ if and only if $\K=\R$.
Set $x_1 = \ell(t_1)$ and $x_2=\ell(t_2)$.
Since $N^+$ is a nilpotent group of step at most $2$, the Campbell-Baker-Hausdorff formula implies that $\exp(x_2) \exp(-x_1) = \exp(Z)$, where $Z\in \mf{n}^+$ is given by
\begin{equation*}
    Z = x_2-x_1 + \frac{1}{2} [x_2,-x_1] = (t_2-t_1)u + \frac{1}{2}(t_2-t_1) [b,u].
\end{equation*}
Note that since $\mf{n}^+_{2\a}$ is the center of $\mf{n}^+$, $[b,u]=[b,u_\a]$ belongs to $\mf{n}^+_{2\a}$.
Hence, we have by~\eqref{eq:Carnot} that 
\begin{align*}
   d_{N^+}(\exp(x_1),\exp(x_2)) &=  \left( (t_2-t_1)^4 \norm{u_\a}^4 + (t_2^2-t_1^2)^2 \norm{u_{2\a}+ \frac{1}{2} [b,u]}^2 \right)^{1/4}.
\end{align*}
Since $\exp(u)$ is at distance $1$ from identity, at least one of $\norm{u_\a}$ and $\norm{u_{2\a}}$ is bounded below by $10^{-1}$.
Moreover, we can find a constant $\th\in (0,10^{-2})$ so that for all $b\in B_1$ and all $y_\a\in \mf{n}^+_\a$ with $\norm{y_\a}\leq \th$ such that $\norm{[b,y_\a]}\leq 10^{-2}$.
Together this implies that
\begin{align*}
    \min\set{t_2-t_1, (t_2^2-t_1^2)^{1/2} } \ll \diam{B_1},
\end{align*}
where $\diam{B_1}$ denotes the diameter of $B_1$.
This proves that $|I(\ell)|=t_2-t_1\ll 1$, where the implicit constant depends only on the choice of $B_1$.
We have thus shown that
\begin{equation}\label{eq:C alpha}
    |\set{t\in I(\ell): p_v(\ell(t)) \leq \e}| \ll \e^{1/4\chi_\K}.
\end{equation}

We now use our assumption that $v$ belongs to the $G$ orbit of a highest weight vector $v_0$.
Since $v_0$ is a highest weight vector, it is fixed by $N^+$. Hence, the Bruhat decomposition, cf.~\eqref{eq:Bruhat} with the roles of $P^-$ and $P^+$ reversed, implies that the orbit $G\cdot v_0$ can be written as
\begin{equation*}
    G\cdot v_0 = P^+\cdot v_0 \bigsqcup P^+w \cdot v_0,
\end{equation*}
where $w$ is the long Weyl ``element".
Recall that $P^+=N^+MA$, where $M$ is the centralizer of $A=\set{g_t}$ in the maximal compact group $K$. In particular, $M$ preserves eigenspaces of $A$ and normalizes $N^+$. Recall further that the norm on $W$ is chosen to be $K$-invariant.

First, we consider the case $v\in P^+w\cdot v_0$ and has unit norm.
For $v'\in W$, we write $[v']$ for its image in the projective space $\P(W)$.
Then, since $w\cdot v_0$ is a joint weight vector of $A$, we see that the image of $P^+w\cdot v_0$ in $\P(W)$ has the form $N^+M \cdot [w\cdot v_0]$.
Setting $v_1:= w\cdot v_0$, we see that
\begin{align}\label{eq:equivariance of sublevel sets}
    S(nm\cdot v_1,\e) = S(mv_1,\e)\cdot n^{-1} = \Ad(m^{-1})(S(v_1,\e))\cdot n^{-1},
\end{align}
where we implicitly used the fact that $M$ commutes with the projection $\pi_+$ and preserves the norm on $W$.
Since the metric on $N^+$ is right invariant under translations by $N^+$ and is invariant under $\Ad(M)$, the above identity implies that it suffices to estimate the diameter of $S(v_1,\e)\cap N^+_1$ in the case $v\in P^+w\cdot v_0$.
Similarly, in the case $v\in P^+\cdot v_0$, it suffices to estimate the diameter of $S(v_0,\e)\cap N_1^+$.

Let $\tilde{S}(v,\e)=\log S(v,\e)$ denote the pre-image of $S(v,\e)$ in the Lie algebra $\mf{n}^+$ of $N^+$ under the exponential map.
By Corollary~\ref{cor:1 point}, for any non-zero $v\in\mc{N}$, either $S(v,\e)$ is empty for all small enough $\e$, or there is a unique global minimizer of $p_v(\cdot)$ on $N^+$, at which $p_v$ vanishes.
    In either case, for any given $v\in\mc{N}\setminus \set{0}$ in the null cone, the set $\tilde{S}(v,\e)$ is convex for all small enough $\e>0$, depending on $v$.
    Let $s_0>0$ be such that $\tilde{S}(v,\e)$ is convex for $v\in \set{v_0,v_1}$ and for all $0\leq \e\leq s_0$.

    Fix some $v\in \set{v_0,v_1}$ and $\e\in [0,s_0]$.
     Suppose that $x_1\neq x_2\in \tilde{S}(v,\e)\cap B_1$.
     Let $r$ denote the distance $d_{N^+}(x_1,x_2)$.
    Let $u'=x_2-x_1$, $u=u'/r$ and $b=x_1$. Set $\ell=\ell_{u,b}$ and note that $\ell_{u,b}(0)=x_1$ and $\ell_{u,b}(r)=x_2$.
    Since $B_1$ is convex, the set $\tilde{S}(v,\e)\cap B_1$ is also convex.
    Hence, the entire interval $(0,r)$ belongs to the set on the left side of~\eqref{eq:C alpha} and, hence, that $r\ll \e^{1/4\chi_\K}$.
    Since $x_1$ and $x_2$ were arbitrary, this shows that the diameter of $\tilde{S}(v,\e)\cap B_1$ is $O(\e^{1/4\chi_\K})$ as desired.


\section{Anisotropic Banach Spaces and Transfer Operators}

\label{sec: spectral gap}

    In this section, we define the Banach spaces on which the transfer operator and resolvent associated to the geodesic flow have good spectral properties.

    The transfer operator, denoted $\Lcal_t$, acts on continuous functions as follows:
    \begin{equation}\label{def: to}
        \Lcal_t f := f\circ g_t , \qquad f\in C(X), t\in\R.
    \end{equation}

    For $z\in \C$, the resolvent $R(z):C_c(X)\to C(X)$ is defined formally as follows:
    \begin{equation*}
        R(z)f := \int_{0}^\infty e^{-z t} \Lcal_t f\;dt.
    \end{equation*}
    
    If $\G$ is not convex cocompact, 
    we fix a choice of $\b>0$ so that Theorem~\ref{thm:Margulis function} holds and set $V=V_\b$.
    If $\G$ is convex cocompact, we take $V=V_\b\equiv 1$ and we may take $\b$ as large as we like in this case.
    Note that the conclusion of Theorem~\ref{thm:Margulis function} holds trivially with this choice of $V$. In particular, we shall use its conclusion throughout the argument regardless of whether $\G$ admits cusps.

    Denote by $C_c^{k+1}(X)^M$ the subspace of $C_c^{k+1}(X)$ consisting of $M$-invariant functions, where $M$ is the centralizer of the geodesic flow inside the maximal compact group $K$.
    In particular, $C_c^{k+1}(X)^M$ is naturally identified with the space of $C_c^{k+1}$ functions on the unit tangent bundle of $\H^d_\K/\G$; cf.~Section~\ref{sec:prelims}.
    The following is the main result of this section.
    
    \begin{thm}[Essential Spectral Gap] \label{thm: resolvent spectrum}
    Let $k\in \N$ be given.
    Then, there exists a seminorm $\norm{\cdot}_k$ on $C_c^{k+1}(X)^M$, non-vanishing on functions whose support meets $\Omega$, and such that for every $z\in \C$, with $\Re(z)>0$, the resolvent $R(z)$ extends to a bounded operator on the completion of $C_c^{k+1}(X)^M$ with respect to $\norm{\cdot}_k$ and having spectral radius at most $1/\Re(z)$.
    Moreover, the essential spectral radius of $R(z)$ is bounded above by $1/(\Re(z)+\s_0)$, where
    \begin{equation*}
        \s_0 := \min\set{k,\b}.
    \end{equation*}
    In particular, if $\G$ is convex cocompact, we can take $\s_0=k$.
    
    \end{thm}
    
    By the completion of a topological vector space $V$ with respect to a seminorm $\norm{\cdot}$, we mean the Banach space obtained by completing the quotient topological vector space $V/W$ with respect to the induced norm, where $W$ is the kernel of $\norm{\cdot}$.

    The proof of Theorem~\ref{thm: resolvent spectrum} occupies Sections~\ref{sec: spectral gap} and~\ref{sec:ess radius}.

    \subsection{Anisotropic Banach Spaces}
    We construct a Banach space of functions on $X$ containing $C^\infty$ functions satisfying Theorem~\ref{thm: resolvent spectrum}.

    Given $r\in\N$, let $\Vcal_r^-$ denote the space of all $C^r$ vector fields on $N^+$ pointing in the direction of the Lie algebra $\mf{n}^-$ of $N^-$ and having norm at most $1$.
    More precisely, $\Vcal_r^-$ consists of all $C^r$ maps $v:N^+\to \mf{n}^-$, with $C^r$ norm at most $1$.
    Similarly, we denote by $\Vcal_r^0$ the set of $C^r$ vector fields $v:N^+\to \mf{a}:=\mrm{Lie}(A)$, with $C^r$ norm at most $1$.
    Note that if $\w\in \mf{a}$ is the vector generating the flow $g_t$, i.e.~$g_t = \exp(t \w)$, then each $v\in \Vcal_r^0$ is of the form $v(n)=\phi(n)\w$, for some $\phi\in C^r(N^+)$ such that $\norm{\phi}_{C^r(N^+)}\leq 1$.
    Define 
    \begin{equation*}
        \Vcal_r = \Vcal_r^-\cup \Vcal_r^0.
    \end{equation*}

    For $v \in \Lie(G)$, denote by $L_v$ the differential operator on $C^1(X)$ given by differentiation with respect to the vector field generated by $v$.
    Hence, for $\vp\in C^1(G/\G)$,
    \begin{equation*}
        L_v\vp(x) = \lim_{s\r0} \frac{\vp(\exp(sv)x)-\vp(x)}{s}.
    \end{equation*}
   
   For each $k\in \N$, we define a norm on $C^k(N^+)$ functions as follows. Letting $\Vcal^+$ be the unit ball in the Lie algebra of $N^+$, $0\leq \ell\leq k$, and $\phi\in C^k(N^+)$, we define $c_\ell(\phi)$ to be the supremum of $|L_{v_1}\cdots L_{v_\ell}(\phi)|$ over $N^+$ and all tuples $(v_1,\dots,v_\ell)\in(\Vcal^+)^\ell$. We define $\norm{\phi}_{C^k}$ to be $\sum_{\ell=0}^k c_\ell(\phi)/(2^\ell \ell!)$. One then checks that for all $\phi_1,\phi_2\in C^k(N^+)$, we have
   \begin{equation}\label{eq:Leibniz}
       \norm{\phi_1\phi_2}_{C^k} \leq \norm{\phi_1}_{C^k}\norm{\phi_2}_{C^k}.
   \end{equation} 
    
    Following~\cite{GouezelLiverani,GouezelLiverani2}, we define a norm on $C_c^{k+1}(X)$ as follows.
    Given $f\in C_c^{k+1}(X)$, $k,\ell$ non-negative integers, $\g=(\g_1,\dots,\g_\ell)\in \Vcal_{k+\ell}^\ell$ (i.e.~$\ell$ tuple of $C^{k+\ell}$ vector fields) and $x\in X$, define
    \begin{equation}\label{eq: ekl gamma x}
        e_{k,\ell,\g}(f;x) := 
        \frac{1}{V(x)}  \sup
        \frac{1}{\mu^u_x\left(\wu{1}\right)}
        \left|
        \int_{\wu{1}} \phi(n) L_{\g_1(n)}\cdots 
        L_{\g_\ell(n)}(f)(g_snx) \;d\mu^u_x(n)\right|,
    \end{equation}
    where the supremum is taken over all $s\in [0,1]$ and all functions $\phi \in C^{k+\ell}(N_1^+)$ which are compactly supported in the interior of $N_1^+$ and having $\norm{\phi}_{C^{k+\ell}(N_1^+)}\leq 1$.

    For $\g\in \Vcal_{k+\ell+1}^\ell$,
    we define $e'_{k,\ell,\g}(f;x)$ analogously to $e_{k,\ell,\g}(f;x)$, but where we take $s=0$ and take the supremum over $\phi\in C^{k+\ell+1}(N_{1/10}^+)$ instead\footnote{The restriction on the supports allows us to handle non-smooth conditional measures; cf.~proof of Prop.~\ref{prop:compact embedding}.} of $C^{k+\ell}(N_1^+)$.
    Given $r>0$, set
    \begin{equation}
        \Omega_r^- := N_r^-\Omega. 
    \end{equation}
    We define
    \begin{equation}\label{eq: ekl}
         e_{k,\ell,\g}(f) := \sup_{x\in \Omega_1^-} e_{k,\ell,\g}(f;x)
         ,\qquad
          e_{k,\ell}(f) = \sup_{\g\in\Vcal_{k+\ell}^\ell} e_{k,\ell,\g}(f).
    \end{equation}
    Finally, we define $\norm{f}_k$ and $\norm{f}'_k$ by
    \begin{equation}\label{eq: norms}
        \norm{f}_{k} :=  \max_{0\leq \ell \leq k} e_{k,\ell}(f), \qquad
        \norm{f}'_k := \max_{0\leq \ell \leq k-1} \sup_{\g\in\Vcal_{k+\ell+1}^\ell, x\in \Omega_{1/2}^-} e'_{k,\ell,\g}(f;x).
    \end{equation}
    Note that the (semi-)norm $\norm{f}'_k$ is weaker than $\norm{f}_k$ since we are using more regular test functions and vector fields, and we are testing fewer derivatives of $f$.

    \begin{remark}\label{rem:seminorm}
    Since the suprema in the definition of $\norm{\cdot}_k$ are restricted to points on $\Omega_1^-$, $\norm{\cdot}_k$ defines a seminorm on $C_c^{k+1}(X)^M$.
    Moreover, since $\Omega_1^-$ is invariant by $g_t$ for all $t\geq 0$, the kernel of this seminorm, denoted $W_k$, is invariant by $\Lcal_t$.
    The seminorm $\norm{\cdot}_k$ induces a norm on the quotient $C_c^{k+1}(X)^M/W_k$, which we continue to denote $\norm{\cdot}_k$.
    \end{remark}
    
    \begin{definition}
    We denote by $\Bcal_{k}$ the Banach space given by the completion of the quotient $C_c^{k+1}(X)^M/W_k$ with respect to the norm $\norm{\cdot}_k$, where $C_c^{k+1}(X)^M$ denotes the subspace consisting of $M$-invariant functions.
    
    \end{definition}
    
    Note that since $\norm{\cdot}'_k$ is dominated by $\norm{\cdot}_k$, $\norm{\cdot}'_k$ descends to a (semi-)norm on $C_c^{k+1}(X)^M/W_k$ and extends to a (semi-)norm on $\Bcal_k$, again denoted $\norm{\cdot}_k'$.

    The following is a reformulation of Theorem~\ref{thm: resolvent spectrum} in the above setup.
    \begin{thm} \label{thm:resolvent spectrum2}
    For all $z\in \C$, with $\Re(z)>0$, and for all $k\in \N$, the operator $R(z)$ extends to a bounded operator on $\Bcal_k$ with spectral radius at most $1/\Re(z)$.
    Moreover, the essential spectral radius of $R(z)$ acting on $\Bcal_k$ is bounded above by $1/(\Re(z)+\s_0)$, where
    \begin{equation*}
        \s_0 := \min\set{k,\b}.
    \end{equation*}
    In particular, if $\G$ is convex cocompact, we can take $\s_0=k$.
    
    \end{thm}
    
    \subsection{Hennion's Theorem and Compact Embedding}

Our key tool in estimating the essential spectral radius is the following refinement of Hennion's Theorem, based on Nussbaum's formula.
 \begin{thm}[cf.~\cite{hennion} and Lemma 2.2 in~\cite{BardetGouezelKeller}]
    \label{thm: hennion}
    Suppose that $\Bcal$ is a Banach space with norm $\norm{\cdot}$ and that $\norm{\cdot}'$ is a seminorm on $\Bcal$ so that the unit ball in $(\Bcal,\norm{\cdot})$ is relatively compact in $\norm{\cdot}'$.
    Suppose $R $ is a bounded operator on $\Bcal$ such that for some $n\in \N$, there exist constants $r>0$ and $C>0$ satisfying
    \begin{equation}\label{eq: hennion ineq}
        \norm{R^nv} \leq r^n \norm{v} + C \norm{v}',
    \end{equation}
    for all $v\in\Bcal$.
    Then, the essential spectral radius of $R$ is at most $r$.
    \end{thm}

    The following proposition, roughly speaking, verifies the compactness assumption of Theorem~\ref{thm: hennion} for $\norm{\cdot}_k$ and $\norm{\cdot}'_k$.
   
    \begin{prop}\label{prop:compact embedding}
    Let $K\subseteq X$ be such that
    \begin{equation*}
        \sup \set{V(x): x\in K} <\infty.
    \end{equation*}
    Then, every sequence $f_n\in C_c^{k+1}(X)^M$, such that $f_n$ is supported in $K$ and has $\norm{f_n}_k \leq 1$ for all $n$, admits a Cauchy subsequence in $\norm{\cdot}'_k$.
    
    \end{prop}

    \subsection{Proof of Proposition~\ref{prop:compact embedding}}

    We adapt the arguments in~\cite{GouezelLiverani,GouezelLiverani2} with the main difference being that we bypass the step involving integration by parts over $N^+$ since our conditionals $\mu_x^u$ need not be smooth in general.
      The idea is to show that since all directions in the tangent space of $X$ are accounted for in the definition of $\norm{\cdot}_k$ (differentiation along the weak stable directions and integration in the unstable directions), one can estimate $\norm{\cdot}'_k$ using finitely many coefficients $e_{k}(f;x_i)$.
    More precisely, we first show that there exists $C\geq 1$ so that for all sufficiently small $\e>0$, there exists a finite set $\Xi \subset \Omega$ so that for all $f\in C_c^{k+1}(X)^M$, which is supported in $K$,
    \begin{equation}\label{eq: finitely many u pieces}
        \norm{f}'_k \leq C \e \norm{f}_{k} + C \sup     \int_{\wu{1}} \phi L_{v_1}\cdots L_{v_\ell} f\;d\mu_{x_i}^u,
    \end{equation}
    where the supremum is over all $0\leq \ell\leq k-1$, all $(v_1,\dots,v_\ell)\in \Vcal_{k+\ell+1}^\ell$, all functions $\phi\in C^{k+\ell+1}(N_2^+)$ with $\norm{\phi}_{C^{k+\ell+1}}\leq 1$
    and all $x_i\in \Xi$.
    
    First, we show how~\eqref{eq: finitely many u pieces} completes the proof. Let $f_n\in C_c^{k+1}(K)$ be as in the statement.
    Let $\e>0$ be small enough so that~\eqref{eq: finitely many u pieces} holds.
    Since $C^{k+\ell+1}(N_2^+)$ is compactly included inside $C^{k+\ell}(N_2^+)$, we can find a finite collection $\set{\phi_j:j} \subset C^{k+\ell}(N_2^+)$ which is $\e$ dense in the unit ball of $C^{k+\ell+1}(N_2^+)$.
    Similarly, we can find a finite collection of vector fields  $\set{(v_1^m,\dots,v^m_\ell):m} \subset \Vcal_{k+\ell}^\ell$ which is $\e$ dense in $\Vcal_{k+\ell+1}^\ell$ in the $C^{k+\ell+1}$ topology.
    Then, we can find a subsequence, also denoted $f_n$, so that the finitely many quantities
    
    \begin{align*}
        \set{\int_{\wu{1}} \phi_j L_{v^m_1}\cdots L_{v^m_\ell} f_n\;d\mu_{x_i}^u : i,j,m }
    \end{align*}
    converge.
    Together with~\eqref{eq: finitely many u pieces}, this implies that
    \begin{equation*}
        \norm{f_{n_1}-f_{n_2}}_k' \ll \e,
    \end{equation*}
    for all large enough $n_1,n_2$, where we used the fact that $\norm{f_n}_k\leq 1$ for all $n$.
    As $\e$ was arbitrary, one can extract a Cauchy subequence by a standard diagonal argument.
     Thus, it remains to prove~\eqref{eq: finitely many u pieces}.
     
     Fix some $f\in C_c^{k+1}(X)^M$ which is supported inside $K$.
    Let an arbitrary tuple $\g=(v_1,\dots,v_\ell)\in \Vcal_{k+\ell+1}^\ell$ be given and set $$\psi = L_{v_1}\cdots L_{v_\ell}f.$$
    Let $\phi\in C^{k+\ell+1}(N_{1/10}^+)$ and write $Q= N_{1/10}^+$. To estimate $e'_{k,\ell,\g}(f;z)$ using the right side of~\eqref{eq: finitely many u pieces}, we need to estimate integrals of the form
    \begin{equation} \label{eq:1 point 1 function}
       \frac{1}{V(z)} \frac{1}{\mu^u_{z}\left(\wu{1}\right)} \int_{\wu{1}} \phi(n) \psi(nz) \;d\mu_{z}^u(n),
    \end{equation}
    for all $z\in \Omega_{1/2}^-$. 
    
    Denote by $\rho:X\to [0,1]$ a smooth function which is identically one on the $1$-neighborhood $\Omega^1$ of $\Omega$ and vanishes outside its $2$-neighborhood.
    Note that if $f$ is supported outside of $\Omega^1$, then the integral in~\eqref{eq:1 point 1 function} vanishes for all $z$ and the estimate follows. 
    The same reasoning implies that
    \begin{align*}
        \norm{\rho f}_k = \norm{f}_k, \qquad \norm{\rho f}'_k = 
        \norm{f}'_k.
    \end{align*}
    Hence, we may assume that $f$ is supported inside the intersection of $K$ with $\Omega^1$.
    In particular, for the remainder of the argument, we may replace $K$ with (the closure of) its intersection with $\Omega^1$.

    This discussion has the important consequence that we may assume that $K$ is a compact set in light of Proposition~\ref{prop:height function properties}.
    Let $K_1$ denote the $1$-neighborhood of $K$ and fix some $z\in K_1\cap \Omega_{1/2}^-$.
    By shrinking $\e$, we may assume it is smaller than the injectivity radius of $K_1$.
    Hence, we can find a finite cover $B_1,\dots,B_M$ of $K_1\cap\Omega^-_{1/2}$ with flow boxes of radius $\e$ and with centers $\Xi:=\set{x_i}\subset \Omega^-_{1/2}$.

   \textbf{Step 1:}
   We first handle the case where $z$ belongs to the same unstable manifold as one of the $x_i$'s.
    Note that we may assume that $Q$ intersects the support of $\mu_z^u$ non-trivially, since otherwise the integral in question is $0$.
    Let $u\in Q$ be one point in this intersection and let $x=uz$. 
    Thus, by~\eqref{eq:N equivariance}, we get
    \begin{align*}
          \int_{\wu{1}} \phi(n) \psi(nz) \;d\mu_{z}^u(n)
          = \int_{Q} \phi(n) \psi(nz) \;d\mu_{z}^u(n)
         = \int_{Qu^{-1}}\phi(nu) \psi(nx) \;d\mu_{x}^u(n).
    \end{align*}
    Let $\phi_u(n):= \phi(nu)$.
    Then, $\phi_u$ is supported inside $Qu^{-1}$.
    Moreover, since $u\in Q$, $Q_u:=Qu^{-1}$ is a ball of radius $1/10$ containing the identity element.
    Hence, $Qu^{-1}\subset N_1^+$ and, thus,
    \begin{equation*}
       \int_{Q_u}\phi(nu) \psi(nx) \;d\mu_{x}^u(n)
        = \int_{\wu{1}}\phi_u(n) \psi(nx) \;d\mu_{x}^u(n).
    \end{equation*}

    Fix some $\e>0$. We may assume that $\e<1/10$. Note that $x$ belongs to the $1$-neighborhood of $K$.
    Then, $x=u_2^{-1} x_i$ for some $i$ and some $u_2\in N^+_\e$, by our assumption in this step that $z$ belongs to the unstable manifold of one of the $x_i$'s.
    By repeating the above argument with $z$, $u$, $x$, $Q$ and $\phi$ replaced with $x$, $u_2$, $x_i$, $Q_u$ and $\phi_u$ respectively, we obtain 
     \begin{align*}
         \int_{\wu{1}} \phi_u(n) \psi(nx) \;d\mu_{x}^u(n)
        = \int_{Q_uu_2^{-1}}\phi_u(nu_2) \psi(nx_i) \;d\mu_{x_i}^u(n).
    \end{align*}
    Note that $Q_u$ is contained in the ball of radius $1/5$ centered around identity.
    Since $u_2\in N_\e^+$ and $\e<1/10$, we see that $Q_uu_2^{-1}\subset N_1^+$.
    It follows that
    \begin{equation*}
       \int_{\wu{1}} \phi_u(n) \psi(nx_i) \;d\mu_{x_i}^u(n)
        = \int_{N_1^+}\phi_{u_2u}(n) \psi(nx_i) \;d\mu_{x_i}^u(n),
    \end{equation*}
    where $\phi_{u_2u}(n)=\phi_u(nu_2)=\phi(nu_2 u)$.
    The function $\phi_{u_2u}$ satisfies $\norm{\phi_{u_2u}}_{C^{k+\ell+1}}=\norm{\phi}_{C^{k+\ell+1}}\leq 1$.
    Finally, let $\vp_1,\vp_2 :N^+\to [0,1]$ be non-negative bump $C^0$ functions where $\vp_1\equiv 1$ on $N_{1}^+$ and while $\vp_2$ is equal to $1$ at identity and its support is contained inside $N_1^+$. 
    Since $y\mapsto \mu_y^u(\vp_i)$ is continuous for $i=1,2$, by~\cite[Lemme 1.16]{Roblin}, and is non-zero on $\Omega_1^-$, we can find, by compactness of $K_1$, a constant $C\geq 1$, depending only on $K$ (and the choice of $\vp_1,\vp_2$), such that
    \begin{equation}\label{eq:bounded conditionals}
      1/C\leq \mu^u_{y}\left(\wu{1}\right) \leq C , \qquad \forall y\in K_1 \cap \Omega_1^-.
    \end{equation}
    Hence, recalling that $\psi=L_{v_1}\cdots L_{v_\ell}f$ and that $V(z)\gg 1$, we conclude that the integral in~\eqref{eq:1 point 1 function} is bounded by the second term in~\eqref{eq: finitely many u pieces}.

    \textbf{Step 2:} We reduce to the case where $z$ is contained in the unstable manifolds of the $x_i$'s.
    Let $i$ be such that $z\in B_i$. 
    Set $z_1 = z$ and let $z_0\in (N_\e^+\cdot x_i)$ be the unique point in the intersection of $N_\e^+\cdot x_i$ with the local weak stable leaf of $z_1$ inside $B_i$.
    Let $p_1^-\in P^-:=MAN^-$ be an element of the $\e$ neighborhood of identity $P^-_\e$ in $P^-$ such that $z_1=p_1^-z_0$.

    We will estimate the integral in~\eqref{eq:1 point 1 function} using integrals at $z_0$.
    The idea is to perform weak stable holonomy between the local strong unstable leaves of $z_0$ and $z_1$.
    To this end, we need some notation.
    Let $Y\in \mf{p}^-$ be such that $p_1^-=\exp(Y)$ and set 
    \[p_t^-=\exp(t Y),\qquad z_t = p_t^-z_0,\]
    for $t\in [0,1]$.
    Let us also consider the following maps
    $  u_t^+: N_1^+ \to N^+$ and $\tilde{p}_t^-: N_1^+ \to P^-  $
    defined by the following commutation relations
    \begin{align*}
        np_t^- = \tilde{p}_t^-(n) u_t^+(n), \qquad \forall n\in N_1^+.
    \end{align*}
    Recall we are given a test function $\phi\in C^{k+\ell+1}(N_{1/10}^+)$. We can rewrite the integral we wish to estimate as follows:
    \begin{align*}
        \int_{\wu{1}} \phi(n) \psi(nz_1) \;d\mu_{z_1}^u(n) 
        =
        \int_{\wu{1}}\phi(n)\psi(n p^-_1 z_0)
        \;d\mu_{z_1}^u(n) 
        =\int \phi(n)\psi( \tilde{p}^-_1(n) u_1^+(n)    z_0)
        \;d\mu_{z_1}^u(n).
    \end{align*}

    Let $U_t^+\subset N^+$ denote the image of $u^+_t$.
    Note that if $\e$ is small enough, $U_t^+\subseteq N_2^+$ for all $t\in[0,1]$.
    We may further assume that $\e$ is small enough so that the map $u_t^+$ is invertible on $U_t^+$ for all $t\in [0,1]$ and write $\phi_t:= \phi\circ (u_t^+)^{-1}$. For simplicity, set
    \begin{equation*}
        p_t^-(n) := \tilde{p}_t^- ( (u_t^+)^{-1}(n)).
    \end{equation*}
    Write $m_t(n)\in M$ and $b_t^-(n)\in AN^-$ for the components of $p^-_t(n)$ along $M$ and $AN^-$ respectively so that
    \[p_t^-(n) = m_t(n) b_t^-(n).\]
    We denote by $J_t$ the Radon-Nikodym derivative of the pushforward of $\mu_{z_1}^u$ by $u_t^+$ with respect to $\mu_{z_t}^u$; cf.~\eqref{eq:stable equivariance} for an explicit formula.
    Thus, changing variables using $n\mapsto u^+_1(n)$, and using the $M$-invariance of $f$, we obtain
    \begin{align*}
        \int_{\wu{1}} \phi(n) \psi(nz_1) \;d\mu_{z_1}^u 
        =\int \phi_1(n)\psi( p^-_1(n) n z_0) J_1(n)
        \;d\mu_{z_0}^u
        = \int \phi_1(n)\tilde{\psi}_1( b^-_1(n) n z_0) J_1(n)
        \;d\mu_{z_0}^u,
    \end{align*}
    where $\tilde{\psi}_t$ is given by
    \begin{align*}
        \tilde{\psi}_t :=  L_{\tilde{v}^t_1}\cdots L_{\tilde{v}^t_\ell}f, \qquad \tilde{v}^t_i(n):= \Ad(m_t((u_t^+)^{-1}(n)))(v_i((u_t^+)^{-1}(n))).
    \end{align*}
    Here, we recall that $\Ad(M)$ commutes with $A$ and normalizes $N^-$ so that $\tilde{v}^t_i$ is a vector field with the same target as $v_i$.
    
    Let $\mf{b}^-$ denote the Lie algebra of $AN^-$ and denote by $\tilde{w}'_t:U_t^+\times [0,1]\to \mf{b}^-$ the vector field tangent to the paths defined by $b_t^-$.
    More explicitly, $\tilde{w}'_t$ is given by the projection of $tY$ to $\mf{b}^-$.
    Denote $\tilde{w}_t(n) := \Ad(m_t(n))(\tilde{w}'_t(n))$.
    Then, using the $M$-invariance of $f$ as above once more, we can write
    \begin{align*}
        \psi( b^-_1(n) n  z_0)-\psi(nz_0)) = \int_0^1 \frac{\partial}{\partial t} \tilde{\psi}_t( b^-_t(n) n    z_0) \;dt
        =\int_0^1 L_{\tilde{w}_t}( \tilde{\psi}_t)( p^-_t(n) n z_0)  \;dt.
    \end{align*}
    To simplify notation, let us set $w_t = \tilde{w}_t\circ u^+_t$, and
    \begin{equation*}
        F_t :=  L_{\tilde{v}^t_1\circ u^+_t}\cdots L_{\tilde{v}^t_\ell\circ u^+_t} f .
    \end{equation*}
    Using a reverse change of variables, we obtain for every $t\in [0,1]$ that
    \begin{align*}
        \int \phi_1(n) L_{\tilde{w}_t}(\tilde{\psi}_t)( p^-_t(n) n z_0) J_1(n)
        \;d\mu_{z_0}^u
        &=\int (\phi_1J_1)\circ u^+_t (n) L_{w_t}(F_t)( \tilde{p}^-_t(n) u^+_t(n) z_0) J_t^{-1}(n)
        \;d\mu_{z_t}^u
        \nonumber\\
        &=\int (\phi_1J_1)\circ u^+_t (n)\cdot 
       L_{w_t}(F_t)(n  z_t)\cdot J_t^{-1}(n)
        \;d\mu_{z_t}^u(n),
    \end{align*}
    where we used the identities $\tilde{p}^-_t(n) u^+_t(n) =np_t^-$ and $z_t=p_t^-z_0$.
    Let us write 
    \[
        \Phi_t(n):= (\phi_1J_1)\circ u^+_t (n)\cdot J_t^{-1}(n),
    \]
    which we view as a test function\footnote{The Jacobians are smooth maps as they are given in terms of Busemann functions; cf.~\eqref{eq:stable equivariance}.}.
    Hence, the last integral above amounts to integrating $\ell+1$ weak stable derivatives of $f$ against a $C^{k+\ell}$ function.
    Moreover, since $\phi$ is supported in $N_{1/10}^+$, we may assume that $\e$ is small enough so that $\Phi_t$ is supported in $N_1^+$ for all $t\in[0,1]$, and meets the requirements on the test functions in the definition of $\norm{f}_k$.
    Since $z=z_1$ belongs to $\Omega^-_{1/2}$ by assumption, we may further shrink $\e$ if necessary so that the points $z_t$ all\footnote{This type of estimate is the reason we use stable thickenings $\Omega_r^-$ of $\Omega$ in the definition of the norm instead of $\Omega$.} belong to $\Omega_1^-$.
    Thus, decomposing $w_t$ into its $A$ and $N^-$ components, and noting that $\norm{w_t}\ll \e$, we obtain the estimate
    \begin{align}\label{eq:z_t terms}
        \int \Phi_t(n)\cdot 
        L_{w_t}(F_t)(n  z_t)
        \;d\mu_{z_t}^u(n) \ll \e \norm{f}_k V(z_t) \mu_{z_t}^u(N_1^+).
    \end{align}

    To complete the argument, note that the integral we wish to estimate satisfies
    \begin{align}\label{eq:split into z_0 and z_t}
        \int_{\wu{1}} \phi(n) \psi(nz_1) \;d\mu_{z_1}^u 
        = \int (\phi_1 J_1)(n)\psi( n z_0)
        \;d\mu_{z_0}^u
        + \int_0^1 \int \Phi_t(n)\cdot 
        L_{w_t}(F_t)(n  z_t)
        \;d\mu_{z_t}^u(n)\;dt.
    \end{align}
     Moreover, recall that $z_0$ belongs to the same unstable manifold as some $x_i\in \Xi$.
     Additionally, since $\phi$ is supported in $N_{1/10}^+$, by taking $\e$ small enough, we may assume that $\phi_1$ is supported inside $N_{1/5}^+$. 
    Hence, arguing similarly to Step 1, viewing $\phi_1 J_1$ as a test function, we can estimate the first term on the right side above using the right side of~\eqref{eq: finitely many u pieces}. 
    
    The second term in~\eqref{eq:split into z_0 and z_t} is also bounded by the right side of~\eqref{eq: finitely many u pieces}, in view of~\eqref{eq:z_t terms}.
    Here we are using that $y\mapsto \mu_y^u(N_1^+)$ and $y\mapsto V(y)$ are uniformly bounded as $y$ varies in the compact set $K_1$; cf.~\eqref{eq:bounded conditionals}.
    This completes the proof of~\eqref{eq: finitely many u pieces} in all cases, since $\phi$ and $z$ were arbitrary.


    \section{The Essential Spectral Radius of Resolvents}
    \label{sec:ess radius}
    In this section, we study the operator norm of the transfer operators $\Lcal_t$ and the resolvents $R(z)$ on the Banach spaces constructed in the previous section. These estimates constitute the proof of Theorem~\ref{thm: resolvent spectrum}.
   With these results in hand, we deduce Theorem~\ref{thm:intro meromorphic} at the end of the section.

 \subsection{Strong continuity of transfer operators}   
    Recall that a collection of measurable subsets $\set{B_i}$ of a space $Y$ is said to have \textit{intersection multiplicity} bounded by a constant $C\geq 1$ if for all $i$, the number of sets $B_j$ in the collection that intersect $B_i$ non-trivially is at most $C$.
    In this case, one has
    \[\sum_i \chi_{B_i}(y) \leq C \chi_{\cup_i B_i}(y), \qquad \forall y\in Y.\] 
    
    The following lemma implies that the operators $\Lcal_t$ are uniformly bounded on $\Bcal_k$ for $t\geq 0$.
    \begin{lem}\label{lem: equicts}
     For every $k,\ell\in \N\cup\set{0}$, $\g\in\Vcal_{k+\ell}^\ell$, $t\geq 0$, and $x\in\Omega_1^-$,
    \begin{equation*}
        e_{k,\ell,\g}(\Lcal_tf;x) \ll_\b e^{- \e(\g) t} e_{k,\ell,\g}(f)(e^{-\b t}+1/V(x)) ,
    \end{equation*}
    where $\e(\g) \geq 0$ is the number of stable derivatives determined by $\g$.
    In particular, $\e(\g)=0$ if and only if $\ell=0$ or all components of $\g$ point in the flow direction.
    
    \end{lem}
    
    \begin{proof}
     Fix some $x\in \Omega$ and $\g=(v_1,\dots,v_\ell)\in \Vcal_{k+\ell}^\ell$.
    Since the Lie algebra of $N^-$ has the orthogonal decomposition $\mf{g}_{-\a}\oplus \mf{g}_{-2\a}$, where $\a$ is the simple positive root in $\mf{g}$ with respect to $g_t$, we have that $g_t$ contracts the norm of each stable vector $v\in \Vcal_{k+\ell}^-$ by at least $e^{-t}$.
    It follows that for all $v\in \Vcal_{k+\ell}^-$ and $w\in \Vcal_{k+\ell}^0$, 
    \begin{equation}\label{eq:weak stable derivatives}
        L_{v}(\Lcal_tf)(x) = e^{-t} L_{\bar{v}^t} (f)(g_t x), \qquad
        L_w(\Lcal_t f)(x) = L_w (f)(g_t x),
    \end{equation}
    for all $f\in C^{k+1}(X)^M$, where $v^t =\mrm{Ad}(g_t)(v)$ and $\bar{v}^t=e^tv^t$ if $v\in \Vcal_{k+\ell}^-$ and $\bar{v}^t = v^t$ if $v\in \Vcal^0_{k+\ell}$. Moreover, we have
    \begin{equation*}\label{eq:uniform hyperbolicity}
        \norm{v^t} \leq e^{-t} \norm{v} \leq e^{-t}.
    \end{equation*}
   Let $\phi$ be a test function, $f\in C^{k+1}(X)^M$, and set $\psi = L_{\bar{v}^t_1}\cdots L_{\bar{v}^t_\ell}f$.
   Then, we get
     \begin{align*}
         \left|\int_{N_1^+} \phi(n) L_{v_1}\cdots L_{v_\ell}(\Lcal_t f)( nx) \;d\mu_x^u(n)\right|
         &= e^{-\e(\g) t} \left|\int_{N_1^+} 
         \phi( n )
         \psi( g_t n x) \;d\mu_{ x}^u(n)\right|.
    \end{align*}

    Let $\set{\rho_i:i\in I}$ be a partition of unity of $\mrm{Ad}(g_t) (\wu{1})$ so that each $\rho_i$ is non-negative, $C^\infty$, and supported inside some ball of radius $1$ centered inside $\Ad(g_t)(N_1^+)$.
    Such a partition of unity can be chosen so that the supports of $\rho_i$ have a uniformly bounded multiplicity\footnote{Note that the analog of the classical Besicovitch covering theorem fails to hold for $N^+$ with the Cygan metric when $N^+$ is not abelian; cf.~\cite[pg.~17]{KoranyiReimann}. Instead, such a partition of unity can be constructed using the Vitali covering lemma with the aid of the right invariance of the Haar measure. To obtain a uniform bound on the multiplicity here and throughout, it is important that such an argument is applied to balls with uniformly comparable radii; cf.~Prop.~\ref{prop:cusp adapted partition} where a suitable substitute to bounded multiplicity is used when the radii are not of comparable size.},
    depending only on $N^+$.
    Denote by $I(\L)$ the subset of indices $i\in I$ such that there is $n_i\in N^+$ in the support of the measure $\mu_{g_tx}^u$ with the property that the support of $\rho_i$ is contained in $N_{1}^+\cdot n_i $.
    In particular, for $i\in I\setminus I(\L)$, $\rho_i\mu_{g_tx}^u$ is the $0$ measure.
    Then, using~\eqref{eq:g_t equivariance} to change variables, we obtain
    \begin{align*}
        \int_{\mrm{Ad}(g_t)(N_1^+)} 
         \phi(g_{-t} n g_t)
         \psi(n g_t x) \;d\mu_{g_t x}^u(n)
         = \sum_{i\in I(\L)} \int_{N_1^+\cdot n_i}  
         \rho_i(n) \phi(g_{-t} n g_t )
         \psi(n g_t x) \;d\mu_{g_t x}^u(n).
    \end{align*}
    Setting $x_i=n_ig_tx$ and changing variables using~\eqref{eq:N equivariance}, we obtain 
    \begin{align}\label{eq:apply equivariance}
         \int_{N_1^+} \phi(n) \psi(g_t nx) \;d\mu_x^u(n)= e^{-\d t} \sum_{i\in I(\L)} \int_{N_1^+}  
         \rho_i(nn_i) \phi(g_{-t} nn_i g_t )
         \psi(nx_i ) \;d\mu_{ x_i}^u(n).
    \end{align}
    
    The bounded multiplicity of the partition of unity implies that the balls $N_1^+\cdot n_i$ have intersection multiplicity bounded by a constant $C_0$, depending only on $N^+$.
    Enlarging $C_0$ if necessary, we may also choose $\rho_i$ so that $\norm{\rho_i}_{C^{k+\ell}} \leq C_0$.
    In particular, $C_0$ is independent of $t$ and $x$.

    For each $i$, let $\bar{\phi}_i(n)= \rho_i(nn_i) \phi(g_{-t} nn_i g_t )$.
    Since $\rho_i$ is chosen to be supported inside $N_1^+n_i$, then $\bar{\phi}_i$ is supported inside $N_1^+$.
    Moreover, since $\rho_i$ is $C^\infty$, $\bar{\phi}_i$ is of the same differentiability class as $\phi$.
    Since conjugation by $g_{-t}$ contracts $N^+$, we see that $\norm{\phi\circ \mrm{Ad}(g_{-t})}_{C^{k+\ell}} \leq \norm{\phi}_{C^{k+\ell}}\leq 1$ (note that the supremum norm of $\phi\circ \mrm{Ad}(g_{-t})$ does not decrease, and hence we do not gain from this contraction).
    Hence, since $\norm{\rho_i}_{C^{k+\ell}}\leq C_0$,~\eqref{eq:Leibniz} implies that
    $\norm{\bar{\phi}_i}_{C^{k+\ell}} \leq C_0$.

    First, let us suppose that $t\geq 1$. Then, using Remark~\ref{rem:commutation of stable and unstable}, since $x\in N_1^-\Omega$, one checks that $x_i$ belongs to $N_1^-\Omega$ as well for all $i$.
    Hence, we obtain
    \begin{align}\label{eq:reduce to ht estimate}
         \left| \int_{N_1^+} \phi(n) 
         \psi(g_t nx) \;d\mu_x^u\right|
        &\leq e^{-\d t}  \sum_{i\in I(\L)}
        \left|\int_{N_1^+}  
          \bar{\phi}_i( n )
         \psi(nx_i ) \;d\mu_{ x_i}^u\right|
         \nonumber\\
         &\leq C_0 e_{k,\ell,\g}(f)\norm{\phi\circ \Ad(g_{-t})}_{C^{k+\ell}}
         e^{-\d t}  \sum_{i\in I(\L) }
         \mu_{x_i}^u(N_1^+) V(x_i).
    \end{align}
    By the log Lipschitz property of $V$ provided by Proposition~\ref{prop:height function properties}, and by enlarging $C_0$ if necessary, we have $V(x_i)\leq C_0 V(nx_i)$ for all $n\in N_1^+$. It follows that
    \begin{equation*}
        \sum_{i\in I(\L) }
         \mu_{x_i}^u(N_1^+) V(x_i) \leq 
         C_0 \sum_{i\in I(\L) }\int_{N_1^+} V(nx_i)\;d\mu_{x_i}^{u}(n).
    \end{equation*}
    Recall that the balls $N_1^+\cdot n_i$ have intersection multiplicity at most $C_0$.
    Moreover, since the support of $\rho_i$ is contained inside $\mrm{Ad}(g_t)(N_1^+)$, the balls $N_1^+n_i$ are all contained in $N_2^+\mrm{Ad}(g_t)(N_1^+)$.
    Hence, applying the equivariance properties~\eqref{eq:g_t equivariance} and~\eqref{eq:N equivariance} once more yields
    \begin{align*}
        \sum_{i\in I(\L) }\int_{N_1^+} V(nx_i)\;d\mu_{x_i}^{u}(n)
       \leq C_0 \int_{N_2^+\mrm{Ad}(g_t)(N_1^+)}
       V(ng_tx) \;d\mu_{g_tx}^u(n) 
       \leq
       C_0 e^{\d t} \int_{N_3^+}
       V(g_tnx) \;d\mu_{x}^u(n) .
    \end{align*}
     Here, we used the positivity of $V$ and that $\mrm{Ad}(g_{-t})(N_2^+)N_1^+\subseteq N_3^+$.
     Combined with~\eqref{eq:apply equivariance} and the contraction estimate on $V$, Theorem~\ref{thm:Margulis function}, it follows that
     \begin{align*}
        \int_{N_1^+} \phi(n) \psi(g_t nx) \;d\mu_x^u
        \leq C_0^3 (c e^{-\b t}V(x) + c) \mu_x^u(N_3^+)e_{k,\ell,\g}(f),
     \end{align*}
     for a constant $c\geq 1$ depending on $\b$.
    By Proposition~\ref{prop:doubling}, we have $\mu_x^u(N_3^+) \leq C_1 \mu_x^u(N_1^+)$, for a uniform constant $C_1\geq 1$, which is independent of $x$. 
    This estimate concludes the proof in view of~\eqref{eq:weak stable derivatives}.

    Now, let $s\in [0,1]$ and $t\geq 0$. If $t+s\geq 1$, then the above argument applied with $t+s$ in place of $t$ implies that
    \begin{align*}
        \left|\int_{N_1^+}\phi(n) L_{v_1}\cdots L_{v_\ell}(\Lcal_tf)(g_{s}nx)\;d\mu_x^u \right|
        \ll_\b e^{- \e(\g) t} e_{k,\ell,\g}(f)(e^{-\b t} V(x)+1)\mu_x^u(N_1^+) ,
    \end{align*}
    as desired.
    Otherwise, if $t+s<1$, then by definition of $e_{k,\ell,\g}$, we have that
    \begin{align*}
    \left|\int_{N_1^+}\phi(n) L_{v_1}\cdots L_{v_\ell}(\Lcal_tf)(g_{s}nx)\;d\mu_x^u \right|
        \leq e_{k,\ell,\g}(f)V(x)\mu_x^u(N_1^+).
    \end{align*}
    Since $t$ is at most $1$ in this case, the conclusion of the lemma follows in this case as well.
    
    \end{proof}

   As a corollary, we deduce the following strong continuity statement which implies that the infinitesimal generator of the semigroup $\Lcal_t$ is well-defined as a closed operator on $\Bcal_k$ with dense domain. When restricted to $C_c^{k+1}(X)^M$, this generator is nothing but the differentiation operator in the flow direction. 
   This strong continuity is also important in applying the results of~\cite{Butterley} to deduce exponential mixing from our spectral bounds on the resolvent in Section~\ref{sec:Dolgopyat}.
   \begin{cor}\label{cor:strong continuity}
    The semigroup $\set{\Lcal_t:t\geq 0}$ is strongly continuous; i.e.~for all $f\in \Bcal_k$,
    \begin{align*}
        \lim_{t \downarrow 0} \norm{\Lcal_tf- f}_k =0.
    \end{align*}
   \end{cor}
   \begin{proof}
   For all $f\in C_c^{k+1}(X)^M$, one easily checks that, since $V(\cdot)\gg 1$ on any bounded neighborhood of $\Omega$, then
   \begin{align*}
       \norm{\Lcal_t f- f}_k \ll \sup_{0\leq s\leq 1} \norm{\Lcal_{t+s} f-\Lcal_sf}_{C^{k}(X)}.
   \end{align*}
   Moreover, since $f$ belongs to $C^{k+1}$, the right side above inequality tends to $0$ as $t\to 0$ by the mean value theorem.
     {
    Now, let $f$ be a general element of $\Bcal_k$ and suppose that $\norm{\Lcal_tf-f}_k \nrightarrow 0$. Then, there is $t_n\to 0$ such that $\norm{\Lcal_{t_n}f-f}_k \to c \neq 0$.
    For every $j\in\N$, let $f_j\in C_c^{k+1}(X)^M$ be such that $\norm{f-f_j}_k<1/j$.
    For each $j$, let $n_j$ be large enough such that $\norm{\Lcal_{t_{n_j}}f_j-f_j}_k <1/j$.
    Then,
    \begin{align*}
        \norm{\Lcal_{t_{n_j}}f-f}_k\leq 
        \norm{\Lcal_{t_{n_j}}f-\Lcal_{t_{n_j}}f_j}_k
        +\norm{\Lcal_{t_{n_j}}f_j-f_j}_k
        + \norm{f_j-f}_k.
    \end{align*}
    The last two terms on the right side are each bounded by $1/j$ by construction. By Lemma~\ref{lem: equicts}, we also have that $\norm{\Lcal_{t_{n_j}}f - \Lcal_{t_{n_j}}f_j}_k$
    is $O(\norm{f-f_j}_k) $. 
    It follows that $\norm{\Lcal_{t_{n_j}}f-f}_k\ll 1/j \to 0$, which contradicts the hypothesis that $\norm{\Lcal_{t_{n}}f-f}_k \to c\neq 0$.
    }
   \end{proof}

    \subsection{Towards a Lasota-Yorke inequality for the resolvent}

    Recall that for all $n\in\N$,
    \begin{equation}
        R(z)^n = \int_0^\infty \frac{t^{n-1}}{(n-1)!} e^{-z t} \Lcal_t \;dt,
    \end{equation}
    as follows by induction on $n$.
    The following corollary is immediate from Lemma~\ref{lem: equicts} and the fact that
    \begin{equation}\label{eq:integral value}
        \left|\int_0^\infty \frac{t^{n-1}}{(n-1)!} e^{-z t}\;dt\right|
        \leq \int_0^\infty \frac{t^{n-1}}{(n-1)!} e^{-\Re(z) t}\;dt =
        1/\Re(z)^n,
    \end{equation}
    for all $z\in\C$ with $\Re(z)>0$.
   \begin{cor}\label{cor:norm of resolvent}
    For all $n,k,\ell \in\N\cup \set{0}$, $f\in C^{k+1}_c(X)^M$ and $z\in \C$ with $\Re(z)>0$, we have
    \begin{equation*}
        e_{k,\ell}(R(z)^nf;x) \ll_\b e_{k,\ell}(f) \left(\frac{1}{(\Re(z)+\b)^n}+\frac{V(x)^{-1}}{\Re(z)^n}\right)
       \ll_\b  e_{k,\ell}(f)/\Re(z)^n.
    \end{equation*}
    In particular, $R(z)$ extends to a bounded operator on $\Bcal_k$ with spectral radius at most $1/\Re(z)$.
    \end{cor}

    Note that Lemma~\ref{lem: equicts} does not provide contraction in the part of the norm that accounts for the flow direction.
    In particular, the estimate in this lemma is not sufficient to control the essential spectral radius of the resolvent.
    The following lemma provides the first step towards a Lasota-Yorke inequality for resolvents for the coefficients $e_{k,\ell}$ when $\ell<k$.
    The idea, based on regularization of test functions, is due to~\cite{GouezelLiverani}.
    The doubling estimates on conditional measures in Proposition~\ref{prop:doubling} are crucial for carrying out the argument.

    \begin{lem}\label{lem: bound in degree 0}
    For all $t\geq 2$ and $0\leq \ell< k$, we have
    \begin{equation*}
        e_{k,\ell}(\Lcal_t f) \ll_{k,\b}   e^{-kt} e_{k,\ell}(f) +  e'_{k,\ell}( f).
    \end{equation*}
    \end{lem}
    
    \begin{proof}

    Fix some $0\leq \ell < k$.
    Let $x\in \Omega_1^-$ and $\phi \in C^{k+\ell}(N^+_1)$.
    Let $(v_i)_i\in \Vcal_{k+\ell}^\ell$ and
    set $F=L_{v_1}\cdots L_{v_\ell}f$.
    We wish to estimate the following:
    \begin{align*}
        \sup_{0\leq s\leq 1} \int_{N_1^+} \phi(n) F(g_{t+s}nx)\;d\mu_x^u.
    \end{align*}
    To simplify notation, we prove the desired estimate for $s=0$, the general case being essentially identical.

    Let $\e>0$ to be determined and choose $\psi_\e$ to be a $C^\infty$ bump function supported inside $N^+_\e$ and satisfying $\norm{\psi_\e}_{C^1}\ll \e^{-1}$.
    Define the following regularization of $\phi$
    \begin{align*}
        \Mcal_\e(\phi)(n) = \frac{\int_{N^+} \phi(un)\psi_\e(u)\;du}{\int_{N^+} \psi_\e(u)\;du},
    \end{align*}
    where $du$ denotes the right-invariant Haar measure on $N^+$.
    Recall the definition of the coefficients $c_r$ above~\eqref{eq:Leibniz}.
    Let $0\leq m <k+\ell$ and $(w_j)\in (\Vcal^+)^m$.
    Then,
    \begin{align*}
        |L_{w_1}\cdots L_{w_m}(\phi-\Mcal_\e(\phi))(n)|
        &\leq 
        \frac{\int |L_{w_1}\cdots L_{w_m}(\phi)(n)-L_{w_1}\cdots L_{w_m}(\phi)(un)|\psi_\e(u)\;du}{\int \psi_\e(u)\;du}
        \nonumber\\
        &\ll c_{m+1}(\phi)\frac{\int \dist(n,un) \psi_\e(u)\;du}{\int \psi_\e(u)\;du}.
    \end{align*}
    Now, note that if $\psi_\e(u)\neq 0$, then $\dist(u,\id)\leq \e$.
    Hence, right invariance of the metric on $N^+$ implies that $c_m(\phi-\Mcal_\e(\phi))\ll \e c_{m+1}(\phi)$.
    
    Moreover, we have that $c_{m}(\Mcal_\e(\phi))\leq c_{m}(\phi)$ for all $0\leq m\leq k+\ell$.
    It follows that $c_{k+\ell}(\phi-\Mcal_\e(\phi))\leq 2c_{k+\ell}(\phi)$.
    Finally, given $(w_i)\in (\Vcal^+)^{k+\ell+1}$, integration by parts gives
    \begin{equation*}
        L_{w_1}\cdots L_{w_{k+\ell+1}}(\Mcal_\e(\phi))(n) = \frac{-\int_{N^+} L_{w_2}\cdots L_{w_{k+\ell+1}}(\phi)(un)\cdot L_{w_1}(\psi_\e)(u)\;du}{\int_{N^+} \psi_\e(u)\;du}.
    \end{equation*}
    In particular, since $\norm{\psi_\e}_{C^1}\ll \e^{-1}$, we get $c_{k+\ell+1}(\Mcal_\e(\phi))\ll \e^{-1}c_{k+\ell}(\phi)$.
    Since $g_t$ expands $N^+$ by at least $e^t$, this discussion shows that for any $t\geq 0$, if $\norm{\phi}_{C^{k+\ell}}\leq 1$, then
    \begin{align}\label{eq:regularization bounds}
        \norm{(\phi-\Mcal_\e(\phi))\circ\Ad(g_{-t})}_{C^{k+\ell}}
        &\ll
        \e\sum_{m=0}^{k+\ell-1} \frac{ e^{-m t}}{2^m}
        +\frac{e^{-(k+\ell)t}}{2^{k+\ell}},
        \nonumber\\
         \norm{\Mcal_\e(\phi)\circ\Ad(g_{-t})}_{C^{k+\ell+1}}
        &\ll \sum_{m=0}^{k+\ell} \frac{ e^{-m t}}{2^m}+ \frac{\e^{-1} e^{-(k+\ell+1)t}}{2^{k+\ell+1}}.
    \end{align}

    Then, taking $\e=e^{-kt}$, we obtain
    \begin{align}\label{eq:introduce regularization}
        \int_{N^+_1} \phi(n) F(g_t nx)
        \;d\mu_x^u
        &=\int \phi(n) F(g_t nx)
        \;d\mu_x^u
        \nonumber\\
        &= \int (\phi-\Mcal_\e(\phi))(n) F(g_t nx)
        \;d\mu_x^u +
        \int \Mcal_\e(\phi)(n) F(g_t nx)
        \;d\mu_x^u.
    \end{align}

    To estimate the second term, we recall that the test functions for the weak norm were required to be supported inside $N_{1/10}^+$.
    On the other hand, the support of $\Mcal_\e(\phi)$ may be larger, but still inside $N_{1+\e}^+$.
    To remedy this issue, we pick a partition of unity $\set{\rho_i:i\in I}$ of $N_2^+$, so that each $\rho_i$ is smooth, non-negative, and supported inside some ball of radius $1/20$.
    We also require that $\norm{\rho_i}_{C^{k+\ell+1}}\ll_k 1$.
    We can find such a partition of unity with bounded cardinality and multiplicity, depending only on $N^+$ (through its dimension and metric).
    
    Similarly to Lemma~\ref{lem: equicts}, we denote by $I(\L)\subseteq I$, the subset of those indices $i$ such that there is some $n_i\in N^+$ in the support of of $\mu_x^u$ so that the support of $\rho_i$ is contained inside $N_{1/10}^+\cdot n_i$.
    In particular, for $i\in I\setminus I(\L)$, $\rho_i\mu_x^u$ is the $0$ measure.

    Now, observe that the functions $n\mapsto \rho_i(nn_i)\Mcal_\e(\phi)(n n_i)$ are supported inside $N_{1/10}^+$.
    Thus, writing $x_i=n_i g_1 x$, using a change of variable, and arguing as in the proof of Lemma~\ref{lem: equicts}, cf.~\eqref{eq:reduce to ht estimate}, we obtain
    
     \begin{align*}
        \int \Mcal_\e(\phi)(n) F(g_t nx)
        \;d\mu_x^u
        &= e^{-\d} \sum_{i\in I(\L)}
        \int (\rho_i\Mcal_\e(\phi))\circ\Ad(g_{-1})(n) F(g_{t-1} ng_1x)
        \;d\mu_{g_1x}^u
        \nonumber\\
        &\ll  
        e'_{k,\ell }( f) \cdot
        \sum_{i\in I(\L)}\norm{(\rho_i\Mcal_\e(\phi))\circ\Ad(g_{-t})}_{C^{k+\ell+1}}\cdot V(x_i)\mu_{x_i}^u(N_1^+).
    \end{align*}
    The point of replacing $x$ with $g_1x$ is that since $x$ belongs to $N_1^-\Omega$, $g_1x$ belongs to $N^-_{1/2}\Omega$, which satisfies the requirement on the basepoints in the definition of the weak norm.
    
    Note that the bounded multiplicity property of the partition of unity, together with the doubling property in Proposition~\ref{prop:doubling}, imply that
    $$ \sum_{i\in I} \mu_{x_i}^u(N_1^+)\ll \mu_x^u(N_3^+)\ll \mu_x^u(N_1^+).$$
    Moreover, combining the Leibniz estimate~\eqref{eq:Leibniz} with~\eqref{eq:regularization bounds}, we see that the $C^{k+\ell+1}$ norm of $(\rho_i\Mcal_\e(\phi))\circ\Ad(g_{-t})$ is $O_k(1)$.
   Hence, by properties of the height function $V$ in Proposition~\ref{prop:height function properties}, it follows that
    \begin{align*}
        \int \Mcal_\e(\phi)(n) F(g_t nx)
        \;d\mu_x^u
        &\ll_k e'_{k,\ell}(f) V(x)\mu_x^{u}(N_1^+).
    \end{align*}
    
    Using a completely analogous argument to handle the issues of the support of the test function, we can estimate the first term in~\eqref{eq:introduce regularization} as follows:
    \begin{align*}
        \frac{1}{V(x)\mu_{x}^u(N_1^+)}
        \int_{N^+_1} (\phi-\Mcal_\e(\phi))(n) F(g_t nx)
        \;d\mu_x^u
        &\ll_k  e^{-kt}e_{k,\ell }(f) .
    \end{align*}
    Since $(v_i)\in \Vcal_{k+\ell}^\ell$, $x\in \Omega_1^-$ and $\phi\in C^{k+\ell}(N_1^+)$ were all arbitrary, this completes the proof.
    \end{proof}

   
   It remains to estimate the coefficients $e_{k,k}$.
   First, the following estimate in the case all the derivatives point in the stable direction follows immediately from Lemma~\ref{lem: equicts}.
   \begin{lem}\label{lem:all stable}
   For all $\g=(v_i)\in (\Vcal_{2k}^-)^k$, we have
   \begin{align*}
       e_{k,k,\g}(R(z)^nf) \ll_\b \frac{1}{(\Re(z)+k)^n}e_{k,k}(f).
   \end{align*}
   \end{lem}
   \begin{proof}
   Indeed, Lemma~\ref{lem: equicts} shows that
   \begin{align*}
       e_{k,k,\g}(\Lcal_t f) \ll e^{-k t} e_{k,k}(f).
   \end{align*}
   Moreover, induction and integration by parts give $|\int_0^\infty t^{n-1} e^{-(z+k)t}/(n-1)!dt|\leq 1/(\Re(z)+k)^n$. This completes the proof.
   \end{proof}
   
    To give improved estimates on the the coefficient $e_{k,k,\g}$ in the case some of the components of $\g$ point in the flow direction, the idea (cf.~\cite[Lem.~8.4]{AvilaGouezel} and~\cite[Lem~4.5]{GiuliettiLiveraniPollicott}) is to take advantage of the fact that the resolvent is defined by integration in the flow direction, which provides additional smoothing. This is leveraged through integration by parts to estimate the coefficient $e_{k,k}$ by $e_{k,k-1}$.
    
    To see how such estimate can be turned into a gain on the norm of the resolvents, following~\cite{AvilaGouezel}, we define the following equivalent norms to $\norm{\cdot}_k$.
    First, let us define the following coefficients:
 \begin{align*}
     e_{k,\ell,s} := \begin{cases}
     e_{k,\ell} & 0\leq \ell < k,
     \\
     \sup_{\g\in (\Vcal_{2k}^-)^k} e_{k,k,\g} 
     & \ell =k,
     \end{cases}
     ,\qquad e_{k,k,\w} := \sup_{\g\in \Vcal_{2k}^k\setminus (\Vcal_{2k}^-)^k} e_{k,k,\g}.
 \end{align*}
 Given $B\geq 1$, define
 \begin{align*}
     \norm{f}_{k,B,s} := \sum_{\ell=0}^k \frac{e_{k,\ell,s}(f)}{B^\ell},
     \qquad \norm{f}_{k,B,\w}:= \frac{e_{k,k,\w}(f)}{B^k}.
 \end{align*}
 Finally, we set
    \begin{equation}\label{eq:equivalent norm}
        \norm{f}_{k,B} := 
       \norm{f}_{k,B,s} + \norm{f}_{k,B,\w}.
    \end{equation}

    \begin{lem}\label{lem: flow by parts}
    Let $n, k \in \N$ and $z\in \C$ with $\Re(z)>0$ be given.
    Then, if $B$ is large enough, depending on $n,k,\b$ and $z$, we obtain for all $f\in C_c^{k+1}(X)^M$ that
        \begin{equation*}
            \norm{R(z)^nf}_{k,B,\w} \leq \frac{1}{(\Re(z)+k+1)^n}
         \norm{f}_{k,B}.
        \end{equation*}
       
    \end{lem}
    \begin{proof}
    
    Fix an integer $n\geq 0$.
    We wish to estimate integrals of the form
    \begin{align*}
        \int_{N_1^+} \phi(u) L_{v_1}\cdots L_{v_k}
        \bigg(\int_0^\infty \frac{t^n e^{-zt}}{n!}
        &\Lcal_{t+s} f\;dt \bigg)(ux)\;d\mu_x^u(u)
        \nonumber\\
       & =\int_{N_1^+} \phi(u) \int_0^\infty \frac{t^n e^{-zt}}{n!}
        L_{v_1}\cdots L_{v_k}(\Lcal_{t+s} f)(ux)\;dt\;d\mu_x^u(u),
    \end{align*}
    with $0\leq s\leq 1$ and at least one of the $v_i$ pointing in the flow direction.
    
    First, let us consider the case $v_k$ points in the flow direction.
    Then, $v_k(u) = \psi_k(u)\w$, where $\w$ is the vector field generating the geodesic flow, for some function $\psi_k$ in the unit ball of $C^{2k}(N^+)$.
    Hence, for a fixed $u\in N_1^+$, integration by parts in $t$, along with the fact that $f$ is bounded, yields 
    \begin{align*}
        &\int_0^\infty 
        \frac{t^n e^{-zt}}{n!}
        L_{v_1}L_{v_2}\cdots L_{v_k}(\Lcal_{t+s}f)(ux)\;dt 
        \nonumber\\
        &=\psi_k(u)z\int_0^{\infty} \frac{t^n e^{-zt}}{n!} L_{v_1}\cdots L_{v_{k-1}}(\Lcal_{t+s}f)(ux)\;dt
        -\psi_k(u)\int_0^{\infty} \frac{t^{n-1} e^{-zt}}{(n-1)!} L_{v_1}\cdots L_{v_{k-1}}(\Lcal_{t+s}f)(ux)\;dt
        \nonumber\\
        &=\psi_k(u)z L_{v_1}\cdots L_{v_{k-1}}(\Lcal_sR(z)^{n+1}f)(ux)
        -\psi_k(u)L_{v_1}\cdots L_{v_{k-1}}(\Lcal_sR^{n}(z)f)(ux).
    \end{align*}
    Recall by Lemma~\ref{lem: equicts} that $e_{k,\ell}(R(z)^nf)\ll_\b e_{k,\ell}(f)/\Re(z)^n$ for all $n\in\N$; cf.~Corollary~\ref{cor:norm of resolvent}.
    It follows that 
     \begin{align*}
         e_{k,k,\g}(R(z)^{n+1}f)\leq e_{k,k-1}(R(z)^nf) +|z|e_{k,k-1}(R(z)^{n+1}f)
         \ll_\b \left(\frac{\Re(z)+|z|}{\Re(z)^{n+1}}\right)
         e_{k,k-1}(f).
     \end{align*}
      
     In the case $v_k$ points in the stable direction instead, we note that $L_{v}L_w=L_wL_v + L_{[v,w]}$ for any two vector fields $v$ and $w$, where $[v,w]$ is their Lie bracket.
     In particular, we can write $L_{v_1}\cdots L_{v_k}$ as a sum of at most $k$ terms involving $k-1$ derivatives in addition to one term of the form $L_{w_1}\cdots L_{w_k}$, where $w_k$ points in the flow direction.
     Each of the terms with one fewer derivative can be bounded by $e_{k,k-1}(R(z)^{n+1}f)\ll_\b e_{k,k-1}(f)/\Re(z)^{n+1}$, while the term with $k$ derivatives is controlled as in the previous case.
    Hence, taking the supremum over $\g\in \Vcal_{2k}^k\setminus (\Vcal_{2k}^-)^k$ and choosing $B$ to be large enough, we obtain the conclusion.
    \end{proof}


\subsection{Decomposition of the transfer operator according to recurrence of orbits} \label{sec:decomposition}
    In order to make use of the compact embedding result in Proposition~\ref{prop:compact embedding}, we need to localize our functions to a fixed compact set. This is done with the help of the Margulis function $V$.
    In this section, we introduce some notation and prove certain preliminary estimates for that purpose.

    Recall the notation in Theorem~\ref{thm:Margulis function}.
    Let $T_0\geq 1$ be a constant large enough so that $e^{\b T_0}>2$. We will enlarge $T_0$ over the course of the argument to absorb various auxiliary uniform constants. Define $V_0$ by
    \begin{equation}\label{eq:V_0}
        V_0 = e^{3\b  T_0}.
    \end{equation}
    Let $\rho_{V_0}\in C^\infty_c(X)$ be a non-negative $M$-invariant function satisfying $\rho_{V_0}\equiv 1$ on the unit neighborhood of $\set{x\in X: V(x)\leq V_0}$ and $\rho_{V_0}\equiv 0$ on $\set{V> 2 V_0}$.
    Moreover, we require that $\rho_{V_0}\leq 1$.
    Note that since $T_0$ is at least $1$, we can choose $\rho_{V_0}$ so that its $C^{2k}$ norm is independent of $T_0$.

    Let $\psi_1 =\rho_{V_0}$ and $\psi_2 = 1-\psi_1$.
    Then, we can write 
    \begin{equation*}
        \Lcal_{T_0}f = \tilde{\Lcal}_1 f + \tilde{\Lcal}_2 f,
    \end{equation*}
    where $\tilde{\Lcal}_i f = \Lcal_{T_0}(\psi_i f)$, for $i\in \set{1,2}$.
    It follows that for all $j\in\N$, we have
    \begin{equation}\label{eq:sum over gamma}
        \Lcal_{jT_0}f = \sum_{\varpi\in\set{1,2}^j} \tilde{\Lcal}_{\varpi_1}\cdots \tilde{\Lcal}_{\varpi_j}f
        =\sum_{\varpi\in\set{1,2}^j}
        \Lcal_{jT_0}(\psi_\varpi f),
        \qquad
         \psi_\varpi = \prod_{i=1}^j \psi_{\varpi_i} \circ g_{-(j-i)T_0}.
    \end{equation}
    Note that if $\varpi_i =1$ for some $1\leq i\leq j$, then, by Proposition~\ref{prop:height function properties}, we have
    \begin{equation}\label{eq:supp psi_gamma}
        \sup_{x\in \mrm{supp} (\psi_\varpi)} V(x) \leq e^{\b I_\varpi T_0} V_0, \qquad 
        I_\varpi =j- \max\set{1\leq i\leq j: \varpi_i=1}.
    \end{equation}
The following lemma estimates the effect of multiplying by a fixed smooth function such as $\psi_\varpi$.
To formulate the lemma, we need the following definition.
\begin{definition}
    Given $\psi\in C^{r}(X)$, we use the notation $\norm{\psi}_{C^r}^u$ to denote the $C^r$-norm of $\psi$ along the unstable foliation.
    More precisely, we set
    \begin{align}\label{def:unstable norm}
        \norm{\psi}_{C^r}^u = \sum_{i=0}^r \frac{c^u_i(\psi)}{2^i i!},
    \end{align}
    where $c^u_i(\psi)$ denotes the maximum of the sup norm of all order-$i$ derivatives of $\psi$ along directions tangent to $N^+$.
\end{definition}

    \begin{lem}\label{lem:multiply by smooth}
    Let $\psi\in C^{2k}(X)$ be given.
    Then, if $B\geq 1$ is large enough, depending on $k$ and $\norm{\psi}_{C^{2k}}$, we have
    \begin{equation*}
        \norm{\psi f}_{k,B,s} \leq  2\norm{\psi}^u_{C^{2k}(X)}  \norm{f}_{k,B,s},
    \end{equation*}
    where $\norm{\psi}^u_{C^{2k}(X)}$ is defined in~\eqref{def:unstable norm}.
    \end{lem}
    \begin{proof}
    Given $0\leq \ell\leq k$ and $0\leq s\leq 1$, we wish to estimate integrals of the form
    \begin{align*}
        \int_{N_1^+} \phi(n) L_{v_1}\cdots L_{v_\ell}(\psi f)(g_snx)\;d\mu_x^u(n).
    \end{align*}
    The term $L_{v_1}\cdots L_{v_\ell}(\psi f)$ can be written as a sum of $2^\ell$ terms, each consisting of a product of an order-$i$ derivative of $\psi$ by an order-$(\ell-i)$ derivative of $f$, for $0\leq i\leq \ell$.
    Viewing the product of $\phi$ by an order-$i$ derivative of $\psi$ as a $C^{k+\ell-i}$ test function, and using~\eqref{eq:Leibniz} to bound the $C^{k+\ell-i}$ norm of such a product, we obtain a bound of the form
    \begin{align*}
        B^{-\ell} e_{k,\ell,s}(\psi f) 
        &\leq B^{-1} C_{k,\psi} \sum_{i=0}^{\ell-1} \binom{\ell}{i} B^{-i} e_{k,i,s}(f)
        + B^{-\ell} \norm{\psi}^u_{C^{2k}} e_{k,\ell,s}(f)
        \nonumber\\
        &\leq B^{-1} C_{k,\psi}2^k \sum_{i=0}^{\ell-1} B^{-i} e_{k,i,s}(f)
        + B^{-\ell} \norm{\psi}^u_{C^{2k}} e_{k,\ell,s}(f),
    \end{align*}
    for a suitably large constant depending on $k$ and the $C^{2k}$-norm of $\psi$.
    Here, we note that the terms that contribute to the $e_{k,\ell,s}(f)$ term in the above sum all have the form $\int_{N_1^+} \phi \psi L_{v_1}\cdots L_{v_\ell}(f)\;d\mu_x^u$.

    Summing over $\ell$, we obtain
    \begin{align*}
         \norm{\psi f}_{k,B,s} &= \sum_{\ell=0}^k \frac{1}{B^\ell} e_{k,\ell,s}(\psi f) 
        \leq 
        B^{-1} C_{k,\psi}2^k
        \sum_{\ell=0}^k  \sum_{i=0}^{\ell-1} B^{-i} e_{k,i,s}(f)
        + \norm{\psi}^u_{C^{2k}} \norm{f}_{k,B,s}
        \nonumber\\
         &\leq  (B^{-1} C_{k,\psi}2^k k + \norm{\psi}^u_{C^{2k}}) 
         \norm{f}_{k,B,s}.
    \end{align*}
    Taking $B$ large enough completes the proof of the lemma.
    \qedhere
    
    \end{proof}

    The above lemma allows us to estimate the norms of the operators $\tilde{\Lcal}_i$, for $i=1,2$ as follows.
    \begin{lem}\label{lem:norms of operator pieces}
    There exists a constant $C_{k,\b} \geq 1$, depending only on $\b$ and $\norm{\rho_{V_0}}_{C^{2k}}$, such that for all large enough $B\geq 1$, we have
    \begin{align*}
     \norm{\tilde{\Lcal}_1 f}_{k,B,s} \leq C_{k,\b} \norm{f}_{k,B,s},
     \qquad \norm{\tilde{\Lcal}_2 f}_{k,B,s}\leq C_{k,\b} e^{-\b T_0} \norm{f}_{k,B,s}.
 \end{align*}
    \end{lem}
    \begin{proof}
    The first inequality follows by Lemmas~\ref{lem: equicts} and~\ref{lem:multiply by smooth}.
    The second inequality follows similarly since
    \begin{align*}
        \psi_2(g_{T_0}nx)\neq 0 \Longrightarrow V(g_{T_0}nx)\geq V_0.
    \end{align*}
 By Proposition~\ref{prop:height function properties}, this in turn implies that, whenever $\psi_2(g_{T_0}nx)\neq 0$ for some $n\in N_1^+$, then $V(x)\gg e^{\b T_0}$, by choice of $V_0$.
    \end{proof}
    
\subsection{Proof of Theorems~\ref{thm: resolvent spectrum} and~\ref{thm:resolvent spectrum2}}

    Theorem~\ref{thm: resolvent spectrum} follows at once from~\ref{thm:resolvent spectrum2}.
     Theorem~\ref{thm:resolvent spectrum2} will follow upon  verifying the hypotheses of Theorem~\ref{thm: hennion}.
    The boundedness assertion follows by Corollary~\ref{cor:norm of resolvent}.
    It remains to estimate the essential spectral radius of the resolvent $R(z)$.
    
    Write $z=a+ib\in \C$.
    Fix some parameter $0<\th< 1$ and define
    \begin{equation*}
        \s:= \min\set{k,\b\th}.
    \end{equation*}
    Let $0<\epsilon<\s/5$ be given. 
    We show that for a suitable choice of $r$ and $B$, the following Lasota-Yorke inequality holds:
    \begin{align}\label{eq:Lasota-Yorke}
    \norm{R(z)^{r+1} f}_{k,B}\leq \frac{\norm{f}_{k,B}}{(a+\s-2\epsilon)^{r+1}}
    + C'_{k,\b, B,r,T_0}\norm{\Psi_{r,\th} f}_k',
\end{align}
where $C'_{k,\b, B,r,T_0}\geq 1$ is a constant depending on the parameters in its subscript, while $\Psi_{r,\th}:X\to [0,1]$ is a smooth function vanishing outside a sublevel set of the Margulis function $V$, and whose support depends on $r$ and $\th$.

    First, we show how~\eqref{eq:Lasota-Yorke} implies the result.
    Hennion's Theorem, Theorem~\ref{thm: hennion}, applied with the norm $\norm{\cdot}=\norm{\cdot}_{k,B}$ and the semi-norm $\norm{\cdot}'=\norm{\Psi_{r,\th}\bullet}_k'$, implies that the essential spectral radius of $R(z)$, with respect to the norm $\norm{\cdot}_{k,B}$, is at most $ 1/(a+\s-2\epsilon)$.
    Equivalence of the norms $\norm{\cdot}_k$ and $\norm{\cdot}_{k,B}$ implies that the same estimate also holds for the essential spectral radius $\rho_{ess}(R(z))$ with respect to $\norm{\cdot}_k$.
Note that the compact embedding requirement follows by Proposition~\ref{prop:compact embedding} again by equivalence of the norms $\norm{\cdot}_k$ and $\norm{\cdot}_{k,B}$.
Since $\epsilon>0$ was arbitrary, this shows that $\rho_{ess}(R(z))\leq 1/(a+\s)$.
Finally, as $0<\th<1$ was arbitrary, we obtain that 
\begin{equation*}
    \rho_{ess}(R(z))\leq \frac{1}{\Re(z)+\s_0},
\end{equation*}
completing the proof.

    To show~\eqref{eq:Lasota-Yorke}, let an integer $r\geq 0$ be given and $J_r\in \N$ to be determined. Using~\eqref{eq:sum over gamma} and a change of variable, we obtain
\begin{align*}
     R(z)^{r+1} f 
    &= \int_0^\infty \frac{t^r e^{-zt}}{r!} \Lcal_t f\; dt
    \\
    &= \int_0^{T_0} \frac{t^r e^{-zt}}{r!} \Lcal_t f\; dt
    + \int_{(J_r+1)T_0}^\infty \frac{t^r e^{-zt}}{r!} \Lcal_t f\; dt
    +
    \sum_{j= 1}^{J_r}  \int_{jT_0}^{(j+1)T_0} \frac{t^r e^{-zt}}{r!}  \Lcal_{t}f\; dt.
\end{align*}

First, by Lemma~\ref{lem: flow by parts}, if $B$ is large enough, depending on $r,k$ and $z$, we obtain 
\begin{align*}
    \norm{R(z)^{r+1}(z)f}_{k,B,\w} \leq  \frac{1}{(a+k+1)^{r+1}} \norm{f}_{k,B}.
\end{align*}

It remains to estimate $\norm{R(z)^{r+1}f}_{k,B,s}$.
Note that $\int_0^{T_0} \frac{t^r e^{-at}}{r!} dt\leq T_0^{r+1}/r!$.
Hence, taking $r$ large enough, depending on $k$, $a$, $\b$ and $T_0$, and using Lemma~\ref{lem: equicts}, we obtain for any $B\geq 1$,
\begin{align*}
    \norm{\int_0^{T_0} \frac{t^r e^{-zt}}{r!} \Lcal_t f\; dt}_{k,B,s} \ll_\b \norm{f}_{k,B} \int_0^{T_0} \frac{t^r e^{-at}}{r!} dt
   \leq  \frac{1}{(a+k+1)^{r+1}}\norm{f}_{k,B}.
\end{align*}
Similarly, taking $J_r$ to be large enough, depending on $k$, $a$, $\b$, and $r$, we obtain for any $B\geq 1$,
\begin{align*}
    \norm{\int_{(J_r+1)T_0}^\infty \frac{t^r e^{-zt}}{r!} \Lcal_t f\; dt}_{k,B,s}
    \ll_\b \norm{f}_{k,B} \int_{(J_r+1)T_0}^\infty \frac{t^r e^{-at}}{r!} \; dt
    \leq  \frac{1}{(a+k+1)^{r+1}}\norm{f}_{k,B}.
\end{align*}

To estimate the remaining term in $R(z)^{r+1}f$, 
let $1\leq j\leq J_r$ and $\varpi=(\varpi_i)_i\in\set{1,2}^j$ be given.
Let $\th_\varpi$ denote the number of indices $i$ such that $\varpi_i=2$.
Then, 
it follows from Lemma~\ref{lem: equicts} and induction on Lemma~\ref{lem:norms of operator pieces} that
 \begin{align}\label{eq:exp recur from cusp}
     \norm{\Lcal_{t+jT_0}(\psi_\varpi f)}_{k,B,s}
     \ll_\b
     \norm{\Lcal_{jT_0}(\psi_\varpi f)}_{k,B,s}
     = \norm{\tilde{\Lcal}_{\varpi_1}\circ\cdots \circ \tilde{\Lcal}_{\varpi_j} f}_{k,B,s}
     \leq C_{k,\b}^{ j} e^{-\b \th_\varpi T_0} \norm{f}_{k,B,s},
 \end{align}
 where $C_{k,\b}$ is the constant provided by Lemma~\ref{lem:norms of operator pieces}.
  We shall assume that $C_{k,\b}$ is taken than the implicit constant in the first inequality.

 Suppose $\th_\varpi\geq \th j$.
 Then, by taking $T_0$ to be large enough so that $C_{k,\b}^{j+1}\leq e^{\epsilon j T_0}$, we obtain
 \begin{align*}
     \norm{\Lcal_{t+jT_0}(\psi_\varpi f)}_{k,B,s}
     \leq e^{-(\b \th-\epsilon)jT_0} \norm{f}_{k,B,s}.
 \end{align*}
    The case $\th_\varpi<\th j$ is addressed in the following lemma.
    Its proof is given in Section~\ref{sec:recurrent contribution lem} below and is an application of Lemmas~\ref{lem: equicts},~\ref{lem: bound in degree 0}, and~\ref{lem:multiply by smooth}.
 
    \begin{lem}\label{lem:recurrent contribution}
        Assume $B\geq 1$ is chosen large enough, depending on $k$ and $r$, and that $T_0\geq 1$ is chosen large enough depending $k,\b$ and $\epsilon$.
        Then, there exists a sublevel set $K_{r,\th}$ of the Margulis function $V$ and a smooth function $\Psi_{r,\th}:X\to [0,1]$ vanishing outside the unit neighborhood of $K_{r,\th}$ so that the following hold.
        For all $1\leq j\leq J_r$, and all $\varpi\in \set{1,2}^j$ with $\th_\varpi <\th j$, we have
        \begin{align*}
            \norm{\Lcal_{t+jT_0}(\psi_\varpi f)}_{k,B,s}
            \leq e^{-(k-\epsilon)(t+  j T_0)} \norm{f}_{k,B,s} + C_{k,\b,B,r,T_0} \norm{\Psi_{r,\th} f}'_k,
        \end{align*}
        for a suitably large constant $C_{k,\b,B,r,T_0} \geq 1$.
    \end{lem}

 Putting the above estimates together, we obtain 
\begin{align*}
\norm{
\sum_{j= 1}^{J_r}  \int_{jT_0}^{(j+1)T_0} \frac{t^r e^{-zt}}{r!}  \Lcal_{t}f\; dt}_{k,B,s}
    &\leq
    \sum_{j= 1}^{J_r}e^{-ajT_0}
    \sum_{\varpi \in \set{1,2}^j} \int_{0}^{T_0} \frac{(t+jT_0)^r e^{-at}}{r!}  \norm{\Lcal_{t+jT_0}(\psi_\varpi f)}_{k,B,s}\; dt
    \nonumber\\
    &\leq \norm{f}_{k,B,s}
    \sum_{j= 1}^{J_r}
    e^{-(a+\s-\epsilon)jT_0} \int_{0}^{T_0} \frac{(t+jT_0)^r e^{-at}}{r!} dt
    \nonumber\\
    &+C_{k,\b, B,r,T_0} \norm{\Psi_r f}'_k
    \sum_{j= 1}^{J_r}
    2^j
    e^{-ajT_0}
     \int_{0}^{T_0} \frac{(t+jT_0)^r e^{-at}}{r!}  \; dt
     \nonumber\\
     &\leq e^{(\s-\epsilon) T_0} \norm{f}_{k,B,s}
     \int_1^{J_r} \frac{t^r e^{-(a+\s-\epsilon)t}}{r!} \;dt
     + C'_{k,\b, B,r,T_0} \norm{\Psi_r f}'_k,
\end{align*}
where we take $C'_{k,\b, B,r,T_0}\geq 1$ to be a constant large enough so that the last inequality holds.

Next, we note that
\begin{align*}
    \int_1^{J_r} \frac{t^r e^{-(a+\s-\epsilon)t}}{r!} \;dt
    \leq \int_0^{\infty} \frac{t^r e^{-(a+\s-\epsilon)t}}{r!} \;dt
    =\frac{1}{(a+\s-\epsilon)^{r+1}}.
\end{align*}
Thus, taking $r$ to be large enough depending on $a$ and $T_0$, and combining the estimates on $\norm{R(z)^{r+1}f}_{k,B,\w}$ and $\norm{R(z)^{r+1}f}_{k,B,s}$, we obtain~\eqref{eq:Lasota-Yorke} as desired. 

\subsubsection{Proof of Lemma~\ref{lem:recurrent contribution}}
\label{sec:recurrent contribution lem}
     
  Let $\varpi\in \set{1,2}^j$ be such that $\th_\varpi<\th j$. 
  By Lemma~\ref{lem: bound in degree 0}, for all $0\leq \ell <k$, we have
\begin{align*}
    e_{k,\ell}(\Lcal_{t+jT_0}(\psi_\varpi f))
    \ll_{k,\b} e^{-k(t+jT_0)} e_{k,\ell}(\psi_\varpi f)
    + e'_{k,\ell}(\psi_\varpi f),
\end{align*}
where we may assume that $T_0$ is at least $2$ so that the hypothesis of Lemma~\ref{lem: bound in degree 0}. 
For the coefficient $e_{k,k}$, Lemma~\ref{lem: equicts} shows that for any $\g\in (\Vcal_{2k}^-)^k$, we have
 \begin{align*}
     e_{k,k,\g}(\Lcal_{t+jT_0}(\psi_\varpi f))
     \ll_\b e^{-k(t+jT_0)}  e_{k,k,s}(\psi_\varpi f).
 \end{align*}
 Hence, summing over $\ell$, we obtain
 \begin{align*}
     \norm{\Lcal_{t+jT_0}(\psi_\varpi f)}_{k,B,s}
     \leq C_{k,\b} e^{-k(t+jT_0)} \norm{\psi_\varpi f}_{k,B,s} + C_{k,\b,B} \norm{\psi_\varpi f}'_k,
 \end{align*}
 for suitable constants $C_{k,\b}\geq 1$ and $C_{k,\b,B}\geq 1$ depending on the parameters in their respective subscripts.
 
 Our next task is to remove the dependence over $\varpi$ in the right side of the above estimate.
 By taking $B$ large enough, depending on the maximum over $1\leq j\leq J_r$ and $\varpi\in \set{1,2}^j$ of the $C^{2k}$-norm of the functions $\psi_\varpi$, we may apply Lemma~\ref{lem:multiply by smooth} to get
 \begin{align*}
     \norm{\psi_\varpi f}_{k,B,s} \leq 2\norm{\psi_\varpi}^u_{C^{2k}} \norm{f}_{k,B,s},
 \end{align*}
 where the unstable norm $\norm{\cdot}^u_{C^{2k}}$ is defined in~\eqref{def:unstable norm}.
 
 By the formula~\eqref{eq:sum over gamma} for $\psi_\varpi$, the functions $\psi_\varpi$ are given by a product of $j$ functions of the form $\rho_{V_0}$ and $1-\rho_{V_0}$ composed by $g_{-t}$ for suitable $t>0$.
 Since composition by $g_{-t}$, $t>0$, is non-expanding on the unstable norm $\norm{\cdot}^u_{C^{2k}}$, we get
 \begin{align*}
     \norm{\psi_\varpi}^u_{C^{2k}} \leq \norm{\rho_{V_0}}^j_{C^{2k}}.
 \end{align*}
 By enlarging the constant $C_{k,\b}$ if necessary, we may assume it is larger than $2\norm{\rho_{V_0}}_{C^{2k}}$.
 Thus, we obtain the bound:
 \begin{align*}
     \norm{\Lcal_{t+jT_0}(\psi_\varpi f)}_{k,B,s}
     &\leq C^{j+1}_{k,\b} e^{-k(t+jT_0)} \norm{f}_{k,B,s} + C_{k,\b,B} \norm{\psi_\varpi f}'_k.
 \end{align*}

To put the term $\norm{\psi_\varpi f}'_k$ in a form where we can apply Hennion's Theorem~\ref{thm: hennion}, we take advantage of the bound $\th_\varpi <\th j$.
 To this end, note that the bound $\th_\varpi < \th j$ and the formula~\eqref{eq:supp psi_gamma} for the support of $\psi_{\varpi}$ imply that there is a sublevel set $K_{r,\th}$ of the Margulis function $V$, depending only on $\th$ and $J_r$, such that the following holds.
 For every $1\leq j\leq J_r$ and all $\varpi\in \set{1,2}^j$ with $\th_\varpi< \th j$, the function $\psi_{\varpi}$ is supported inside $K_{r,\th}$.
Let $\Psi_{r,\th}:X\to [0,1]$ denote a smooth bump function which is identically $1$ on $K_{r,\th}$ and vanishes outside the unit neighborhood of $K_{r,\th}$.
Then, for every $\varpi$ with $\th_\varpi <\th j$, we have that $\psi_\varpi = \psi_\varpi \Psi_{r,\th}$.
Hence, arguing as in the proof of Lemma~\ref{lem:multiply by smooth} with $\norm{\cdot}_k'$ in place of $\norm{\cdot}_{k,B,s}$, we obtain
\begin{align*}
    \norm{\psi_\varpi f}'_k = \norm{\psi_\varpi \Psi_{r,\th} f}'_k\ll_{k,T_0,J_r} \norm{\Psi_{r,\th} f}'_k .
\end{align*}
Here, the dependence of the implicit constant arises from the norm $\norm{\psi_\varpi}_{C^{2k}}$.

Hence, taking $T_0$ large enough so that $C^{j+1}_{k,\b} \leq e^{\epsilon k (t+jT_0)}$, and combining the above estimates, we obtain
\begin{align*}
    \norm{\Lcal_{t+jT_0}(\psi_\varpi f)}_{k,B,s} 
    \leq  e^{ - (k -\epsilon)(t+ jT_0) }
    \norm{ f}_{k,B,s}
    + C_{k,\b,B,r,T_0} \norm{\Psi_{r,\th} f}'_k,
\end{align*}
for a suitably large constant $C_{k,\b,B,r,T_0}\geq 1$.


\subsection{Proof of Theorem~\ref{thm:intro meromorphic}}
\label{sec:proof of intro meromorphic}

Recall the notation in the statement of the theorem.
We note that switching the order of integration in the definition of the Laplace transform shows that
\begin{equation*}
    \hat{\rho}_{f,g}(z) = \int R(z)(f) g\;d\bms, \qquad \Re(z)>0.
\end{equation*}
In particular, the poles of $\hat{\rho}_{f,g}$ form a subset of the set of poles the resolvent $R(z)$.

On the other hand, Corollary~\ref{cor:strong continuity} implies that the infinitesimal generator $\mf{X}$ of the semigroup $\Lcal_t$ is well-defined as a closed operator on $\Bcal_k$ with dense domain.
Moreover, $R(z)$ coincides with the resolvent operator $(\mf{X}-z\id)^{-1}$ associated to $\mf{X}$, whenever $z$ belongs to the resolvent set (complement of the spectrum) of $\mf{X}$.

We further note that the spectra of $\mf{X}$ and $R(z)$ are related by the formula $\s(\mf{X}) = z-1/\s(R(z))$.
In particular, by Theorem~\ref{thm:resolvent spectrum2}, in the half plane $\Re(z)>-\s_0$, the poles of $R(z)$ coincide with the eigenvalues of $\mf{X}$.
In view of this relationship between the spectra, the fact that the imaginary axis does not contain any poles for the resolvent, apart from $0$, follows from the mixing property of the geodesic flow with respect to $\bms$ as shown in Lemma~\ref{lem:spectrum on imaginary axis} below.

Finally, we note that in the case $\G$ has cusps, $\b$ was an arbitrary constant in $(0,\Delta/2)$, so that we may take $\s_0$ in the conclusion of Theorem~\ref{thm:resolvent spectrum2} to be the minimum of $k$ and $\Delta/2$ in this case.
This completes the proof of Theorem~\ref{thm:intro meromorphic}.

\subsection{Resonances on the imaginary axis}
In this section, we study the intersection of the spectrum of $\mf{X}$ with the imaginary axis.

\begin{lem}\label{lem:spectrum on imaginary axis}
    The intersection of the spectrum of $\mf{X}$ with the imaginary axis consists only of the eigenvalue $0$ which has algebraic multiplicity one. 
\end{lem}

First, we need the following lemma relating our norms to correlation functions.

\begin{lem}\label{lem:norm controls correlations}
    For all $f,\vp\in C^{2}_c(X)^M$, we have that
    $
        \int f\cdot \vp\;d\bms \ll_\vp \norm{f}'_1$,
    where the implied constant depends on $\norm{\vp}_{C^2}$ and the injectivity radius of its support.
    \end{lem}
    \begin{proof}
    Using a partition of unity, we may assume $\vp$ is supported inside a flow box.
    The implied constant then depends on the number of elements of the partition of unity needed to cover the support of $\vp$.
    Inside each such flow box, the measure $\bms$ admits a disintegration in terms of the conditional measures $\mu_x^u$ averaged against a suitable measure on the transversal to the strong unstable foliation.
    Thus, the lemma follows by definition of the norm by viewing the restriction of $\vp$ to each local unstable leaf as a test function.
    \end{proof}

\begin{proof}[Proof of Lemma~\ref{lem:spectrum on imaginary axis}]
    In what follows, we endow elements $\vp$ of $C_c^2(X)$ with the norm $\norm{\vp}'_{C^2}$ given by multiplying the $C^2$-norm of $\vp$ with a suitable power of the reciprocal of the injectivity radius of its support so that $\norm{\vp}'_{C^2}$ dominates the implicit constant depending on $\vp$ in Lemma~\ref{lem:norm controls correlations}. Such power exists by the proof of the lemma. 
    The dual space $C_c^2(X)^\ast$ is endowed with the corresponding dual norm.

    First, we note that, since $\Bcal_k\subseteq \Bcal_1$ for all $k\geq 1$, it suffices to prove the lemma for the action of $\mf{X}$ on $\Bcal_1$.
    Let $\Phi:\Bcal_1 \to C_c^2(X)^\ast$ denote the linear map which extends the mapping $f\mapsto (\vp\mapsto \int f\vp\;d\bms)$ from $C_c^2(X)^M$ to the dual space $C_c^2(X)^\ast$.
    The fact that this mapping extends continuously to $\Bcal_1$ follows by Lemma~\ref{lem:norm controls correlations}.
    We claim that $\Phi$ is injective. This claim is routine in the absence of cusps, and we briefly outline why it also holds in general. 

    To prove this claim, note first that the coefficients $e_{1,0}(\cdot;x)$ and $e_{1,1}(\cdot;x)$ extend from $C_c^2$ to define seminorms on $\Bcal_1$.
    In particular, given any $f\in\Bcal_1$ and $f_n\in C_c^2(X)^M$ tending to $f$ in $\Bcal_1$, we have $e_{1,\ell}(f;x) = \lim_{n\to\infty} e_{1,\ell}(f_n;x)$ for $\ell=0,1$ and for every $x\in N_1^-\Omega$.
    Since the coefficient $e_{1,\ell}(f)$ is defined by taking a supremum over $x$, it follows that we can find a sequence $x_m\in N^-\Omega$ such that 
    $e_{1,\ell}(f;x_m)$ converges to $e_{1,\ell}(f)$.
    In particular, we obtain 
    \begin{align}\label{eq:switch sup and lim}
        e_{1,\ell}(f) = \lim_{m\to\infty }\lim_{n\to\infty} e_{1,\ell}(f_n; x_m).
    \end{align}

    Now, suppose $ f\in\Bcal_1$ is in the kernel of $\Phi$ and
    let $f_n\in C_c^2(X)^M$ be a sequence of functions converging to $f$. 
    By continuity, $\Phi(f_n)$ tends to $0$ in $C_c^2(X)^\ast$. One then checks that this implies that for every fixed $x\in N_1^-\Omega$, we have that\footnote{This is similar to the argument in the proof of~\eqref{eq: finitely many u pieces}. One proceeds by thickening test functions on $N_1^+\cdot x$ to functions supported in a small box around $x$ and controlling the difference between the integrals using $e_{1,0}(f_n;x)$ and the integral against the thickened functions using $e_{1,1}(f_n)$. The seminorms $e_{1,1}(f_n)$ remain bounded since $f_n\to f$, while the integrals against thickened functions tend to $0$ since $\Phi(f_n)\to 0$.} 
    $e_{1,0}(f_n;x)\to 0$ as $n\to \infty$.
    Hence, by~\eqref{eq:switch sup and lim}, we get that $e_{1,0}(f)=0$.
    Since $\norm{f}_1'\leq e_{1,0}(f)$,  this shows that $\norm{f}'_1=0$, and hence $\Phi$ is injective as claimed.

    We now show that this injectivity implies the lemma.
    Via the relationship between the spectra of $\mf{X}$ and the resolvents (cf.~Section~\ref{sec:proof of intro meromorphic}), Theorem~\ref{thm:resolvent spectrum2} implies that the intersection of the spectrum $\s(\mf{X})$ with the imaginary axis consists of a discrete set of eigenvalues.
    Similarly, finiteness of the multiplicities of each of these eigenvalues is a consequence of quasi-compactness of the resolvent. 
    
    Let $b\in \R$ be such that $ib$ is one such eigenvalue with eigenvector $0\neq f\in\Bcal_1$ and note that this implies that $\Lcal_t f = e^{ibt} f$.
    We show that $\Phi(f)$ is a multiple of the measure $\bms$. This implies that $b=0$ by injectivity since $\bms$ is the image of the constant function $1$ under $\Phi$.
    To do so, we use the fact that the geodesic flow is mixing\footnote{We refer the reader to~\cite[Corollary 5.4]{BaladiDemersLiverani} for this deduction using only ergodicity of the flow.} with respect to $\bms$ by work of Rudolph~\cite{Rudolph} and Babillot~\cite{Babillot}.
    Let $\vp\in C_c^2(X)$ be arbitrary and let $\th_n = \int f_n\;d\bms$ and $\xi = \int\vp\;d\bms$.
    Then, for every $t\geq 0$ and $n\in \N$, we have
    \begin{align}\label{eq:triangle ineq 1}
        \left| \Phi(f)(\vp) - \th_n \xi \right| \leq 
        \left| \Phi( f)(\vp) - \int \vp \Lcal_tf_n\;d\bms \right|
        + \left| \int \vp \Lcal_tf_n\;d\bms -\th_n \xi \right|.
    \end{align}
    By mixing, for every fixed $n$, the second term can be made arbitrarily small by taking $t$ large enough.
    Moreover, since $\Phi(f) = e^{-ibt} \Phi(\Lcal_t f)$, the first term is bounded by
    \begin{align}\label{eq:triangle ineq 2}
       \left|e^{-ibt} \Phi(\Lcal_t f)(\vp) - e^{-ibt}\int \vp \Lcal_tf_n\;d\bms \right|
       +| e^{-ibt}-1| \left|  \int \vp \Lcal_tf_n\;d\bms \right|.
    \end{align}
    The first term in~\eqref{eq:triangle ineq 2} is equal to $|\Phi(\Lcal_t(f-f_n)(\vp)|$, which is $O_\vp(\norm{f-f_n}_1)$ in view of Lemmas~\ref{lem:norm controls correlations} and~\ref{lem: equicts}.
    Similarly, since $f_n$ converges to $f$ in $\Bcal_1$, the second term is $O_\vp(| e^{-ibt}-1|\norm{f}_1)$.
    To bound this term, note that one can find arbitrarily large $t$ so that $e^{ibt}$ is arbitrarily close to $1$.
    
    Therefore, using a diagonal argument, this implies that we can find a sequence $t(n)$ tending to infinity so that the upper bound in~\eqref{eq:triangle ineq 1} tends to $0$ with $n$.
    If $\xi\neq 0$, the above argument implies that $\th_n$ is $O_\vp(\Phi(f)(\vp))$ and hence converges (along a subsequence) to some $\th\in \R$. In particular, the values of $\Phi(f)$ and $ \th \bms$ agree on $\vp$ in this case. If $\xi=0$, then the above argument shows that $\Phi(f)(\vp)=0$ so that the same conclusion also holds.
    
    The assertion on the algebraic multiplicity, which in particular involves ruling out the presence of Jordan blocks, is standard and can be deduced from quasi-compactness of the resolvent and the bound on its norm given in Corollary~\ref{cor:norm of resolvent} following very similar lines to~\cite[Corollary 5.4]{BaladiDemersLiverani} to which we refer the interested reader for details.
    \qedhere 
    
\end{proof}

\subsection{Exponential recurrence from the cusp and Proof of Theorem~\ref{thm:exp recurrence intro}}\label{sec:exp recur}

As a corollary of our analysis, we obtain the following stronger form of Theorem~\ref{thm:exp recurrence intro} regarding the exponential decay of the measure of orbits spending a large proportion of their time in the cusp. This result is crucial to our arguments in later sections. 
The deduction of Theorem~\ref{thm:exp recurrence intro} in its continuous time formulation from the following result follows using Proposition~\ref{prop:height function properties} and is left to the reader.

\begin{thm}\label{thm:exp recurrence}
For every $\e >0$, there exists $r_0\asymp_\b 1/\e$ such that the following holds for all $m\in\N, r\geq r_0, 0<\th<1$ and $x\in N_1^- \Omega$.
Let $H = e^{3\b r_0}$, and let $\chi_H$ be the indicator function of the set $\set{x: V(x)> H}$.
Then,
    \begin{align*}
        \mu_x^u
        \left(n\in N_1^+: \sum_{1\leq\ell\leq m}  \chi_H(g_{r\ell} nx) > \th m\right)
        \leq e^{-(\b\th -\e )m } V(x) \mu_x^u(N_1^+).
    \end{align*}
\end{thm}

\begin{proof}
    
    The argument is very similar to the proof of the estimate~\eqref{eq:exp recur from cusp}, with small modifications allowing for the height $H$ to be independent of the step size $r$.
    This subtle difference from~\eqref{eq:exp recur from cusp} will be important later in the proof of Corollary~\ref{cor:aff non-conc of projections}.
    
    Let $r_0\geq 1$ to be chosen later in the argument depending on $\e$ and $\b$ and set $V_0 = e^{2\b r_0}$.
    As before, let $\rho_{V_0}:X\to [0,1]$ denote a smooth compactly supported function which is identically $1$ on $\set{V\leq V_0}$ and vanishing outside $\set{V>2V_0}$.
    Let $\psi = 1-\rho_{V_0}$.
    Let $r\geq r_0$ and define the following operators:
    \begin{align*}
        \tilde{\Lcal}_1(f) := \Lcal_rf, \qquad \tilde{\Lcal}_2(f) = \Lcal_r(\psi f).
    \end{align*}
    Note that, unlike our previous arguments, the operators $\tilde{\Lcal}_i$ do not provide a decomposition of $\Lcal_r$, i.e., $\Lcal_r\neq \tilde{\Lcal}_1 +\tilde{\Lcal}_2$. 
    Given $m\in\N$ and $\varpi \in \set{1,2}^m$, let $\Lcal_\varpi = \tilde{\Lcal}_{\varpi_1} \circ \cdots \circ \tilde{\Lcal}_{\varpi_m}$. We also have that
    \begin{align*}
        \Lcal_\varpi(f) = \Lcal_{mr}(\psi_\varpi f), \qquad \text{where } \quad
        \psi_\varpi = \prod_{\ell:\varpi_\ell=2} \psi\circ g_{(\ell-m)r}.
    \end{align*}
    Similarly to Lemma~\ref{lem:norms of operator pieces}, Lemma~\ref{lem: equicts} implies the bounds
    \begin{align}\label{eq:operator pieces exp recur}
        e_{1,0}(\tilde{\Lcal}_1f)\ll_\b e_{1,0}(f), \qquad e_{1,0}(\tilde{\Lcal}_2 f) \ll_\b e^{-\b r_0} e_{1,0}(\psi f)\ll e^{-\b r_0} e_{1,0}(f).
    \end{align}
    
    Let $H=e^{3\b r_0}$. 
    We shall assume that $r_0$ is large enough so that $H>2V_0$.
    Define
    \begin{align*}
        E_\varpi = \set{n\in N_1^+: \varpi_\ell = 2 \Rightarrow V(g_{\ell r}nx)>  H , \text{ for all } \ell=1,\dots,m}.
    \end{align*}
    Then, for all $n\in N_1^+$,
    \begin{align}\label{eq:lower bound on psi_varpi}
        \psi_\varpi(g_{mr}nx)\geq \mathbbm{1}_{E_\varpi}(n).
    \end{align}
    Indeed, if $\mathbbm{1}_{E_\varpi}(n)=1$, 
    and $\ell$ is such that $\varpi_\ell=2$, then $V(g_{\ell r}nx)>H>2V_0$ and, hence,
    $\psi(g_{\ell r}nx)=1$.
    It follows that
    \begin{align*}
        \psi_\varpi(g_{mr}nx)
        = \prod_{\ell:\varpi_\ell=2} \psi(g_{\ell r}nx) =1.
    \end{align*}
    This verifies~\eqref{eq:lower bound on psi_varpi}.
    
    Denote by $\th_\varpi$ the number of indices $\ell$ for which $\varpi_\ell=2$.
    Then, we see that
    \begin{align*}
        \set{n\in N_1^+: \sum_{1\leq \ell\leq m} \chi_H(g_{r\ell} nx) > \th m}
        \subseteq \bigcup_{\varpi:\th_\varpi >\th m} E_\varpi.
    \end{align*}
    We wish to apply~\eqref{eq:operator pieces exp recur} with $f$ the constant function on $X$.
    One checks that this $f$ belongs to the space $\Bcal_{1}$ and $e_{1,0}(f)\ll 1$.
    Let $C_1\geq 1$ denote a constant larger than $e_{1,0}(f)$ and the two implicit constants in~\eqref{eq:operator pieces exp recur}.
    Then, applying~\eqref{eq:operator pieces exp recur} iteratively $m$ times, and using~\eqref{eq:lower bound on psi_varpi}, we obtain
    \begin{align*}
        \muxu(E_\varpi) \leq e_{1,0}(\Lcal_\varpi(f))\leq C_1^{m} e^{-\b \th_{\varpi} r_0} V(x)\muxu(N_1^+) e_{1,0}(f) \leq C_1^{m+1} e^{-\b \th mr_0} V(x)\muxu(N_1^+).
    \end{align*}
    Since there are at most $2^m$ choices of $\varpi$, the result follows by taking $r_0$ large enough so that $2^mC_1^{m+1}\leq e^{\e mr_0}$.
    \qedhere
   
\end{proof}


\section{Fractal Mollifiers}
\label{sec:mollifiers}

In this section, we introduce certain mollification operators on smooth functions on $X$. These operators have the advantage that, roughly speaking, their Lipschitz norms are dominated by the norms introduced in~\eqref{eq: norms}. This property is very convenient in the estimates carried out in Section~\ref{sec:Dolgopyat}. The idea of using mollifiers to handle analogous steps is due to~\cite{BaladiLiverani}.

\subsection{Definition and regularity of mollifiers}
Fix a non-negative $C^\infty$ bump function $\psi$ supported inside $N_{1/2}^+$ and having value identically $1$ on $N_{1/4}^+$.
We also choose $\psi$ to be symmetric and $\Ad(M)$-invariant, i.e.
\begin{equation}\label{eq:psi symmetry}
    \psi(n) = \psi(n^{-1}), \qquad \psi(mnm^{-1})=\psi(n), \quad \forall n\in N^+, m\in M.
\end{equation}
Given $\e>0$, define $\M_\e:C(X)\to C(X)$ be the operator defined by
\begin{align}\label{eq:psi_epsilon}
    \M_\e(f)(x) = \int \frac{\psi_{\e}(n)}{\int \psi_\e \;d\mu_{nx}^u} f(nx)\;d\mu_{x}^u(n),
\qquad 
    \psi_{\e}(n) =\psi(\Ad(g_{-\log \e})(n)).
\end{align}
Note that $\psi_\e$ is supported inside $N_{\e/2}^+$.

\begin{remark}
The condition that $\psi_\e(\id)=\psi(\id)=1$ implies that for $x\in X$ with $x^+\in \L_\G$, 
\begin{equation}
    \mu_x^u(\psi_\e)>0, \qquad \forall \e>0.
\end{equation}
In particular, since the conditional measures $\mu_x^u$ are supported on points $nx$ with $(nx)^+\in \L_\G$, the mollifier $\M_\e(f)$ is a well-defined function on all of $X$.
That $\M_\e(f)$ is continuous follows by continuity of the map $x\mapsto \mu_x^u$ in the weak-$\ast$ topology; cf.~\cite[Lemme 1.16]{Roblin}.
\end{remark}

\begin{remark}\label{remark:M invariance}
We note that $\Ad(M)$-invariance of $\psi_\e$ and the conditional measures $\mu_x^u$ (cf.~\eqref{eq:M equivariance}) implies that $\M_\e(f)$ is $M$-invariant whenever $f$ is.
\end{remark}

To simplify arguments related to the regularity of the function $n\mapsto \psi_\e(n)/\mu_{nx}^u(\psi_\e)$, we introduce the following slightly stronger version of the norm $\norm{\cdot}_1$ which suffices for our purposes.

\begin{definition}
    [The Banach Space $\Bcal_\star$]
    \label{def:Bcal_star}
Let $C^{k,\a}(N_1^+)$ denote the space of $C^k-$functions $\phi$ on $N_1^+$, all of whose derivatives of order $k$ are $\a$-H\"older continuous functions on $N_1^+$.
We endow this space with the standard norm denoted $\norm{\phi}_{C^{k,\a}}$.
We define coefficients $e^\star_{1,0}(f)$ and $e^\star_{1,1}(f)$, similarly to the coefficients $e_{1,0}$ and $e_{1,1}$ respectively in~\eqref{eq: ekl gamma x} and~\eqref{eq: ekl}, but where, in both coefficients, the supremum is taken over all test functions $\phi\in C^{0,1}(N_1^+)$ with $\norm{\phi}_{C^{0,1}}\leq 1$, instead of  $C^1(N_1^+)$ and $C^2(N_1^+)$.
Using these definitions, we introduce the following seminorm on $C_c^2(X)$:
\begin{align}\label{eq:norm star}
    \norm{f}^\star_1 =e^\star_{1,0}(f) + e^\star_{1,1}(f).
\end{align}
We denote by $\Bcal_\star$ the Banach space completion of the quotient space $C_c^2(X)^M$ of $M$-invariant compactly supported $C^2$-functions by the kernel of the seminorm $\norm{\cdot}^\star_1$ with respect to the induced norm on the quotient. 
\end{definition}

The first result asserts that $\M_\e(f)$ is a good approximation of $f$.
\begin{prop}\label{prop:mollifiers are close}
For all $0<\e\leq 1/10$, and $t\geq 1$, we have
\begin{align*}
    e^\star_{1,0}(\Lcal_t(f-\M_\e(f))) \ll (\e+1) e^{-t}  e^\star_{1,0}(f).
\end{align*}
\end{prop}
In light of this statement, we will in fact only use $\M_\e$ with $\e=1/10$. However, for clarity, we state and prove the remaining results for a general value of $\e$.

The following results estimate the regularity of mollifiers.
Recall the constant $\Delta_+\geq 0$ in~\eqref{eq:Delta}.
The first result is an estimate of $L^\infty$ type.
\begin{prop}\label{prop:upper estimate on mollifiers}
For all $0<\e\leq 1$ and $x\in N_1^-\Omega$, we have
\begin{align*}
    |\M_\e(f)(x)| \ll \e^{-\Delta_+-1} e^\star_{1,0}(f)V(x).
\end{align*}
\end{prop}

Finally, we need the following Lipschitz estimate on mollifiers along the stable direction.
Recall the stable parabolic group $P^-=N^-AM$ parametrizing the weak stable manifolds of $g_t$. 
\begin{prop}\label{prop:stable derivatives of mollifiers}
For all $0<\e\leq 1/10$, $p^-\in P^-$, and $x\in X$ so that $x$ belongs to $N^-_{3/4}\Omega$ and $p^-$ is of the form $u^-g_tm$ for $u^-\in N^-_{1/10}$, $|t|\leq 1/10$ and $m\in M$, we have that
\begin{equation*}
    |\M_\e (f)(p^-x) - \M_{\e }(f)(x)|\ll
    \dist(p^-,\id)
    \e^{-\Delta_+-2}\cdot \norm{f}_1^\star V(x).
\end{equation*}
\end{prop}

The above results are straightforward in the case of smooth mollifiers, however some care is required in our case due to the fractal nature of the conditionals and (possible) non-compactness of $\Omega$.
This is in part the reason for the non-standard shape of the chosen mollifier.
The proofs of the above results are rather technical and can be skipped on a first reading.

\subsection{Preliminary estimates}
We begin by providing some tameness estimates for our mollifiers. 
The first lemma extends the applicability of Proposition~\ref{prop:doubling} to points that are near, but not necessarily in, $\Omega$.

\begin{lem}\label{lem:doubling away from omega}
For all $x\in N_1^-\Omega$, and $0<\e \leq 1$, we have
\begin{equation*}
    \frac{\mu^u_{nx}(N_{5\e}^+)}{\mu_{nx}^u(\psi_\e)} \ll 1,
\end{equation*}
uniformly over $n\in N_1^+$ in the $(\e/10)$-neighborhood of the support of $\mu_x^u$.
\end{lem}

\begin{proof}
Since $\psi_\e\equiv 1$ on $N^+_{\e/4}$, 
$  \mu_{nx}^u(\psi_\e) \geq \mu_{nx}^u(N_{\e/4}^+)$.
Let $u$ be in the support of $\mu_{x}^u$, which is at distance $\e/10$ from $n$.
In particular, $ux\in N_2^-\Omega$ by Remark~\ref{rem:commutation of stable and unstable}.
Hence, using a change of variables and Proposition~\ref{prop:doubling}, we obtain
\begin{equation*}
   \frac{\mu_{nx}^u(N_{5\e}^+)}{\mu_{nx}^u(\psi_\e)}
    \leq \frac{\mu_{nx}^u(N_{5\e}^+)}{\mu_{nx}^u(N^+_{\e/4})}
   \leq \frac{\mu^u_{u x}(N^+_{5\e}\cdot (n u^{-1}))}{\mu^u_{u x}(N^+_{\e/4}\cdot (n u^{-1}))}
   \leq  \frac{\mu^u_{u x}(N^+_{6\e})}{\mu^u_{u x}(N^+_{\e/8})}
   \ll 1.
\end{equation*}

\end{proof}

The next statement is roughly a Lipschitz estimate on conditional measures.
\begin{lem}\label{lem:reciprocal difference}
For all $0<\e \leq 1$ and $x\in N_1^-\Omega$, we have the following.
For all $n_1,n_2\in N_1^+$ with $d_{N^+}(n_1,n_2)\leq \e/2$, we have
\begin{equation*}
    \left|\frac{1}{\mu_{n_1x}^u(\psi_\e)}- \frac{1}{\mu_{n_2x}^u(\psi_\e)}\right| \ll \frac{\e^{-1}d_{N^+}(n_1,n_2)}{\mu_{n_2x}^u(\psi_\e)},
\end{equation*}
provided $n_1$ is at distance at most $\e/10$ from the support of $\mu_x^u$. 
\end{lem}

\begin{proof}
Let $\s=n_1n_2^{-1}$.
Since $\psi_\e$ is supported inside $N_{\e/2}^+$.
we have by the symmetry of $\psi$ in~\eqref{eq:psi symmetry} and the right invariance of the metric $d_{N^+}$ on $N^+$ that
\begin{align*}
    |\mu_{n_1x}^u(\psi_\e) - \mu_{n_2x}^u(\psi_\e)|
    &\leq \int |\psi_\e(n)-\psi_\e(n\s)|\;d\mu_{n_1x}^u(n) 
    =\int |\psi_\e(n^{-1})-\psi_\e(\s^{-1}n^{-1})|\;d\mu_{n_1x}^u(n) 
    \\
    &\ll \norm{\psi_\e}_{C^1} d_{N^+}(n_1,n_2) \mu^u_{n_1x}(N_\e^+),
\end{align*}
where on the last line we used that the integrands are non-zero only on the union $N_{\e/2}^+\cup N_{\e/2}^+\s\subseteq N_\e^+$. Lemma~\ref{lem:doubling away from omega} also implies that $ \mu_{n_1x}^u(N_\e^+)/\mu_{n_1x}^u(\psi_\e)\ll 1$.
The lemma follows since $\norm{\psi_\e}_{C^1}\ll \e^{-1}$. 
\end{proof}

\subsection{Regularity of mollifiers and proof of Proposition~\ref{prop:mollifiers are close}}

Let $\vp\in C^{0,1}(N_{1}^+)$ be a test function and let $x\in N_1^-\Omega$. 
Set $\vp_t = \vp\circ \Ad(g_{-t})$ and $x_t=g_t x$.
Then, using~\eqref{eq:g_t equivariance} to change variables, we obtain
\begin{align*}
    \int \vp(n) \M_\e(f)(g_t nx)\;d\mu_x^u(n)
    = e^{-\d t} \int \vp_t(n) \M_\e(f)(nx_t)\;d\mu_{x_t}^u(n).
\end{align*}

We can rewrite the integral on the right side in a convenient form using the following series of formal manipulations.
Let $\psi_{\e,y}(n)=\psi_\e(n)/\mu_y^u(\psi_\e)$.
First, using the definition of $\M_\e$ and~\eqref{eq:N equivariance} to change variables, we get
\begin{align*}
    \int \vp_t(n)\M_\e(f)(nx_t) \;d\mu_{x_t}^u(n)
    &= \int \vp_t(n)\int \psi_{\e,n'nx_t}(n') f(n'nx_t)\;d\mu_{nx_t}^u(n') \;d\mu_{x_t}^u(n)
    \\  
    &=\int \vp_t(n)\int \psi_{\e,n'x_t}(n'n^{-1}) f(n'x_t)\;d\mu_{x_t}^u(n') \;d\mu_{x_t}^u(n).
\end{align*}
Next, using Fubini's Theorem and the symmetry of $\psi_\e$ provided by~\eqref{eq:psi symmetry}, we get
\begin{align*}
    \int \vp_t(n)\M_\e(f)(nx_t) \;d\mu_{x_t}^u(n)
    &=\int \bigg(\int \vp_t(n)\psi_{\e,n'x_t}(n'n^{-1})\;d\mu_{x_t}^u(n)\bigg) f(n'x_t)\;d\mu_{x_t}^u(n')
    \\
    &= \int \bigg(\int \vp_t(n)\psi_{\e,n'x_t}(n(n')^{-1})\;d\mu_{x_t}^u(n)\bigg) f(n'x_t)\;d\mu_{x_t}^u(n').
\end{align*}
Finally, we obtain the desired convenient form of the integral upon changing variables  using~\eqref{eq:N equivariance} once more to get
\begin{align*}
    \int \vp_t(n)\M_\e(f)(nx_t) \;d\mu_{x_t}^u(n)
    =\int \bigg(\int \vp_t(nn')\psi_{\e,n'x_t}(n)\;d\mu_{n'x_t}^u(n)\bigg) f(n'x_t)\;d\mu_{x_t}^u(n').
\end{align*}
It is thus natural to define the following function:
\begin{equation*}
    \Phi_{\e,x,t}(n'):=\int \vp_t(nn')\psi_{\e, n'x_t}(n)\;d\mu_{n'x_t}^u(n)= \frac{\int \vp_t(nn')\psi_{\e}(n)\;d\mu_{n'x_t}^u(n)}{\mu_{n'x_t}^u(\psi_\e)} .
\end{equation*}
Note that, since $\vp_t$ and $\psi_\e$ are supported in $N_{e^t}^+$ and $N_{\e/2}^+$ respectively, $\Phi_{\e,x,t}$ is supported inside $N_{e^t+\e/2}^+\subset N_{2e^t}^+$.

We wish to estimate integrals of the form
\begin{align*}
    \int  (\vp_{t+s}(n') - \Phi_{\e,x,t+s}(n') f(n'x_{t+s})\;d\mu_{x_{t+s}}^u(n'),
\end{align*}
for arbitrary $t\geq 1$, $s\in [0,1]$, basepoints $x$ and test functions $\vp$.
First, we note that it suffices to estimate the integrals when $s=0$ since $t$ is at least $1$ by assumption.
We proceed by essentially regarding $\vp_t -\Phi_{\e,x,t}$ itself as a test function.
Note that $\Phi_{\e,x,t}$ may not be well-defined for arbirary $n'\in N^+$, since $\mu_{n'x_t}^u(\psi_\e)$ could be $0$ for those $n'$ with $(n'x_t)^+\notin \L_\G$.
However, $\Phi_{\e,x,t}$ is well-defined on the $(\e/4)$-neighborhood of the support of $\mu_{x_t}^u$ by definition of $\psi_\e$.

For this reason, let $\th_\e:N^+\to [0,1]$ be a smooth bump function which is identically $1$ on the $(\e/100)$-neighborhood of the support of the measure $\mu_{x_t}^u$ and vanishes outside of its $(\e/50)$-neighborhood.
We can choose such a function to satisfy
\begin{equation*}
    \norm{\th_\e}_{C^0}\leq 1, \qquad \norm{\th_\e}_{C^1}\ll \e^{-1},
\end{equation*}
for instance by convolving (with respect to the Haar measure) the indicator function of the $(\e/100)$-neighborhood of the support with $\psi_{\e/200}$.
Then, we observe that
\begin{equation*}
\int (\vp_t-\Phi_{\e,x,t})(n') f(n'x_t)\;d\mu_{x_t}^u(n') = \int ((\vp_t-\Phi_{\e,x,t})\th_\e)(n')f(n'x_t)\;d\mu_{x_t}^u(n').
\end{equation*}
The upshot is that $\vartheta :=(\vp_t-\Phi_{\e,x,t})\th_\e$ is a well-defined function on $N^+$.
Thus, arguing exactly as in the proof of Lemma~\ref{lem: equicts}, the conclusion of the proposition will follow as soon as we estimate the norm $\norm{\vartheta}_{C^{0,1}}$; cf.~\eqref{eq:reduce to ht estimate}.

We begin by estimating $\norm{\vartheta}_{C^0}$.
Let $n'\in N^+$ be in the support of $\th_\e$.
Note that
\begin{equation}\label{eq:prop measure}
    \vp_t(n') = \frac{\int \vp_t(n') \psi_\e(n) \;d\mu_{n'x_t}^u(n)}{\mu^u_{n'x_t}(\psi_\e)},
\end{equation} 
and, hence,
\begin{equation*}
    |\vp_t(n')-\Phi_{\e,x,t}(n')| \leq 
    \frac{\int |\vp_t(n')-\vp_t(nn')| \psi_\e(n) \;d\mu^u_{n'x_t}(n)}{\mu^u_{n'x_t}(\psi_\e)}.
\end{equation*}
We further observe that if $\psi_{\e}(n)\neq 0$ for some $n\in N^+$, then $n$ is at distance at most $\e/2$ from identity.
Moreover, since $\norm{\vp}_{C^{0,1}}\leq 1$ and $g_t$ expands $N^+$ by at least $e^t$, the Lipschitz constant of $\vp_t$ is at most $e^{-t}$.
Hence, using the right invariance of the metric on $N^+$, for any such $n$, $|\vp_t(n')-\vp_t(nn')|\leq e^{-t}\e/2$.
As $n'$ was arbitrary and $|\th_\e(n')|\leq 1$, 
it follows that $\norm{\vartheta}_{C^0}\leq e^{-t}\e/2$.

It remains to estimate the Lipschitz constant of $\vartheta$. Let $n_1,n_2\in N_1^+$ be arbitrary points in the support of $\th_\e$.
Then, note that since $\norm{\theta_\eta}_{C^0}\leq 1$ and $\norm{\th_\e}_{C^1} \ll \e^{-1}$, we have
\begin{align*}
    &|\vartheta(n_1)-\vartheta(n_2)|
    \nonumber\\
    &\ll 
    |(\vp_t-\Phi_{\e,x,t})(n_1)
    - (\vp_t-\Phi_{\e,x,t})(n_2)|+ |(\vp_t-\Phi_{\e,x,t})(n_1)|\norm{\th_\e}_{C^1}
    d_{N^+}(n_1,n_2)
    \nonumber\\
    &\ll |(\vp_t-\Phi_{\e,x,t})(n_1)
    - (\vp_t-\Phi_{\e,x,t})(n_2)|
    + e^{-t}  d_{N^+}(n_1,n_2).
\end{align*}

Let $\s=n_1n_2^{-1}$.
Using~\eqref{eq:prop measure} and a change of variable, we have
\begin{align*}
   (\vp_t-\Phi_{\e,x,t})(n_1)
    - &(\vp_t-\Phi_{\e,x,t})(n_2)
    \nonumber\\
    &=\frac{\int (\vp_t(n_1)-\vp_t(nn_1))(\psi_\e(n)-\psi_\e(n\s))\;d\mu^u_{n_1 x_t}(n) }{\mu^u_{n_1 x_t}(\psi_\e)}
    \nonumber\\
    &+ \int (\vp_t(n_2)-\vp_t(nn_2))\psi_\e(n)\;d\mu_{n_2x_t}^u(n)
    \times \left(\frac{1}{\mu_{n_1x_t}^u(\psi_\e)}- \frac{1}{\mu_{n_2x_t}^u(\psi_\e)} \right).
\end{align*}

In estimating the Lipschitz constant, without loss of generality, we may assume that the distance between $n_1$ and $n_2$ is at most $\e/2$.
Hence, arguing as before and using Lemma~\ref{lem:reciprocal difference}, we obtain the following estimate on the second term:
\begin{align*}
    \int (\vp_t(n_2)-\vp_t(nn_2))
    \psi_\e(n)\;d\mu_{n_2x_t}^u(n)
    \times 
    &\left(\frac{1}{\mu_{n_1x_t}^u(\psi_\e)}- \frac{1}{\mu_{n_2x_t}^u(\psi_\e)} \right)
    \nonumber\\
    &\ll \e e^{-t} \mu^u_{n_2 x_t}(\psi_\e) \frac{\e^{-1}d_{N^+}(n_1,n_2)}{\mu^u_{n_2 x_t}(\psi_\e)} 
    \leq  e^{-t} d_{N^+}(n_1,n_2).
\end{align*}

To estimate the first term, note that symmetry of $\psi_\e$ (cf.~\eqref{eq:psi symmetry}) implies that $|\psi_\e(n)-\psi_\e(n\s)|$ is $O( \e^{-1}  d_{N^+}(n_1,n_2))$.
Moreover, since $\s=n_1n_2^{-1}$ and the support of $\psi_\e$ are contained in $N^+_{\e/2}$, the function $\psi_\e(n)-\psi_\e(n\s)$ is supported inside $N_{\e}^+$.
Hence,
\begin{align*}
    \int (\vp_t(n_1)-\vp_t(nn_1))(\psi_\e(n)-\psi_\e(n\s))\;d\mu^u_{n_1 x_t}(n)
    \ll \e e^{-t}
    \times \e^{-1}  d_{N^+}(n_1,n_2)
    \times  \mu^u_{n_1 x_t}(N^+_\e).
\end{align*}
Combined with Lemma~\ref{lem:doubling away from omega} and the estimate on the second term, this shows that
\begin{align*}
    |\vartheta(n_1)-\vartheta(n_2)|
    \ll e^{-t} d_{N^+}(n_1,n_2),
\end{align*}
thus completing the proof.

\subsection{Pointwise estimates and proof of Proposition~\ref{prop:upper estimate on mollifiers}}
As in the proof of Proposition~\ref{prop:mollifiers are close}, let $\th_\e:N^+\to [0,1]$ denote a smooth function that is identicaly $1$ on the $(\e/100)$-neighborhood of the support of $\mu_x^u$ and vanishing outside its $(\e/50)$-neighborhood.
We again note that we can find such $\th_\e$ with $\norm{\th_\e}_{C^1}\ll \e^{-1}$.
Set $\Psi(n)=\psi_\e(n)\mu_x^u(N_1^+)/\mu_{nx}^u(\psi_\e)$ and note that~\eqref{eq:N equivariance} implies that the function $\Psi\th_\e $ belongs to $C^0(N_1^+)$.
Moreover, we have that
\begin{align*}
    \mathbb{M}_\e(f)(x) &= \frac{1}{\mu_x^u(N_1^+)}\int \Psi(n)f(nx)\;d\mu_x^u = 
    \frac{1}{\mu_x^u(N_1^+)}
    \int (\Psi\th_\e)(n)f(nx)\;d\mu_x^u
    \nonumber\\
    &\ll e^\star_{1,0}(f)V(x) \norm{\Psi\th_\e}_{C^{0,1}}.
\end{align*}

Hence, the result follows once we estimate the norm $\norm{\Psi \th_\e}_{C^{0,1}}$.
We begin by proving that $\norm{\Psi\th_\e}_{C^0}$ is $O( \e^{-\Delta_+})$.
As a first step, we show that
\begin{equation}\label{eq:ratio of ball to psi}
    \frac{\mu_x^u(N_1^+)}{\mu^u_{nx}(\psi_\e)}\ll \e^{-\Delta_+}, \qquad \forall n\in N^+, \psi_\e(n)\th_\e(n)\neq 0.
\end{equation}
Since $\norm{\th_\e\psi_\e}_{C^0}\leq 1$, this will show that $\norm{\Psi\th_\e}_{C^0}\ll \e^{-\Delta_+}$.

Fix some $n$ with $\psi_\e(n)\th_\e(n)\neq 0$.
Then, we can find $u$ in the $\e/2$ ball around identity in $N^+$ such that $ux$ belongs to $N_2^-\Omega$ (cf.~Remark~\ref{rem:commutation of stable and unstable}) and $u$ is at distance at most $10^{-2}\e$ from $n$.
Since $\psi_\e \equiv 1$ on $N_{\e/4}^+$, we have by~\eqref{eq:N equivariance} that
\begin{equation}\label{eq:psi is a fat indicator}
    \mu_{nx}^u(\psi_\e) \geq \mu_{nx}^u(N_{\e/4}^+)
    =\mu_{ux}^u(N_{\e/4}^+\cdot (nu^{-1}))
    \geq \mu_{ux}^u(N_{\e/10}^+).
\end{equation}
Similarly, we have that
\begin{equation*}
    \mu_x^u(N_1^+) \leq  
    \mu_{ux}^u(N_2^+).
\end{equation*}

Let $k\in \N$ be the smallest integer such that $2^{-k}\leq \e/4$.
Applying Proposition~\ref{prop:doubling} with $\s =2^{k+1}$ and $r=2^{-k}$, since $ux\in \Omega$, we obtain

\begin{equation*}
   \mu_{ux}^u(N_2^+)=\mu_{ux}^u(N_{2^{k+1}2^{-k}}^+) 
   \ll 2^{(k+1)\Delta_+} \mu_{ux}^u(N_{2^{-k}}^+) 
   \ll \e^{-\Delta_+} \mu_{ux}^u(N_{\e/4}^+).
\end{equation*}
Together with~\eqref{eq:psi is a fat indicator}, this concludes the proof of~\eqref{eq:ratio of ball to psi}.

Next, we estimate the Lipschitz norm of $\Psi\th_\e$ as a function on $N_1^+$.
Let $n_1,n_2\in N^+$ be such that $n_1n_2^{-1}\in N^+_{\e/10}$, and $(\th_\e\psi_\e)(n_i)\neq 0$ for $i=1,2$.
Then, Lemma~\ref{lem:reciprocal difference} and~\eqref{eq:ratio of ball to psi} imply that
\begin{align*}
    |\Psi(n_1)-\Psi(n_2)|
    &\leq 
    \mu_{x}^u(N_1^+) \left(\left|\frac{1}{\mu_{n_1x}^u(\psi_\e)}
    - \frac{1}{\mu_{n_2x}^u(\psi_\e)}\right|
    + 
    \frac{|\psi_\e(n_1)-\psi_\e(n_2)|}{\mu_{n_2x}^u(\psi_\e)}
    \right)
    \nonumber\\
    &\ll \e^{-1}d_{N^+}(n_1,n_2)
    \frac{ \mu_{x}^u(N_1^+)}{\mu_{n_2x}^u(\psi_\e)}
    \ll \e^{-\Delta_+-1} d_{N^+}(n_1,n_2).
\end{align*}
Since $\norm{\th_\e}_{C^0}\leq 1$ and $\norm{\th_\e}_{C^1}\ll \e^{-1}$, this shows that the Lipschitz norm of $\Psi\th_\e$ is at most $\e^{-\Delta_+-1}$ and concludes the proof.


\subsection{Weak stable derivatives and proof of Proposition~\ref{prop:stable derivatives of mollifiers}}

The idea of the proof is based on performing local stable holonomy between the strong unstable disks $N_1^+\cdot x$ and $N_1^+\cdot p^-x$ and proceeding exactly as in the proof of Prop.~\ref{prop:compact embedding}.
The main ingredient is an estimate on the regularity of the test functions arising from composing $\psi_\e(n)/\mu^u_{nx}(\psi_\e)$ with holonomy maps from $x$ to intermediate points between $x$ and $p^-x$ along the weak stable manifold. We omit the details of the proof since it follows by elaborating the same ideas in the proof of Prop.~\ref{prop:mollifiers are close}.
We only remark that for $p^-=u^-g_tm$ as in the statement, letting $w$ in the Lie algebra of $N^-$ be so that $u^-=\exp(w)$, then for all $r\in \R$ with $|r|\leq |t|$ and all $s\in [0,1]$, one checks that the points $\exp(sw)g_rmx$ all belong to $N_1^-\Omega$.
This is relevant in ensuring that the basepoints arising over the course of carrying out the analogous estimate to~\eqref{eq:z_t terms} all satisfy the requirement on basepoints for the norm $\norm{\cdot}_1^\star$.


\section{Spectral gap for resolvents with large imaginary parts}
\label{sec:Dolgopyat}

In this section, we establish the key estimate in the proof of Theorems~\ref{thm:intro mixing} and~\ref{thm:intro strip}.
The estimates in Sections~\ref{sec: spectral gap} and~\ref{sec:ess radius} allow us to show that there is a half plane $\set{\Re(z)>-\eta}$, for a suitable $\eta>0$, containing at most countably many isolated eigenvalues for the generator of the geodesic flow. To show exponential mixing, it remains to rule out the accumulation of such eigenvalues on the imaginary axis as their imaginary part tends to $\infty$.

\begin{remark}\label{remark:choice of V}
Throughout the rest of this section, if $X$ has cusps, we require the Margulis function $V=V_\b$ in the definition of all the norms we use to have 
\begin{equation}\label{eq:small beta}
    \gls{beta}= \Delta/4
\end{equation}
in the notation of Theorem~\ref{thm:Margulis function}.
In particular, the contraction estimate in Theorem~\ref{thm:Margulis function} holds with $V^p$ in place of $V$ for all $1\leq p \leq 2$. Recall that the constant $\Delta$ is given in~\eqref{eq:Delta}.
\end{remark}

Similarly to~\eqref{eq:equivalent norm}, we define for $B\neq 0$ an equivalent norm to $\norm{\cdot}^\star_1$ defined in~\eqref{eq:norm star} as follows:
\begin{equation}\label{eq:norm star b}
    \norm{f}^\star_{1,B} := e^\star_{1,0}(f) + \frac{e^\star_{1,1}(f)}{B}.
\end{equation}

The following result is one of the main technical contributions of this article.

\begin{thm}\label{thm:Dolgopyat}
There exist constants $b_\star\geq 1$, and $\varkappa, a_\star,\s_\star>0$,
such that the following holds.
For all $z=a_\star+ib\in \C$ with $|b|\geq b_\star$ and for $m= \lceil \log |b|\rceil$, we have that

\begin{equation*}
    e^\star_{1,0}(R(z)^mf) \leq C_\G 
    \frac{\norm{f}^\star_{1,B}}{(a_\star+\s_\star)^m},
\end{equation*}
where $C_\G\geq 1$ is a constant depending only on the fundamental group $\G$ and $  B = |b|^{1+\varkappa}$.

\end{thm}

\begin{remark}
    The constants $\varkappa,a_\star$, and $\s_\star$ depend only on non-concentration parameters of the Patterson-Sullivan measure near proper subvarieties of the boundary at infinity; cf.~Definition~\ref{def:aff non-conc} for the precise definition of non-concentration and Corollary~\ref{cor:aff non-conc of projections} where this non-concentration is established.
    This non-concentration property is used to apply the results of Section~\ref{sec:flattening} in the proof of Prop.~\ref{prop:close pairs} and Theorem~\ref{thm:counting bad frequencies}, which are the key steps in the proof of Theorem~\ref{thm:Dolgopyat}.
\end{remark}

\subsection{Sketch of the proof}
\label{sec:sketch}

We begin with a rough sketch of the proof of Theorem~\ref{thm:Dolgopyat}. To describe the main new idea based on additive combinatorics concisely, we will ignore many technical difficulties, including those posed by the presence of cusps. 

Fix $z=a+ib$, where $a>0$ is suitably small and $|b|$ is sufficiently large. Let $m=\lceil \log |b|\rceil$.
The integrals we wish to estimate take the form
\begin{align*}
     \int_0^{\infty} \frac{t^{m-1} }{(m-1)!} e^{-zt} \int_{N_1^+} \phi(n)  \Lcal_tf(nx)\;d\mu_x^u dt.
\end{align*}

Using convergence of the integral defining $R(z)^m$, we will trivially estimate over the parts of the integral where $t$ is very large and relatively small. 
The bulk of the work lies in finding $\s>0$ so that the following bound holds in the range $T\asymp \log |b|$:
\begin{align*}
    \left|\int_{T}^{T+1} e^{-zt} \int_{N_1^+} \phi(n) f(g_t nx)\;d\mu^u_x dt \right|
    \ll e^{-aT} \norm{f} |b|^{-\s}.
\end{align*}

The key to the proof is to exploit the oscillations of the phase function $e^{ibt}$ and (average) Fourier decay properties of the measures $\mu_x^u$ obtained in Section~\ref{sec:flattening}. To do this, we put the integral in a form where this phase is integrated against $\mu_x^u$. This is achieved using non-joint integrability of the stable and unstable foliations (i.e. that $N^+$ and $N^-$ do not commute).

To this end, we partition the space into flow boxes ${B_\rho}$, apply the geodesic flow by amount $T$, and group pieces of the expanded unstable manifold according to which flow box they land in to get:
\begin{align*}
    e^{-\d T}
    \sum_{\text{flow boxes } B_\rho} \quad \sum_{\substack{\text{ conn. comps. of } \\ g_TN_1^+x \cap B_\rho}} 
    \quad 
    \int_0^1 e^{-zt}\int_{\text{conn. component}(x_{\rho,\ell})} 
    \phi \circ g_{-T} \cdot
    f(g_t n x_{\rho,\ell}) \;d\mu^u_{x_{\rho,\ell}}(n) dt.
\end{align*}
Here, the points $\set{x_{\rho,\ell}:\ell}$ are transverse intersection points of the expanded unstable manifolds $g_TN_1^+x$ with a fixed transversal to the unstable foliation inside the flow box $B_\rho$.

Since the derivatives of $\phi \circ g_{-T}$ along $N^+$ are $O(e^{-T})$, these functions are nearly constant along each connected component and we can ignore them in the sequel.

Fix a box $B_\rho$ and a reference point $y_\rho\in B_\rho$.
We view the above integrals as taking place on the \textit{weak} unstable manifold of $x_{\rho,\ell}$ for each $\ell$.
We change variables using local strong stable holonomy so that all the integrals are taking place along the local weak unstable manifold of our fixed reference point $y_\rho$ to get:
\begin{align*}
     \sum_\ell \int_0^1\int_{N_1^+} e^{-ib(t- \t_\ell(n))} \;d\mu_{y_\rho}^u dt,
\end{align*}
where we ignored the Jacobian of the change of variables for simplicity.

The functions $\t_\ell(n)$ are known as the \textit{temporal distance functions} in the literature.
Roughly, they are defined dynamically as follows: for every $(t,n) \in (0,1)\times N_1^+$, there exists a unique pair $(\t_\ell(n), u_\ell(t,n))\in (0,1)\times N_1^+$ such that the point $g_{t-\t_\ell(n)} u_\ell(t,n) \cdot y_\rho $ is on the same strong stable leaf of $g_t nx$.
Now, non-joint integrability of the strong stable and unstable foliations imply that the derivative of $\t_\ell(n)$ is uniformly bounded away from $0$.
Variants of this property are crucial in carrying out analysis of oscillatory integrals.

To proceed, we partition the domain of integration into small balls $\set{A_j}$ so we approximate $\t_\ell$ by its linearization.
Fix one such ball and apply a suitable amount of geodesic flow to scale this ball to be a ball of radius $1$ to get
\begin{align}\label{eq:sketch Dolgopyat}
    \sum_\ell \int_{N_1^+} e^{-ib \langle v_{\ell,j}, n\rangle } \;d\mu_{z_j}^u(n),
\end{align}
where, roughly speaking, $v_{\ell,j}$ is the derivative of $\t_\ell$ at the center of the ball $A_j$.

Up to this point, the argument is very similar to that appearing in Liverani's original work~\cite{Liverani} and its subsequent generalizations, e.g.~\cite{BaladiLiverani,BaladiDemersLiverani,GiuliettiLiveraniPollicott,GiuliettiLiveraniPolicott-Erratum}.
The crucial difference comes at the next step.
In these previous works, the measure $\mu^u_{z_i}$ was always absolutely continuous to Lebesgue, and proof proceeds by integration by parts. 

In our case, these measures have fractional Hausdorff dimension in general.
The novelty of our approach is to note that the sum in~\eqref{eq:sketch Dolgopyat}, when properly normalized, is an \textit{average} over Fourier coefficients of the measure $\mu^u_{z_i}$. This makes it amenable to our flattening results (Cor.~\ref{cor:flattening}) which establish verifiable criteria under which measures enjoy polynomial Fourier decay outside of an arbitrarily sparse set of frequencies.

\subsection{Proof of Theorem~\ref{thm:Dolgopyat}}  
The remainder of this section is dedicated to the proof of Theorem~\ref{thm:Dolgopyat}. Let $a\in (0,1]$ to be determined. We assume that $z=a+ib$ with $b>0$, the other case being identical.
For convenience, we summarize the notation used in this section in Table~\ref{tab:section 9}.    

\begin{table}
    \centering
    \begin{tabular}{c|c}
     Notation & Definition \\
    \hline\\
        $\d=\d_\G$ & critical exponent\\
        $\L=\L_\G$, $\Omega=\Omega_\G$ & limit set and non-wandering set\\
      $\Delta,\Delta_+$   & \eqref{eq:Delta}  \\
       $\b$  & Remark~\ref{remark:choice of V}\\
       $m$ & $\lceil \log b \rceil$ \\
       $T_0$ & time discretization \\
       $p_j$ & partition of time variable~\eqref{eq:time partition}\\ 
       $j$ & summand index in resolvent~\eqref{eq:j small}\\ 
       $\Ecal_j,\Ecal_b$ & sets of recurrent orbits~\eqref{eq:Ecal_j} and~\eqref{eq:tilde Ecal_b}\\
       $\a$ & proportion of time in cusp~\eqref{eq:Ecal_j} \\
       $K_j$ & fixed compact set \eqref{eq:K_j and iota_j and iota_b} \\
       $\iota_j,\iota_b$ & dimensions of flow boxes~\eqref{eq:K_j and iota_j and iota_b} \\
       $p_{j,w}$ & \eqref{eq:p j w} \\
       $w$ & discretization of $(0,2T_0)$ \eqref{eq:p j w}\\ 
       $g_j^w$ & $g_{w+jT_0}$ \eqref{eq:gjw} \\
       $\mathbb{M}$ & mollifier \eqref{eq:gjw}\\
       $F$ & $\mathbb{M}(f)$ \eqref{eq:gjw}\\ 
       $\Pcal, \Pcal_b$ & collections of partitions of unity~\eqref{eq:recurrent boxes} \\ 
       $\g$ & $1/2$\\ 
       $g^\g$ & amount of time we flow $g_{(w+jT_0)/2}$ \eqref{eq:ggamma}\\
       $x_j$ & $g^\g x $ \eqref{eq:x_j} \\
       $N_1^+(j)$ & neighborhood of $N_1^+$ \eqref{eq:N_1^+(j)} \\
       $F_\g$ & $\Lcal_{(w+jT_0)/2} F$ \eqref{eq:F_gamma}\\ 
       $y_\rho$ & center of flow box $B_\rho$ \eqref{eq:box notation} \\
       $T_\rho$ & transversal to strong unstable in $B_\rho$ \eqref{eq:box notation} \\
       $I_{\rho,j}$ & indexes unstable leaves landing in $B_\rho$ at time $(w+jT_0)/2$ \\
       $W_\ell$ & $\ell^{th}$ unstable piece in $B_\rho$\\
       $x_{\rho,\ell}$ & center of $W_\ell$ \eqref{eq:centers} \\
       $s_{\rho,\ell}$ & return time to compact for $x_{\rho,\ell} $ \eqref{eq:return time of x_rho,ell}\\
       $W_\rho$ & local unstable leaf of $y_\rho$ \eqref{eq:W_rho} \\ 
       $\t_\ell$ & the temporal distance function~\eqref{eq:def of tau_ell}\\ 
       $\phi_{\rho,\ell}$ & test function after change of variables~\eqref{eq:def of tau_ell}\\
       $J\Phi_\ell$ & Jacobian of stable holonomy~\eqref{eq:def of tau_ell} \\
       $\varkappa$ & \eqref{eq:stable derivative} \\ 
       $J_\rho$ & support of integration in $t$ \eqref{eq:pre-CauchySchwarz}\\ 
       $A_i$ & cusp adapted partition~\eqref{eq:remove divergent orbits} \\
        $t_i, r_i, \yrhoi$ & cusp-adapted partition parameters \eqref{eq:after moving to cpt}\\
       $w^i_{k,\ell}$ & frequencies \eqref{eq:frequency vectors def}\\
    
       $C_{\rho,j}/ S_{\rho,j}$ & close/separated pairs of unstable disks \eqref{eq:S_rho,j}\\
       $\k_0$ & Proposition~\ref{prop:close pairs} \\ 
       $\e,\l$ & Theorem~\ref{thm:counting bad frequencies} \\
       \hline 
    \end{tabular}
    \vspace{5pt}
    \caption{Summary of notation in the proof of Theorem~\ref{thm:Dolgopyat}.}
    \label{tab:section 9}
\end{table}

\subsection*{Time partition}     
Let $p:\R\to [0,1]$ be a smooth bump function supported in $(-1,1)$ such that
\begin{equation}\label{eq:p}
    \sum_{j\in\Z}p(t-j) =1, \qquad \forall t\in\R.
\end{equation}
Let $\gls{T0}>0$ be a parameter to be chosen large depending only on $\G$ and let
\begin{equation}\label{eq:m and b}
    m=\lceil\log b\rceil.
\end{equation}
Changing variables, we obtain
\begin{align}\label{eq:expand resolvent}
     &R(z)^{m}  
    = \int_0^\infty  \frac{t^{m-1} e^{-zt}}{(m-1)!} \Lcal_t \; dt
    \nonumber\\
    &= \int_0^\infty \frac{t^{m-1}e^{-zt}}{(m-1)!}
    p(t/T_0) \Lcal_t \;dt
    +\sum_{j= 0}^\infty  \frac{((j+2)T_0)^{m-1} e^{-zjT_0}}{(m-1)!} \int_\R p_j(t)
    e^{-zt}  \Lcal_{t+jT_0}\; dt,
\end{align}
where we define $p_j$ as follows:
\begin{equation}\label{eq:time partition}
    \gls{pj}(t):=
    \left(\frac{jT_0+t}{(j+2)T_0} \right)^{m-1}
    p\left(\frac{t-T_0}{T_0}\right).
\end{equation}
Note that $p_j$ is supported in the interval $(0,2T_0)$ for all $j\geq 0$.

\subsection*{Contribution of very small and very large times}

We will estimate the contribution of each term in the sum over $j$ in~\eqref{eq:expand resolvent} individually.
We will restrict our attention to values of $j$ of size $\asymp \log b$.
We begin by estimating the first term in~\eqref{eq:expand resolvent} trivially.
Since $p(t/T_0)$ is supported in $(-T_0,T_0)$ and $a\leq 1$, by taking $b$ large enough and using the triangle inequality for the seminorm $e^\star_{1,0}$ and Lemma~\ref{lem: equicts}, we obtain
\begin{align}\label{eq:first term}
    e^\star_{1,0}\left( \int_0^\infty \frac{t^{m-1}e^{-zt}}{(m-1)!}
    p(t/T_0) \Lcal_t  f\;dt\right)
    \ll e^\star_{1,0}(f) T_0^{m-1}/(m-1)!
    \ll \frac{ e^\star_{1,0}(f)}{(a+1)^{m}}.
\end{align}

Next, we let
\begin{align}\label{eq:eta1 and eta2}
    \eta_1 = 4/3, \qquad \eta_2 = 2/a.
\end{align}
Note that since $a\leq 1$, we have $\eta_1<\eta_2$.
We wish to find a trivial bound on the terms corresponding to $j\notin [\eta_2,\eta_2]m/T_0$.
First, we have the following bound on the sum of the terms $j\leq \eta_1m/T_0$.
\begin{align}\label{eq:initial segment of sum}
    \sum_{j:jT_0<\eta_1 m-2T_0}  \frac{((j+2)T_0)^{m-1} e^{-ajT_0}}{(m-1)!} \int_\R p_j(t)
    e^{-at}  e^\star_{1,0}
    \left(\Lcal_{t+jT_0}f \right)\;dt
    \ll e^\star_{1,0}(f) \int_0^{\eta_1m}\frac{t^{m-1} e^{-at}}{(m-1)!}
    dt.
\end{align}
Similarly, we have the following bound on the tail of the sum:
\begin{align}\label{eq:tail}
   \sum_{j: jT_0> \eta_2 m} \frac{((j+2)T_0)^{m-1} e^{-ajT_0}}{(m-1)!} \int_\R p_j(t)
    e^{-at} e^\star_{1,0}(\Lcal_{t+jT_0}f)\; dt
\ll  e^\star_{1,0}(f)\int_{\eta_2 m}^\infty \frac{t^{m-1}e^{-at}}{(m-1)!}\;dt 
\end{align}
The following lemma provides the desired bound on the integrals appearing in the above bounds.

\begin{lem}\label{lem:large j}
Suppose that 
$a\eta >1$. Then, there exists $\th\in (0,1)$ such that
\begin{align*}
    \int_{\eta m}^\infty \frac{t^{m-1}e^{-at}}{(m-1)!}\;dt
    \ll_{a,\eta} \left(\th/a\right)^m.
\end{align*}
On the other hand, if $a\eta < 1/e$, then there exists $\th\in (0,1)$ such that 
\begin{align*}
    \int_0^{\eta m}\frac{t^{m-1}e^{-at}}{(m-1)!}\;dt \ll  \left(\th/a\right)^m.
\end{align*} 
In both cases, we may take $\th= a\eta e^{1-a\eta}$.
\end{lem}

\begin{proof}

Integration by parts and induction on $m$ yield
\begin{align}\label{eq:integral tail formula}
  \int_{\eta m}^\infty \frac{t^{m-1}e^{-at}}{(m-1)!}\;dt
    =\frac{ e^{-a\eta m}}{a^{m}}
    \sum_{k=0}^{m-1}\frac{(a\eta m)^k}{k!}
    = \frac{ e^{-a\eta m} (a\eta m)^{m}}{a^{m}m!}
    \sum_{k=0}^{m-1}\frac{m\cdots (k+1)}{(a\eta m)^{m-k}}.
\end{align}
Note that the $k^{th}$ term of the latter sum is at most $(a\eta)^{-m+k}$. 
Moreover, from Stirling's formula, we have that $m!\gg m^{m+1/2} e^{-m}$.
Hence, when $a\eta>1$, we get
\begin{align*}
    \int_{\eta m}^\infty \frac{t^{m-1}e^{-at}}{(m-1)!}\;dt
    \ll \frac{e^{(1-a\eta) m} (a\eta)^m }{a^m}.
\end{align*}
Taking $\th=a\eta e^{1-a\eta}$ and noting that $xe^{1-x}$ is strictly less than $1$ for all $x\geq 0$ with $x\neq 1$, concludes the proof of the first assertion.
For the second assertion, assume that $a\eta e<1$. Then, combining~\eqref{eq:integral value} with~\eqref{eq:integral tail formula}, we get
\begin{align*}
    \int_0^{\eta m} \frac{t^{m-1}e^{-at}}{(m-1)!}\;dt
    =\frac{ e^{-a\eta m}}{a^{m}}
    \sum_{k=m}^{\infty}\frac{(a\eta m)^k}{k!}.
\end{align*}
Stirling's formula shows that the $k^{th}$-term of the above sum is $O((a\eta e)^k)$, for all $k\geq m$. Since $a\eta e<1$, we get that the integral is $\ll (\th/a)^m$ for $\th= a\eta e^{1-a\eta}$.
\end{proof}

In view of~\eqref{eq:first term},~\eqref{eq:initial segment of sum},~\eqref{eq:tail}, and Lemma~\ref{lem:large j}, in what follows, we restrict our attention to the terms where $j$ satisfies $-2T_0+4m/3\leq jT_0 \leq 2 m/a$.
We shall assume $b$ is large enough so that $4m/3-2T_0 \geq 2m/3$. In particular, we estimate the terms satisfying
\begin{equation}
    \label{eq:j small}
    2m/3\leq jT_0 \leq 2 m/a.
\end{equation}
Finally, in view of~\eqref{eq:integral value}, we have that
\begin{align}\label{eq:sum over j}
    \sum_{j: jT_0\geq 2m/3}\frac{((j+2)T_0)^{m-1} e^{-ajT_0}}{(m-1)!}
    \leq e^{2aT_0} \left(1+6/m\right)^m \int_0^\infty \frac{t^{m-1}e^{-at}}{(m-1)!} \;dt
    \ll  \frac{e^{2aT_0}}{a^m},
\end{align}
where we used the bound $\left(1+6/m\right)^m\ll 1$ for all large enough $m$.

\subsection*{Contribution of points in the cusp}
We estimate the contribution of each term in the sum over $j$ in~\eqref{eq:expand resolvent} individually. 
We begin by reducing to the case where the basepoint has bounded height.
Let $\gls{alpha}\geq 0$ be a small parameter to be chosen at the end of the argument and satisfies
\begin{align}\label{eq:a and alpha}
    \a \leq a/40.
\end{align}
Let $x\in N_1^-\Omega$ be arbitrary. 
Suppose that $V(x)>e^{\b\a jT_0}$.
Then, Lemma~\ref{lem: equicts} implies that
\begin{equation*}
     e^\star_{1,0}(\Lcal_{t+jT_0}f;x) \ll_\b e^{-\b\a jT_0}e^\star_{1,0}(f).
 \end{equation*}
 In light of~\eqref{eq:sum over j}, summing the above errors over $j$, we obtain an error term of the form
 \begin{align}\label{eq:cusp contribution}
     e^{2aT_0}e^\star_{1,0}(f)
     \frac{1}{\left(a+\b\a \right)^m}.
 \end{align}
Thus, we may assume for the remainder of the section that
\begin{equation}\label{eq:V(x)}
    V(x)\leq e^{\b\a jT_0/2}.
\end{equation}


\subsection*{Approximation with mollifiers}

To begin our estimates, fix a suitable test function $\phi$ for $e^\star_{1,0}$.
In particular, $\phi$ has $C^{0,1}(N^+)$ norm at most $1$.
The integrals we wish to estimate take the form
\begin{align*}
    \int_{N_1^+}\phi(n) \int_\R p_j(t) &e^{-zt}\Lcal_{t+ jT_0} (f)(g_snx) \; dt d\mu_x^u(n) 
    \nonumber\\
    &= \int_\R e^{-zt}
    \int_{N_1^+}p_j(t) \phi(n) f (g_{s+t+ jT_0} nx) \;d\mu_x^u(n)dt,
\end{align*}
for all $s\in [0,1]$. We again only provide the estimate in the case $s=0$ to simplify notation, the general case being essentially identical.

Recall that $p_j$ is supported in the interval $(0,2T_0)$. In particular, the extra $t$ in $\Lcal_{t+ jT_0}$ could be rather large, which will ruin certain trivial estimates later.
To remedy this, recall the partition of unity of $\R$ given in~\eqref{eq:p} and set
\begin{equation}\label{eq:p j w}
    \gls{pjw}(t) := p_j(t+w)p(t), \qquad \forall w\in\Z.
\end{equation} 
Using a change of variable, we obtain
\begin{align}\label{eq:smaller time steps}
    \int_\R e^{-zt}
    \int_{N_1^+}p_j(t) &\phi(n) f (g_{t+ jT_0} nx) \;d\mu_x^u(n)dt
    \nonumber\\
    &=
    \sum_{w\in\Z}e^{-zw}
    \int_\R e^{-zt}
    \int_{N_1^+}p_{j,w}(t) \phi(n) f (g_{t+w+ jT_0} nx) \;d\mu_x^u(n)dt.
\end{align}
Note the above sum is supported on $0\leq w\ll T_0$, and the support of each integral in $t$ is now $(-1,1)$.
For the remainder of the section, we fix some $\gls{w}\in\Z$ in that support.

Let $ \M := \M_{1/10}$,
where for $\e>0$, $\M_\e$ denotes the mollifier defined in Section~\ref{sec:mollifiers}.
To simplify notation, we set
\begin{equation}\label{eq:gjw}
    \gls{gjw} := g_{w+jT_0},\qquad 
    \gls{F}:=\M(f).
\end{equation}
Since $\phi\in C^{0,1}(N_1^+)$ with $\norm{\phi}_{C^{0,1}}\leq 1$,
it follows by Proposition~\ref{prop:mollifiers are close} that
\begin{align*}
    \bigg|\int_{N_1^+}\phi(n)
    \Lcal_{t} 
    (f-F) (g_j^w nx ) 
    \;d\mu_{ x}^u\bigg|
    \ll 
      e^{-(t+w+ jT_0)}  e^\star_{1,0}(f) V(x) \mu_{x}^u(N_1^+).
\end{align*}
Arguing as in~\eqref{eq:cusp contribution}, summing the above errors over $j$, we get an error term of the form
\begin{align}\label{eq:replace with mollifier}
    O_{T_0}\left(\frac{e^\star_{1,0}(f) V(x) \mu_{x}^u(N_1^+)}{\left(a+1\right)^{m}}\right).
\end{align}
Hence, we may replace $f$ with $F$ in~\eqref{eq:smaller time steps}.
We will frequently use the following observation.
Writing $F=F-f+f$ and using Proposition~\ref{prop:mollifiers are close}, we have that
\begin{equation}\label{eq:mollifier is at most f}
    e^\star_{1,0}(F) \ll e^\star_{1,0}(f).
\end{equation}



\subsection*{Partitions of unity and flow boxes}

We begin by finding convenient partitions of the space by \textit{flow boxes}, i.e., sets of the form $P^-_r N^+_s\cdot x$ for $r,s>0$ and $x\in G/\G$ and such that the map $P^-_r N^+_s \ni g \mapsto gx$ is injective.
To this end, we have to restrict our attention to the part of the space where the injectivity radius is bounded away from $0$.
Define
\begin{equation}\label{eq:K_j and iota_j and iota_b}
    \gls{Kj}:=\set{y\in X: V(y)\leq e^{(2\b \a j+3\b) T_0}}, 
    \qquad \gls{iotaj}:=\min\set{1/10,\inj(K_j)},
    \qquad \iota_b := b^{-2/3}.
\end{equation}
We note that Proposition~\ref{prop:height function properties} implies that
\begin{equation}\label{eq:bound iota_j}
    \iota_j^{-1} \ll e^{(4\a j+6)T_0},
\end{equation}
where we used the fact that $\chi_\K\leq 2$; cf.~\eqref{eq:multiple of simple root}.
\begin{remark}
    Since we are working in the regime where $\a$ is small and $j$ is bounded linearly in $\log b$, cf.~\eqref{eq:j small},~\eqref{eq:m and b}, and~\eqref{eq:a and alpha}, the bound~\eqref{eq:bound iota_j} implies that
    $\iota_b$ is much smaller than $\iota_j$ in general.
\end{remark}

The following lemma provides an efficient cover of $K_j\cap N^-_{1/2}\Omega$ by flow boxes which are very narrow in the unstable direction.
This will be useful in the proof of Lemma~\ref{lem:temp function formula} where we linearize the phase functions of the oscillatory integrals that arise over the course of the proof.
\begin{lem}\label{lem:count flow boxes}
    The collection of flow boxes $\set{P^-_{\iota_j} N^+_{\iota_b}\cdot x:x\in K_j\cap N^-_1\Omega}$ admits a finite subcover $\Bcal$ of $N^-_{1/2}\Omega$ with uniformly bounded multiplicity; i.e.~for all $x\in K_j\cap N^-_{1/2}\Omega$, $\sum_{B\in\Bcal} \mathbbm{1}_B(x)\ll 1$.
\end{lem}

\begin{proof}

    Let $\Qcal$ denote a cover of the unit neighborhood of $K_j$ by flow boxes of the form $ P^-_{\iota_j} N^+_{\iota_j}\cdot x$, where $\iota_j$ is as in~\eqref{eq:K_j and iota_j and iota_b}.
    With the help of the Vitali covering lemma, such cover can be chosen to have multiplicity $C_G\geq 1$, depending only on the dimension of $G$.
    We will build our collection of boxes $\Bcal$ by refining this cover as follows.
    
    Let $\Qcal^0$ denote the subcollection of boxes $Q\in\Qcal$ such that $Q$ intersects $  K_j\cap N_{1/2}^-\Omega$ non-trivially.
    For each $Q\in \Qcal^0$, fix some $x_Q\in Q\cap N_{1/2}^-\Omega$.
    Then, we can find a finite set of points $\set{u_i:i\in I_Q}\subset N^+_{2\iota_j}$ such that the points $x_i:=u_i x_Q$ belong to $N_{1}^-\Omega$ and so that the balls $N^+_{\iota_b}\cdot x_i$ provide a cover of $N^-_{1/2}\Omega\cap N^+_{\iota_j} \cdot x_Q $
    with uniformly bounded multiplicity (i.e. with multiplicity that is independent of $b$ and $j$).
    This is again possible thanks to the Vitali covering lemma.
   Now, define 
    \begin{align*}
        \Bcal :=\set{ P^-_{\iota_j} N^+_{\iota_b} \cdot u_i x_Q: i\in I_Q, Q \in \Qcal^0 }.
    \end{align*}
    To bound the multiplicity of $\Bcal$, let $x\in K_j\cap N_{1/2}^-\Omega$ be arbitrary, and note that
    \begin{align*}
        \sum_{B\in \Bcal} \mathbbm{1}_B(x) =
        \sum_{Q\in \Qcal^0} \sum_{i\in I_Q} \mathbbm{1}_{P^-_{\iota_j} N^+_{\iota_b} \cdot u_i x_Q} (x)
        \ll \sum_{Q\in \Qcal^0} \mathbbm{1}_{\cup_{i\in I_Q} P^-_{\iota_j} N^+_{\iota_b} \cdot u_i x_Q} (x).
    \end{align*}
    Moreover, if $Q=P^-_{\iota_j} N^+_{\iota_j} \cdot  x'_Q$ for some $x'_Q$, then the union $\cup_{i\in I_Q} P^-_{\iota_j} N^+_{\iota_b} \cdot u_i x_Q$ is contained inside $Q^+:= P^-_{\iota_j} N^+_{2\iota_j} \cdot  x'_Q$. Finally, bounded multiplicity of $\Qcal^0$ implies that $\sum_{Q\in\Qcal^0} \mathbbm{1}_{Q^+}(x) \ll 1$.
    This concludes the proof.
    \qedhere
\end{proof}

Let $\Bcal$ be the finite cover provided by Lemma~\ref{lem:count flow boxes} and let $\Pcal$ denote a partition of unity 
subordinate to it. 
For each $\rho\in \Pcal$, we denote by $B_\rho $ the element of $\Bcal$ containing the support of $\rho$.
In particular, such partition of unity can be chosen so that for all $\rho\in \Pcal$, we have
\begin{equation}\label{eq:norm of rho at most b}
    \norm{\rho}_{C^1} \ll b^{2/3}.
\end{equation}
Over the course of the argument, we need to apply the geodesic flow to enlarge the width of the boxes $B_\rho$ in the $N^+$ direction to be $\asymp 1$, for e.g.~to apply Theorems~\ref{thm:Margulis function} and~\ref{thm:exp recurrence}. 
It will be important to ensure that these boxes meet the compact set $K_j$ after flowing.
To this end, we define the following subset of $\Pcal$ consisting of boxes which return to $K_j$ at time $b^{2/3}$:
\begin{align}\label{eq:recurrent boxes}
    \Pcal_b := \set{\rho\in \Pcal: g_{-\log \iota_b}B_\rho \cap K_j \neq \emptyset }.
\end{align}
Note that for each $B\in \Bcal$, $g_{-\log \iota_b}B$ has diameter $O(1)$.

\subsection*{Transversals}
We fix a system of transversals $\{T_\rho\}$ to the strong unstable foliation inside the boxes $B_\rho$.
Since $B_\rho$ meets $N_{1/2}^-\Omega$ for all $\rho\in \Pcal$, we take $\gls{yrho}$ in the intersection $B_\rho\cap N_{1/2}^-\Omega$. 
In this notation, we can find neighborhoods of identity $P^-_\rho \subset P^-= MAN^-$ and $N^+_\rho \subset N^+$ such that
\begin{equation}\label{eq:box notation}
    B_\rho = P^-_\rho N^+_\rho  \cdot y_\rho, \qquad \gls{Trho} = P^-_\rho \cdot y_\rho.
\end{equation}
We also let $M_\rho, A_\rho$, and $N^-_\rho$ be neighborhoods of identity in $M, A$, and $N^-$ respectively so that $P^-_\rho=M_\rho A_\rho N^-_\rho$.

\subsection*{Localizing away from the cusp}

Our next step is to restrict the support of the integral away from the cusp.
Define the following smoothed cusp indicator function $\zeta_j:X\to [0,1]$:
\begin{equation*}
    \zeta_j(y) := 1- \sum_{\rho\in\Pcal} \rho(y).
\end{equation*} 
Let 
\begin{equation}\label{eq:ggamma}
    \gls{gamma} = 1/2, \qquad \gls{ggamma} := g_{\g(w+jT_0)}.
\end{equation}
It will be convenient to take $T_0$ large enough so that
\begin{align}\label{eq:gamma time at least 1}
    (1-\g)(w+jT_0)=\g(w+jT_0)\geq 4.
\end{align}
First, we note that Proposition~\ref{prop:upper estimate on mollifiers} implies 
$|\Lcal_t  F(g_j^wnx)| \ll e^\star_{1,0}(f) \Lcal_{t}V(g_j^wnx)$. 
Hence, since $|\phi|$ is bounded by $1$ and $\zeta_j$ is non-negative, we obtain
\begin{align*}
    \Bigg| \int_{N_1^+}\phi(n)\zeta_j(g^\g nx) \Lcal_t F(g_j^wnx)\;d\mu_x^u \Bigg|
    \ll e^\star_{1,0}(f) \left|\int_{N_1^+}\zeta_j(g^\g nx) \Lcal_{t} V(g_j^w nx) \;d\mu_x^u \right|.
\end{align*}
To proceed, we show that the support of the above integral is in the cusp to apply Theorem~\ref{thm:Margulis function}.
\begin{lem}\label{lem:supp of zeta_j}
    For every $n\in \supp(\mu_x^u)\cap N_1^+$, we have $\zeta_j(g^\g nx)>0 \Longrightarrow V(g^\g nx)> e^{2\b \a jT_0}$.
\end{lem}
\begin{proof}
    Let $n\in \supp(\mu_x^u)\cap N_1^+$.
    Since $x\in N_1^-\Omega$, it follows by~\eqref{eq:switching order of N- and N+} and Remark~\ref{rem:commutation of stable and unstable} that $nx\in N_2^-\Omega$.
    Then, $g^\g nx \in N^-_{r}\Omega$, for $r =2e^{-\g(w+jT_0)}$.
    By~\eqref{eq:gamma time at least 1}, we get that $g^\g nx \in N^-_{1/2}\Omega$.
    On the other hand, if $\zeta_j(g^\g nx)>0$, then $g^\g nx \notin K_j \cap N^-_{1/2}\Omega$.
    The lemma now follows by definition of $K_j$ in~\eqref{eq:K_j and iota_j and iota_b}.
\end{proof}
Lemma~\ref{lem:supp of zeta_j} and the Cauchy-Schwarz inequality thus yield
\begin{align*}
    \left|\int_{N_1^+}\zeta_j(g^\g nx) \Lcal_{t} V(g_j^w nx) \;d\mu_x^u \right|^2
    \leq  \mu_x^u\left(n\in N_1^+:  V(g^\g nx)> e^{2\b\a jT_0}\right)  
    \times \int_{N_1^+} \Lcal_{t} V^2(g_j^wnx)\;d\mu_x^u.
\end{align*}
By Theorem~\ref{thm:Margulis function} and Chebyshev's inequality, we have that the set on the right side has measure $O(e^{-2\b\a jT_0} V(x) \mu_x^u(N_1^+))$.
Moreover, recall that we are assuming that $V^2$ satisfies the Margulis inequality in Theorem~\ref{thm:Margulis function}; cf.~Remark~\ref{remark:choice of V}.
Hence, applying Theorem~\ref{thm:Margulis function} once more shows that the integral on the right side is at most $O(V^2(x)\mu_x^u(N_1^+))$.
These bounds together yield
\begin{align*}
    \Bigg| \int_{N_1^+}\phi(n)\zeta_j(g^\g nx) \Lcal_t F(g_j^wnx)\;d\mu_x^u \Bigg|
  & \ll e^\star_{1,0}(f)\mu_x^u(N_1^+) V^{3/2}(x) e^{-\b\a jT_0}.
\end{align*}
Using the bound on $V(x)$ in~\eqref{eq:V(x)}, we get
\begin{align}\label{eq:localize away from cusp}
    & \int_{N_1^+}\phi(n) \Lcal_t  F(g_j^wnx)\;d\mu_x^u(n)
  \nonumber\\
     &=  \sum_{\rho\in \Pcal} \int_{N_1^+}\phi(n) \rho(g^\g nx) \Lcal_t F(g_j^wnx)\;d\mu_x^u
     + O\left(e^\star_{1,0}(f)\mu_x^u(N_1^+) V(x) e^{-3\b\a jT_0/4} \right) .
\end{align}
Using~\eqref{eq:sum over j} to sum the above errors over $j$, we obtain an error term of the form
\begin{align}\label{eq:zeta_j contribution}
   O\left(
    \frac{ e^\star_{1,0}(f) \mu_x^u(N_1^+) V(x)}{ (a+3\b\a/4)^m}
    \right) .
\end{align}

\subsection*{Saturation and localization to flow boxes}

Next, we partition the integral over $N_1^+$ into pieces according to the flow box they land in under flowing by $g^\g$.
To simplify notation, we write
\begin{equation}\label{eq:x_j}
      \gls{xj}:= g^\g x.
\end{equation}
We denote by $\gls{N1j}$ a neighborhood of $N_1^+$ defined by the property that the intersection 
$$ B_\rho\cap (\mrm{Ad}(g^\g)(N_1^+(j))\cdot x_j)$$
consists entirely of full local strong unstable leaves in $B_\rho$.
We note that since $\mrm{Ad}(g^\g)$ expands $N^+$ and $B_\rho$ has radius $<1$, $N_1^+(j)$ is contained inside $N_2^+$.
Since $\phi$ is supported inside $N_1^+$, we have
\begin{equation}\label{eq:N_1^+(j)}
    \chi_{N_1^+}(n)\phi(n)  = \chi_{N_1^+(j)}(n)\phi(n), \qquad \forall n \in N^+.
\end{equation}
For simplicity, we set
\begin{equation*}
    \vp_j(n) := \phi(\Ad(g^\g)^{-1} n ), \qquad \Acal_j :=  \mrm{Ad}(g^\g)(N_1^+(j)).
\end{equation*}
For $\rho\in \Pcal$, we let $\tilde{\Wcal}_{\rho,j}$ denote the collection of connected components of the set
\begin{equation*}
    \set{n\in\Acal_j: nx_j\in B_\rho}.
\end{equation*}
To simplify notation, let
\begin{equation}\label{eq:F_gamma}
    \gls{Fgamma} := \Lcal_{(1-\g)(w+jT_0)} (F).
\end{equation}
In view of~\eqref{eq:N_1^+(j)}, changing variables using~\eqref{eq:g_t equivariance} yields
\begin{align}\label{eq:localize space}
  \sum_{\rho\in \Pcal} \int_{N_1^+}\phi(n) \rho(g^\g nx) \Lcal_t F(g_j^wnx)\;d\mu_x^u
 =e^{-\d\g(w+jT_0) } 
    \sum_{\rho\in \Pcal, W\in \tilde{\Wcal}_{\rho,j}}
    \int_{n\in W}  \vp_j(n) 
        \rho(nx_j)  F_\g(g_tn x_j)\; d\mu_{x_j}^u.
\end{align}

\subsection*{Contribution of non-recurrent orbits}
In this subsection, we wish to restrict our attention to those $n\in N_1^+$ for which the orbit $(g_t nx)$ spends most of its time away from the cusp.

To this end, let $\e_1 = \b\a/2(1+2\a)$ and apply Theorem~\ref{thm:exp recurrence} with $\e=\e_1$ to find $H=H(\e_1,T_0) \geq 1$ such that the conclusion of the theorem holds. Let $\chi_H$ denote the indicator function of the set $\set{x:V(x)>H}$ and define
\begin{align*}
    \tilde{\Ecal}_j = \set{n\in N_1^+: \int_0^{(1+2\a)\g(w+ jT_0)} \chi_H(g_tnx)\;dt  >2\a\g(w+ jT_0) }.
\end{align*}
We wish to define a saturated version of the set $\tilde{\Ecal}_j$, which we denote by $\Ecal_j$.
The goal of doing so is to discard all the disks $W\in\tilde{\Wcal}_{\rho,j}$ with the property that it contains the image of a point in $\tilde{\Ecal}_j$ under $g^\g$.
In particular, $\Ecal_j$ is contained inside the $O(e^{-\g(w+ jT_0)}\iota_b)$-neighborhood of $\tilde{\Ecal}_j$.
This is made precise in the following lemma.

\begin{lem}\label{lem:saturate tilde Ecal_j}
    Let $\rho$ and $W\in \tilde{\Wcal}_{\rho,j}$ be arbitrary. 
    Then, for all $n_1,n_2\in N_1^+$ with $g^\g n_ix$, $i=1,2$, we have that $V(g_t n_1x) \asymp V(g_t n_2 x)$, uniformly over all $\rho, W$ and $0\leq t\leq (1+2\a)\g(w+jT_0)$.
\end{lem}
\begin{proof}
    Since $B_\rho$ has width $\iota_b$ along $N^+$, we have $d_{N^+}(n_1,n_2) \ll e^{-\g(w+jT_0)} \iota_b$.
    Hence, letting $n^t = g_t n_1 n_2^{-1} g_{-t}$, we get $d_{N^+}(n^t,\id)\ll e^{2\a \g(w+jT_0)}\iota_b$, for all $t \leq (1+2\a)\g(w+jT_0)$.
    By~\eqref{eq:j small} and~\eqref{eq:a and alpha},  $d_{N^+}(n^t,\id)\ll_{T_0} b^{-1/2}\ll 1$. 
    The lemma follows from Prop.~\ref{prop:height function properties} since $g_t n_1x = n^t g_tn_2x$. 
\end{proof}

 Given $n\in N_1^+$, let $\rho(n)$ be a flow box index satisfying $g^\g n x\in B_{\rho(n)}$. We also let $W(n)\in \tilde{\Wcal}_{\rho(n),j} $ be such that $g^\g nx \in W(n)\subset B_{\rho(n)}$.
With this notation, we define $\Ecal_j$ as follows:
\begin{align}\label{eq:Ecal_j}
    \Ecal_j := \set{n\in N_1^+: \text{ there is } n'\in\tilde{\Ecal}_j \text{ such that } W(n) = W(n')}.
\end{align}
By Lemma~\ref{lem:saturate tilde Ecal_j}, there is a uniform constant $C\geq 1$ such that
\begin{align*}
    \Ecal_j \subseteq \set{n\in N_1^+: \int_0^{(1+2\a)\g(w+ jT_0)} \chi_{H/C}(g_tnx)\;dt  >2\a\g(w+ jT_0)  }.
\end{align*}
Hence, Theorem~\ref{thm:exp recurrence} shows that $\muxu(\Ecal_j)$ is $ O(e^{-(\b\th -\e_1)(1+2\a)\g(w+ jT_0)} V(x) \mu_x^u(N_1^+))$, where $\th = 2\a/(1+2\a)$.
Here, we used Prop.~\ref{prop:height function properties} to deduce this continuous time version of Theorem~\ref{thm:exp recurrence} from its discrete time formulation.
Moreover, recalling that $\g=1/2$, we obtain by~\eqref{eq:V(x)} that $V(x)\leq e^{\b\a\g jT_0}$.
Since $\e_1=\b\a/2(1+2\a)$, these bounds thus yield
\begin{align}\label{eq:measure of Ecal_j in Dolgopyat}
    \mu_x^u(\Ecal_j) \ll e^{-\b\a\g(w+ jT_0)/2} \mu_x^u(N_1^+).
\end{align}

Next, we let $\Wcal_{\rho,j}$ denote the connected components that avoid $\Ecal_j$. More precisely, let
\begin{align}\label{eq:Wcal_rho,j}
    \Wcal_{\rho,j}:= \set{W\in \tilde{\Wcal}_{\rho,j}: W \neq W(n) \text{ for any } n\in \Ecal_j }.
\end{align}
We now restrict the sum in~\eqref{eq:localize space} to the subsets $\Wcal_{\rho,j}$.
Reversing the change of variables in~\eqref{eq:localize space} and using the fact that our test functions have $C^0$-norm at most $1$, we get
\begin{align}\label{eq:split into Ecal_j and complement}
    \eqref{eq:localize space}
    = e^{-\d\g(w+jT_0) } 
    \sum_{\rho\in \Pcal, W\in \Wcal_{\rho,j}}
    \int_{n\in W}  \vp_j(n) 
    \rho(nx_j)  F_\g(g_tn x_j)\; d\mu_{x_j}^u 
     + O\left(\int_{\Ecal_j}  |\Lcal_{t}F(g_j^wnx)| \;d\mu^u_x \right).
\end{align}

To estimate the integral on the right side, we argue as before using Proposition~\ref{prop:upper estimate on mollifiers}, to get that $ |\Lcal_t  F(g_j^wnx)|$ is at most $O(e^\star_{1,0}(f) \Lcal_{t}V(g_j^wnx))$.
Then, Cauchy-Schwarz gives
\begin{align*}
    \left|\int_{\Ecal_j} \Lcal_tV(g_j^wnx)\;d\mu_x^u\right|^2
    \leq \mu_x^u(\Ecal_j)\times \int_{N_1^+}\Lcal_tV^2(g_j^wnx)\;d\mu_x^u.
\end{align*}
Applying Theorem~\ref{thm:Margulis function} on integrability of $V^2$ and~\eqref{eq:measure of Ecal_j in Dolgopyat}, we obtain
\begin{align*}
    \int_{\Ecal_j}  |\Lcal_{t}F(g_j^wnx)| \;d\mu^u_x 
    \ll e^{-\b\a\g(w+ jT_0)/4} e^\star_{1,0}(f)V(x)\mu_x^u(N_1^+).
\end{align*}

Our next step is to restrict the sum in~\eqref{eq:split into Ecal_j and complement} to the recurrent boxes $\Pcal_b$ defined in~\eqref{eq:recurrent boxes} using a similar argument.
Let
\begin{align}\label{eq:tilde Ecal_b}
    \tilde{\Ecal}_b = \set{n\in N_1^+: g^\g nx \notin \bigcup_{\rho \in \Pcal_b} B_\rho} ,
\end{align}
and define its saturation $\Ecal_b$ analogously to~\eqref{eq:Ecal_j}.
Then, by definition of $\Pcal_b$ and a similar argument to Lemma~\ref{lem:saturate tilde Ecal_j}, we can find a uniform constant $C\geq 1$ so that $\Ecal_b$ is contained in the set of $n\in N_1^+$ so that $V(g_{-\log \iota_b}g^\g nx) > e^{2\b\a jT_0}/C $. 
Theorem~\ref{thm:Margulis function} then gives that $\Ecal_b$ has measure $O( e^{-2\b\a jT_0} V(x)\mu_x^u(N_1^+) )$.
Hence, splitting the sum in~\eqref{eq:split into Ecal_j and complement} into sum over $\Pcal_b$ and $\Pcal\setminus \Pcal_b$, and recalling that $w\leq 2T_0$, we get
\begin{align}\label{eq:split into Ecal_b and complement}
    \eqref{eq:localize space} =
    e^{-\d\g(w+jT_0) } 
    \sum_{\rho\in \Pcal_b, W\in \Wcal_{\rho,j}}
    \int_{ W}  \vp_j(n) 
        \rho(nx_j)  F_\g(g_tn x_j)\; d\mu_{x_j}^u 
    + O(e^{-\b\a\g(w+ jT_0)/4} V(x) \mu_x^u(N_1^+)).
\end{align}
As in~\eqref{eq:sum over j}, summing over $j$, we obtain an error term of the form
 \begin{align}\label{eq:divergent contribution}
     O\left(\frac{e^\star_{1,0}(f)}{(a+\b\a/4)^m}\right).
 \end{align} 
The remainder of the section, is dedicated to estimating the sum on the right side of~\eqref{eq:split into Ecal_b and complement}.

\subsection*{Centering the integrals}
It will be convenient to center all the integrals in~\eqref{eq:localize space} so that their basepoints belong to the transversals $T_\rho$ of the respective flow box $B_\rho$; cf.~\eqref{eq:box notation}.

Let $\gls{Irhoj}$ denote an index set for $\Wcal_{\rho,j}$.
For $W\in \Wcal_{\rho,j}$ with index $\ell\in I_{\rho,j}$, let $n_{\rho,\ell}\in W$, $m_{\rho,\ell}\in M_\rho$, $n_{\rho,\ell}^-\in N^-_\rho$, and $t_{\rho,\ell}\in(-\iota_j,\iota_j)$ be such that
\begin{equation}\label{eq:centers}
    \gls{xrhoell} := m_{\rho,\ell}g_{-t_{\rho,\ell}}n_{\rho,\ell}\cdot x_j = n_{\rho,\ell}^- \cdot y_\rho \in T_\rho.
\end{equation}
Arguing as in the proof of Lemma~\ref{lem:supp of zeta_j}, since $x$ belongs to $N_1^-\Omega$, we have that
\begin{equation}\label{eq:x_rho,ell in omega}
    x_{\rho,\ell}\in N_1^-\Omega.
\end{equation}
Moreover, if we let $u_\ell= \Ad((g^\g)^{-1})(n_{\rho,\ell}) \in N^+_1(j)$, then
in light of the fact that the components $W\in \Wcal_{\rho,j}$ correspond to recurrent orbits, cf.~\eqref{eq:Wcal_rho,j} for a precise definition, we may and will assume that there is $\gls{sell}>0$ such that 
\begin{align}\label{eq:return time of x_rho,ell}
    \g(w+jT_0)\leq  s_{\rho,\ell} 
    \leq (1+2\a)\g(w+jT_0) , \qquad
    V(g_{s_{\rho,\ell}}u_\ell x) \ll_{T_0} 1.
\end{align}

\subsection*{Regularity of test functions}
For each such $\ell$ and $W$, let $\gls{Well}=\Ad(m_{\rho,\ell}g_{t_{\rho,\ell}})(Wn_{\rho,\ell}^{-1})$ and
\begin{equation}\label{eq:phi tilde}
    \gls{phitilde}(t,n) := p_{j,w}(t-t_{\rho,\ell})\cdot
         e^{z t_{\rho,\ell}}\cdot \phi( \Ad(m_{\rho,\ell} g^\g g_{ -t_{\rho,\ell}})^{-1}(n n_{\rho,\ell}))
         \cdot \rho(g_{t_{\rho,\ell}} nx_{\rho,\ell}).
\end{equation}  
Note that $\widetilde{\phi}_{\rho,\ell}$ has bounded support in the $t$ direction and~\eqref{eq:norm of rho at most b} implies
\begin{equation} \label{eq:phi tilde is bounded}
    \norm{\widetilde{\phi}_{\rho,\ell}}_{C^0(\R\times N^+)}\ll 1,
    \qquad \norm{\widetilde{\phi}_{\rho,\ell}(t,\cdot)}_{C^{0,1}( N^+)} \ll \iota_b^{-1},
\end{equation}
for all $t\in\R$.
Moreover, recalling~\eqref{eq:time partition}, we see that
\begin{equation}\label{eq:phi tilde Lipschitz}
    \norm{\widetilde{\phi}_{\rho,\ell}}_{C^{0,1}(\R\times N^+)} \ll \iota_b^{-1} m.
\end{equation}
Integrating the main term in~\eqref{eq:split into Ecal_b and complement} in $t$, and changing variables using~\eqref{eq:g_t equivariance} and~\eqref{eq:N equivariance}, we get
\begin{align}\label{eq:center integrals on transversal}
         e^{-\d\g(w+jT_0) }  \int_\R
         &e^{-zt}
         p_{j,w}(t) \sum_{\rho\in \Pcal_b, W\in \Wcal_{\rho,j}}
          \int_{n\in W}  
          \vp_j( n  )
         \rho(nx_j) F_\g(g_tn x_j) \;d\mu_{x_j}^u(n) dt
         \nonumber\\
         &=e^{-\d\g(w+jT_0) }  
         \sum_{\rho\in \Pcal_b} 
         \sum_{\ell\in I_{\rho,j}} 
         \int_\R  e^{-zt}
           \int_{n\in W_\ell}  
          \widetilde{\phi}_{\rho,\ell}(t,n)
           F_\g(g_{t+t_{\rho,\ell}} n x_{\rho,\ell}) \;d\mu_{x_{\rho,\ell}}^u(n)dt,
\end{align}  
where we also used $M$-invariance of $F_\g$; cf.~Remark~\ref{remark:M invariance}.

\subsection*{Mass estimates}
   
We record here certain counting estimates which will allow us to sum error terms in later estimates over $\Pcal$.
Note that by definition of $N_1^+(j)$, we have $\bigcup_{\rho\in\Pcal,W\in \Wcal_{\rho,j}}W\subseteq \Acal_j$. 
Thus, using the log-Lipschitz and contraction properties of $V$, it follows that
\begin{align}\label{eq:total mass}
    \sum_{\rho\in\Pcal,\ell\in I_{\rho,j}}
    \mu_{ x_{\rho,\ell}}^u(W_\ell)V( x_{\rho,\ell})
    & \ll \int_{\Acal_j}V(nx_j)\;d \mu_{x_j}^u(n)
    \nonumber\\
    &=
    e^{\d\g(w+jT_0)} \int_{N_1^+(j)}V(g_j^wnx)\;d \mu_{ x}^u(n)
    \ll  e^{\d\g(w+jT_0)} \mu_{ x}^u(N_1^+)V(x), 
\end{align}
where we used that $|t_{\rho,\ell}|<1$ and the last inequality follows by Proposition~\ref{prop:doubling} since $N_1^+(j)\subseteq N_2^+$.
We also used the uniformly bounded multiplicity of the partition of unity $\Pcal$.

We also need the following weighted number of flow boxes parametrized by $\Pcal$.
\begin{lem}\label{lem:weighted number of boxes}
 Recall that $\iota_b=b^{-2/3}$. Then, we have
 \begin{align*}
     \sum_{ \rho\in \Pcal } \mu^u_{y_\rho}(N^+_{\iota_b}) \ll e^{O_\b(\a jT_0)}.
 \end{align*} 
\end{lem}
\begin{proof}
Recall the Bowen-Margulis-Sullivan measure defined below~\eqref{eq:BMS}, and its conditional measures $\mu^s_\bullet$ along orbits of $N^-$ defined analogously to~\eqref{eq:unstable conditionals}.
Recall further that each $B_\rho$ is of the form $P^-_\rho N^+_\rho\cdot y_\rho$, where $P^-\rho$ and $N^+_\rho$ are identity neighborhoods of radius $\asymp \iota_j$ and $\asymp \iota_b$ respectively.
Bounded multiplicity of $\Pcal$ implies that $\sum_\rho \bms(B_\rho)\ll 1$.
Hence, the local product structure of $\bms$ implies that 
\begin{align*}
\bms(B_\rho)\asymp \iota_j^{\dim M +1} \mu^u_{y_\rho}(N^+_{\iota_b}) \mu^s_{y_\rho}(N^-_{\iota_j}),    
\end{align*}
where $M$ is the centralizer of the geodesic flow inside the maximal compact group $K$, and $\mu^s_\bullet$ are the conditional measures along $N^-$-orbits defined similarly to~\eqref{eq:unstable conditionals}.
This estimate implicitly uses the uniform doubling property from Prop.~\ref{prop:doubling} to assert that the ratio of the measures of all local strong (un)stable disks inside $B_\rho$ is uniformly $O(1)$.
Finally, by definition of $\mu^s_{y_\rho}$, we have that $\mu^s_{y_\rho}(N^-_{\iota_j}) \gg e^{-\d \dist(o,y_\rho)}$. 
By Lemma~\ref{lem:ht vs dist}, we have that $e^{\dist(y_\rho,o)} \ll V(y_\rho)^{O_\b(1)} \ll e^{O_\b(\a jT_0)}$, since $y_\rho $ belongs to the unit neighborhood of the set $K_j$ defined in~\eqref{eq:K_j and iota_j and iota_b}. The lemma follows by combining the above estimates with~\eqref{eq:bound iota_j}.
\end{proof}

\begin{remark}
The proof of Lemma~\ref{lem:weighted number of boxes} shows that the sum in question is $O_\G(1)$ when $\G$ is convex cocompact. The point of the above lemma is that this sum has, in general, much fewer terms than the sum in~\eqref{eq:total mass}. 
\end{remark}

\subsection*{Stable holonomy} Fix some $\rho\in \Pcal_b$.
Recall the points $y_{\rho}\in T_\rho$ and $n^-_{\rho,\ell}\in N^-_\rho$ satisfying~\eqref{eq:centers}.
The product map $M\times N^-\times A \times N^+\to G$ is a diffeomorphism on a ball of radius $1$ around identity; cf.~Section~\ref{sec:holonomy}. 
Hence, given $\ell\in I_{\rho,j}$, we can define maps $\tilde{u}_\ell$, $\tilde{\t}_\ell$, $m_\ell$ and $\tilde{u}^-_\ell$ from $ W_\ell$ to $N^+$, $\R$, $M$ and $N^-$ respectively by the following formula
\begin{equation}\label{eq:stable hol commutation}
    g_{t+t_{\rho,\ell}} n n_{\rho,\ell}^- =
    g_{t+t_{\rho,\ell}} m_\ell(n)\tilde{u}^-_\ell(n) g_{\tilde{\t}_\ell(n)} \tilde{u}_\ell(n)
    =m_\ell(n)\tilde{u}^-_\ell(t,n) g_{t+t_{\rho,\ell}+\tilde{\t}_\ell(n)} \tilde{u}_\ell(n),
\end{equation}
where we set $\tilde{u}^-_\ell(t,n)=\Ad(g_{t+t_{\rho,\ell}})(\tilde{u}^-_\ell(n))$. We define the following change of variable map:
\begin{align}\label{eq:stable hol map}
    \Phi_\ell : \R\times W_\ell\to \R\times N^+,
\qquad \Phi_\ell(t,n)=(t+\tilde{\t}_\ell(n),\tilde{u}_\ell(n)).
\end{align}
We suppress the dependence on $\rho$ and $j$ to ease notation.
Then, $\Phi_\ell$ induces a map between the weak unstable manifolds of $x_{\rho,\ell}$ and $y_\rho$, also denoted $\Phi_\ell$, and defined by
\begin{equation*}
    \Phi_\ell( g_t nx_{\rho,\ell}) = g_{t+\tilde{\t}_\ell(n)} \tilde{u}_\ell(n)y_\rho.
\end{equation*}
In particular, this induced map coincides with the local strong stable holonomy map inside $B_\rho$.

Note that we can find a neighborhood $\gls{Wrho}\subset N^+$ of identity of radius $\asymp \iota_b$ such that
\begin{equation}\label{eq:W_rho}
    \Phi_\ell(\R\times W_\ell) \subseteq \R\times W_\rho,
\end{equation}
for all $\ell\in I_{\rho,j}$.
Moreover, we may assume that $b$ is large enough (and hence $\iota_b$ is small enough), depending only on $G$, so that all the maps $\Phi_\ell$ in~\eqref{eq:stable hol map} are invertible on $\R\times W_{\rho}$.
Hence, we can define the following:
\begin{align}\label{eq:def of tau_ell}
    \gls{tauell}(n) &= \tilde{\t}_\ell( \tilde{u}_\ell^{-1}(n))+t_{\rho,\ell}\in \R,
    \qquad u^-_\ell(t,n) = \tilde{u}_\ell^-(t-\t_\ell(n),\tilde{u}_\ell^{-1}(n)) \in N^-,
    \nonumber\\
     \gls{phirhoell}(t,n) &= e^{-a(t-\t_\ell(n))}\times J\Phi_\ell(n)\times  
    \widetilde{\phi}_{\rho,\ell}(t-\t_\ell(n), \tilde{u}_\ell^{-1}(n)), 
\end{align}
and $\gls{Jphiell}$ denotes the Jacobian of the change of variable $\Phi_\ell$; cf.~\eqref{eq:stable equivariance}.

Changing variables and using $M$-invariance of $F_\g$, we obtain
\begin{align}\label{eq:stable hol}
    \eqref{eq:center integrals on transversal}
    = e^{-\d\g(w+jT_0) }  
         \sum_{\rho\in \Pcal_b} 
        \sum_{\ell\in I_{\rho,j}}
         \int_\R  \int_{W_\rho}
         e^{-ib(t-\t_\ell(n))}
         \phi_{\rho,\ell}(t,n)
    F_\g(u^-_\ell(t,n)g_{t} ny_{\rho})    \;d\mu_{y_{\rho}}^u(n)dt.
\end{align}

\subsection*{Stable derivatives}
Our next step is to remove $F_\g$ from the sum over $\ell$ in~\eqref{eq:stable hol}.
Due to non-joint integrability of the stable and unstable foliations, our estimate involves a derivative of $f$ in the flow direction.
In particular, in view of the way we obtain contraction in the norm of flow derivatives in Lemma~\ref{lem: flow by parts}, this step is the most ``expensive" estimate in our argument.
In essence, all the prior setup was aimed at optimizing the gain in this step.

Recall the definition of $F_\g$ in~\eqref{eq:F_gamma}.
Since $y_\rho$ belongs to $N_{1/2}^-\Omega$ and $u^-_\ell(t,n)$ belongs to a neighborhood of identity in $N^-$ of radius $O( \iota_j)$, uniformly over $(t,n)$ in the support of our integrals, Proposition~\ref{prop:stable derivatives of mollifiers} yields
\begin{align}\label{eq:mollifier derivative}
    |F_\g(u^-_\ell(t,n)g_{t} ny_{\rho}) - F_\g(g_{t} ny_{\rho})|
    \ll  e^{-(1-\g)(w+jT_0)} \norm{f}^\star_1
     V(y_{\rho}),
\end{align}
where we implicitly used the fact that $W_\rho \subset N_1^+$ and $|t|\leq 1$ so that $V(g_tny_\rho)\ll V(y_\rho)$. 
Indeed, the additional gain is due to the fact that $g_s$ contracts $N^-$ by at least $e^{-s}$ for all $s\geq 0$.

To sum the above errors over $\ell$ and $\rho$, we wish to use~\eqref{eq:total mass}. We first note that Propositions~\ref{prop:doubling} and~\ref{prop:height function properties} allow us to use closeness of $y_\rho$ and $x_{\rho,\ell}$ along with regularity of holonomy to deduce that 
\begin{equation}\label{eq:the x rho ells are comparable}
    V(y_\rho)  \mu^u_{y_\rho}(W_\rho)\asymp V(x_{\rho,\ell})  \mu^u_{x_{\rho,\ell}}(W_\ell).
\end{equation}
Here, we also use the fact that both $x_{\rho,\ell}$ and $y_\rho$ belong to $N_1^-\Omega$; cf.~\eqref{eq:x_rho,ell in omega}. 
Hence, we can use~\eqref{eq:total mass} to estimate the sum of the errors in~\eqref{eq:mollifier derivative}
yielding
\begin{align*}
    \eqref{eq:stable hol} 
    &= e^{-\d\g(w+jT_0) }  
    \sum_{\rho\in \Pcal_b} \sum_{\ell\in I_{\rho,j}}
         \int_\R  \int_{W_\rho}
        \Bigg(\sum_{\ell\in I_{\rho,j}}
        e^{-ib(t-\t_\ell(n))}
         \phi_{\rho,\ell}(t,n)\Bigg)
        F_\g(g_{t} ny_{\rho})    \;d\mu^u_{y_\rho} dt
    \nonumber\\
        &+ O\left(e^{-(1-\g)(w+jT_0)} \norm{ f}^\star_1
     \mu^u_{x}(N^+_{1}) V(x)\right),
\end{align*}
where we used that the above integrands have uniformly bounded support in the $\R$ direction, independently of $\ell$ (and $\rho$).
Indeed, this boundedness follows from that of the partition of unity $p_j$; cf.~\eqref{eq:time partition}.
We also used~\eqref{eq:phi tilde is bounded} to bound the $C^0$ norm of $\phi_{\rho,\ell}$.
Summing over $j$ and $w$ using~\eqref{eq:sum over j}, and recalling that $\g=1/2$, we obtain
\begin{align*}
     O_{T_0}\left(\frac{\norm{ f}^\star_1
     \mu^u_{x}(N^+_{1}) V(x) }{ (a+1/2)^m}\right).
\end{align*}

Recall the norm $\norm{\cdot}^\star_{1,B}$ defined in~\eqref{eq:norm star b} and note that $\norm{\cdot}^\star_1\leq B\norm{\cdot}^\star_{1,B}$.
Choosing $a$ and $\gls{varkappa}>0$ small enough, we can ensure that $e^{1+\varkappa}/(a+1/2)$ is at most
$1/(a+1/6)$.
With this choice, taking $B=b^{1+\varkappa}$ yields an error term of the form:
\begin{align}\label{eq:stable derivative}
     O\left(\frac{ \norm{ f}^\star_{1,B}
     \mu^u_{x}(N^+_{1}) V(x)}{ (a+1/6)^{m}}\right).
\end{align}

\subsection*{Mollifiers and Cauchy-Schwarz}
\label{sec:CauchySchwarz}
We are left with estimating integrals of the form:
\begin{align}\label{eq:pre-CauchySchwarz}
       \int_{\R\times W_\rho} \Psi_{\rho}(t,n)
    F_\g(g_{t} ny_{\rho})   \;d\mu^u_{y_\rho} dt,
    \qquad
    \Psi_{\rho}(t,n):=\sum_{\ell\in I_{\rho,j}} e^{-ib(t-\t_\ell(n))}
         \phi_{\rho,\ell}(t,n). 
\end{align}

We begin by giving an apriori bound on $\Psi_\rho$.
Denote by $\gls{Jrho}\subset\R$ the bounded support of the integrand in $t$ coordinate of the above integrals.
Note that~\eqref{eq:phi tilde is bounded} and the fact that $|t|\ll 1$ imply
 \begin{align}\label{eq:bound phi without tilde}
     \norm{\phi_{\rho,\ell}}_{L^\infty(J_\rho\times W_\rho)}
     \ll  1, \qquad \norm{\Psi_\rho}_{L^\infty(J_\rho\times W_\rho)} \ll \# I_{\rho,j}.
 \end{align}

To simplify notation, we let 
\begin{equation*}
    \gls{r}=(1-\g)(w+jT_0).
\end{equation*}
Note that we have that $y_\rho\in N_1^-\Omega$, $|J_\rho|\ll 1$, and $r\geq 1$.
Hence, Proposition~\ref{prop:upper estimate on mollifiers}, along with~\eqref{eq:mollifier is at most f}, the definition of $F_\g$ in~\eqref{eq:F_gamma} and the Cauchy-Schwarz inequality, yield
\begin{align*}
    \left|\int_{\R\times W_\rho} \Psi_\rho(t,n) F_\g(g_tny_\rho)\;d\mu^u_{y_\rho} dt\right|^2
   \ll 
       e^\star_{1,0}(f)^2 
     \int_{J_\rho\times W_\rho} |\Psi_\rho(t,n)|^2 \;d\mu^u_{y_\rho} dt
     \int_{W_\rho} V^2(g_rny_\rho)\;d\mu_{y_\rho}^u.
\end{align*}
The following lemma estimates the integral of $V^2$ on the right side of the above inequality.
\begin{lem}\label{lem:Margulis ineq on W_rho}
    We have the bound $\int_{W_\rho} V^2(g_rny_\rho)\;d\mu_{y_\rho}^u \ll_{T_0} e^{4\b\a jT_0} \murho(W_\rho)$.
\end{lem}
\begin{proof}
    Recall that $W_\rho$ has radius $\asymp\iota_b =b^{-2/3}$, and hence the expanded disk $W_\rho^b = \Ad(g_{-\log\iota_b})(W_\rho)$ has radius $\asymp 1$. 
    We also recall from~\eqref{eq:j small} that $r \geq -\log\iota_b$. 
    We also that $\rho$ is an element of $\Pcal_b$ defined in~\eqref{eq:recurrent boxes}, so that $V^2(g_{-\log \iota_b} y_\rho) \ll_{T_0} e^{4\b\a jT_0}$.
    Let $r_1 = r + \log\iota_b \geq 0$ and $y^b_\rho = g_{-\log b }y_\rho$. 
    Changing variables using~\eqref{eq:g_t equivariance}, and using Remark~\ref{remark:choice of V} and the Margulis inequality for $V^2$ in Theorem~\ref{thm:Margulis function}, we deduce the lemma from the following estimate
    \begin{align*}
        \int_{W_\rho} V^2(g_rny_\rho)\;d\mu_{y_\rho}^u
        = \iota_b^\d \int_{W^b_\rho} V^2(g_{r_1}n y^b_\rho)\;d\mu_{y^b_\rho}^u
        \ll \iota_b^\d V^2(y^b_\rho) 
        \mu^u_{y^b_\rho}(W^b_\rho) 
        = V^2(y^b_\rho) \mu^u_{y_\rho}(W_\rho).
    \end{align*}   
    \qedhere
\end{proof}
The above lemma hence yields the bound
\begin{align}\label{eq:Cauchy-Schwarz}
    \left|\int_{\R\times W_\rho} \Psi_\rho(t,n) F_\g(g_tny_\rho)\;d\mu^u_{y_\rho} dt\right|^2 
     &\ll_{T_0} e^\star_{1,0}(f)^2 e^{4\b\a jT_0} \mu^u_{y_\rho}(W_\rho)
     \int_{J_\rho\times W_\rho} |\Psi_\rho(t,n)|^2 \;d\mu^u_{y_\rho} dt.
\end{align}

\subsection*{Cusp-adapted partitions}
To estimate the right side of~\eqref{eq:Cauchy-Schwarz}, it will be convenient to linearize the phase functions $\t_{k}$.
For this purpose, we need to pick a cover of $W_\rho$ by balls with radius determined by a certain return time  of their centers to a given compact set.
\begin{prop}
     \label{prop:cusp adapted partition}
     There exists $\b_0 \asymp \b$ such that the following holds.
    For all $b\geq 1$ and $\rho\in \Pcal_b$, there exist a cover $\set{A_i:i}$ of $W_\rho$ and a set $\Rcal_\rho\subseteq W_\rho$ with $\murho(W_\rho\setminus \Rcal_\rho)\ll_{T_0} b^{-\b_0}e^{2\a\b jT_0}\murho(W_\rho)$ such that for all $i$ with $A_i\cap \Rcal_\rho\neq \emptyset$, we have
    \begin{enumerate}
        \item $A_i$ has the form $A_i=N^+_{r_i}\cdot u_i$ for some $r_i>0$ and $u_i\in W_\rho$.
        \item If $t_i=-\log r_i$, then $V(g_{t_i}u\yrho)\ll_\b 1$ for all $u\in A_i$.
        \item $b^{-8/10}\ll r_i\ll b^{-7/10}$.
        \item\label{item:bounded mult} $ \sum_i \murho(A_i) \ll \murho(W_\rho)$.
    \end{enumerate}
\end{prop}
\begin{proof}
Let $y^b_\rho = g_{-\log \iota_b } y_\rho$, where $\iota_b = b^{-2/3}$. 
Then, since $\rho\in\Pcal_b$, by~\eqref{eq:recurrent boxes}, $V(y^b_\rho)\ll_{T_0} e^{2\b\a jT_0}$.
As in Lemma~\ref{lem:Margulis ineq on W_rho}, the disk $W_\rho^b = \Ad(g_{-\log \iota_b})(W_\rho)$ has radius $\asymp 1$.

Let $r_0\geq 1$ be the constant provided by Theorem~\ref{thm:exp recurrence} applied with $\e=\b/40$.
Let $m_0= \lceil r_0^{-1}\log b \rceil$ and let $H=e^{3\b r_0}$ be the height provided by Theorem~\ref{thm:exp recurrence}.
Let $\chi_H$ denote the indicator function of the set of points of height at least $H$, i.e. the set $\set{y:V(y)>H}$.
Then, Theorem~\ref{thm:exp recurrence} yields
\begin{align*}
    \mu^u_{y^b_\rho}\left( n\in W^b_\rho:\sum_{1\leq \ell\leq m_0} \chi_H(g_{\ell r_0} n y^b_\rho) > m_0/20 \right) \ll_\b b^{-\b/40r_0}  V(y^b_\rho) \mu^u_{y^b_\rho}(W^b_\rho).
\end{align*}
Denote the set on the left side in the above estimate by $\Ecal_{\rho}^b$.
Let $c=7/10-2/3$ and $d=8/10-2/3$.
We claim that, if $b$ is large enough, then for every $n\in W^b_\rho\setminus \Ecal_\rho^b$, we can find $\eta\in [c,d]$ such that $V(g_{\eta \log b}n y^b_\rho)\leq H$.
Indeed, suppose not. 
Then, it follows that
\begin{align*}
    \sum_{1\leq \ell\leq m_0} \chi_H(g_{\ell r_0} nx) \geq \frac{\log b}{10r_0}-1 \geq m_0/10 -2.
\end{align*}
This contradicts the fact that $n\notin \Ecal^b_\rho$ when $b$ is large enough.

Let $\Rcal^b_\rho:=\mrm{supp}(\mu^u_{y^b_\rho})\cap W^b_\rho\setminus \Ecal^b_\rho$, and define $\Rcal_\rho$ to be its preimage in $W_\rho$.
More precisely, $\Rcal_\rho = \Ad(g_{\log\iota_b})(\Rcal^b_\rho) \subseteq W_\rho$.
Define a function $\varsigma:\Rcal_\rho \to [7/10,8/10]$ by setting $\varsigma(n)$ to be the least value of $\eta\in [7/10,8/10]$ such that
$V(g_{\eta \log b}n y_\rho)\leq H$.
Consider the cover $\set{\tilde{A}_u:u\in \Rcal_\rho}$, where each $A_u$ is the ball around each $u$ of radius $b^{-\varsigma(u)}$. 
Using the Vitali covering lemma and the uniform doubling in Prop.~\ref{prop:doubling}, we can find a finite subcover $\set{A_{u_i}:i}$ such that $ \sum_i \murho(A_{u_i}) \ll \murho(W_\rho)$.
This completes the proof by taking $\b_0 = \b/40r_0$, $A_i:=A_{u_i}$, and $r_i = 5b^{-\varsigma(u_i)}$.
   \qedhere
\end{proof} 
Let $\set{A_i}$ be the cover provided by Proposition~\ref{prop:cusp adapted partition}.
Combining this result with~\eqref{eq:bound phi without tilde}, we obtain
\begin{align}\label{eq:remove divergent orbits}
    \int_{J_\rho\times W_\rho} |\Psi_\rho(t,n)|^2\;d\mu_{y_\rho}^udt
    \leq \sum_i \int_{J_\rho\times A_i} |\Psi_\rho(t,n)|^2\;d\mu_{y_\rho}^udt
    + O_{T_0}\left( b^{-\b_0} \# I_{\rho,j}^2  e^{2\b\a jT_0}\mu_{y_\rho}^u(W_\rho)
    \right).
\end{align}

\subsection*{Linearizing the phase}
We now turn to estimating the sum of oscillatory integrals in~\eqref{eq:remove divergent orbits}.
For $k,\ell\in I_{\rho,j}$, we let
\begin{equation*}
    \psi_{k,\ell}(t,n) := \phi_{\rho,k}(t,n)
    \overline{\phi_{\rho,\ell}(t,n)}.
\end{equation*}
Expanding the square, we get
\begin{align}\label{eq:expanding square}
    \sum_i \int_{J_\rho\times A_i} |\Psi_\rho(t,n)|^2\;d\mu_{y_\rho}^udt
    = 
     \sum_i \sum_{k,\ell\in I_{\rho,j}}
     \int_{J_\rho\times A_i}  e^{-ib(\t_{k}(n)-\t_{\ell}(n))} \psi_{k,\ell}(t,n)  \;d\mu^u_{y_{\rho}} dt .
\end{align}

Using~\eqref{eq:g_t equivariance} and~\eqref{eq:N equivariance}, we change variables in the integrals using the maps taking each $A_i$ onto $N_1^+$.
More precisely, recall that $A_i$ is a ball of radius $r_i$ around $u_i\in W_\rho$.
Letting  
\begin{align}\label{eq:after moving to cpt}
    t_i = -\log r_i, \quad
    \yrho^i = g_{t_i}u_i \yrho, 
    \quad
    \t^i_k = \t_k(\Ad(g_{-t_i})(n)u_i), 
    \quad
    \psi^i_{k,\ell}(t,n) = \psi_{k,\ell} (t, \Ad(g_{-t_i})(n)u_i),
\end{align}
we can bound the above sum as follows:
\begin{align}\label{eq:from A_i to N_1}
    \eqref{eq:expanding square} 
    \leq  \sum_{i} e^{-\d t_i}
    \sum_{k,\ell\in I_{\rho,j}}
       \left| \int_{J_\rho\times N_1^+} 
       e^{-ib(\t^i_{k}(n)-\t^i_\ell(n))} \psi^i_{k,\ell}(t,n) 
       d\murhoi dt
       \right|.
\end{align}
 We note that the radius $r_i$ of $A_i$ satisfies 
\begin{align}\label{eq:size of t_i}
   b^{-8/10}\ll  e^{-t_i}=r_i \ll b^{-7/10}.
\end{align}
We also recall from Proposition~\ref{prop:cusp adapted partition} that $r_i$ was chosen so that
\begin{align}\label{eq:height of y_rho,i}
    V(\yrho^i)\ll 1, \qquad \forall i.
\end{align}
This is important for the proof of Theorem~\ref{thm:counting bad frequencies} below.

Next, we use the coordinate parametrization of $N^+$ by its Lie algebra $\mf{n}^+:=\mrm{Lie}(N^+)$ via the exponential map.
We suppress composition with $\exp$ from our notation for simplicity and continue to denote by $\murhoi$ and $N_1^+$ their preimage to $\mf{n}^+$ under $\exp$.

Recall from Section~\ref{sec:Carnot} the parametrization of $N^-$ by its Lie algebra $\mf{n}^-=\mf{n}^-_\a\oplus \mf{n}^-_{2\a}$ via the exponential map and similarly for $N^+$.
Let $w_i=(v_i,r_i)\in\mf{n}^+_\a\times \mf{n}^+_{2\a}$ be such that $u_i=\exp(w_i)$, where $u_i$ is the center of the ball $A_i$.
Recall the notation for transverse intersection points $n^-_{\rho,k}$ in~\eqref{eq:centers}.
For each $k\in I_{\rho,j}$, write
\begin{align*}
    n^-_{\rho,k} = \exp(u_k+s_k)
\end{align*}
with $u_k\in\mf{n}_\a^-$ and $s_k\in \mf{n}_{2\a}^-$.
With this notation, we have the following formula for the temporal functions $\t_k$.
The proof of this lemma is given in Section~\ref{sec:temporal function}.
\begin{lem}\label{lem:temp function formula}
    For every $i$, there exists a bilinear form $\langle\cdot,\cdot\rangle :  \nminus \times \nonepls \rightarrow \R$ such that the following holds.
    For every $k\in I_{\rho,j}$, there is a constant $c^i_{k}\in \R$ such that for all $n=\exp(v,r)\in N_1^+$ with $v\in\nonepls$ and $r\in \ntwopls$, we have that
    \begin{align*}
        \t^i_k(n)-\t^i_\ell(n) = c^i_{k,\ell}
        + e^{-t_i} \langle u_k-u_\ell +s_k-s_\ell , v\rangle
        + O(b^{-4/3}).
    \end{align*}
    Moreover, for every $(u,s)\in \noneminus\times \ntwominus$, the linear functional $\langle u+s,\cdot\rangle:\nonepls \to \R$ satisfies
    \begin{align*}
        \norm{\langle u+s,\cdot\rangle} \gg \norm{u},
    \end{align*}
    where $\norm{\langle u+s,\cdot\rangle} := \sup_{\norm{v}=1} |\langle u+s,v\rangle|$.
\end{lem}

We apply Lemma~\ref{lem:temp function formula} to linearize the phase and amplitude functions in~\eqref{eq:from A_i to N_1}.
Let
\begin{align}\label{eq:frequency vectors def}
    w^i_{k,\ell} := e^{-t_i}(u_k -u_\ell + s_k-s_\ell).
\end{align}
Note that $\Ad(g_{-t_i})$ contracts $N^+$ by at least $e^{-t_i}\ll b^{-7/10}$; cf.~\eqref{eq:size of t_i}.
Hence, in light of~\eqref{eq:phi tilde Lipschitz}, the Lipschitz norm of $\psi^i_{k,\ell}$ along $N^+$ is $O(b^{2/3-7/10})$.
Moreover, linearizing the phase function introduces an error $O(b^{1-4/3})$.
Hence, recalling that $|J_\rho|\ll 1$, we get
\begin{align}\label{eq:linearize phase}
    \eqref{eq:from A_i to N_1}
    \ll 
    \sum_i e^{-\d t_i}
    \sum_{k,\ell\in I_{\rho,j}}
    \left| \int_{N_1^+} e^{-ib  \langle w^i_{k,\ell},v\rangle}  d\murhoi
       \right|
        +b^{-3/100} 
        \murho(W_\rho)
        \# I_{\rho,j}^2,
\end{align}
where we used the estimate $\sum_i \murhoi(A_i) \ll \murho(W_\rho)$.

\subsection*{Excluding close pairs of unstable manifolds}
Consider the following partition of $I_{\rho,j}^2$:
\begin{equation}\label{eq:S_rho,j}
    C_{\rho,j} = \set{(k,\ell)\in I_{\rho,j}^2: 
     \norm{u_k-u_\ell}\leq b^{-1/10} }, 
     \qquad
     S_{\rho,j} = 
    I_{\rho,j}^2\setminus C_{\rho,j}.
\end{equation}
Then, $C_{\rho,j}$ parametrizes pairs of unstable manifolds which are too close along the $ \ntwominus$ direction in the stable foliation.
In particular, since $ \ntwominus=\set{0}$ when $X$ is real hyperbolic, $C_{\rho,j}$ simply parametrizes pairs of unstable manifolds which are too close along the stable foliation in this case.
With this notation, the sum on the right side of~\eqref{eq:linearize phase} can be estimated as follows:
\begin{align}\label{eq:exclude close pairs} 
    \sum_i e^{-\d t_i}
    \sum_{k,\ell\in I_{\rho,j}}
    \left| \int_{N_1^+} e^{-ib  \langle w^i_{k,\ell},v\rangle}  d\murhoi
       \right|
      \ll
    \# C_{\rho,j} \mu^u_{y_\rho}( W_\rho)
    + \sum_i e^{-\d t_i}
    \sum_{(k,\ell) \in S_{\rho,j}}
       \left| \int_{N_1^+} e^{-ib  \langle w^i_{k,\ell},v\rangle}  d\murhoi
       \right|.
\end{align}

 We estimate the first term in~\eqref{eq:exclude close pairs} via the following proposition, proved in Section~\ref{sec:close}.
 
\begin{prop}\label{prop:close pairs}
Assume that the parameter $\a$ is chosen sufficiently small.
Then, there exists a constant $\k_0>0$ such that for all $\ell\in I_{\rho,j}$, 
\[ \# \set{k \in I_{\rho,j}:(k,\ell)\in C_{\rho,j}} \ll_{T_0}
(b^{-\k_0/10} + e^{-\k_0 \g(w+jT_0)}) e^{\d \g(w+jT_0)}.
 \]
We may take $\k_0=\k/2$, where $\k$ is the constant provided by Theorem~\ref{thm:flat implies friendly}.
\end{prop}

\begin{remark}\label{rem:role of friendliness}
    When $X$ is non-real hyperbolic, Prop.~\ref{prop:close pairs} requires a polynomial decay estimate for PS measures near certain proper subspaces of the boundary, Theorem~\ref{thm:flat implies friendly}.
    When $X$ additionally has cusps, the latter result in turn requires the full strength of the $L^2$-flattening results in Section~\ref{sec:flattening}.
    These estimates are not needed in the real hyperbolic case.
\end{remark}

   In what follows, we shall assume that $\a$ is chosen small enough so that Prop.~\ref{prop:close pairs} holds.
Summarizing our estimates in~\eqref{eq:remove divergent orbits},~\eqref{eq:linearize phase},~\eqref{eq:exclude close pairs}, and Proposition~\ref{prop:close pairs}, we have shown that
\begin{align}\label{eq:before inverse theorem}
    &\int_{J_\rho\times W_\rho} |\Psi_\rho(t,n)|^2\;d\mu^u_{y_\rho} dt
    \nonumber\\
    &\ll
     \sum_i e^{-\d t_i}
    \sum_{(k,\ell) \in S_{\rho,j}}
       \left| \int_{N_1^+} e^{-ib  \langle w^i_{k,\ell},v\rangle}  d\murhoi
       \right|
    \nonumber\\
       &+
    \left((b^{-\b_0}e^{2\b\a jT_0}+b^{-3/100})\#I_{\rho,j}
    + (b^{-\k_0/10 } + e^{-\k_0\g(w+jT_0)}) e^{\d \g(w+jT_0)} 
    \right)
    \times \#I_{\rho,j}
    \mu^u_{y_\rho}(W_\rho).
\end{align}

\subsection{The role of additive combinatorics} 
To proceed, we wish to make use of the oscillations due to the large frequencies $b w^i_{k,\ell}$ to obtain cancellations.
First, we note that Lemma~\ref{lem:temp function formula} and the separation between pairs of unstable manifolds with indices in $S_{\rho,j}$ imply that the frequencies $bw^i_{k,\ell}$ have large size.
More precisely, the linear functionals $\langle w^i_{k,\ell},\cdot\rangle: \nonepls \to \R$ satisfy
\begin{align}\label{eq:size of separated frequencies}
    b^{-9/10}\ll  \norm{\langle w^i_{k,\ell},\cdot \rangle }\ll b^{-7/10} .
\end{align}
Let $\pi: \npls \to \nonepls$ denote the projection parallel to $\ntwopls$ and note that the integrands on the right side of~\eqref{eq:before inverse theorem} depend only on the $\nonepls$ component of the variable.
To simplify notation, we let\footnote{Note that $\pi$ is the identity map in the real hyperbolic case.}
\begin{align}\label{eq:proj of PS}
    \nu_i := \pi_\ast \murhoi\left|_{N_1^+} \right. .
\end{align}

\begin{remark}\label{rem:projections}
    It is worth emphasizing that the linearization provided by Lemma~\ref{lem:temp function formula} only depends on the unstable directions with weakest expansion under the flow.
    The reason we do so is that our metric on $\npls$ is not invariant by addition when $X$ is not real hyperbolic (it is invariant by the nilpotent group operations), but our non-concentration estimates for the measures $\mu^u_\bullet$ only hold for this metric. This in particular means the results of Section~\ref{sec:flattening} do not apply to these measures in this case, which is the reason we work with projections.
    It is possible to develop the theory in Section~\ref{sec:flattening} for measures and convolutions on nilpotent groups such as $N^+$ to avoid working with projections, however we believe the approach we adopt here is more amenable to generalizations beyond the algebraic setting of this article.
\end{remark}

For $w\in \nminus$, let
\begin{align}\label{eq:Fourier coefficient at x^i_rho,k}
    \hat{\nu}_i(w):= 
    \int_{\nonepls} e^{-i \langle w, v\rangle} \;d\nu_i(v).
\end{align}
Note that the total mass of $\nu_i$, denoted $|\nu_i|$, is $\murhoi(N_1^+)$.
Let $\gls{lambda}>0$ be a small parameter to be chosen using Theorem~\ref{thm:counting bad frequencies} below.
Define the following set of frequencies where $\hat{\nu}_i$ is large:
\begin{align}
    B(i,k,\l):= \set{\ell\in I_{\rho,j}: (k,\ell)\in S_{\rho,j} \text{ and } |\hat{\nu}_i(bw_{k,\ell}^i)| > b^{-\l} |\nu_i|}.
\end{align}
Then, splitting the sum over frequencies according to the size of the Fourier transform $\hat{\nu}_i$ and reversing our change variables to go back to integrating over $A_i$, we obtain
\begin{align}\label{eq:bound by count}
    \sum_i e^{-\d t_i} \sum_{(k,\ell)\in S_{\rho,j}}
    \int_{\nonepls} e^{-ib \langle w^i_{k,\ell}, v\rangle}
     \;d\nu_i(v)
    \ll
     \left( 
     \max_{i,k} \# B(i,  k,\l) 
    +  b^{-\l} \#I_{\rho,j} 
    \right)  \# I_{\rho,j}
    \mu^u_{y_\rho}(W_\rho),
 \end{align}
 where we again used the estimate $\sum_i \murhoi(A_i)\ll \mu^u_{y_\rho}(W_\rho)$.
The following key counting estimate for $B(i,k,\l)$ is
deduced from Corollary~\ref{cor:flattening}. Its proof is given in Section~\ref{sec:count bad freqs}.
\begin{thm}\label{thm:counting bad frequencies}
For every $\e>0$, there exists $\l>0$ such that for all $i$ and $k$, we have
\begin{align*}
    \# B(i,  k,\l) \ll_{\e}
    b^{\e} \left( b^{-\k_0/10} 
    + e^{-\k_0\g(w+jT_0)} \right)
    e^{\d \g(w+jT_0)} ,
\end{align*}
where $\k_0>0$ is the constant provided by Proposition~\ref{prop:close pairs}.
\end{thm}

 \subsection*{Combining estimates on oscillatory integrals}
Let $\b_0$ and $\k_0>0$ be as in Propositions~\ref{prop:cusp adapted partition} and~\ref{prop:close pairs} respectively.
In what follows, we assume $\e$ is chosen smaller than $\k_0/100$ and that $\l\leq \min\set{\b_0, 3/100,\k_0/20} $.
Let 
\begin{align*}
    Q = (b^{-\k_0/20} + b^{\e} e^{-\k_0 \g(w+jT_0)}) e^{\d \g(w+jT_0)}.
\end{align*}
Theorem~\ref{thm:counting bad frequencies}, combined with~\eqref{eq:Cauchy-Schwarz},~\eqref{eq:before inverse theorem} and~\eqref{eq:bound by count}, yields:
\begin{align}\label{eq:combined good and bad frequencies}
    \int_{\R\times W_\rho} \Psi_\rho(t,n) F_\g(g_tny_\rho)\;d\mu^u_{y_\rho} dt
    \ll e^\star_{1,0}(f) e^{3\b\a jT_0}
      \mu^u_{y_\rho}(W_\rho)
      \times \bigg(b^{-\l/2} \#I_{\rho,j}
      + \sqrt{\#I_{\rho,j}\times Q}\bigg),
\end{align}
where we used the elementary inequality $\sqrt{x+y}\leq \sqrt{x}+\sqrt{y}$ for any $x,y\geq 0$.

Our next goal is to estimate the sum of the above bound over $\rho$.
Recall that $\murho(W_\rho)\asymp \mu^u_{x_{\rho,\ell}}(W_\ell) \asymp \murho(N^+_{\iota_b})$ for all $\ell\in I_{\rho,j}$ by Prop.~\ref{prop:doubling}.
Hence, the Cauchy-Schwarz inequality yields
\begin{align*}
    \sum_{\rho\in\Pcal_b}  \mu^u_{y_\rho}(W_\rho)
      \sqrt{ \#I_{\rho,j} }
     &\ll
    \left(\sum_{\rho\in\Pcal_b} \mu_{y_\rho}^u(W_\rho)
    \times
    \sum_{\rho\in\Pcal_b,\ell \in I_{\rho,j}} \mu^u_{x_{\rho,\ell}}(W_\ell)\right)^{1/2}
    \nonumber\\
    &\ll  e^{O_\b(\a jT_0)} \times  \mu_x^u(N_1^+)^{1/2} 
    e^{\d(\g(w+jT_0)/2 }
    ,
\end{align*}
where the second inequality follows by Lemma~\ref{lem:weighted number of boxes} and~\eqref{eq:total mass}.
By definition of $\muxu$ and Lemma~\ref{lem:ht vs dist}, we have that $\muxu(N_1^+)\gg e^{-\d \dist(x,o)} \gg V(x)^{O_\b(1)}$.
Hence, we get
\begin{align}\label{eq:estimate on close pairs}
    \sum_{\rho\in\Pcal_b}  \mu^u_{y_\rho}(W_\rho)
      \sqrt{ \#I_{\rho,j} }
    \ll \mu_x^u(N_1^+) 
    \times  e^{\d(\g(w+jT_0)/2 +O_\b(\a jT_0)},
\end{align}
We also note that a similar argument to~\eqref{eq:total mass} yields $
     \sum_{\rho\in\Pcal_b}  \mu^u_{y_\rho}(W_\rho) \#I_{\rho,j}
    \ll  
     e^{\d\g(w+jT_0)} \mu_x^u(N_1^+)$.

 Recall that $\l/2\leq \k_0/40$.
It follows that upon combining the above estimate with~\eqref{eq:combined good and bad frequencies} and~\eqref{eq:estimate on close pairs}, we obtain the following bound on the sum of the integrals in~\eqref{eq:combined good and bad frequencies}:
\begin{align*}
 & e^{-\d\g(w+jT_0)} 
    \sum_{\rho\in \Pcal_b } 
     \int_{\R\times W_\rho} \Psi_\rho(t,n) F_\g(g_tny_\rho)\;d\mu^u_{y_\rho} dt 
          \nonumber\\ 
   & \ll  e^\star_{1,0}(f)   \mu_{x}^u(N_1^+)\times e^{O_\b(\a jT_0/2)}
  \times
         \left(  
     b^{-\l/2}
    +
         b^{\e/2} e^{-\k_0 \g(w+jT_0)/2} \right), 
\end{align*} 
where we again used the inequality $\sqrt{x+y}\leq \sqrt{x}+\sqrt{y}$.

Using~\eqref{eq:sum over j} to sum the above error terms over $j$ and $w$, we obtain 
\begin{align}\label{eq:oscillatory int error}
    O_{T_0,\e}\Bigg ( 
   e^\star_{1,0}(f)  \mu_{x}^u(N_1^+) 
    \times \Bigg[ 
   \frac{b^{-\l/2}}{(a -O_\b(\a))^m}
    + \frac{b^{\e/2}}
    {(a +\k_0\g/2)^m}
    \Bigg]\Bigg).
\end{align}

To simplify the above bound, recall that $\l$ is chosen according to Theorem~\ref{thm:counting bad frequencies} and hence its size depends on $\e$, however $\k_0$ is given by Proposition~\ref{prop:close pairs} and is independent of $\e$. Moreover, $\g=1/2,\l$ and $\k_0$ are independent of $a$ and $\a$, and we are free to choose the parameter $\a$ as small as needed.
We also recall that $m=\lceil\log b\rceil$; cf.~\eqref{eq:m and b}.
Hence, we may choose $a,$ and $\e$ small enough relative to $\k\g$ to ensure that
\begin{align*}
    \frac{e^{\e/2}}{a+\k\g/2} \leq 
    \frac{1}{a+\k\g/3}.
\end{align*}
Using the bound $e^{-\l/2}\leq 1/(1+\l/2)$ and taking $\a$ small enough, depending on $a,\b$ and $\l$, we obtain
\begin{align*}
    \frac{e^{-\l/2}}{a -O_\b(\a)} \leq 
    \frac{1}{a+a\l/4}.
\end{align*}
Hence, taking $a$ small enough so that $a\l/4\leq \k\g/3$, the error term in~\eqref{eq:oscillatory int error} becomes
\begin{align}\label{eq:osc int error simplified}
    O\left( \frac{e^\star_{1,0}(f) V(x)  \mu_{x}^u(N_1^+)}{(a +a\l/4)^m} \right),
\end{align}
where we used the inequality $V(x)\gg 1$.

\subsection{Parameter selection and conclusion of the proof} \label{sec:parameter}
In this subsection, we finish the proof of Theorem~\ref{thm:Dolgopyat} assuming Lemma~\ref{lem:temp function formula}, Proposition~\ref{prop:close pairs}, and Theorem~\ref{thm:counting bad frequencies}.

Collecting the error terms in Lemma~\ref{lem:large j},~\eqref{eq:first term},~\eqref{eq:cusp contribution},~\eqref{eq:replace with mollifier},~\eqref{eq:zeta_j contribution},~\eqref{eq:divergent contribution},~\eqref{eq:stable derivative}, and~\eqref{eq:osc int error simplified}, and letting $\s_\star>0$ be the minimum of all the gains in these error terms, we obtain
\begin{align*}
    e^\star_{1,0}(R(z)^{m}f)
    \ll  \frac{\norm{f}^\star_{1,B}}{(a+\s_\star)^m}.
\end{align*}
Letting $C_\G$ denote the implied constant, this estimate concludes the proof of Theorem~\ref{thm:Dolgopyat}.



\section{The temporal function and proof of Lemma~\ref{lem:temp function formula}}
\label{sec:temporal function}

In this section, we give an explicit formula for the temporal functions $\t_{k,\ell}$ appearing in Section~\ref{sec:Dolgopyat} and prove Lemma~\ref{lem:temp function formula}.
Our argument is Lie theoretic.
We refer the reader to~\cite[Chapter 1]{Knapp} for background on the material used in this section. 
Similar results are known more generally outside of the homogeneous setting by more dynamical/geometric arguments building on work of Katok and Burns~\cite{KatokBurns}.

\subsection{Proximal representations and temporal functions}

Let $\rho: G\r H:=\mrm{SL}_n(\R)$ be a proximal irreducible representation of $G$, i.e., $\rho$ is irreducible and the top eigenspace of $\rho(g_1)$ is one-dimensional.
The existence of such a representation is guaranteed by~\cite{Tits-RepsReductiveGps}.
In what follows, we suppress $\rho$ from the notation and view $G$ as a subgroup of $H$ and view elements of the Lie algebra of $G$ as (traceless) matrices in $\mf{h}=\mrm{Lie}(H)$.

Let $\set{e_i}$ denote the standard basis of $\R^n$. Without loss of generality, we assume that $e_1$ is a top eigenvector for $g_1$ and denote by $e^\l>1$ the corresponding top eigenvalue.
Up to a change of basis, we shall further assume that $N^+$ (resp.~$N^-$) consists of upper (resp.~lower) triangular matrices.

Given a matrix $h\in \mrm{SL}_n(\R)$, we let $\pi_0(h)$ denote its top left entry.
In particular, $\pi_0(h) =1$ for all $h\in N^+\cup N^-$.
Moreover, since $\rho$ is proximal, $M$ acts trivially on the top eigenspace of $g_1$, where we recall that $M$ denotes the centralizer of the geodesic flow inside the maximal compact subgroup of $G$.
It follows that $\pi_0(m)=1$ for all $m\in M$.
Finally, we have the simple formula
\begin{align*}
    t = \l^{-1}\log \pi_0(g_t).
\end{align*}
These observations will allow us to compute the functions $\t_k^i$ using elementary matrix calculations.

Let $X\in \mf{n}^-$ and $Y\in \mf{n}^+$ be sufficiently close to $0$.
As the product map $ M\times A\times N^+ \times N^-\to G$ is a diffeomorphism near identity, 
there exist unique $\t_X(Y)\in \R$ and $\phi_X(Y)\in \mf{n}^+$ such that
\begin{align}\label{eq:product temporal}
    \exp(Y)\exp(X) \in N^-M g_{\t_X(Y)} \exp(\phi_X(Y)), 
\end{align}
where $\exp(Z) = \sum_{n\geq 0} \frac{Z^n}{n!}$.
To compute $\t_X(Y)$, for a matrix $M$, let $r_1(M)$ and $c_1(M)$ denote its top row and first column respectively.
Then,~\eqref{eq:product temporal} shows that $\pi_0(\exp(Y)\exp(X)) = 1+ r_1(Y)\cdot c_1(\exp(X)) + O(\norm{Y}^2)$.
Hence, 
\begin{align}
    \t_X(Y) = \l^{-1}\log (1+ r_1(Y)\cdot c_1(\exp(X)) + O(\norm{Y}^2)).
\end{align}

\subsection{Proof of the first assertion of Lemma~\ref{lem:temp function formula}}

Fix $k\in I_{\rho,j}$
 and recall the elements $n_{\rho,k}^-\in N^-$ which were defined by the displacement of the points $x_{\rho,k}$ from $y_\rho$ along $N^-$ inside the flow box $B_\rho$; cf.~\eqref{eq:centers}.
 We also recall the elements $u_k\in\noneminus$ and $s_k\in \ntwominus$ chosen so that $n^-_{\rho,k} = \exp(u_k+s_k)$.
In what follows, we set $ X_k:=u_k+s_k$.
Given $Y\in \npls$, we write $Y_\a$ and $Y_{2\a}$ for its $\nonepls$ and $\ntwopls$ components respectively.

Recall the vectors $w_i=v_i+r_i\in\npls$ defined above Lemma~\ref{lem:temp function formula}, where $v_i$ and $r_i$ denoted the $\nonepls$ and $\ntwopls$ components of $w_i$ respectively.
We also recall the return times $t_i$ in~\eqref{eq:after moving to cpt}.
For $Y\in \npls$, let $Y^i = \log(\exp(\Ad(g_{-t_i})(Y))\exp(w_i))\in\npls$.
In particular, $Y^i$ takes the form
\begin{align*}
    Y^i = w_i +  e^{-t_i}Y_\a + e^{-t_i}[Y_\a,v_i]/2 + e^{-2t_i}Y_{2\a}  .
\end{align*}
In this notation, we have by definition of the functions $\t_k$ (cf.~\eqref{eq:def of tau_ell}) and $\t_k^i$ (cf.~\eqref{eq:after moving to cpt}) that $\t_k^i(Y)=\t_k(Y^i)$.
Moreover, it follows from the definition of $\t_k$ and~\eqref{eq:stable hol commutation} that 
\begin{align*}
    g_{t_{\rho,k}}\exp(\phi_{X_k}^{-1}(Y^i)) \exp(X_k) \in N^-M g_{\t_k(Y^i)} \exp(Y^i).
\end{align*}
Rearranging this identity and using the fact that $g_t$ normalizes $N^-M=MN^-$, we obtain
\begin{align*}
    \exp(Y^i) \exp(-X_k) \in N^- M g_{t_{\rho,k}-\t_k(Y^i)}\exp(\phi_{X_k}^{-1}(Y^i)),
\end{align*}
where we used the fact that $\exp(X_k)^{-1}=\exp(-X_k)$.
We thus obtain the formula
\begin{align}\label{eq:temp formula}
    \t_k(Y^i) =  t_{\rho,k}-\t_{-X_k}(Y^i) ,
\end{align}
where the notation is as in~\eqref{eq:product temporal} above.

Define a bilinear form $\langle \cdot,\cdot \rangle: \nminus \times \nonepls \r \R$ by
\begin{align}\label{eq:bilinear form}
    \langle X, Y_\a \rangle := \l^{-1}  r_1(Y_\a) \cdot c_1\left(X \right).
\end{align}
Recall that our flow boxes $B_\rho$ have radius $\leq b^{-2/3}$ in the unstable direction; cf.~\eqref{eq:recurrent boxes}.
In particular, $\norm{w_i}\ll b^{-2/3}$.
Moreover, by~\eqref{eq:size of t_i}, we have that $e^{-t_i}\ll b^{-7/10}$. Hence, we get that $\norm{Y^i}\ll b^{-2/3}$. 
Note further that since $e^{-t_i}v_i = O(b^{-4/3})$ by~\eqref{eq:size of t_i}, we also have that $Y^i = w_i+e^{-t_i}Y_\a + O(b^{-4/3})$.
Finally, note by Lemma~\ref{lem:elementary properties}\eqref{item:perp roots} below that for all $n\geq 2$ and $X\in \mf{n}^+$, $r_1(Y_\a)\cdot c_1(X^n)=0$.
In particular, since $Y_\a$ is strictly upper triangular, $r_1(Y_\a)\cdot c_1(\exp(X_k))=r_1(Y_\a)\cdot c_1(X_k)$.

Let $d_k^i := 1+r_1(w_i)\cdot c_1(X_k)$.
Then, using the estimate $\log(1+x)=x +O(|x|^2)$ for $x$ near $0$, the above discussion yields
\begin{align*}
    \t_k(Y^i) = t_{\rho,k}- \l^{-1} d_k^i +  
    e^{-t_i} \langle X_k, Y_\a \rangle +O(b^{-4/3}).
\end{align*}
Thus, taking $c^i_{k,\ell} = t_{\rho,k}- \l^{-1} d_k^i-t_{\rho,\ell}+ \l^{-1} d_\ell^i$, we obtain
\begin{align*}
    \t_k^i(Y)-\t_\ell^i(Y) = c^i_{k,\ell} + e^{-t_i}
    \langle X_k-X_\ell ,Y_\a \rangle + O(b^{-4/3}),
\end{align*}
which completes the proof of the first assertion of Lemma~\ref{lem:temp function formula}.

\subsection{Norm of the bilinear form}
The following lemma is needed to prove the second assertion of Lemma~\ref{lem:temp function formula} in the next subsection.
Recall that $\mf{h}$ denotes the Lie algebra of $\mrm{SL}_n(\R)$. Let $\mf{h}_{\b}$ denote the $\Ad(g_t)$-eigenspace with eigenvalue $e^{\b t}$. In particular, $\ntwominus \subseteq \mf{h}_{-2\a}$. 
\begin{lem}\label{lem:elementary properties}
    \begin{enumerate}
        
        \item\label{item:perp roots} 
        For every $Y\in \mf{h}_\a$ and $Z\in \mf{h}_{-k\a}$, $k\geq 2$, we have $r_1(Y)\cdot c_1(Z)=0$.

        \item\label{item:injective} The restriction of the linear map $X\mapsto c_1(X)$ to $\noneminus$ is injective.

    \end{enumerate}
\end{lem}

\begin{proof}
    
    Let $\th:\Lie(G)\to \Lie(G)$ denote a Cartan involution preserving $\Lie(K)$ and acting by $-\id$ on $\Lie(A)$. In particular, $\th$ sends $\npls$ onto $\nminus$ while respecting their decompositions into $\Ad(g_t)$-eigenspaces. 
    By~\cite{Mostow-Selfadjoint}, we may assume that $\th$ is the restriction to $\Lie(G)$ of the map on $\mf{h}$ that sends each matrix to its negative transpose.
    Finally, let $B(\cdot,\cdot)$ denote the Killing form on $\Lie(G)$ and recall that the quadratic form $-B(\cdot, \th(\cdot))$ is positive definite.
    
To show Item~\eqref{item:perp roots}, note that, since $Z$ is strictly lower triangular and $Y$ is strictly upper triangular, we have $\pi_0([Y,Z]) = r_1(Y)\cdot c_1(Z)$, where we recall that $\pi_0(\cdot)$ is the top left entry.
On the other hand, $[Y,Z]$ is strictly lower triangular since $k\geq 2$ and $[\mf{h}_\a,\mf{h}_{-k\a}] \subseteq \mf{h}_{-(k-1)\a}$.
The claim follows.
    
For Item~\eqref{item:injective}, fix an arbitrary non-zero element $X\in \noneminus$ and let $Z=[X,\th(X)]$. 
Note that, since $\th$ is a Lie algebra homomorphism, then $\th(Z) = -Z$.
In particular, $Z$ belongs to $\Lie(A)$.
Moreover, we claim that $Z\neq 0$, and, hence, $\pi_0(Z)\neq 0$.
Indeed, let $\w\in \Lie(A)$ be such that $[\w,\cdot]$ fixes $\nonepls$ pointwise.
Then, by properties of the Killing form, we obtain
\begin{align*}
    B(Z,\w) = B(X,[\th(X),\w]) = -B(X,\th(X))\neq 0,
\end{align*}
since $X\neq 0$ and $-B(\cdot,\th(\cdot))$ is positive definite.
In particular, arguing as in the first part, we have that $\norm{c_1(X)}^2= |\pi_0(Z)| \neq 0$, concluding the proof.
\qedhere
\end{proof}

\subsection{Proof of the second assertion of Lemma~\ref{lem:temp function formula}}
Fix $(u,s)\in \noneminus\times \ntwominus$. 
We show that $\sup |\langle u+s, Y\rangle| \gg \norm{u}$, where the supremum ranges over all $Y\in \nonepls$ with $\norm{Y}=1$.
Let $Y= \th(u)/\norm{u}$.
    Then, by Lemma~\ref{lem:elementary properties}\eqref{item:perp roots}, we have that 
    \begin{align*}
        \langle u+s, Y\rangle = \l^{-1} r_1(Y)\cdot c_1(u).
    \end{align*}
    Hence, since $\th(u)$ is the negative transpose of $u$,  $  \langle u+s, Y\rangle = -\l^{-1}\norm{c_1(u)}^2/\norm{u}$.
    It follows by Lemma~\ref{lem:elementary properties}\eqref{item:injective} and equivalence of norms that $\norm{u}\asymp \norm{c_1(u)}$, concluding the proof of Lemma~\ref{lem:temp function formula}.

\section{Dimension Increase Under Iterated Convolutions}
\label{sec:flattening}

The goal of this section is to prove that measures that do not concentrate near proper affine subspaces in $\R^d$ become smoother under iterated self-convolutions in the sense of quantitative increase in their $L^2$-dimension; cf.~Theorem~\ref{thm:quant ellq improvement} below. This result immediately implies Theorem~\ref{thm:ellq improvement}.
As a corollary, we deduce that the Fourier transforms of such measures enjoy polynomial decay outside of a very sparse set of frequencies; cf.~ Corollary~\ref{cor:flattening}.
In fact, we prove that such results hold for certain \textit{projections} of non-concentrated measures.

Corollary~\ref{cor:flattening} provides the key ingredient in the proof of Theorem~\ref{thm:counting bad frequencies} where it is applied to (projections of) conditional measures of the BMS measure.
Moreover, the proof of Proposition~\ref{prop:close pairs} in the case of cusped non-real hyperbolic manifolds requires a polynomial non-concentration estimate near hyperplanes which we deduce from Theorem~\ref{thm:quant ellq improvement}; cf.~Theorem~\ref{thm:flat implies friendly}.

\subsection{General setting}\label{sec:projection setup}

Throughout this section, $N$ denotes a connected nilpotent group, equipped with a right-invariant metric. 
Affine subspaces of $N$ are defined analogously to Definition~\ref{def:aff subspace of nilpotent}.
Given $\rho>0$, $W\subset N$ and $x\in N$, we write $W^{(\rho)}$ and $B(x,\rho)$ for $\rho$-neighborhood of $W$ and $\rho$-ball around $x$ respectively.
We fix a surjective homomorphism $\pi:N \r \R^d $ from $N$ onto $\R^d$.
We assume that $N$ is equipped with a compatible $1$-parameter group of dilation automorphisms, which we denote $g_t$, such that for all $x\in N$, $t\in\R$ and $r>0$,
\begin{align}\label{eq:scaling equivariance under pi}
    \pi(g_t \cdot x) = e^t \pi(x), \qquad \text{and} \qquad \dist(g_t\cdot x,\id) = e^t \dist(x,\id).
\end{align}

 We further fix $\widetilde{\L}$ to be a lattice in $N$ (i.e.~a discrete cocompact subgroup) and let $D\subset N$ be a fundamental domain for $\widetilde{\L}$ containing the identity element.
By scaling $\widetilde{\L}$ using $g_{-t}$ if necessary, we shall assume without loss of generality that $D\subseteq B(\id,1)$.
Up to composing $\pi$ with a linear change of basis, we shall assume that
\begin{align*}
    \pi(\widetilde{\L})=\Z^d, \qquad 
    \pi(D) = [0,1)^d.
\end{align*}
Finally, we fix a norm on $\R^d$ and assume that its induced metric is compatible with the metric on $N$ in the sense that there is a uniform $c\geq 1$ such that for all sets $E$ and $\rho >0$, we have
\begin{align}\label{eq:pullback metric}
    \pi^{-1}(E^{(\rho)}) \subseteq \pi^{-1}(E)^{(c\rho)}.
\end{align}

\begin{examples}\label{ex:examples of framework}
\begin{enumerate}
    \item In our application in this article, $N$ will be $N^{\pm}$, equipped with the Cygan metric from Section~\ref{sec:Carnot}, and  $\pi$ will be the projection onto the abelianization $N/[N,N]\cong \R^d$ (which is the identity map in the real hyperbolic case where $N$ is abelian).

    \item Another interesting example that falls under our setting is $N=\R^D$ for some $D\geq d$, $\pi$ is a standard projection, and $g_t$ is a diagonal matrix, where one equips $N$ with a suitable analog of the Cygan metric which satisfies the above scaling properties.
    In particular, we anticipate that the results of this section will have applications towards the study of fractal geometric properties of self-affine measures and their projections.
\end{enumerate}
\end{examples}

\subsection{Non-uniform affine non-concentration} 

We begin by introducing our non-concentration hypothesis, which allows for exceptional sets of points and scales where concentration may happen.

\begin{definition}\label{def:aff non-conc}
    
    Let positive functions $\l$, $\vp$, and $C$ on $(0,1]$ be given such that $\vp(x) \xrightarrow{x\to 0} 0$.
    We say a Borel measure $\mu$ on $N$ is \textit{$(\l,\vp,C)$-affinely non-concentrated at almost every scale} (or $(\l,\vp,C)$-ANC for short)  
    if the following holds. 
    For every $0<\e,\th\leq 1$, $k\in\N$, and $r\geq C(\th)$:
    \begin{enumerate}
        
        \item 
        There is an exceptional set $\Ecal = \Ecal(k,\e,\th,r)\subset  B(\id,2)$ with $\mu(\Ecal)\leq C(\th) 2^{-\l(\th) k} \mu(B(\id,2))$.

        \item\label{item:non-conc def 4} For every $x\in B(\id,2)\cap \mrm{supp}(\mu)\setminus \Ecal$, there is a set of good scales $\Ncal(x)\subseteq [0,k]\cap \N$ with $\# \Ncal(x) \geq (1-\th) k$.
        
        \item For every $x\in B(\id,2)\cap \mrm{supp}(\mu)\setminus \Ecal$, every affine subspace $W < N$, every $\ell\in\Ncal(x)$ and $\rho \asymp 2^{-r\ell}$, we have
        \begin{align}\label{eq:affine non-conc}
            \mu( W^{(\e \rho)} \cap B(x,\rho)) \leq (\vp(\th)+C(\th)\vp(\e)) \mu(B(x,\rho)).
        \end{align}
    \end{enumerate}
    We say $\mu$ is ANC at almost every scale when the parameters are understood from context.
    \end{definition}

In words, this definition says that $\mu$ exhibits strong non-concentration near proper subspaces at nearly all scales outside of a small exceptional set, however the size of the exceptional set is allowed to depend on the strength and frequency of non-concentration.

\begin{remark}

    \begin{enumerate}
        \item Note that we do not require $\mu$ to be a probability measure or compactly supported. 
        Instead, we require that non-concentration holds uniformly over all balls centered in a given ball around identity (outside of some exceptional set).
        This flexibility allows us to avoid edge effects in verifying~\eqref{eq:affine non-conc} for the restrictions of the PS conditional measures $\mu_x^u$ to bounded balls.

        \item For purposes of following the arguments in this section, there is no harm in considering the example $\l(x)=\b x$ for some $\b>0$ and the stronger bound
     \begin{align*}
            \mu( W^{(\e \rho)} \cap B(x,2^{-r\ell})) \leq C(\th)\vp(\e) \mu(B(x,\rho)),
    \end{align*}
    in place of~\eqref{eq:affine non-conc}. In fact, the measures $\muxu$ are shown to satisfy this bound in Corollary~\ref{cor:aff non-conc of projections}.
    \end{enumerate}
\end{remark}

\subsection{The $L^2$-flattening theorem}

For $k\in \N$, let 
\begin{align*}
    \L_k:= 2^{-k} \Z^d,
\end{align*}
and let $\Dcal_k$ be the dyadic partition of $\R^d$ given by translates of $2^{-k}[0,1)^d$ by $\L_k$.
    For $x\in \R^d$, we denote by $\Dcal_k(x)$ the unique element of $\Dcal_k$ containing $x$.
    For a Borel probability measure $\nu$, we define $\nu_k\in \mrm{Prob}(\L_k)$ to be the scale-$k$ discretization of $\nu$, i.e. 
    \begin{align}\label{eq:discretized measure}
        \nu_k = \sum_{\l\in\L_k} \nu(\Dcal_k(\l)) \d_{\l}.
    \end{align}
    For any $\mu\in \mrm{Prob}(\L_k)$ and $0<q<\infty$, we set $
        \norm{\mu}_q^q :=  \sum_{\l\in\L_k} \mu(\l)^q$.
    The \textit{convolution} $\mu\ast \nu$ of two probability measures $\mu$ and $\nu$ on $\R^d$ is defined by
    \begin{align*}
        \mu\ast \nu (A) = \int \int 1_A(x+y) \;d\mu(x)\;d\nu(y),
    \end{align*}
    for all Borel sets $A\subseteq \R^d$.
      Recall the setup in Section~\ref{sec:projection setup}.

    \begin{thm}\label{thm:quant ellq improvement}
    Let $\l$, $\vp$ and $C$ be given.
    For every $\e>0$, there exist $n,k_1\in \N$ such that the following holds.
    Let $\tilde{\mu}$ be a $(\l,\vp,C)$-ANC Borel measure on the gorup $N$.  
    Let $\mu$ be the projection to $ \R^d$ of $\tilde{\mu}\left|_{B(\id,1)}\right.$ under $\pi$, normalized to be a probability measure.
   Then, for every $k\geq k_0$, we have
    \begin{align*}
        \norm{\mu_k^{\ast n}}^2_2 \ll_{\e,d,n} 2^{-(d-\e)k},
    \end{align*}
    with implicit constant depending only on $d$ and the non-concentration parameters of $\tilde{\mu}$.
    In particular, for all $P\in \Dcal_k$, we have
    \begin{align*}
        \mu_k^{\ast 2n}(P) \ll_{\e,d,n} 2^{-(d-\e)k}.
    \end{align*}
\end{thm}

The following is a more precise version of Corollary~\ref{cor:flattening intro}.

\begin{cor}\label{cor:flattening}
    Let $\tilde{\mu}$ and $\mu$ be as in Theorem~\ref{thm:quant ellq improvement}.
    Then,
    for every $\e>0$, there is $\d>0$, depending only on the non-concentration parameters of $\tilde{\mu}$, such that for every $T\geq 1$, the set
    \begin{align*}
        \set{w\in \R^d: \norm{w}\leq T \text{ and } |\hat{\mu}(w)|\geq T^{-\d}}
    \end{align*}
    can be covered by $O_\e(T^{\e})$ balls of radius $1$, where
    $\hat{\mu}$ denotes the Fourier transform of $\mu$.
    The implicit constant depends only on $\e$, the diameter of the support of $\mu$ and the non-concentration parameters.

\end{cor}

\begin{remark}   
\begin{enumerate}
        \item As noted in~\cite{BaYu}, the proof of Theorem~\ref{thm:quant ellq improvement} and Corollary~\ref{cor:flattening} goes through under the weaker hypothesis replacing the ball $B(x,\rho)$ in~\eqref{eq:affine non-conc} with the larger ball $B(x,c\rho)$, for some fixed $c\geq 1$.
        Indeed, the proof relies on the discretized form of~\eqref{eq:affine non-conc} in Lemma~\ref{lem:non-conc of discretization}, where such weaker inequality naturally appears.
        This weaker hypothesis is very useful however for applying the above results to non-doubling measures; cf.~\cite{BaYu}.

        \item Note that Def.~\ref{def:aff non-conc} requires ANC to hold at points in $B(\id,2)$, while Theorem~\ref{thm:quant ellq improvement} and Corollary~\ref{cor:flattening} concern the restriction of the measure to $B(\id,1)$. The same arguments work for any two fixed nested balls, after suitably enlarging the implicit constants.
    In particular, such requirements can be vacuously satisfied if $\tilde{\mu}$ is compactly supported.
    This flexibility however allows us to avoid certain edge effects when working with restrictions of $\tilde{\mu}$ to a ball.
    \end{enumerate}
\end{remark}

\subsection{Preliminary lemmas on discretized measures}

The first lemma asserts that convolution and discretization essentially commute.  This justifies the statement of Theorem~\ref{thm:quant ellq improvement}. 
    \begin{lem}\label{lem:disc and conv commute}
       Let $\mu$ and $\nu$ be Borel probability measures on $\R^d$. Then, for all $q>1$ and $k\in \N$, we have $\norm{(\mu\ast\nu)_k}_q \asymp_{q,d} \norm{\mu_k\ast \nu_k}_q$.
    \end{lem}
\begin{proof}
    This lemma is a direct consequence of the fact that a ball of radius $\rho$ with $2^{-k-1}<\rho\leq 2^{-k}$, $k\in\Z$, can be covered with $O_d(1)$ elements of $\Dcal_k$; cf.~\cite[Lemma 4.3]{Shmerkin-Furstenberg} for a detailed proof in the case $d=1$, which readily generalizes to higher dimensions. 
\end{proof}

For each $k\in\N$, set
\begin{align*}
    \widetilde{\L}_k := g_{-k\log 2}(\widetilde{\L}).
\end{align*}
In particular, $\pi(\widetilde{\L}_k)=\L_k = 2^{-k}\Z^d$.
Denote by $\Dcal_k$ the partition of $N$ consisting of translates of $g_{-k\log 2}(D)$ by $\widetilde{\L}_k$.
Using the lattices $\widetilde{\L}_k$, we define the scale-$k$ discretization $\tilde{\mu}_k$ of a Borel measure $\tilde{\mu}$ on $N$ analogously to~\eqref{eq:discretized measure}.
Given $k$ and a set $E\subset \R^d$, we let
\begin{align*}
    \L_k(E) = \set{w\in \L_k: \Dcal_k(w)\cap E \neq \emptyset}.
\end{align*}
We define $\widetilde{\L}_k(E)$ analogously for $E\subset N$.
For a set $E\subseteq \R^d$, we define its scale-$k$ smoothing by
\begin{align}\label{eq:set discretization}
    E_k := \bigsqcup_{v\in \L_k(E)} \Dcal_k(v).
\end{align}
Discretizations of subsets of $N$ are defined analogously.
The following lemma relates the discretization of a measure $\tilde{\mu}$ on $N$ to the discretization of its projection on $\R^d$.
    \begin{lem}\label{lem:discretization and projections}
    Let $\tilde{\mu}$ and $\mu$ be as in Theorem~\ref{thm:quant ellq improvement} and $k\geq 0$.
        Let $E\subseteq \R^d$ be a Borel set
        and $F=\pi^{-1}(E_k)\cap B(\id,1)$. 
        Then,  $\tilde{\mu}_{k}(F_k)\geq \mu_{k}(E)\tilde{\mu}(B(\id,1)) $.
    \end{lem}
    \begin{proof}
    Note that for any measurable set $G$ and measure $\nu$, we have the following properties by the definition of discretizations: $G\subseteq G_k$ and $\nu(G)\leq \nu(G_k)=\nu_k(G_k)$. 
    Let $C= \tilde{\mu}(B(\id,1))$.
    Then, applying the previous observation to $\mu$ and $\tilde{\mu}$, we get
    \begin{align*}
        \mu_k(E) \leq \mu(E_k) = C^{-1}\tilde{\mu}(F)\leq C^{-1}\tilde{\mu}(F_k) = C^{-1}\tilde{\mu}_k(F_k).
    \end{align*}
    \qedhere
    \end{proof}

    The next lemma shows that affine non-concentration passes to discretizations.
    In what follows, in light of~\eqref{eq:scaling equivariance under pi}, we may and will assume that the constant $c\geq 1$ in~\eqref{eq:pullback metric} is chosen large enough so that the following diameter bound holds:
\begin{align}\label{eq:diameter c}
    \diam{\Dcal_j(v)} \leq c2^{-j}, \qquad \forall j\in\N, v\in \L_j \cup \widetilde{\L}_j.
\end{align}

\begin{lem}\label{lem:non-conc of discretization}

    Let $\tilde{\mu}$ be as in Theorem~\ref{thm:quant ellq improvement}.
    Let $\th\in (0,1)$ and a sufficiently large natural number $r\geq C(\th)$ be given.
    Then, for all $\e\geq 2^{-r}$ and sufficiently large $k\in\N$, the scale-$kr$ discretized measure $\mu_{kr}$ 
    is affinely non-concentrated in the following sense.

    Let $\Ecal=\Ecal(k,\e,\th,r)\subset N$ denote the exceptional set for $\tilde{\mu}$ provided by Definition~\ref{def:aff non-conc}.
    \begin{enumerate}
         \item There is an exceptional set $\Edisc = \Edisc(k,\e,\th,r)$ with $\tilde{\mu}_{kr}(\Edisc)\leq 2C(\th) 2^{-\l(\th) k}\tilde{\mu}(B(\id,2))$.
        
        \item \label{item:NUANCD N(w)}
        For every $w\in  \widetilde{\L}_{kr}(B(\id,1))\cap\mrm{supp}(\tilde{\mu}_{kr})\setminus \Edisc$, there is a set of good scales $\Ncal(w)\subseteq [0,k]$ with
         $\# \Ncal(w) \geq (1-\th) k-O(1)$.
         Moreover, there is $x\in \Dcal_{kr}(w) \cap \supp(\tilde{\mu})\setminus \Ecal$ such that $\Ncal(w)= \Ncal(x)\cap [0,k-1]$, where $\Ncal(x)$ is as in Definition~\ref{def:aff non-conc}.

        \item For every $w\in \widetilde{\L}_{kr}(B(\id,1))\cap\mrm{supp}(\tilde{\mu}_{kr})\setminus \Edisc$, every affine subspace $W<N$ and every $\ell\in\Ncal(w)$, setting $\rho_\ell = 2^{-r\ell}$, we have
        \begin{align}\label{eq:discretized nonconc}
            \tilde{\mu}_{kr} (W^{(\e \rho_\ell)} \cap \Dcal_{r\ell}(w)))
            \leq (\vp(\th)+ C(\th) \vp(2\e/c)) \tilde{\mu}(B(w,3c\rho_\ell)),
        \end{align}
        where $c\geq 1$ is as in~\eqref{eq:diameter c}.
    \end{enumerate}
\end{lem}

    \begin{proof}
       
        Define
        $
             \Edisc= \set{w\in \widetilde{\L}_{kr}(B(\id,1)): \tilde{\mu}(\Dcal_{kr}(w)\cap \Ecal) > \tilde{\mu}(\Dcal_{kr}(w))/2 }$.
        Then, we get 
        \begin{align*}
            \tilde{\mu}_{kr}(\Edisc) 
            = \sum_{w \in \Edisc} \tilde{\mu}(\Dcal_{kr}(w))
            < 2\sum_{w \in \Edisc} \tilde{\mu}(\Dcal_{kr}(w)\cap \Ecal) \leq 2\tilde{\mu}(\Ecal)
            \leq 2C(\th) 2^{-\l(\th)k} \tilde{\mu}(B(\id,2)).
        \end{align*}
        For each $w\in \widetilde{\L}(B(\id,1))\cap\supp(\tilde{\mu}_{kr})\setminus \Edisc$, fix an arbitrary $x\in \supp(\tilde{\mu})\cap\Dcal_{kr}(w)\setminus \Ecal$.
        For each such $w$, since $\Dcal_{kr}(w)$ intersects $B(\id,1)$, $\Dcal_{kr}(w)$ is contained in $B(\id,2)$ whenever $k$ is large enough. 
        In particular, $x\in B(\id,2)$.
        Set
        \begin{align*}
            \Ncal(w) = \Ncal(x)\cap [0,k-2].
        \end{align*}
        Then, $\#\Ncal(w)\geq (1-\th)k-2$.
        For $w$ and $x$ as above, let $\ell\in \Ncal(w)$ and set $\rho_\ell = 2^{-r\ell}$.
        Then, given any proper affine subspace $W\subset N$, we have
        \begin{align*}
           \tilde{\mu}_{kr}(W^{(\e \rho_\ell)} \cap \Dcal_{r\ell}(w)))
            =\sum_{v \in \widetilde{\L}_{kr} \cap W^{(\e \rho_\ell)} \cap \Dcal_{r\ell}(w)} \tilde{\mu}(\Dcal_{kr}(v)).
        \end{align*}
        
        Next, by~\eqref{eq:diameter c}, the cell $\Dcal_{kr}(v)$ has diameter $\leq c2^{-rk}$.
        In particular, if $r$ is large enough relative to $c$, since $\e\geq 2^{-r}$ and $\ell\leq k-2$, we have that $\Dcal_{kr}(v)$ is contained inside $W^{(2\e \rho_\ell)}$ for all $v\in W^{(\e \rho_\ell)}$.
        Similarly, we have that $\Dcal_{r\ell}(w) $ is contained inside $B(x, 2c\rho_\ell)$.
        It follows that 
        \begin{align*}
            v \in \widetilde{\L}_{kr} \cap W^{(\e \rho_\ell)} \cap \Dcal_{r\ell}(w)
            \Longrightarrow
            \Dcal_{kr}(v) \subset W^{(2\e \rho_\ell)} \cap B(x,2c\rho_\ell).
        \end{align*}
        We thus get that
        \begin{align*}
            \tilde{\mu}_{kr}(W^{(\e \rho_\ell)} \cap \Dcal_{r\ell}(w)))
            \leq \tilde{\mu}(W^{(2\e \rho_\ell)} \cap B(x,2c\rho_\ell)).
        \end{align*}
        Hence, since $\tilde{\mu}$ is affinely non-concentrated, and $x\in B(\id,2)$ and $\ell\in \Ncal(x)$, we obtain
        \begin{align*}
            \tilde{\mu}_{kr}(W^{(\e \rho_\ell)} \cap \Dcal_{r\ell}(w)))
            \leq (\vp(\th)+ C(\th) \vp(2\e/c)) \tilde{\mu}(B(x,2c\rho_\ell)).
        \end{align*}
        
        Finally, we observe that since $x\in \Dcal_{kr}(w)$, the ball $B(x,2c\rho_\ell)$ is contained in $B(w,3c\rho_\ell)$.
        Together with the above estimate, this yields~\eqref{eq:discretized nonconc} and concludes the proof.
        \qedhere
    \end{proof}

    We end this section with the following useful lemma regarding intersection multiplicities.
    \begin{lem}\label{lem:intersection multiplicity on N}
        Let $C\geq 1$ be given. Then, for all $\ell\in\N$, the balls $\set{B(v,C\rho_\ell): v\in \widetilde{\L}_\ell}$ have intersection multiplicity $O_{N,C}(1)$.
    \end{lem}
    \begin{proof}
        Fix some $x\in N$ and let $E(x)$ denote the set of $v\in\widetilde{\L}_\ell$ with $x\in B(v,C\rho_\ell)$.
        Then, $E(x)\subseteq \widetilde{\L}_\ell\cap B(v_0,2C\rho_\ell)$, for any fixed $v_0\in E(x)$. 
        By right-invariance of the metric, $\# E(x)$ is at most $\# \widetilde{\L}_\ell \cap B(\id,2C\rho_\ell)$. 
        Applying the scaling automorphism $g_{\ell\log 2}$ to the latter set, and using~\eqref{eq:scaling equivariance under pi}, we conclude that $\#E(x)\leq \#\widetilde{\L}\cap B(\id, 2C)$. 
        The lemma follows by discreteness of the lattice $\widetilde{\L}$.
    \end{proof}
    
\subsection{Asymmetric Balog-Szemer\'edi-Gowers Lemma}

The following is the asymmetric version of the Balog-Szemr\'edi-Gowers Lemma due to Tao and Vu, which is the first key ingredient in the proof of Theorem~\ref{thm:quant ellq improvement}.
For a finite set $A\subset \R^d$, $|A|$ denotes its cardinality.

\begin{thm}[Corollary 2.36,~\cite{TaoVu-book}]
\label{thm:asymmetric Balog-Szemeredi-Gowers}
    Let $A,B\subset \R^d$ be finite sets such that $\norm{1_A\ast 1_B}^2_{2}\geq 2\a |A| |B|^2$ and $|A| \leq L |B|$ for some $0<\a \leq 1$ and $L\geq 1$.
    Let $\e'>0$ be given.
    Then, there exist sets $A'\subseteq A$ and $B'\subseteq B$ such that
    \begin{enumerate}
        \item{$A'$ and $B'$ are sufficiently dense:} $|A'| \gg_{\e'} \a^{O_{\e'}(1)} L^{-\e'} |A|$ and $|B'|\gg_{\e'} \a^{O_{\e'}(1)} L^{-\e'}|B|$.
        \item{$A'$ is approximately invariant by $B'$:} $|A' + B'| \ll_{\e'} \a^{-O_{\e'}(1)} L^{\e'} |A'|$. 
    \end{enumerate}
\end{thm}

\begin{remark}
    The quoted result is stated in terms of the additive energy $E(A,B)$ in \textit{loc.~cit.}, which is nothing but $\norm{1_A\ast 1_B}_{2}^2$.
\end{remark}

\subsection{Hochman's inverse theorem for entropy}

In order to be able to bring our affine non-concentration hypothesis into play, we will need to convert the approximate additive invariance provided by the Balog-Szemer\'edi-Gowers Lemma into exact additive obstructions to flattening under convolution, i.e.~affine subspaces.
Our key tool for this step is Hochman's inverse entropy theorem for convolutions of measures
We need some notation before stating the result.
    
    For a Borel probability measure $\nu$ on $\R^d$, the entropy $H_k(\nu)$ of $\nu$ at scale $k$ is defined to be
    \begin{align*}
        H_k(\nu) := -\frac{1}{k}\sum_{P\in\Dcal_k} \nu(P) \log_2 \nu(P).
    \end{align*}
    By concavity of $\log$ and Jensen's inequality, we have the following elementary inequality
    \begin{align}\label{eq:entropy and covers}
        H_k(\nu) \leq \frac{\log_2 \#\set{P\in\Dcal_k: \nu(P)\neq 0}}{k}.
    \end{align}
    It also follows from Jensen's inequality that the above inequality becomes equality if and only if $\nu$ gives equal weights to the elements $P$ of $\Dcal_k$ with $\nu(P)\neq 0$.

     Given a Borel probability measure $\nu$ on $\R^d$ and $z\in \R^d$ with $\nu(\Dcal_k(z))>0$, we define the component measure $\nu^{z,k}$ by
    \begin{align*}
        \int f\;d\nu^{z,k}:= \frac{1}{\nu(\Dcal_k(z))} \int_{\Dcal_k(z)} f(T(y))\;d\nu(y),
    \end{align*}
    where $T:\Dcal_k(z)\to \Dcal_0(\mathbf{0})$ is the affine map given by composing scaling by $2^k$ with translation by the element of $\L_k$ sending $\Dcal_k(z)$ to $\Dcal_k(\mathbf{0})$.

    Given a Borel subset $\Pcal \subseteq \mrm{Prob}(\R^d)$ and $k\in\N$, we define
    \begin{align}\label{eq:probability notation}
        \mathbb{P}_{0\leq i \leq k} (\nu^{z,i} \in \Pcal)
        := \frac{1}{k+1} \sum_{i=0}^k \int 1_{\Pcal}(\nu^{z,i})\;d\nu(z).
    \end{align}

    Given a linear subspace $0\leq V \leq \R^d$, $\e>0$ and a probability measure $\nu$, we say that $\nu$ is $(V,\e)$-\textit{concentrated} if there is a translate $L$ of $V$ such that $\nu(L^{(\e)})>1-\e$.
    We say that $\nu$ is $(V,\e,m)$-\textit{saturated} for a given $m\in\N$ if
    \begin{align}\label{eq:saturated}
        H_m(\nu) \geq H_m(\pi_W\nu) + \mrm{dim} V -\e,
    \end{align}
    where $W = V^{\perp}$ and $\pi_W\nu$ is the pushforward of $\nu$ under the orthogonal projection to $W$.

\begin{thm}[Theorem 2.8,~\cite{Hochman-InverseEntropy}]
\label{thm:Hochman}
    For every $\e, R>0$ and $r\in \N$, there are $\s>0$ and $m_0,k_0\in \N$ such that for all $k\geq k_0$ and all Borel probability measures $\nu$ and $\mu$ on $[-R,R]^d$ satisfying
    \begin{align*}
        H_{kr}(\mu \ast \nu) < H_{kr}(\nu)+\s,
    \end{align*}
    there exists a sequence of subspaces $0 \leq V_0,\dots, V_k \leq\R^d$ such that
    \begin{align*}
        \mathbb{P}_{0\leq i \leq k}\left(
            \begin{array}{c}
                 \mu^{x,ir} \text{ is } (V_i,\e)-\text{concentrated and}
                 \\
               \nu^{x,ir} \text{ is } (V_i,\e,m_0)-\text{saturated}
            \end{array}
         \right) >1-\e.
    \end{align*}
\end{thm}

    \begin{remark}
        Theorem~\ref{thm:Hochman} is stated in~\cite{Hochman-InverseEntropy} in the case $r=1$. However, the extension to general $r$ is rather routine since it roughly corresponds to working in base $2^r$ in place of base $2$.
    \end{remark}

 \subsection{Flattening of discretized measures}

    The following quantitative result is the main ingredient in the proof of Theorem~\ref{thm:quant ellq improvement}.

\begin{prop}\label{prop:discretized flattening}
        Let positive functions $\l,\vp,$ and $C$ on $(0,1]$ be given.
        Then, for every $0<\g <1$, there exist $\eta>0$, and $r\in \N$, depending on $\g, \l,C,$ and $\vp$, such that the following holds. 
        
        For every $n\geq 0$, there exists $k_1=k_1(\l,\vp,C,\g)\in\N$, so that the following hold for all $k\geq k_1$ and
        any probability measure supported on $2^{-kr}\Z^d \cap B(0,2^n)$ and satisfying
        \begin{align}\label{eq:large L2}
         \norm{\nu}^2_2>2^{-(1-\g)d kr}.  
        \end{align}
        Let $\tilde{\mu}$ be a $(\l,\vp,C)$-ANC Borel measure on $N$.  
        Let $\mu$ be the projection to $\R^d$ of $\tilde{\mu}\left|_{B(\id,2)}\right.$ under $\pi$, normalized to be a probability measure.
        Then, 
        \begin{align}\label{eq:L2 improvement}
            \norm{\mu_{kr} \ast \nu}_2
            \leq 2^{-\eta kr} \norm{\nu}_2.
        \end{align}    
    \end{prop}

    This proposition says that the convolution of an arbitrary measure $\nu$ with a non-concentrated measure causes $\nu$ to ``spread out", i.e.~leads to a quantitative reduction in the $\ell^2$ norm of $\nu$, unless $\norm{\nu}_2$ is already very close to $0$.

    \subsubsection{From measures to sets}
    
    The remainder of this subsection is dedicated to the proof of Proposition~\ref{prop:discretized flattening}.
    Let $\g>0$ and $\eta>0$ be small parameters and $r,k\in \N$ be large integers to be specified over the course of the proof.
    We frequently assume that $\g$ is sufficiently small so that various properties hold and the values of $\eta,r$ and $k$ will depend only on $\g$ and the non-concentration parameters.
    Suppose towards a contradiction that~\eqref{eq:large L2} holds but~\eqref{eq:L2 improvement} fails.

    We first translate the failure of~\eqref{eq:L2 improvement} from measures to indicator functions of certain sets using standard arguments. 
    This allows us to apply the Balog-Szemer\'edi-Gowers Lemma.

    \begin{lem}[Lemma 3.3,~\cite{Shmerkin-Furstenberg}]
    \label{lem:set expansion to measure expansion}
        For every $\eta>0$ and $n\geq 0$, the following holds for all large enough $\ell$. Let $\mu$ and $\nu$ be probability measures such that $
        \supp(\mu)\subseteq \L_\ell\cap [0,1)^d$ and $\supp(\nu)\subseteq \L_\ell\cap B(0,2^n)$.
        Assume that $\norm{\mu\ast \nu}_2$ is at least $2^{-\eta \ell}\norm{\nu}_2$.
        Then, there exist $j,j'\leq 4\eta \ell$ such that 
        \begin{align}\label{eq:bulk of nu}
            A &:= \set{x\in \L_\ell: 2^{-j-1} \norm{\nu}_2^2 < \nu(x) \leq 2^{-j}\norm{\nu}_2^2},\\
            B &:= \set{x\in\L_\ell: 2^{-j'-1-d\ell} < \mu(x)\leq 2^{-j'-d\ell}}
        \end{align}
        satisfy
        \begin{enumerate}
            \item $\norm{1_A\ast 1_B}^2_2 \geq 2^{-4\eta \ell} |A| |B|^2$,
            \item $\norm{\nu|_A}_2 \geq 2^{-2\eta \ell} \norm{\nu}_2$, and
            \item $\mu(B)\geq 2^{-2\eta \ell}$.
        \end{enumerate}
        In particular, there exists a subset $A_0\subseteq A$ such that 
        \begin{enumerate}
            \item $A_0$ is contained in $w+[0,1)^d$, for some $w\in \Z^d$.

            \item $\norm{1_{A_0}\ast 1_B}^2_2 \geq 2^{-5\eta \ell} |A| |B|^2$
        \end{enumerate}
    \end{lem}
    \begin{proof}
        The properties of $A$ and $B$ were proved in~\cite{Shmerkin-Furstenberg} for measures on $\R$ and for $n=0$, however the short argument, based on the pigeonhole principle, goes through for general $d$ and $n$ with minimal modifications.

    To find $A_0$ with the claimed properties, let $A=\sqcup_{w\in \Z^d} A_w$ be the partition of $A$ defined by $A_w = A\cap w+[0,1)^d$. 
        Since $A\subset B(0,2^n)$, we have that $\#\set{w: A_w\neq \emptyset}\ll  2^{dn}$.
        By linearity of convolution, the triangle inequality and Cauchy-Shwarz, we get
        \begin{align*}
            2^{-4\eta \ell} |A| |B|^2 \leq \norm{1_A\ast 1_B}_2^2 \leq \left(\sum_w \norm{1_{A_w}\ast 1_B}_2 \right)^2
            \ll 2^{dn} \sum_w \norm{1_{A_w}\ast 1_B}_2^2.
        \end{align*}
        The corollary follows for all $\ell$ large enough, depending on $d,n, $ and $\eta$, by taking $A_0=A_w$, for $w$ such that $\norm{1_{A_w}\ast 1_B}_2^2$ is maximal.
        \qedhere
    \end{proof}

    \subsubsection{From $\ell^2$-concentration to entropy concentration}
   Let $A_0\subseteq A$ and $B$ be as in Lemma~\ref{lem:set expansion to measure expansion}, applied with $\ell=kr$ and $\mu=\mu_{kr}$.
    Taking $\eta$ small enough, we get by~\eqref{eq:large L2}, the definition of $A$, and Chebyshev's inequality that
    \begin{equation}\label{eq:|A|}
       |A_0|\leq |A| \leq 2^{4\eta kr+1 +(1-\g)dkr} \leq 2^{(1-\g/2)dkr+1}.
    \end{equation}
    We now apply Theorem~\ref{thm:asymmetric Balog-Szemeredi-Gowers} with $A=A_0$, $\a=2^{-5\eta kr-1}$, $L=\max\set{1,|A_0|/|B|}$, and
    \begin{align*}
        0<\e'<1
    \end{align*}
    a parameter to be chosen small enough depending on $\e$.
    Let $A'\subseteq A_0$ and $B'\subseteq B$ be the sets provided by Theorem~\ref{thm:asymmetric Balog-Szemeredi-Gowers}.
    
    Let $\nu'$ and $\mu'$ be the uniform probability measures supported on $A'$ and $B'$ respectively.
    Combining the above estimate with~\eqref{eq:entropy and covers}, we obtain
    \begin{align*}
        H_{kr}(\mu'\ast\nu') 
        \leq \frac{\log_2 |A'+B'|}{kr} 
        \leq \frac{\log_2 |A'|}{kr} +  O_{\e'}(\eta)+\log_2 L^{\e'}/kr.
    \end{align*}
    Since $\nu'$ is the uniform measure on $A'$, the remark following~\eqref{eq:entropy and covers} thus implies that
    \begin{align*}
        H_{kr}(\mu'\ast\nu') \leq H_{kr}(\nu') + O_{\e'}(\eta)+\log_2 L^{\e'}/kr.
    \end{align*}
    By~\eqref{eq:|A|}, we have $\log_2 L^{\e'}\leq \e'\log_2 |A|\leq \e'((1-\g/2)d kr+1)$.

    Recall from Lemma~\ref{lem:set expansion to measure expansion} $A_0$, and hence the support of $\nu'$, is contained in a box of the form $w+[0,1)^d$.
    Moreover, the above inequality remains unchanged by translating $\nu'$. Hence, for the purposes of applying Theorem~\ref{thm:Hochman}, we may without loss of generality assume in the sequel that
    \begin{align*}
        \supp(\nu')\subset [0,1)^d.
    \end{align*}

    Let $R=O(1)$ be such that $[0,1)^d$ is contained in the $R$-ball around the origin.
    Let $\s>0$ and $k_0\in\N$ be the parameters provided by Theorem~\ref{thm:Hochman} applied with this $R$ and
    \begin{align*}
        \e = 2^{-r}.
    \end{align*}
    We shall assume that $k$ is chosen to be larger than $k_0$.
    Hence, taking $\e'$ small enough (depending on $\s$) and $\eta$ small enough (depending on $\e'$ and $\s$), we obtain
    \begin{align}\label{eq:entropy didn't improve}
        H_{kr}(\mu'\ast\nu') < H_{kr}(\nu') + \s.
    \end{align}

    We show that the conclusion of Theorem~\ref{thm:Hochman} is incompatible with the non-concentration properties of the measure $\mu$.
    Let $V_0,\dots,V_k$ be the subspaces provided by Theorem~\ref{thm:Hochman}
    and
    \begin{align*}
        \Scal = \set{0\leq i \leq k:  V_i = \R^d}.
    \end{align*}
    We begin by showing that a significant proportion of the $V_i's$ are proper subspaces.
    Intuitively, being $\R^d$-saturated on most scales means the measure $\nu$ is close to being absolutely continuous to Lebesgue on $\R^d$ in the sense that its $\ell^2$-norm would be very close to $2^{-dk}$. This would contradict~\eqref{eq:large L2}.

    \begin{lem}~\label{lem:count saturation indices}
        If $\e$ is chosen small enough and $k$ large enough depending on $\g$, then $\# \Scal< (1-\g/10) k$.
    \end{lem}
    
    \begin{proof}
    Let $\g_1=\g/10$ and suppose that $\# \Scal \geq (1-\g_1) k$.
    Then, Theorem~\ref{thm:Hochman} and the definition of saturation (cf.~\eqref{eq:saturated}) imply that
    \begin{align*}
        \frac{1}{k+1} \sum_{i=0}^k \int H_{m_0}((\nu')^{z,ir}) \;d\nu'(z) \geq (1-\g_1) (1-\e)(d -\e) = (1-\g_1)d-O(\e).
    \end{align*}
    By~\cite[Lemma 3.4]{Hochman-Annals}\footnote{The cited result is stated for step-size $r=1$, however its short proof extends to work for any $r$ with minor changes.}, this yields the following estimate on $H_{kr}(\nu')$: 
    \begin{align*}
        H_{kr}(\nu') \geq (1-\g_1)d  - O(\e)- O_r\left( \frac{m_0}{k}\right) 
        \geq (1-\g_1)d - O(\e),
    \end{align*}
    where the second inequality holds whenever $k$ is large enough depending on $r$ and $m_0$.
    Moreover, by the remark following~\eqref{eq:entropy and covers}, we have
    $H_{kr}(\nu')= \log_2 |A'|/kr\leq \log_2 |A|/kr$.
    Hence, we obtain that
    $|A| \geq 2^{( (1-\g_1)d-O(\e))kr}$.
    This contradicts~\eqref{eq:|A|} when $\e$ is small enough compared to $\g$.
    \qedhere 
    \end{proof}

    \subsubsection{Concentration of large sets at many scales}
    Roughly speaking, our strategy is as follows. Armed with Lemma~\ref{lem:count saturation indices}, we show that the concentration provided by Theorem~\ref{thm:Hochman} holds on a set of relatively large measure and on a definite proporrtion of scales.
    On the other hand, the non-concentration property of $\mu$ and induction on scales shows that such set must have very small measure, yielding a contradiction.

    Recall that $\pi:N\r  \R^d$ is our fixed surjective homomorphism.
    Let $0<\g_2<\g/40$ be a small parameter to be chosen depending only on $\g$.
    Let $\Edisc$ be the exceptional set provided by Lemma~\ref{lem:non-conc of discretization} for our choices of $k,r$, and with $\th=\g_2$
    and $2c\e$ in place of $\e$, where $c$ is the constant in~\eqref{eq:pullback metric}. 
    We set $\g_3 = \g/40 - \g_2$.
    We use the notation
    \begin{align*}
        \rho_i= 2^{-ir}.
    \end{align*}

    \begin{lem}\label{lem:build concentrated set}
   Suppose $\e$ is small enough depending on $\g$, $\e'$ is small enough depending on $\e$, $\eta$ is small enough depending on $\e'$, and $k$ is large enough depending on all the previous parameters.
    
        Then, there exist a set $F\subseteq \widetilde{\L}_{kr}(B(\id,1))\subset N$ and a $2$-separated set of scales $\ell_1 < \ell_2 < \cdots < \ell_m$, where $m=\lceil \g_3 k\rceil $ such that the following hold for every $1\leq i \leq m$:
        
        \begin{enumerate}
        \item\label{item:lower bound on F} $\tilde{\mu}_{kr}(F)\geq 2^{-\sqrt{\e'}kr -k-1}$.
        \item\label{item:F is near V} For every $w\in \widetilde{\L}_{r\ell_i}(F)$, there exists an affine subspace $\widetilde{V}_w$ such that $
            F\cap \Dcal_{r\ell_i}(w)
            \subseteq 
            \widetilde{V}_w^{(c\e \rho_{\ell_i})}$,
        where $c\geq 1$ is as in~\eqref{eq:pullback metric}.

         \item\label{item:V's are proper} $\widetilde{V}_w$ is a proper affine subspace for every $w\in \widetilde{\L}_{r\ell_i}(F)$.
        \item\label{item:E outside of Ecal} $F$ is disjoint from the exceptional set for non-concentration, i.e. $F\cap \Edisc=\emptyset$.
        \item\label{item:varpi are good for non-conc} $\ell_i$ is a good scale for non-concentration at every point in $F$, i.e.~$\ell_i\in \Ncal(x)$ for all $x\in F$.
    \end{enumerate}
    \end{lem}

    \begin{remark}
        The proof of Lemma~\ref{lem:build concentrated set} in fact shows that for each fixed scale $\ell_i$, the projection of the spaces $\set{\widetilde{V}_w:w\in \widetilde{\L}_{r\ell_i}(F) }$ to $\R^d$ are all parallel to one another.
    \end{remark}

    We begin by deriving a lower bound on the measure of $B'$ with respect to our original discretized measure $\mu_{kr}$ (not $\mu'$).
    Recall the parameter $\e'$ chosen above~\eqref{eq:entropy didn't improve}. 
    \begin{lem}\label{lem:lower bound on B'}
       
        If $\eta$ is chosen sufficiently small depending on $\e'$, then for all sufficiently large $k$, 
        \begin{align*}
            \mu_{kr}(B') \geq 2^{-2d\e' kr}.
        \end{align*}
    \end{lem}
    \begin{proof}
        Recall that the set $B$ was defined in~\eqref{eq:bulk of nu} and $B'\subseteq B$ is provided by Theorem~\ref{thm:asymmetric Balog-Szemeredi-Gowers} with $\a=2^{-4\eta kr-1}$ and $L=\max\set{1,|A|/|B|}$.
        We calculate using Lemma~\ref{lem:set expansion to measure expansion} and Theorem~\ref{thm:asymmetric Balog-Szemeredi-Gowers}:
        \begin{align*}
            \mu_{kr}(B') = \sum_{u\in B'}\mu(u) \geq 2^{-j'-d kr-1}|B'|
            \gg_{\e'} 2^{-j'-dkr}2^{-O_{\e'}(\eta kr)}   L^{-\e'} |B|
            \geq 2^{-j'-dkr} 2^{-O_{\e'}(\eta kr)} |B| |A|^{-\e'}.
        \end{align*}
        By~\eqref{eq:|A|}, we have that $|A|^{\e'}\ll  2^{d\e' kr}$.
        Moreover, Lemma~\ref{lem:set expansion to measure expansion} implies that
        \begin{align*}
            2^{-2\eta kr}\leq \mu_{kr}(B) \leq 2^{-j'-dkr} |B|.
        \end{align*}
        The lemma then follows once $\eta$ is chosen sufficiently small depending on $\e'$.
    \end{proof}

    Next, we define the following set of scales where the concentration provided by Theorem~\ref{thm:Hochman} gives non-trivial information:
    \begin{align*}
        \Ccal := \set{0,\dots,k}\setminus\Scal =  \set{0\leq i\leq k :  V_i \lneq \R^d }.
    \end{align*}
    By Lemma~\ref{lem:count saturation indices}, we know that 
    \begin{align}\label{eq:cardinality of Ccal}
      |\Ccal| \geq \g_1 k, \qquad    \g_1=\g /10.
    \end{align}
    Our next goal is to transfer the concentration information provided by Theorem~\ref{thm:Hochman} for $\mu'$ to the measure $\mu$.
    To do so, we convert the probabilistic concentration provided in the theorem into geometric containment into subspace neighborhoods.
    
    Recall that $\set{\Dcal_{\ell } :\ell\in\N}$ is a refining sequence of dyadic partitions of $\R^d$ and $\L_\ell = 2^{-\ell}\Z^d$.
    For $i\in\Ccal$ and $w\in \L_{ir}$, 
    let $z\in \Dcal_{ir}(w)$ be such that for $V_w := V_i +z$, we have
    \begin{align}\label{eq:meaning of concentration}
         \mu'( V_w^{(\e\rho_i)} \cap \Dcal_{ir}(w))
        \geq (1-\e) \mu'(\Dcal_{ir}(w)).
    \end{align}
    If no such $z$ exists, we let $V_w = V_i+w$. 
    Denote by $Q_i$ the set of concentrated points at scale $ir$, i.e.,
    \begin{align*}
        Q_i = \bigcup_{w\in \L_{ir}} V_w^{(\e\rho_i)}\cap \Dcal_{ir}(w).
    \end{align*}
    For every $w\in \tilde{\L}_{ir}$, let $\widetilde{V}_{w}$ denote the affine space $\pi^{-1}(V_\pi(w))$, where $\pi:N\r \R^d$ is our fixed surjective homomorphism.
    For $x\in \R^d$, we set
    \begin{align*}
        \Ccal(x) = \set{i\in \Ccal: x\in Q_i}.
    \end{align*}
    In particular, for $x \in B'$, $\Ccal(x)$ consists of scales at which $x$ witnesses the concentration of $B'$.

    \begin{lem}
        \label{lem:everywhere concentration}
        If $\e$ is small enough and $k$ is large enough, depending on $\g$, then the subset
        \begin{align}\label{eq:B''}
            B'' = \set{x\in B': |\Ccal(x)| \geq |\Ccal|/2}
        \end{align}
        satisfies $
        \mu_{kr}(B'')\geq 2^{-3d\e' kr}$.
    \end{lem}
    
    \begin{proof}
        Let $E=B'\setminus B''$. First, we give an upper bound on the measure of $E$ with respect to $\mu'$.
        Let $Q_i^c = \R^d\setminus Q_i$.
        Then, the concentration provided by  Theorem~\ref{thm:Hochman} implies that
        \begin{align*}
            \mathbb{P}_{0\leq i\leq k} ( (\mu')^{x,ir} \text{ is } (V_i,\e)-\text{concentrated})
            >1-\e.
        \end{align*}
        To unpack the above inequality, let us denote by $\Theta_{ir}\subseteq \L_{ir}$ the subset consisting of those $w\in \L_{ir}$ for which~\eqref{eq:meaning of concentration} holds.
        For every $x\in \R^d$ and $1\leq i\leq k$, let $w(x)\in \L_{ir}$ be such that $\Dcal_{ir}(x) = w(x) + 2^{-ir}[0,1)^d$.
        In this notation, the above inequality reads
        \begin{align*}
            \sum_{0\leq i\leq k} 
            \mu'\left(x\in \R^d:  \mu'( V_{w(x)}^{(\e\rho_i)} \cap \Dcal_{ir}(x))
        \geq (1-\e) \mu'(\Dcal_{ir}(x)) \right)
            > (1-\e)(k+1)
        \end{align*}
        
        Note that if $\mu'( V_{w(x)}^{(\e\rho_i)} \cap \Dcal_{ir}(x))
        \geq (1-\e) \mu'(\Dcal_{ir}(x))$ holds for some $x$ and $i$, then the inequality holds for all $y\in \Dcal_{ir}(x)$ in place of $x$.
        Hence, we get
        \begin{align*}
        (1-\e)k &<
        \sum_{0\leq i\leq k}
        \mu'\left(x\in\R^d:  \mu'( V_{w(x)}^{(\e\rho_i)} \cap \Dcal_{ir}(x))
        \geq (1-\e) \mu'(\Dcal_{ir}(x)) \right)
        \nonumber\\
        &=\sum_{0\leq i\leq k}
        \sum_{w\in \Theta_{ir} } \mu'(\Dcal_{ir}(w))
        \nonumber\\
        &
        \leq (1-\e)^{-1} \sum_{0\leq i\leq k}
        \sum_{w\in \Theta_{ir}} \mu'(\Dcal_{ir}(w) \cap V_{w}^{(\e\rho_i)} )
        \leq (1-\e)^{-1} \sum_{0\leq i\leq k}
        \mu'(Q_i).
        \end{align*}
        It follows that $
            \int \sum_{i\in \Ccal} \mathbbm{1}_{Q_i^c}(x) \;d\mu'(x) <2\e k$.
        On the other hand, we have by~\eqref{eq:cardinality of Ccal} that
        \begin{align*}
            \int \sum_{i\in \Ccal} \mathbbm{1}_{Q_i^c}(x) \;d\mu'(x)
            \geq \int_E \sum_{i\in \Ccal} \mathbbm{1}_{Q_i^c}(x) \;d\mu'(x) 
            \geq |\Ccal| \mu'(E)/2 
            \geq \g k \mu'(E)/20.
        \end{align*}
        Recalling that $\mu'$ is the uniform measure on $B'$, we can assert that these inequalities imply that $|B''|\geq (1- 40\e/\g) |B'|$.
        Hence, the assertion of the lemma follows from Lemma~\ref{lem:set expansion to measure expansion} by the same argument as in the proof of Lemma~\ref{lem:lower bound on B'}.
    \end{proof}

    \begin{table}[ht]
        \centering
        \begin{tabular}{c|c}
            Parameter & Definition\\
            \hline
            
            $\e$  &    $2^{-r}$ \\
            $\rho_i$ & $2^{-ir}$ \\
           $\g_1$  & $\g/10$  \\
             $\g_2$  & small parameter depending on $\g$\\
            $\g_3$ & $\g_1/4-\g_2$\\
           $\e'$ & small parameter depending on $\e$ and $\g$ \\
           $\eta$ & small parameter depending on $\e'$  \\
           $m$ & $\lceil \g_3k\rceil$ \\
           \hline
        \end{tabular}
        \vspace{10pt}
        \caption{Summary of parameters chosen in the proof of Proposition~\ref{prop:discretized flattening}.}
        \label{tab:my_label}
    \end{table}

    \subsubsection{Lifting concentration to $N$ and proof of Lemma~\ref{lem:build concentrated set}}

    Note that the scales $\Ccal(x)$ may vary with $x$.
    Similarly, the scales at which our affine non-concentration hypothesis holds also vary from point to point.
    To arrive at a contradiction, we partition $B''$ into sets where there is a fixed subset of scales of $\Ccal$ at which the aforementioned phenomena hold simultaneously and find an upper bound on the measure of each piece separately.

    Let $0<\g_2<\g_1/4$ be a small parameter to be chosen depending only on $\g$.
    Let $\Edisc$ be the exceptional set provided by Lemma~\ref{lem:non-conc of discretization} for our choices of $k,r$, and with $\th=\g_2$
    and $2\e$ in place of $\e$. 
    By taking $r\geq C(\g_2)$ large enough, then Lemma~\ref{lem:non-conc of discretization} implies that $\tilde{\mu}_{kr}(\Edisc)\leq 2C(\g_2)2^{-\l(\g_2)k}$.
    Recall the definition of smoothed scale-$k$ sets in~\eqref{eq:set discretization} and let
    Let
    \begin{align*}
        B''' = \widetilde{\L}_{kr}(E) \setminus \Edisc \subseteq \widetilde{\L}_{kr}(B(\id,1)),
        \qquad
        \text{where }
        E = \pi^{-1}(B''_{kr}) \cap B(\id,1).
    \end{align*}
    Then, by Lemma~\ref{lem:discretization and projections}, taking $\e'$ small enough depending on $r$ and $\l(\g_2)$, we can ensure that 
    \begin{align}\label{eq:lower bound B'''}
        \tilde{\mu}_{kr}(B''')
        \geq 2^{-3d\e' kr} - 2C(\g_2)2^{-\l(\g_2)k} \tilde{\mu}(B(\id,2)) 
        \geq 2^{-k\sqrt{\e'} },
    \end{align}
    for all large enough $k$.
    Recall the sets of good scales $\Ncal(\cdot)$ provided by Lemma~\ref{lem:non-conc of discretization}.
    By a slight abuse of notation, for $x\in B'''$, we let 
    \begin{align*}
        \Ccal(x) = \Ccal(\pi(x)).
    \end{align*}
    
    We define
    \begin{align*}
        \Gcal(x) = \text{maximal } 2\text{-separated subset of } \Ccal(x)\cap \Ncal(x).
    \end{align*}
    By~\eqref{eq:cardinality of Ccal} and the definition of $B''$ in~\eqref{eq:B''}, setting $\g_3=\g_1/4-\g_2$, we also have
    \begin{align*}
        |\Gcal(x)| \geq ((\g_1/2-\g_2)k-2)/2\geq \g_3 k, \qquad
        \forall x\in B''',
    \end{align*} 
    where the second inequality holds whenever $k$ is large enough.

    Given $\varpi\subseteq \set{0,\dots, k}$, we let
    \begin{align*}
        B'''_\varpi := \set{x\in B''': \varpi\subseteq \Gcal(x)}.
    \end{align*}
    Then, the sets $\set{B'''_\varpi: |\varpi|= \lceil \g_3 k\rceil}$ provide a cover of $B'''$. Hence, we have that
    \begin{align}\label{eq:partition B'''}
        \tilde{\mu}_{kr}(B''') \leq  \sum_{ |\varpi| = \lceil \g_3 k\rceil} \tilde{\mu}_{kr}(B'''_\varpi).
    \end{align}    
    Fix a set $\varpi \subset [0,k] \cap \N$ for which $ \tilde{\mu}_{kr}(B'''_\varpi)$ is maximal and let $F=B'''_\varpi$.
    Since the sum in~\eqref{eq:partition B'''} has most $2^{k+1}$ terms,~\eqref{eq:lower bound B'''} implies that $\tilde{\mu}_{kr}(F) \geq 2^{-\sqrt{\e'}k-k-1}$.

     It remains to prove that the set of scales given by $\varpi$ satisfy items~\eqref{item:F is near V} and~\eqref{item:V's are proper} of the lemma.
    We need the following observation regarding compatibility of dyadic partitions under our projection.
    \begin{lem}\label{sublem:compatible partitions}
       Let $ \ell\geq 0$, $\tilde{w}\in \widetilde{\L}_{\ell}(F)$ and $w=\pi(\tilde{w})$.
       Then, $\pi(F\cap \Dcal_{\ell}(\tilde{w}))\subseteq B''\cap \Dcal_{\ell}(w)$.
    \end{lem}
    \begin{proof}
    For all $\ell$, we have $\Dcal_{\ell}(\tilde{w})= g_{-\ell\log 2}(D)\cdot \tilde{w}$ by definition, and hence by~\eqref{eq:scaling equivariance under pi}, we get
        \begin{align*}
            \pi(\Dcal_{\ell}(\tilde{w})) = 2^{-\ell} \pi(D) + w = \Dcal_\ell(w).
        \end{align*}
    Next, we show that $\pi(F)\subseteq B''$.
        Let $U=B(\id,1)$ and fix some $x\in F$.
        Then, $x\in \widetilde{\L}_{kr}(\pi^{-1}(B''_{kr})\cap U)$.
        Hence, the intersection $\Dcal_{kr}(x)\cap \pi^{-1}(B''_{kr})\cap U$ is non-empty.
        Let $y$ be a point in this intersection, so that $y\in B''_{kr}=\sqcup_{v\in B''} \Dcal_{kr}(v)$, where we used that $B''$ is a subset of $\L_{kr}$.
        Letting $v\in B''$ be such that $y\in\Dcal_{kr}(v)$, it follows that $\pi(\Dcal_{kr}(x)) $ intersects $\Dcal_{kr}(v)$.
        On the other hand, we have shown that $\pi(\Dcal_{kr}(x)) =\Dcal_{kr}(\pi(x)) $. Since $\pi(x)\in \L_{kr}$, it follows that $\pi(x)=v$, concluding the proof.
    \end{proof}
    Let $\ell\in \varpi$ and $\tilde{w}\in \widetilde{\L}_{r\ell}(F)$.
    Let $x\in F\cap \Dcal_{r\ell}(\tilde{w})$ and $w=\pi(\tilde{w})$.
    Then, $\pi(x)\in \Dcal_{r\ell}(w)\cap B''$ by Lemma~\ref{sublem:compatible partitions}, and $\ell\in \Ccal(\pi(x))$ by definition.
    Hence, by definition of $B''$ in~\eqref{eq:B''}, we have that $\pi(x)\in V_w^{(\e\rho_\ell)}$.
    It follows by~\eqref{eq:pullback metric} that $x\in \widetilde{V}_{\tilde{w}}^{(c\e\rho_\ell)}$.
    Moreover, since $\ell\in \Ccal(\pi(x))$, we have that $V_w$ is a proper subspace, and hence so is $\widetilde{V}_{\tilde{w}}$.
    This completes the proof of Lemma~\ref{lem:build concentrated set}.

    \subsubsection{ANC implies a contradiction to Lemma~\ref{lem:build concentrated set}}
    
     In this section, we complete the proof of Proposition~\ref{prop:discretized flattening} by showing that the ANC condition gives a contradiction to Lemma~\ref{lem:build concentrated set} via an induction on scales argument showing that the multiscale structure of $F$ given in the lemma implies that it has very small measure.

     Let $F$, $\ell_i$, and $\widetilde{V}_w$ be as in Lemma~\ref{lem:build concentrated set}.
     We recall that $\widetilde{\L}_\ell(F)$ denotes those elements $v\in \widetilde{\L}_\ell$ for which the corresponding cells $\Dcal_\ell(v)$ intersect $F$ non-trivially.

    As a first step, we have the following basic estimate that will allow us to proceed by induction on scales:
    \begin{align}\label{eq:from F to a cover of F}
    \tilde{\mu}_{kr}(F) 
    \leq \sum_{v\in \widetilde{\L}_{r\ell_m}(F)} \tilde{\mu}_{kr}(\Dcal_{r\ell_m}(v))
    = \sum_{w\in \widetilde{\L}_{r\ell_{m-1}}(F) }
     \sum_{ \substack{ v\in \widetilde{\L}_{r\ell_m}(F) \\ \Dcal_{r\ell_m}(v) \subset \Dcal_{r\ell_{m-1}}(w) } } \tilde{\mu}_{kr}(\Dcal_{r\ell_m}(v))
    .
\end{align}

    Recall by~\eqref{eq:diameter c} that the diameter of each element of $\Dcal_{n}$ is at most $c\rho_{n}$ for a fixed uniform constant $c\geq 1$.
    Moreover, since the $\ell_i$'s are $2$-separated and $\e = 2^{-r}$, we may assume that $r$ is large enough, depending on $c$ so that
    \begin{align}\label{eq:apply 2-sep}
        c\rho_{\ell_j}\leq \e\rho_{\ell_i}/10, \qquad \forall 1\leq i<j\leq m.
    \end{align}
    Hence, if $\widetilde{V}_w^{(c\e\rho_{\ell_i})}$ intersects a box $\Dcal_{r\ell_{i+1}}(v)$ non-trivially, then we have
    \begin{align}\label{eq:F is near V}
        \Dcal_{r\ell_{i+1}}(v) \subseteq \widetilde{V}_w^{(2c\e\rho_{\ell_i})}.
    \end{align}
    This containment, along with item~\eqref{item:F is near V} of Lemma~\ref{lem:build concentrated set}, imply that for every $1\leq i<m$ and $w\in \widetilde{\L}_{r\ell_i}(F)$, we have that
    \begin{align*}
        \sum_{ \substack{ v\in \widetilde{\L}_{r\ell_{i+1}}(F) \\ \Dcal_{r\ell_{i+1}}(v) \subset \Dcal_{r\ell_{i}}(w) } } \tilde{\mu}_{kr}(\Dcal_{r\ell_{i+1}}(v))
        \leq 
        \tilde{\mu}_{kr}(\widetilde{V}_w^{(2c\e\rho_{\ell_i})} \cap \Dcal_{r\ell_i}(w))
        .
    \end{align*}

    Recall that $\tilde{\mu}$ satisfies the ANC condition in Def.~\ref{def:aff non-conc}.
    Hence, for all $i$ and $w\in\widetilde{\L}_{r\ell_i}(F)$, items~\eqref{item:V's are proper},~\eqref{item:E outside of Ecal} and~\eqref{item:varpi are good for non-conc} of Lemma~\ref{lem:build concentrated set}, along with Lemma~\ref{lem:non-conc of discretization} imply that
    \begin{align}
        \tilde{\mu}_{kr}\left( \widetilde{V}_w^{(2c\e\rho_{\ell_i})} \cap \Dcal_{r\ell_i}(w) \right) \leq
       \d(\g_2,\e) \tilde{\mu}\left( B(w,3c\rho_{\ell_i}) \right),
    \end{align}
    where $\d(\g_2,\e)= \left(\vp(\g_2)+ C(\g_2)\vp(4\e)\right)$.
    Note that the above inequality has the discretized measure $\tilde{\mu}_{kr}$ on the left side and has the original measure $\tilde{\mu}$ on the right side.
    Applying this estimate with $i=m-1$ and combining it with~\eqref{eq:from F to a cover of F}, we obtain
    \begin{align}\label{eq:flat pre-induction}
        \tilde{\mu}_{kr}(F)
        &\leq \d(\g_2,\e)
        \sum_{w\in \widetilde{\L}_{r\ell_{m-1}}(F)} 
        \tilde{\mu}\left( B(w,3c\rho_{\ell_{m-1}})\right)
        .
    \end{align}
    Our next lemma will allow us to apply induction on the above estimate.

    \begin{lem}
    There exists a uniform constant $C_N\geq 1$, depending only on the metric on the nilpotent group $N$, so that for all $2\leq i\leq m$, we have
        \begin{align*}
            \sum_{w\in \widetilde{\L}_{r\ell_{i}}(F)} 
        \tilde{\mu}\left( B(w, 4c\rho_{\ell_{i}})\right)
        &\leq C_N \d(\g_2,\e)
        \sum_{v\in \widetilde{\L}_{r\ell_{i-1}}(F)} 
        \tilde{\mu}\left( B(v,4c\rho_{\ell_{i-1}})\right).
        \end{align*}
    \end{lem}
    \begin{proof}
        
    We begin by noting the following equality that relates the scale $\rho_{\ell_{i}}$ to the scale $\rho_{\ell_{i-1}}$:
    \begin{align*}
        \sum_{w\in \widetilde{\L}_{r\ell_{i}}(F)} 
        \tilde{\mu}\left( B(w,4c\rho_{\ell_{i}})\right)
        = 
        \sum_{v\in \widetilde{\L}_{r\ell_{i-1}}(F)}
        \sum_{ \substack{ w\in \widetilde{\L}_{r\ell_{i}}(F) \\ \Dcal_{r\ell_{i}}(w) \subset \Dcal_{r\ell_{i-1}}(v) } } 
        \tilde{\mu}\left( B(w,4 c\rho_{\ell_{i}})\right)
        .
    \end{align*}

    Let $w$ be such that $w\in \widetilde{\L}_{r\ell_{i}}(F)$ and  $ \Dcal_{r\ell_{i}}(w) \subset \Dcal_{r\ell_{i-1}}(v) $.
    Since $\Dcal_{r\ell_{i}}(w)$ meets $F$, we have that $w$ is at distance at most $ c\rho_{\ell_{i}}\leq \e \rho_{\ell_{i-1}}/10$ from a point in $F$.
    We also have that $F\cap \Dcal_{r\ell_{i-1}}(v)$ is contained inside the subspace neighborhood $\widetilde{V}_v^{(c\e \rho_{\ell_{i-1}})}$ by our choices above.
    Hence, by~\eqref{eq:apply 2-sep}, we obtain that the ball $B(w,4c\rho_{\ell_{i}})$ is contained inside $\widetilde{V}_v^{(2c\e \rho_{\ell_{i-1}})}$.
    Finally, similar considerations imply that $B(w,4 c\rho_{\ell_{i}})$ is contained in $B(v,2c\rho_{\ell_{i-1}})$.
    Put together, we arrive at the following inclusion:
    \begin{align*}
        \bigcup_{ \substack{ w\in \widetilde{\L}_{r\ell_{i}}(F) \\ \Dcal_{r\ell_{i}}(w) \subset \Dcal_{r\ell_{i-1}}(v) } } 
         B(w,4c\rho_{\ell_{i}})
         \subseteq 
         \widetilde{V}_v^{(2c\e \rho_{\ell_{i-1}})}
          \cap B(v,2c\rho_{\ell_{i-1}}) .
    \end{align*}

    On the other hand, by Lemma~\ref{lem:intersection multiplicity on N},  the intersection multiplicity of the balls on the left side of the above equation is uniformly bounded by $O_{N,c}(1)$.
    Letting $C_N$ denote this bound on multiplicity, we thus obtain
    \begin{align*}
        \sum_{ \substack{ w\in \widetilde{\L}_{r\ell_{i}}(F) \\ \Dcal_{r\ell_{i}}(w) \subset \Dcal_{r\ell_{i-1}}(v) } } 
        \tilde{\mu}\left( B(w,4 c\rho_{\ell_{i}})\right)
        \leq 
        C_N \tilde{\mu}\left( \widetilde{V}_v^{(2c\e \rho_{\ell_{i-1}})}
          \cap  B(v,2c\rho_{\ell_{i-1}})
        \right).
    \end{align*}

    To apply our non-concentration hypothesis, we wish to find a suitable point $x$ in the support of $\tilde{\mu}$ which is sufficiently close to $v$. 
    To this end, recall by Lemma~\ref{lem:non-conc of discretization}\eqref{item:NUANCD N(w)} that there is $x\in \Dcal_{r\ell_{i-1}}(v)$ such that $x$ is outside the exceptional set for $\tilde{\mu}$ and such that the set of good scales $\Ncal(x)$ contains the set $\Ncal(v)$.
    In particular, since $\ell_{i-1}\in \Ncal(v)$ by construction, we also have that $\ell_{i-1}\in \Ncal(x)$.
    Finally, since $\Dcal_{r\ell_{i-1}}(v)$ intersects the unit ball around identity, it is contained in $B(\id,2)$ whenever $r$ is larger than an absolute constant.
    In particular, $x\in B(\id,2)$, and hence, we obtain by our non-concentration hypothesis that
    \begin{align*}
        \tilde{\mu}\left( \widetilde{V}_v^{(2c\e \rho_{\ell_{i-1}})}
          \cap  B(v,2c\rho_{\ell_{i-1}}) \right)
          \leq \tilde{\mu}\left( \widetilde{V}_v^{(3c\e \rho_{\ell_{i-1}})}
          \cap  B(x,3 c\rho_{\ell_{i-1}}) \right)
          \leq \d(\g_2,\e) \tilde{\mu}(B(x, 3 c\rho_{\ell_{i-1}})).
    \end{align*}
    Using that $x\in\Dcal_{r\ell}(v)$ once more, we get that the right side of the above inequality is at most $ \d(\g_2,\e) \tilde{\mu}(B(v,4c \rho_{\ell_{i-1}}))$.
    This completes the proof of the lemma.
    \end{proof}
    
    Applying the above lemma $(m-2)$-times to the right side of~\eqref{eq:flat pre-induction}, we obtain
    \begin{align*}
        \tilde{\mu}_{kr}(F)\leq \d(\g_2,\e)
        \sum_{w\in \widetilde{\L}_{r\ell_{m-1}}(F)} 
        \tilde{\mu}\left( B(w,4c\rho_{\ell_{i}})\right)
        \leq (C_N \d(\g_2,\e))^{m-1} 
        \sum_{v\in \widetilde{\L}_{r\ell_{1}}(F)} 
        \tilde{\mu}\left( B(v,4c\rho_{\ell_{1}})\right).
    \end{align*}

    Since $F\subset B(\id,1)$, each of the balls $B(v,4c\rho_{\ell_1})$ is contained in $B(\id,2)$ for all $v\in \widetilde{\L}_{r\ell_1}(F)$ whenever $r$ is large enough. 
    Moreover, by Lemma~\ref{lem:intersection multiplicity on N}, those balls have uniformly bounded multiplicity.
    We thus obtain the bound $\tilde{\mu}_{kr}(F)\ll (C_N \d(\g_2,\e))^{m-1}$.

On the other hand, by Lemma~\ref{lem:build concentrated set}\eqref{item:lower bound on F}, we have the lower bound $\tilde{\mu}_{kr}(F)\geq 2^{-\sqrt{\e'} k-k-1}$.
Moreover, by taking $\g_2$ small enough, and taking $\e$ sufficiently small depending on $\g_2$, we can ensure that $C_N \d(\g_2,\e)$ is at most $1$.
Taking $k$ large enough so that $1/k\leq \g_3/2$, we arrive at the inequality
\begin{align*}
    2^{-\sqrt{\e'}-1-1/k} \leq C_0 \left(C_N \d(\g_2,\e) \right)^{\g_3-1/k} \leq C_0 \left(C_N \d(\g_2,\e)\right)^{\g_3/2} ,
\end{align*}
where $C_0\geq 1$ is the implicit constant in the previous inequality.
Recall that $\g_1 = \g/10$, $\g_2$ is to be chosen smaller than $\g_1/4$, and $\g_3=\g_1/4-\g_2$.
Hence, by choosing $\g_2$ first to be sufficiently small relative to $\g_1$, then choosing $\e$ very small, depending on $\g_2$, we can make the right side of the above inequality at most $1/2$.
This gives a contradiction since $\e'<1$.

\subsection{Proof of Theorem~\ref{thm:quant ellq improvement}}

    Let $\eta >0$ and $r\in \N$ be the parameters provided by Proposition~\ref{prop:discretized flattening} applied with $\g = \e/d$.
    Let $n\in \N$ be the smallest integer such that $(n-1)\eta\geq (d-\e)/2$.
    Note that by Young's inequality, for all $a,b,k\in \N$, we have that
    \begin{align*}
        \norm{\mu_k^{\ast a}\ast \mu_k^{\ast b}}_2 \leq \norm{\mu_k^{\ast a}}_1 \norm{\mu_k^{\ast b}}_2
        = \norm{\mu_k^{\ast b}}_2.
    \end{align*}

    We first observe that this inequality implies that it suffices to prove the first assertion of the theorem for multiples of $r$.
    Indeed, given any probability measure $\nu$, $k\in \N$, and $0\leq s<r$, we have
    \begin{align*}
        \sum_{P\in\Dcal_{kr+s}} \nu(P)^2
        = \sum_{Q\in\Dcal_{kr}} \sum_{P\in \Dcal_{kr+s}, P\subseteq Q} \nu(P)^2 
        \leq 2^{dr}\sum_{Q\in\Dcal_{kr}} \nu(Q)^2.
    \end{align*}
    
    Let $k_1$ be the parameter provided by Prop.~\ref{prop:discretized flattening} applied with $\lceil\log n\rceil$, where $n$ is as above.
    Let $k\geq k_1$ be given and suppose that 
    \begin{align}\label{eq:good after ell steps}
        \norm{\mu_{kr}^{\ast \ell}}_2^2 \leq 2^{- (d-\e)kr},
    \end{align}
    for some $\ell\in \N$ with $1\leq \ell\leq n$.
    By Lemma~\ref{lem:disc and conv commute}, we also have that $\norm{\mu^{\ast n}_{kr}}_2^2 \asymp \norm{\mu^{\ast (n-\ell)}_{kr} \ast \mu^{\ast \ell}_{kr}}_2^2 $. 
    Hence, since convolution with $\mu^{\ast (n-\ell)}_{kr}$ does not increase the $\ell^2$-norm, we get
    that $\norm{\mu_{kr}^{\ast n}}_2^2 \ll 2^{ - (d-\e)kr}$ as desired.
    Now, suppose that~\eqref{eq:good after ell steps} fails for all $1\leq \ell\leq n$.
    Then, applying Prop.~\ref{prop:discretized flattening} $(n-1)$-times by induction, and using Lemma~\ref{lem:disc and conv commute}, we obtain for some uniform constant $c\geq 1$
    \begin{align*}
        \norm{\mu_{kr}^{\ast n}}_2 \leq c \norm{\mu_{kr}\ast \mu_{kr}^{\ast (n-1)}}_2
        \leq c 2^{-\eta kr} \norm{\mu_{kr}^{\ast (n-1)}}_2 
        \leq \cdots
        \leq c^{n-1}2^{-\eta (n-1)kr} \norm{\mu_{kr}}_2
         \leq c^{n-1}2^{-\eta (n-1)kr} ,
    \end{align*}
    where the second inequality follows since $\norm{\mu_{kr}}_2\leq 1$.
    This proves the first assertion by our choice of $n$.
    The (short) deduction of the second assertion from the first can be found for instance in~\cite[Proof of Lemma 5.2]{MosqueraShmerkin}.

\subsection{Proof of Theorem~\ref{thm:ellq improvement}, Corollary~\ref{cor:flattening intro}, and Corollary~\ref{cor:flattening} from Theorem~\ref{thm:quant ellq improvement}}

Note that being uniformly affinely non-concentration immediately implies that $\mu$ is affinely non-concentrated at almost every (in fact at every) scale with an empty exceptional set.
Hence, the second assertion of Theorem~\ref{thm:quant ellq improvement} immediately implies that $\dim_\infty \mu^{\ast n}$ tends to $d$ as $n\to \infty$.
The same holds for $\dim_q\mu^{\ast n}$ due to the inequality $\dim_q \mu \geq \dim_\infty\mu$ for all $q>1$.
Finally, the first assertion of Theorem~\ref{thm:ellq improvement} follows readily from Proposition~\ref{prop:discretized flattening}; cf.~\cite[Proof of Theorem 1.1]{RossiShmerkin} for details of this deduction.

Similarly, Corollary~\ref{cor:flattening intro} is a special case of Corollary~\ref{cor:flattening}.
Hence, it remains to deduce 
Corollary~\ref{cor:flattening} from Theorem~\ref{thm:quant ellq improvement} via the well-known relationship between $L^2$-dimension and Fourier transform. 
Namely, by~\cite[Proof of Claim 2.8]{FengNguyenWang}, we have\footnote{The reference~\cite{FengNguyenWang} proves this fact in the case $d=1$, however the proof works equally well for $\R^d$ for any $d$.} 
\begin{align}\label{eq:Feng}
    \int_{\norm{\xi}\leq 1/r} |\hat{\nu}(\xi)|^2\;d\xi
    \ll_d r^{-2d} \int \nu(B(x,r))^2 \;dx.
\end{align}
for every $r>0$ and any Borel probability measure $\nu$ on $\R^d$.
Moreover, if $k\in \N$ is such that $2^{-(k+1)}<r\leq 2^{-k}$, then $B(x,r)$ can be covered by $O_d(1)$ elements of the partition $\Dcal_k$.
Hence,
\begin{align}\label{eq:Feng discretized}
    \int \nu(B(x,2^{-k}))^2 \;dx \ll_d 2^{-dk} \sum_{P\in \Dcal_k} \nu(P)^2 = 2^{-dk}\norm{\nu_k}_2^2. 
\end{align}

Now, let $\mu$ be a measure satisfying the hypotheses of Corollary~\ref{cor:flattening} and let $\e>0$ be arbitrary.
By Theorem~\ref{thm:quant ellq improvement}, 
there are natural numbers $n,k_1$, depending only on $\e$ and the non-concentration parameters of $\mu$, such that for all $k\geq k_1$,
\begin{equation*}
    \norm{\mu^{\ast n}_k}_2^2 \ll_{\e,d,n} 2^{-(d-\e/2)k}.
\end{equation*}

Given $T> 2^{k_1}$, let $r=1/T$ and $k\in \N$ be such that $2^{-(k+1)}<r\leq  2^{-k}$.
We can apply~\eqref{eq:Feng} and~\eqref{eq:Feng discretized} with $\nu=\mu_k^{\ast n}$ to get that
\begin{equation}\label{eq:moment bound on Fourier transform}
    \int_{\norm{\xi}\leq T} |\hat{\mu}(\xi)|^{2n}\;d\xi \ll_{\mu,\e} T^{2d} 2^{-dk -(d-\e/2)k} \ll T^{\e/2}.
\end{equation}

The conclusion of the corollary will now follow by Chebyshev's inequality and the fact that the Fourier transform is Lipschitz.
Indeed,
note that
\begin{equation*}
    \left|\hat{\mu}(\xi_1)-\hat{\mu}(\xi_2)\right| \ll_\mu \norm{\xi_1-\xi_2},
\end{equation*}
where the implicit constant depends only on the radius of the smallest ball around the origin containing the support of $\mu$.

In particular, given $\d>0$, if $|\hat{\mu}(\xi)|>T^{-\d}$ for some $\xi$, then $|\hat{\mu}|$ is at least $T^{-\d}/2$ on a ball of radius $T^{-2\d}$, when $T$ is large enough depending on $\d$ and $\mu$.
It follows that 
\begin{align*}
    &\# \set{v\in \Z^d: \norm{v}\leq T, \text{ and there exists } \xi \in v+[0,1)^d \text{ such that } |\hat{\mu}(\xi)|>T^{-\d}  } \times T^{-2\d d}
    \nonumber\\
    &\ll \left|\set{\xi\in \R^d: \norm{\xi}\leq 2T \text{ and } |\hat{\mu}(\xi)|>T^{-\d}/2 }\right|.
\end{align*}
Chebyshev's inequality applied to~\eqref{eq:moment bound on Fourier transform} implies that the right side of the above inequality is $O(T^{\e/2+2\d n})$.
Thus, the number of radius one balls needed to cover the set of frequencies $\xi$ of norm at most $ T$ such that $|\hat{\mu}(\xi)| > T^{-\d}$ is $O(T^{\e/2 +\d(2n+2d)})$.
Thus, taking $\d=\e/2(2n+2d)$, we obtain the assertion of Corollary~\ref{cor:flattening} as desired.

\subsection{Polynomial affine non-concentration}

In this section, we show that Theorem~\ref{thm:quant ellq improvement} implies quantitative non-concentration estimates near proper subspaces.

\begin{thm}\label{thm:flat implies friendly}
    Let $\tilde{\mu}$ and $\mu$ be as in Theorem~\ref{thm:quant ellq improvement}.
    Then, there exist $\k>0$ and $C$, depending on the non-concentration parameters of $\tilde{\mu}$, such that for all $\e>0$ and all proper affine subspaces $W<\R^d$, we have that $\mu(W^{(\e)}) \leq C\e^\k$.
\end{thm}

We first need the following useful observation which translates polynomial non-concentration for self-convolution into a similar estimate for the original measure.

\begin{lem}\label{lem:sqrt friendly}
    Let $\nu$ be a Borel probability measure, $\e,\a,C>0$ be arbitrary constants, and $W<\R^d$ be an affine hyperplane. Let $V=W+W$. Assume that $\nu^{\ast 2}(V^{(\e)})\leq C\e^\a$. Then, $\nu(W^{(\e/2)}) \leq C\e^{\a/2}$.
\end{lem}

\begin{proof}
    Note that the definition of convolution implies
    \begin{align*}
        \nu^{\ast 2}(V^{(\e)}) = \int \int \mathbbm{1}_{V^{(\e)}}(x+y)\;d\nu(x)\;d\nu(y)
        = \int \nu \left( V^{(\e)}-x \right) \;d\nu(x).
    \end{align*}
    Hence, by Chebyshev's inequality and our hypothesis on $\nu$, the set
    \begin{align*}
        B= \set{ x\in\R^d: \nu \left( V^{(\e)}-x \right) >\e^{\a/2} }
    \end{align*}
    has $\nu$ measure at most $C\e^{\a/2}$.
    Hence, the conclusion of the lemma follows if $W^{(\e/2)}$ is contained inside $B$.
    Otherwise, let $x\in W^{(\e/2)}\setminus B$ and observe that $W^{(\e/2)}$ is contained inside $V^{(\e)}-x$. However, the latter set has $\nu$ measure at most $\e^{\a/2}$ since $x\notin B$. Hence, the lemma follows in this case as well.
\end{proof}

We are now ready for the proof of Theorem~\ref{thm:flat implies friendly}.
\begin{proof}[Proof of Theorem~\ref{thm:flat implies friendly}]
  
      By Theorem~\ref{thm:quant ellq improvement}, we can find $n\in\N$ and $r_0>0$, depending only on the non-concentration parameters of $\tilde{\mu}$, such that
    \begin{align}\label{eq:measure of r-ball}
        \mu^{\ast 2^n}(B(x,r)) \ll_d 2^{2^ndm} r^{d-1/2},
    \end{align}
    for all $0<r\leq r_0$ and all $x\in \R^d$.
    Fix one such value of $n$ once and for all.
    Let $\nu=\mu^{\ast 2^{n}}$ and let $B\subset \R^d$ be a large ball containing the supports of $\mu^{\ast k}$ for all $0\leq k\leq 2^n$.
    In light of Lemma~\ref{lem:sqrt friendly}, it will suffice to find $C\geq 1$ and $\a>0$ so that $\nu(V^{(\e)})\leq C\e^\a$ for all proper affine hyperplanes $V$.
    
    Let $0<\e\leq 1$ and a proper affine hyperplane $V<\R^d$ be arbitrary. 
    Then, note that $V^{(\e)}\cap B$ can be covered by $O_{B,d}(\e^{-(d-1)})$ balls of radius $\e$ with multiplicity depending only on $d$.
    Then,~\eqref{eq:measure of r-ball} implies that $\nu(V^{(\e)})\leq C'\e^{1/2}$ for a suitable constant $C'=C'(m,n,d)\geq 1$.
    Since $V$ was arbitrary, Lemma~\ref{lem:sqrt friendly} and induction on $n$ show that
    $\mu(W^{(\e/2^n)})\leq C'\e^{\k}$ for $\k=2^{-n-1}$ for all proper hyperplanes $W$.
    Since $\e>0$ was arbitrary, this completes the proof by taking $C=C'2^{\k n}$.
\end{proof}


\section{Non-concentration of Patterson-Sullivan Measures}
\label{sec:friendly}

In this section, we verify the non-concentration hypothesis in Corollary~\ref{cor:flattening} for
the measures $\mxu$.
This enables us to apply these results to prove Proposition~\ref{prop:close pairs} and Theorem~\ref{thm:counting bad frequencies} which are the remaining pieces in the proof of Theorem~\ref{thm:Dolgopyat}.

\begin{definition}\label{def:aff subspace of nilpotent}
    Let $\Lcal'$ denote the collection of all proper linear subspaces of the Lie algebra $\mf{n}^+$ and denote by $\Lcal$ the set of all $N^+$-translates of images of elements of $\Lcal'$ under the exponential map.
Elements of $\Lcal$ will be referred to as \textbf{affine subspaces} of $N^+$.
For $\e>0$ and $L\in \Lcal$, let $L^{(\e)}$ be the $\e$-neighborhood of $L$.
\end{definition}

Recall that we fixed a choice of a Margulis function $V$ in Remark~\ref{remark:choice of V} and define
\begin{equation}
\label{eq:t(epsilon)}
    t(\e) := \sup_{x\in N_1^-\Omega} \sup_{L\in\Lcal} \frac{\mu_x^u(N_1^+\cap L^{(\e)})}{V(x)\mu_x^u(N_1^+)}.
\end{equation}
We also recall that $\G$ is a geometrically finite subgroup of $G=\mathrm{Isom}^+(\H_\K^d)$.

\begin{thm}
\label{thm:t(epsilon) goes to 0}
    Assume that $\G$ is Zariski-dense inside $G$.
    We have that $t(\e)\to 0$ as $\e\to 0$.
\end{thm}

As a consequence, we verify the hypotheses of Corollary~\ref{cor:flattening}.
\begin{cor}\label{cor:aff non-conc of projections}
    For every $x\in N_1^-\Omega$, the measure $\mu_x^u$ is affinely non-concentrated at almost every scale in the sense of Definition~\ref{def:aff non-conc} with parameters depending only on $V(x)$; cf.~\eqref{eq:parameters}.
\end{cor}

\subsection{Proof of Theorem~\ref{thm:t(epsilon) goes to 0}}

Our key tool is the following result which is a consequence of the ergodicity of the geodesic flow. The case of real hyperbolic spaces of this result was known earlier in~\cite{FlaminioSpatzier} by different methods.

\begin{prop}[{\hspace{0.1pt}\cite[Corollary 9.4]{EdwardsLeeOh-Torus}}]
\label{prop:null subvarieties}
    For all $x\in X$ and $L\in \Lcal$, $\mu_x^u(L)=0$.
\end{prop}

Theorem~\ref{thm:t(epsilon) goes to 0} follows from the above result and a compactness argument.
Indeed, fix an arbitrary $\eta>0$ and note that for all $x$ with $V(x)>1/\eta$, the inner supremum in the definition of $t(\e)$ is bounded above by $\eta$, for any choice of $\e>0$. We now show that $t(\e)<\eta$ for all sufficiently small $\e$ by restricting our attention to the bounded set of $x\in N_1^-\Omega$ where $V(x)\leq 1/\eta$.
Suppose not and let $x_n\in N_1^- \Omega$, $L_n\in\Lcal$, $\e_n>0$ be sequences such that $V(x_n) \leq 1/\eta$, $\e_n\to 0$, and 
\begin{align}
     \label{eq:sequence of hyperplane nbhds}
     \liminf_{n\to \infty}\frac{\mu_{x_n}^u(N_1^+\cap L_n^{(\e_n)})}{\mu_{x_n}^u(N_1^+)}  >0.
\end{align}
Passing to a subsequence if necessary, we may assume $x_n\to y \in N_1^-\Omega$ and $L_n$ converges to some $P\in\Lcal$ (in the Hausdorff topology on compact sets).
On the other hand, when $x_n$ is sufficiently close to $y$, we can change variables using~\eqref{eq:g_t equivariance},~\eqref{eq:N equivariance}, and~\eqref{eq:stable equivariance} to get
\begin{align*}
    \mu_{x_n}^u(N_1^+\cap L_n^{(\e_n)}) = \int f_n J_n \;d\mu_y^u,
\end{align*}
where $J_n$ is the Jacobian of the change of variables and $f_n$ is the indicator function of the image of $N_1^+\cap L_n^{(\e_n)}$ under this change of variables.
By Proposition~\ref{prop:null subvarieties}, since $L_n$ converges to $L$, $f_n$ converges to $0$ pointwise $\mu_y^u$-almost everywhere.
Additionally, $J_n$ converges to $1$ everywhere since $x_n$ converges to $y$. 
Finally, $\mu_{x_n}^u(N_1^+)$ remains bounded away from $0$ since $x_n$ remain within a bounded set for all $n$.
This gives a contradiction to~\eqref{eq:sequence of hyperplane nbhds} and concludes the proof.


\subsection{Non-concentration and proof of Corollary~\ref{cor:aff non-conc of projections}}

In this section, we show that the conditional measures $\mu^u_\bullet$ are affinely non-concentrated in the sense of Def.~\ref{def:aff non-conc}.
Our key tools are Theorem~\ref{thm:exp recurrence} and Theorem~\ref{thm:t(epsilon) goes to 0}.

Let $0<\th,\e< 1$ be arbitrary.
Let $H, r_0=O_{\b,\th}(1)$ be the constants provided by Theorem~\ref{thm:exp recurrence} when applied with $\e=\b\th/2$ and let $r\geq r_0$.
For $\ell\in\N$, let $t_\ell = r\ell\log 2$ and 
define 
 \begin{align*}
     \Ecal = \set{n\in N_1^+: \sum_{1\leq \ell \leq k} \chi_H(g_{t_\ell} n x) \geq \th k}.
 \end{align*}
Then, by Theorem~\ref{thm:exp recurrence}, we have that $\mu_x^u(\Ecal\cap N_1^+)\ll e^{-\b\th k/2} V(x)\mu_x^u(N_1^+)$.

It remains to show that our desired non-concentration holds outside of $\Ecal$.
For $n\in N_1^+$, define the set of scales $\Ncal(n)$ as follows:
\begin{align*}
    \Ncal(n) = \set{1\leq \ell\leq k:    V(g_{t_\ell} n x)\leq  H}.
\end{align*}
Let $n\in N_1^+\cap \mrm{supp}(\mu_x^u)\setminus \Ecal$.
By definition, we have $\# \Ncal(n) \geq (1-\th)k$.

Let $\ell\in \Ncal(n)$ and let $W<N^+$ be a proper affine subspace.
Recall the function $t(\e)$ defined in~\eqref{eq:t(epsilon)}.
Let $\rho\asymp 2^{-r\ell}$.
Let $z = g_{-\log \rho} nx$ and $W_n = \Ad(g_{-\log \rho})(Wn^{-1})$.
Then, changing variables and using the definition of $t(\e)$ along with the fact that $V(z)\ll H$, we obtain
\begin{align}\label{eq:non-conc of projections}
     \mu_x^u(W^{(\e \rho)} \cap N^+_{\rho}\cdot n )
     = \rho^\d \mu_z^u(W_n^{(\e)} \cap N_1^+)
     \ll H t(\e) \times \rho^\d \mu_z^u(N_1^+) = H t(\e) \times \mu_x^u(N_\rho^+\cdot n),
\end{align}
where the last equality follows by reversing the change of variables since $\rho^\d \mu_z^u(N_1^+) = \mu_x^u(N^+_{\rho}\cdot n )$.

Let $C_1\geq 1$ be the larger of the implicit constants in the bound on the measure of $\Ecal$ and in~\eqref{eq:non-conc of projections}.
These two estimates imply that $\mu_x^u$ satisfies Definition~\ref{def:aff non-conc} by taking
\begin{align}\label{eq:parameters}
    C(\th) := C_1 V(x)H, \qquad
    \vp(\e) := C_1t(\e), \qquad
    \l(\th) := (\b\th \log 2)/2.
\end{align}
That $\vp(\e)$ tends to $0$ as $\e\to 0$ follows by Theorem~\ref{thm:t(epsilon) goes to 0}.


\subsection{Counting close frequencies and proof of Proposition~\ref{prop:close pairs}}
\label{sec:close}

The idea of the proof is similar to that of~\cite[Lemma 6.2]{Liverani}, with the significant added difficulty being the non-concentration result for PS measures established in Theorem~\ref{thm:flat implies friendly}.
We note however that the case of real hyperbolic manifolds is much simpler in that it does not require Theorem~\ref{thm:flat implies friendly} and instead uses only the doubling result in Proposition~\ref{prop:doubling}. 

Recall our definition of the transverse intersection points $x_{\rho,\ell}$ in~\eqref{eq:centers} and of $N_1^+(j)$ in the paragraph above~\eqref{eq:N_1^+(j)}.
For each $\ell\in I_{\rho,j}$, fix some $u_\ell \in N^+_1(j)\subseteq N_3^+$ such that
\begin{equation}\label{eq:choice of p_ell}
    x_{\rho,\ell} = g^\g p^+_\ell \cdot x = n_{\rho,\ell}^- \cdot y_\rho, \qquad
    p^+_{\ell}:= m_{\rho,\ell} g_{t_{\rho,\ell}}   u_\ell.
\end{equation}
Here, we are using that the groups $A=\set{g_t:t\in\R}$ and $M$ commute.
Denote by $P^+$ the parabolic subgroup $N^+AM$ of $G$.
Since $M$ is compact, $|t_{\rho,\ell}|<1$, and $N^+_1(j)$ is contained in $N_3^+$, there is a uniform constant $C>0$ such that
\begin{equation}\label{eq:bounded set of elements}
    \set{p^+_\ell :\ell\in I_{\rho,j}}\subset P^+_C,
\end{equation}
where $P^+_C$ denotes the ball of radius $C$ around identity in $P^+$.

Fix some $\ell_0\in I_{\rho,j}$ and denote by $C_{\rho,j}(\ell_0)$ the set of indices $\ell\in I_{\rho,j}$ such that $(\ell_0,\ell)\in C_{\rho,j}$. 
To simplify notation, we set
\begin{align*}
     \epsilon:=b^{-1/10}, \qquad t_\star:= \g(w+jT_0).
\end{align*}
Let $Z=\exp(\ntwominus) \subset N^-$.
In particular\footnote{This is the reason Theorem~\ref{thm:flat implies friendly} is not needed in this case.}, $Z=\set{\id}$ is the trivial group in the real hyperbolic case.
Recalling the definition of the Cygan metric in~\eqref{eq:Carnot}, the definition of $C_{\rho,j}$ implies that
\begin{equation*}
    d_{N^-} (n^-_{\rho,\ell}(n^-_{\rho,\ell_0})^{-1},Z) 
    \leq 
    b^{-1/10} .
\end{equation*}
Denote by $Z^{(\epsilon)}$ the $\epsilon$-neighborhood of $Z$ inside $N^-$. Let
\begin{align*}
    \tilde{u}_\ell^- = n^-_{\rho,\ell}(n^-_{\rho,\ell_0})^{-1}
    \in \zeps\cap N_{\iota_j}^-,
\end{align*}
where we recall that the points $n_{\rho,\ell}^-$ belong to $N^-_{\iota_j/10}$ by definition of our flow boxes $B_\rho$; cf.~paragraph preceding~\eqref{eq:norm of rho at most b}.
Note that
\begin{align*}
    g^\g p^+_\ell \cdot x  = \tilde{u}_\ell^{-}\cdot  g^\g p^+_{\ell_0} \cdot x, \qquad
    \forall \ell\in C_{\rho,j}(\ell_0).
\end{align*}
In particular, for $u^-_\ell = \Ad(g^\g)^{-1}(\tilde{u}_\ell^-)$, since $g^\g=g_{t_\star}$ (cf.~\eqref{eq:ggamma}), we have that
\begin{equation}\label{eq:strong stable ball}
   p^+_\ell x = u_\ell^-\cdot p^+_{\ell_0}x \in  (Z^{(e^{t_\star}\epsilon)} \cap N^-_{e^{t_\star}\iota_j}) \cdot p^+_{\ell_0}x, 
   \qquad \forall \ell\in C_{\rho,j}(\ell_0).
\end{equation}
Our counting estimate will follow by estimating from below the separation between the points $p_\ell^+x$, combined with a measure estimate on the sets $(Z^{(e^{t_\star}\epsilon)} \cap N^-_{e^{t_\star}\iota_j})\cdot p^+_{\ell_0}x$.

To this end, recall the sublevel set $K_j$ and the injectivity radius $\iota_j$ in~\eqref{eq:K_j and iota_j and iota_b}. 
Recall also by~\eqref{eq:V(x)} that $x$ belongs to $K_j$. 
It follows that the injectivity radius at every point of the weak unstable ball $P^+_C\cdot x$ is $\gg \iota_j$. 
This implies that there is a radius $r_j$ with $\iota_j\ll r_j\leq \iota_j$ such that for every $\ell\in C_{\rho,j}(\ell_0)$, the map $n^- \mapsto n^- \cdot p^+_\ell x$
is an embedding of $ N^-_{r_j}$ into $X$.

Let $\set{B_m}$ be a cover of $P^+_C$ by $O(\iota_j^{-\dim P^+})$ balls of radius $r_j$. Then, similar injectivity radius considerations imply that, for every $m$, the disks 
\begin{align*}
    \set{N^-_{r_j}\cdot p^+_{\ell} x: \ell \in C_{\rho,j}(\ell_0), p^+_\ell\in B_m}
\end{align*}
are disjoint.
Indeed, otherwise, we can find $n\in N^-_{2r_j}$ and $\ell_1,\ell_2\in C_{\rho,j}$ with $p^+_{\ell_1}, p^+_{\ell_2}\in B_m$ such that $n p^+_{\ell_1}x= p^+_{\ell_2}x$.  
By choosing $r_j$ sufficiently smaller than $\iota_j$, this gives a contradiction to the fact that the injectivity radius at $x$ is $\gg \iota_j$ since $(p^+_{\ell_2})^{-1}n p^+_{\ell_1} = (p^+_{\ell_2})^{-1}n p^+_{\ell_2}\cdot (p^+_{\ell_2})^{-1}p^+_{\ell_1} $ is at distance $O(r_j)$ from identity.
The above disjointness, together with~\eqref{eq:strong stable ball}, imply that the disks $\set{N^-_{r_j}\cdot u^-_\ell: \ell\in C_{\rho,j}(\ell_0), p^+_\ell\in B_m}$ form a disjoint collection of disks inside $Z^{(e^{t_\star}\epsilon+\iota_j)} \cap N^-_{(e^{t_\star}+1)\iota_j}$.
In particular, we get that
\begin{align}\label{eq:count close frequencies by measure}
    \# \set{\ell\in C_{\rho,j}(\ell_0): p^+_{\ell}\in B_m}
    \leq 
    \frac{ \mu^s_{p^+_{\ell_0}x} \left( Z^{(e^{t_\star}\epsilon+\iota_j)} \cap N^-_{(e^{t_\star}+1)\iota_j} \right)
    }
    {\min_{\ell \in C_{\rho,j}(\ell_0)}  \mu_{p^+_{\ell_0}x}^s(N^-_{r_j}\cdot u^-_\ell )}
    ,
\end{align}
where $\mu^s_\bullet$ are the conditional measures along $N^-$-orbits defined similarly to~\eqref{eq:unstable conditionals}.

To obtain good bounds on the ratio in~\eqref{eq:count close frequencies by measure} for a given $\ell$, it will be important to change the basepoint $p^+_{\ell_0}x$ to another point of the form $g_{s}p^+_\ell x$ with uniformly bounded height.
Fix some arbitrary $\ell\in C_{\rho,j}(\ell_0)$ and recall~\eqref{eq:choice of p_ell} and~\eqref{eq:strong stable ball}.
Let $s_{\rho,\ell}\in [ t_\star,(1+2\a)t_\star]$ be the return time defined in~\eqref{eq:return time of x_rho,ell} and set
\begin{align*}
    y_\ell = g_{s_{\rho,\ell}} p_\ell^+ x.
\end{align*}
Note that our choice of $u^-_\ell$ implies that
\begin{align*}
    Z^{(e^{t_\star}\epsilon+\iota_j)} \cap N^-_{(e^{t_\star}+1)\iota_j} 
    \subseteq \left(Z^{(2e^{t_\star}\epsilon+\iota_j)} \cap N^-_{2(e^{t_\star}+1)\iota_j} \right) 
    \cdot u_\ell^-. 
\end{align*}
In particular, we can use the set on the right side to estimate the numerator of~\eqref{eq:count close frequencies by measure}.
Let
\begin{align*}
    Q := Z^{(2e^{t_\star}\epsilon+\iota_j)} \cap N^-_{2(e^{t_\star}+1)\iota_j},
    \qquad 
    Q^\prime := \Ad(g_{s_{\rho,\ell}})(Q).
\end{align*}
 Then, changing variables using~\eqref{eq:N equivariance} and~\eqref{eq:g_t equivariance}, we have
 \begin{align*}
     \frac{\mu^s_{p^+_{\ell_0} x}(Q \cdot u^-_\ell)}{\mu^s_{p^+_{\ell_0} x}(N^-_{r_j} \cdot u^-_\ell) }
     = \frac{\mu^s_{p^+_{\ell} x}(Q )}{\mu^s_{p^+_{\ell} x}(N^-_{r_j}) }
     = \frac{\mu^s_{y_\ell}(Q^\prime)}{ \mu^s_{y_\ell}(N^-_{e^{-s_{\rho,\ell}} r_j}) }.
 \end{align*}
Moreover, by the global measure formula, Theorem~\ref{thm:global measure formula}, since $V(y_\ell)\ll_{T_0} 1$, we have that
 \begin{align}\label{eq:lower bound on denominator in counting}
     \mu^s_{y_\ell}(N^-_{e^{-s_{\rho,\ell}} r_j}) \gg_{T_0} e^{-\d s_{\rho,\ell}} r_j^\d \gg e^{-\d (1+2\a)t_\star} \iota_j^\d.
 \end{align}
 Here, we used~\cite[Theorem 2.2]{Corlette} to relate strong stable disks of the form $N^-_{r}\cdot y_\ell$ to their shadows on the boundary; cf.~\eqref{eq:approximate shadow} for a precise formulation.

\begin{lem}\label{lem:measure of Q'}
    We have the bound $\mu^s_{y_\ell}(Q') \ll b^{-\k/10} + e^{-\k t_\star}$, where $\k>0$ is a uniform constant provided by Theorem~\ref{thm:flat implies friendly}.
\end{lem}
\begin{proof}
    We wish to apply Theorem~\ref{thm:flat implies friendly} to the measure $\mu^s_{y_\ell}$.
    This result concerns decay of measures of the intersection of subspace neighborhoods with $N_1^-$.
    First, we show that $Q'\subseteq Z^{(\varrho)}\cap N_1^-$ for $\varrho = 2\epsilon + \iota_je^{-t_\star}  $.
    Indeed, let $r_1 = (2e^{t_\star}\epsilon+\iota_j)e^{-s_{\rho,\ell}} $ and $r_2=2(e^{t_\star}+1)\iota_j e^{-s_{\rho,\ell}}$.
    Then, $Q' = Z^{(r_1)}\cap N^-_{r_2}$.
    Since $s_{\rho,\ell}\geq t_\star$ and $\iota_j\leq 1/10$ (cf.~\eqref{eq:K_j and iota_j and iota_b}), we have $r_2\leq 1$. 
    Similarly, $r_1\leq \varrho$.

    Next, we note that the non-concentration hypothesis of Theorem~\ref{thm:flat implies friendly} is verified in Corollary~\ref{cor:aff non-conc of projections}.
    Moreover, the corollary also shows that the non-concentration parameters can be chosen uniform over all $y_\ell$ in view of the fact that $V(y_\ell) \ll 1$.
    Let $\mu$ be the projection of $\mu^s_{y_\ell}\left|_{N_1^-}\right.$ to the abelianization $N^-/[N^-,N^-]\cong \mf{n}_\a^-$, normalized to be a probability measure.
    Then, Theorem~\ref{thm:flat implies friendly} provides constants $C_1,\k>0$, independent of $\ell$, so that $\mu(W^{(\e)})\leq C_1 \e^\k$ for all $\e>0$ and all proper affine subspaces $W$, where such subspaces are defined in Def.~\ref{def:aff subspace of nilpotent}.
    The lemma now follows since $Q'$ is contained in the $\varrho$-neighborhood of a translate of the preimage of $0$ under this projection.
\end{proof}

Recall that the cover $\set{B_m}$ has cardinality at most $O(\iota_j^{-\dim P^+})$.
Also, recall by~\eqref{eq:bound iota_j} that $\iota_j^{-1} \ll_{T_0} e^{4\a t_\star}$.
Hence,~\eqref{eq:count close frequencies by measure},~\eqref{eq:lower bound on denominator in counting}, and Lemma~\ref{lem:measure of Q'} imply that
\begin{align*}
    \# C_{\rho,j}(\ell_0) \ll_{T_0} e^{O(\a t_\star)} (\epsilon^{\k} + e^{-\k t_\star }) e^{\d t_\star}.
\end{align*}
Finally, note that~\eqref{eq:j small} provides the bound $e^{\a t_\star}\ll_{T_0} b^{2\a/a}$.
Hence, if $\a$ is small enough, the above bounds imply that $\# C_{\rho,j}(\ell_0)$ is $\ll_{T_0}  (\epsilon^{\k_0} + e^{- \k_0 t_\star}) e^{\d t_\star} $, for $\k_0=\k/2$.
This concludes the proof.

\subsection{Flattening and proof of Theorem~\ref{thm:counting bad frequencies}}
\label{sec:count bad freqs}

We wish to apply Corollary~\ref{cor:flattening}.
Recall that $\nu_i$ has total mass $\murhoi(N_1^+)$ and let $\mu= \nu_i/\murhoi(N_1^+)$.
In particular, $\mu$ is a probability measure supported on the unit ball in $\nonepls$.

We fix identifications $\nonepls \cong \K^{p}\cong \noneminus$ for some $p\in\N$; cf.~Section~\ref{sec:Carnot}. 
Note further that the restriction of the metric in~\eqref{eq:Carnot} to $\nonepls$ is Euclidean.
In particular, we will fix a linear isomorphism of $\nonepls$ and $\noneminus$ with $\R^d$, where $d=p \dim \K$.

By Corollary~\ref{cor:aff non-conc of projections}, the measure $\mu^u_{y_\rho^i}$ is affinely non-concentrated at almost all scales in the sense of Def.~\ref{def:aff non-conc}.
Hence, Corollary~\ref{cor:flattening} provides $\l>0$ such that, for $T=b^{4/10}$, the set
\begin{align*}
    \mf{B}(\l):=\set{w\in \R^d: \norm{w}\leq T \text{ and } |\hat{\mu}( w)| \geq T^{-\l }}
\end{align*}
can be covered by $O_\e(T^\e)$ balls of radius $1$.
The result will follow once we estimate the spacing between the functionals $\langle w^i_{k,\ell},\cdot \rangle$.

To simplify notation, let $ w_\ell := \langle w^i_{k,\ell},\cdot \rangle$.
By~\eqref{eq:size of separated frequencies}, when $b$ is large enough, we have that $b\norm{w_\ell} \leq T$.
In particular, we can view the set $B(i,k,\l)$ as a subset of $\mf{B}(\l)$ above using the map $\ell \mapsto bw_\ell$.
By Lemma~\ref{lem:temp function formula}, the definition of $w^i_{k,\ell}$ in~\eqref{eq:frequency vectors def}, and~\eqref{eq:size of t_i}, we have that
\begin{align*}
    \norm{w_{\ell_1}-w_{\ell_2}} \gg b^{2/10}\norm{u_{\ell_2}-u_{\ell_1}}.
\end{align*}
In particular, by Proposition~\ref{prop:close pairs}, any ball of radius $1$ in $\R^d$ contains at most 
\[ 
O_{T_0} \left( (b^{-\k/10} + e^{-\k\g(w+jT_0)}) e^{\d \g(w+jT_0)} \right) 
\]
of the vectors $w_\ell$.
This completes the proof of Theorem~\ref{thm:counting bad frequencies}.


 \section{Proof of Theorems~\ref{thm:intro mixing} and~\ref{thm:intro strip}}
\label{sec:Butterley}

The goal of this section is to complete the proofs of Theorems~\ref{thm:intro mixing} and~\ref{thm:intro strip}.
The key ingredients are Theorems~\ref{thm:resolvent spectrum2} and~\ref{thm:Dolgopyat}.
The deduction is through a form of the Paley-Wiener theorems adapted for this purpose obtained in~\cite{Butterley,Butterley-erratum}.

\subsection{Paley-Wiener theorems}

Let $(\Bcal, \norm{\cdot}_\Bcal)$ be a Banach space equipped with a weaker norm $\norm{\cdot}_\Acal$.
Let $\Lcal_t$ be a bounded one-parameter semigroup of operators on $\Bcal$ in the norm $\norm{\cdot}_\Bcal$.
Denote by $\mf{X}$ the infinitesimal generator of $\Lcal_t$ and let $R(z)$, $\Re(z)>0$, be its resolvent.

\begin{thm}[{\hspace{0.1pt}\cite[Theorem 1]{Butterley,Butterley-erratum}}]
    \label{thm:Butterley}
Assume that $\Lcal_t$ is strongly continuous\footnote{Cf.~\cite{Butterley} for a result that does not require strong continuity.}, and
\begin{enumerate}
    \item
    \label{item:Assumption 1}
    $\Lcal_t$ is weakly-Lipschitz, i.e., for all $t\geq 0$ and $f\in \Bcal$, $
        \norm{\Lcal_t f - f}_\Acal \ll t \norm{f}_\Bcal$.
    \item
    \label{item:Assumption 2}
    there exists $\l>0$ such that $\rho_{ess}(R(z))\leq 1/(\Re(z)+\l)$ for all $z\in \C$ with $\Re(z)>0$, where $\rho_{ess}$ denotes the essential spectral radius.

    \item
    \label{item:Assumption 3}
    there exist positive constants $C, \a, \b$ and $0<\g <\log(1+\l/\a)$ such that, for all $z\in \C$ with $\Re(z)=\a$ and $|\Im(z)|\geq \b$, we have $
        \norm{R(z)^{q}}_\Bcal \leq  C/(\a +\l)^{q}$,
    where $q=\lceil \g \log |\Im(z)|\rceil$.
    Here, $\l$ is the constant in Assumption~\eqref{item:Assumption 2}.

\end{enumerate}
Then, there exists an operator valued function $t\mapsto \Pcal_t$ taking values in the space of bounded operators on $\Bcal$, and for $1\leq j\leq N\in\N$, there exist $z_j\in \C$ with $-\l <\Re(z_j)\leq \b$, a 
finite rank projector $\Pi_j$, and a nilpotent operator $\Ncal_j$, so that the following hold:
\begin{enumerate}
    \item For all $1\leq j,k\leq N$ and $t\geq 0$, we have 
    \begin{equation*}
        \Pi_j \Pi_k = \d_{jk} \Pi_j, \quad 
        \Pi_j \Pcal_t = \Pcal_t \Pi_j = 0, \quad 
        \Pi_j \Ncal_j = \Ncal_j\Pi_j = \Ncal_j.
    \end{equation*}

    \item For all $t\geq 0$,
    $   \Lcal_t = \Pcal_t + \sum_{j=1}^N e^{t z_j} \exp(t\Ncal_j) \Pi_j$.

    \item For all $f$ in the domain of $\mf{X}$, $t\geq 0$ and $0<\ell <\l$,
        $\norm{\Pcal_t f}_\Acal \ll_\ell e^{-\ell t} \norm{\mf{X} f}_\Bcal $.
\end{enumerate}
\end{thm}

\subsection{Verification of the hypotheses of Theorem~\ref{thm:Butterley}}
Recall the Banach space $\Bcal_\star$ defined below~\eqref{eq:norm star} and the weak norm $\norm{\cdot}'_1$ defined in~\eqref{eq: norms}.
The link between the norms we introduced and decay of correlations is furnished in Lemma~\ref{lem:norm controls correlations}. 
In particular, this lemma implies that decay of correlations (for mean $0$ functions) would follow at once if we verify that $\norm{\Lcal_t f}'_1$ decays in $t$ with a suitable rate.
Theorem~\ref{thm:Butterley} shows that such decay follows from suitable spectral bounds on the resolvent.
Hence, it remains to verify the hypotheses of Theorem~\ref{thm:Butterley}.
We take
\begin{align*}
    \norm{\cdot}_\Acal = \norm{\cdot}'_1,
    \qquad
    \norm{\cdot}_\Bcal = \norm{\cdot}_1^\star
\end{align*}
 in the notation of Theorem~\ref{thm:Butterley}.
Strong continuity of $\Lcal_t$ is provided by Corollary~\ref{cor:strong continuity}, while
Theorem~\ref{thm:resolvent spectrum2} verifies Assumption~\eqref{item:Assumption 2} of Theorem~\ref{thm:Butterley}\footnote{Corollary~\ref{cor:strong continuity} and Theorem~\ref{thm:resolvent spectrum2} are obtained for the norms $\norm{\cdot}_k$, $k\geq 1$, however the proof extends readily to the norm $\norm{\cdot}^\star_1$ taking $\norm{\cdot}'_1$ as its associated norm.}.
The following lemma verifies Assumption~\eqref{item:Assumption 1}.

\begin{lem}\label{lem:weak Lipschitz}
For all $t\geq 0$, $
    \norm{\Lcal_t f-f}_1' \ll t \norm{f}^\star_1$.
\end{lem}

\begin{proof}
Recall that the norm $\norm{\cdot}'_1$ only involves the coefficient $e'_{1,0}$; cf.~\eqref{eq: norms}.
Let $x\in N_1^-\Omega$ and $t\geq 0$.
Then, given any test function $\phi$ for $e'_{1,0}$, we have that 
\begin{align*}
    \int_{N_1^+}\phi(n)(f(g_{t}nx)-f(nx))\;d\mu_x^u
    = \int_0^t \int_{N_1^+}\phi(n) L_\w f(g_{r}nx)\;d\mu_x^udr,
\end{align*}
where $L_\w$ denotes the derivative with respect to the vector field generating the geodesic flow.
Hence, Lemma~\ref{lem: equicts} implies that
\begin{align*}
    \left|\int_{N_1^+}\phi(n)(f(g_{t}nx)-f(nx))\;d\mu_x^u\right|
    \leq V(x)\mu_x^u(N_1^+) \int_0^t e^\star_{1,1}(\Lcal_{r}f)\;dr
    \ll tV(x)\mu_x^u(N_1^+) e^\star_{1,1}(f),
\end{align*}
where $e^\star_{1,1}$ is the coefficient defined above~\eqref{eq:norm star}.
This completes the proof.
\end{proof}

Finally, the following corollary verifies Assumption 3 of Theorem~\ref{thm:Butterley}.

\begin{cor} \label{cor:combine Dolgopyat bound}
Let the notation be as in Theorem~\ref{thm:Dolgopyat}.
Then, there exist constants $c_\star, \l_\star>0$, 
such that the following holds.
For all $z=a_\star+ib\in \C$ and for $q=\lceil c_\star \log|b|\rceil$, we have the following bound on the operator norm of $R(z)$:

\begin{align*}
    \norm{R(z)^{q}}_{1}^\star \leq  \frac{1}{(a_\star +\l_\star)^{q}},
\end{align*}
whenever $|b|\geq b_\G$, where $b_\G\geq 1$ is a constant depending on $\G$.
\end{cor}

\begin{proof} 
First, we verify the corollary for the norm $\norm{\cdot}^\star_{1,B}$.
Let $e^\star_{1,1,b}$ be the scaled seminorm $e^\star_{1,1}/|b|^{1+\varkappa}$.
Note that the arguments of Lemmas~\ref{lem:all stable} and~\ref{lem: flow by parts} imply that for $z=a_\star+ib$ with $|b|\geq a_\star$, we have
\begin{align*}
    e^\star_{1,1,b}(R(z)^mf) \leq C_\G
    \frac{\norm{f}^\star_{1,B}(a_\star+|z|)}{a_\star^m b^{1+\varkappa}}
    \leq \frac{3C_\G\norm{f}^\star_{1,B}}{a_\star^m |b|^{\varkappa}},
\end{align*}
for some constant $C_\G\geq 1$ depending only on $\G$, where we used the fact that $a_\star+|z|\leq 3|b|$.

Moreover, if $m=\lceil \log |b| \rceil\geq 3/2$, we have that $|b|^\varkappa\geq e^{\varkappa m/2}\geq (1+\varkappa/2)^m$ and hence $a_\star^m |b|^\varkappa $ is at least $(a_\star+\varkappa/2)^{m}$.
It follows that, for all $f\in \Bcal_\star$, we have 
\begin{align*}
    e^\star_{1,1,b}(R(z)^mf) 
    \leq \frac{3C_\G\norm{f}^\star_{1,B}}{(a_\star+\varkappa/2)^m}.
\end{align*}
This estimate, combined with the estimate in Theorem~\ref{thm:Dolgopyat} implies that whenever $|b|\geq b_\star $,
\begin{align*}
    \norm{R(z)^m}^\star_{1,B} \ll_\G  (a_\star+\s_1)^{-m},
\end{align*}
where $\s_1>0$ is the minimum of $\s_\star$ and $\varkappa/2$.
In particular, if $|b|$ is large enough, depending on $\G$, we can absorb the implied constant in the estimate above to obtain
\begin{align*}
    \norm{R(z)^m}^\star_{1,B}\leq (a_\star+\s_1/2)^{-m}.
\end{align*}

Let $p\in \N$ be a large integer to be chosen shortly.
To obtain the claimed estimate for the norm $\norm{\cdot}^\star_1$, note that for any $f$ in the Banach space $\Bcal_\star$, since $\norm{\cdot}^\star_{1,B}\leq \norm{\cdot}_1^\star\leq B\norm{\cdot}^\star_B=|b|^{1+\varkappa} \norm{\cdot}^\star_{1,B}$, iterating the above estimate yields
\begin{align*}
    \norm{R(z)^{2pm}f}^\star_1
    \leq  B\norm{R(z)^{2pm}f}^\star_B
    \leq \frac{B\norm{R(z)^{pm}f}^\star_{1,B}}{(a_\star+\s_1/2)^{pm}}
    \leq \frac{B\norm{f}^\star_1}{(a_\star+\s_1/2)^{2pm}}.
\end{align*}
Since $m=\lceil \log |b|\rceil$, choosing $p$ large enough, depending only on $a_\star$ and $\s_1$, we can ensure that $B/(a_\star+\s_1/2)^{pm}\leq 1/a_\star^{pm}$.
In particular, taking $\l_\star$ to be the positive solution of the quadratic polynomial $x\mapsto x^2+2a_\star x-a_\star\s_1/2$, we obtain the desired estimate with $c_\star=2p $.
\end{proof}

\subsection{Proofs of the main theorems}

Let $\mf{X}$ denote the generator of the semigroup $\Lcal_t$ acting on $\Bcal_\star$ (which exists by Corollary~\ref{cor:strong continuity}).
In light of the above results and Theorem~\ref{thm:Butterley}, we obtain the following decomposition of the transfer operator $\Lcal_t$:
$
    \Lcal_t = \Pcal_t + \sum_{i=1}^N e^{t\l_i} e^{t\Ncal_i} \Pi_i  $,
where $\Pcal_t, \Ncal_i, \l_i$ and $\Pi_i$ are as in Theorem~\ref{thm:Butterley}.
Moreover, for a suitable $\s>0$ depending only on $\l_\star$ in Corollary~\ref{cor:combine Dolgopyat bound} and on $\s_0$ given by Theorem~\ref{thm:resolvent spectrum2},  we have that
\begin{align*}
    \norm{\Pcal_tf}'_1 \ll e^{-\s t} \norm{\mf{X}f}^\star_1,
\end{align*}
for all $t\geq 0$ and $f\in \Bstar$.
Finally, it follows by Lemma~\ref{lem:spectrum on imaginary axis}\footnote{Lemma~\ref{lem:spectrum on imaginary axis} is obtained for a slightly different norm but the proof is identical.} that the only eigenvalue $\l_i$ lying on the imaginary axis is $0$ and that its associated nilpotent operator $\Ncal_i$ vanishes.
This concludes the proof.

\bibliography{bibliography}{} 
\bibliographystyle{amsalpha}

\end{document}